\documentclass[a4paper,11pt,openright]{book} % one-sided
\usepackage[a4paper, 
		scale=1.0,  							% Uses whole paper size (instead of 0.7)
		layoutoffset=0pt,					% No offset on the margins of the page (thats why we have margins anyway...)
		ignoremp, 								%Ignore marginal notes
		includehead,						%Includes head of page into total body
		marginparsep=0pt,			 %Distance between text and marginal notes
		bottom=4cm, top=2cm,     	%UoE margins  
		left=4cm, right=2.5cm,			%UoE margins
		verbose=true,						%Displays parameters in terminal, useful to find errors.
		bindingoffset=0pt					%We already have a large inner margin...
		]{geometry}	 %  Please note that the margins apply to thecreated .pdf and it is thus crucial that you don't shrink or auto-rotate and center your document when printing! 
\usepackage{enumerate}
\usepackage[utf8]{inputenc}					 %recommended 
\usepackage[T1]{fontenc}					 %recommended, together with inputenc. Helps fixing umlaut errors. 
\usepackage{lmodern}
\usepackage [english]{babel} 				% man­ages cul­tur­ally-de­ter­mined ty­po­graph­i­cal (and other) rules, and hy­phen­ation pat­terns for a wide range of lan­guages
\usepackage{titlesec, blindtext} 							% allows for easy change of section styles
\usepackage{fancyhdr} 							% allows fancy headers
\usepackage{setspace}							 %allows local control over line spacing, e.g. single for contents, double for main body text
\usepackage{epigraph} 							%for quotations at start or end of chapter
\usepackage{array} 									%various table improvements eg. custom column definitions
\usepackage{fixltx2e} 								%corrects for errors between different latex versions
\usepackage[section]{placeins}				 % new page for each section
\usepackage[toc, page, header]{appendix} 	% used for appendices
\usepackage{float} 										% for floats
\usepackage{color}
\usepackage{graphicx} 								%uses graphics in .pdf format
\graphicspath{{./figures/}{./figures/chapter1/}{./figures/chapter2/}{./figures/chapter3/}{./figures/chapter4/}} 
\usepackage{rotating} 									%allows the possiblity to rotate graphics, tables, etc.
\usepackage{caption} 									%allows left justification of figure and table captions & much more
\usepackage{subcaption}									% subfigures!
\usepackage{wrapfig} 									% allows to wrap text around figures
\usepackage{tikz}
\usetikzlibrary{decorations.pathreplacing,calc,positioning,arrows,shapes}
\providecommand{\nodepoint}[3]{\node[inner sep =0mm, outer sep =0mm] (#1) at (#2,#3) {\huge\bf$\cdot$}}
\usepackage{algorithm}
\usepackage[noend]{algpseudocode}
\usepackage{xcolor,colortbl}							 %allows coloured cells in tables

\usepackage{textcomp} 									%symbols
\usepackage{amsfonts,amssymb,amsmath,amsthm}				 %maths symbols 
\usepackage{doi} 										 %hyperlinks doi numbers to dx.doi.org

\usepackage{pdfpages} 					% To load pdfs into your document
\DeclareUnicodeCharacter{00A0}{ }     % To get rid of this error "! Package inputenc Error: Unicode char \u8: not set up for use with LaTeX." Which often comes with references
\DeclareUnicodeCharacter{2009}{ } % % Define thinspace as space  
\DeclareUnicodeCharacter{2212}{-} 
\let\origdoublepage\cleardoublepage
\newcommand{\clearemptydoublepage}{%
  \clearpage
  {\pagestyle{empty}\origdoublepage}%
}
\let\cleardoublepage\clearemptydoublepage

\reversemarginpar	% to get notes on left side instead of right
\newcommand{\R}{\mathbb{R}}
\newcommand{\N}{\mathbb{N}}
\newcommand{\K}{\mathcal{K}}
\newcommand{\eqdef}{\overset{\text{def}}{=}} 
\newcommand{\dd}{d} 
 \newcommand{\COV}[1]{\mathbf{Cov} \left[ #1 \right]}
\newcommand{\E}[1]{\mathbf{E}\left[#1\right] } % D^2 f % \mathcal{B}
\newcommand{\norm}[1]{\lVert#1\rVert}
\newcommand{\dotprod}[1]{\left< #1\right>}
\newcommand{\Tr}[1]{\mathbf{Tr}\left( #1\right)}
\providecommand{\Null}[1]{\mathbf{Null}\left( #1\right)}
\providecommand{\Rank}[1]{\mathbf{Rank}\left( #1\right)}
\providecommand{\myRange}[1]{\mathbf{Range}\left( #1\right)}
\newtheorem{lemma}{Lemma}
\newtheorem{proposition}[lemma]{Proposition}
\newtheorem{corollary}[lemma]{Corollary}
\newtheorem{definition}[lemma]{Definition}
\newtheorem{assumption}[lemma]{Assumption}
\newtheorem{theorem}[lemma]{Theorem}
\newtheorem{remark}[lemma]{Remark}
\fancyheadoffset{0.0in} 									% margin size at header
\renewcommand{\chaptermark}[1]{\markboth{\chaptername \ \thechapter. \ #1}{}} 		% Chapter name on right pages
\usepackage[firstinits=true,backend=bibtex, refsegment=chapter, doi=false,isbn=false,url=false,eprint=false,maxbibnames=10]{biblatex}
\renewbibmacro{in:}{\ifentrytype{article}{}{\printtext{\bibstring{in}\intitlepunct}}} % don't print "in" for articles
\defbibheading{secbib}[\bibname]{%
  \section*{#1}%
  \markboth{#1}{#1}}

\DeclareFieldFormat{sentencecase}{\MakeSentenceCase{#1}}

\renewbibmacro*{title}{%
  \ifthenelse{\iffieldundef{title}\AND\iffieldundef{subtitle}}
    {}
    {\ifthenelse{\ifentrytype{article}\OR\ifentrytype{inbook}%
      \OR\ifentrytype{incollection}\OR\ifentrytype{inproceedings}%
      \OR\ifentrytype{inreference}}
      {\printtext[title]{%
        \printfield[sentencecase]{title}%
        \setunit{\subtitlepunct}%
        \printfield[sentencecase]{subtitle}}}%
      {\printtext[title]{%
        \printfield[titlecase]{title}%
        \setunit{\subtitlepunct}%
        \printfield[titlecase]{subtitle}}}%
     \newunit}%
  \printfield{titleaddon}}

\renewbibmacro*{volume+number+eid}{%
		  \printfield{volume}%
		  \setunit*{\addnbspace}% NEW (optional); there's also \addnbthinspace
		  \printfield{number}%
		  \setunit{\addcomma\space}%
		  \printfield{eid}}
\DeclareFieldFormat[article]{number}{\mkbibparens{#1}}

\begin{document}
%%----------------------------------------------------------------------- % %
%% 	THESIS TITLE PAGE
%%----------------------------------------------------------------------- % %
%\pagenumbering{roman}
\begin{titlepage}
\begin{center}

%\doublespacing
\begin{spacing}{2.5}
{\Huge \textbf{Sketch and Project: Randomized Iterative Methods for Linear Systems and Inverting Matrices}}% thesis title
\end{spacing}
\vspace{1.0cm}

% to include a copy of the Edinburgh crest on the front cover

\begin{figure}[h]
\begin{center}
\includegraphics[width=0.5\textwidth, natwidth=610,natheight=642]{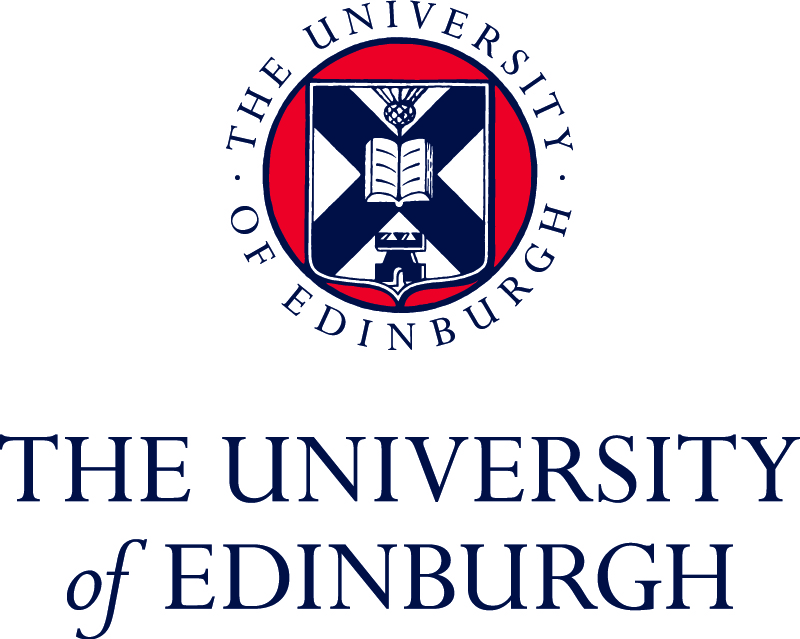}
\end{center}
\end{figure}
\vspace{1cm}

\singlespacing

{\LARGE Robert Mansel Gower} % author

\vspace{1.5cm}
{\large Supervisor:\\
 Dr. Peter Richt\'arik}
\singlespacing
\vspace{1.5cm}
{\large % \sffamily
Thesis submitted for the degree of Doctor of Philosophy\\
\vspace{0.9 cm}
School of Mathematics\\
\vspace{0.5 cm}
The University of Edinburgh\\
\vspace{0.5cm}
2016% date of submission
}

\end{center}

%\begin{flushright}
%Word Count:\\
%123456 %number of words
%\end{flushright}

\end{titlepage}
%\cleardoublepage % %  to make sure Abstract page is on a right page and not on the back of the Title Page! 
% % % % % % % % % % % % % % % % % % % % % % % % %
% % The following is for the headings of the frontmatter pages %
%\titleformat{\chapter}[hang]{\normalfont\Huge}{\thechapter}{0pt}{} 
%\titlespacing*{\chapter}{0cm}{10pt}{20pt}[0pt]
%\pagenumbering{roman} %roman numerals
%%----------------------------------------------------------------------- % %
%% 	ABSTRACT PAGE
%%----------------------------------------------------------------------- % %
\onehalfspace
\addcontentsline{toc}{chapter}{Abstract} 
\chapter*{Abstract  \markboth{Abstract}{}}
\small

Probabilistic ideas and tools have recently begun to permeate into several fields where they had traditionally not played a major role, including fields such as numerical linear algebra and optimization. One of the key ways in which these ideas influence these fields is via the development and analysis of randomized algorithms for solving standard and new problems of these fields. Such methods are typically easier to analyze, and often lead to faster and/or more scalable and versatile methods in practice.

This thesis explores the design and analysis of new randomized iterative methods for solving linear systems and inverting matrices. The methods are based on a novel sketch-and-project framework. By sketching we mean, to start with a difficult problem and then randomly generate a simple problem that contains all the solutions of the original problem. After sketching the problem, we calculate the next iterate by projecting our current iterate onto the solution space of the sketched problem.

The starting point for this thesis is the development of an archetype randomized method for solving linear systems. Our method has six different but equivalent interpretations: sketch-and-project, constrain-and-approximate, random intersect, random linear solve, random update and random fixed point. By varying its two parameters -- a positive definite matrix (defining geometry), and a random matrix (sampled in an i.i.d. fashion in each iteration) -- we recover a comprehensive array of well known algorithms as special cases, including the randomized Kaczmarz method, randomized Newton method, randomized coordinate descent method and random Gaussian pursuit. We also naturally obtain variants of all these methods using blocks and importance sampling. However, our method allows for a much wider selection of these two parameters, which leads to a number of new specific methods. We prove exponential convergence of the expected norm of the error in a single theorem, from which existing complexity results for known variants can be obtained. However, we also give an exact formula for the evolution of the expected iterates, which allows us to give lower bounds on the convergence rate. 

We then extend our problem to that of finding the projection of given vector onto the solution space of a linear system. For this we develop a new randomized iterative algorithm: {\em stochastic dual ascent (SDA)}. The method is dual in nature, and iteratively solves the dual of the projection problem. The dual problem is a non-strongly concave quadratic maximization problem without constraints.
 In each iteration of SDA, a dual variable is updated by a carefully chosen point in a subspace  spanned by the columns of a random matrix drawn independently from a fixed distribution. The distribution plays the role of a parameter of the method.  Our complexity results hold for a wide family of distributions of random matrices, which opens the possibility to fine-tune the stochasticity of the method to particular applications.   
 We prove that  primal iterates associated with the dual process converge to the projection exponentially fast in expectation, and give a formula and an insightful lower bound for the convergence rate. We also prove that the same rate applies to dual function values, primal function values and the duality gap. Unlike traditional iterative methods, SDA converges under virtually no additional assumptions on the system (e.g., rank, diagonal dominance) beyond consistency. In fact, our lower bound  improves as the rank of the system matrix drops. By mapping our dual algorithm to a primal process, we uncover that the SDA method is the dual method with respect to the sketch-and-project method from the previous chapter. Thus our new more general convergence results for SDA carry over to the sketch-and-project method and all its specializations (randomized Kaczmarz, randomized coordinate descent...etc). When our method specializes to a known algorithm, we either recover the best known rates, or improve upon them. Finally, we show that the framework can be applied to the distributed average consensus problem to obtain  an array of new algorithms. The randomized gossip algorithm arises as a special case. 

In the final chapter, we extend our method for solving linear system to inverting matrices, and develop a family of methods with specialized variants that maintain symmetry or positive definiteness of the iterates. All the methods in the family converge globally and exponentially, with explicit rates. In special cases, we obtain stochastic block variants of several quasi-Newton updates, including   bad Broyden (BB), good Broyden (GB),  Powell-symmetric-Broyden (PSB), Davidon-Fletcher-Powell (DFP) and Broyden-Fletcher-Goldfarb-Shanno (BFGS). Ours are the first stochastic versions of these updates shown to converge to an inverse of a fixed matrix. Through a dual viewpoint we uncover a fundamental link between quasi-Newton updates and approximate inverse preconditioning.  Further, we develop an adaptive variant  of the randomized block BFGS (AdaRBFGS), where we modify the distribution underlying the stochasticity of the method throughout the iterative process to achieve faster convergence. By inverting several matrices from varied applications, we demonstrate that AdaRBFGS is highly competitive when compared to the well established Newton-Schulz and approximate preconditioning methods.  In particular, on large-scale problems our method outperforms the standard methods by orders of magnitude. The development of efficient methods for estimating the inverse of very large matrices is a  much needed tool for preconditioning and variable metric methods in the  big data era.

%In the final chapter we study preconditioning a sequences of linear systems. At the heart of Newton based optimization methods is a sequence of symmetric linear systems. Each consecutive system in this sequence is similar to the next, so solving them separately is a waste of computational effort. Here we describe automatic preconditioning techniques for iterative methods for solving such sequences of systems by maintaining and updating an estimate of the inverse system matrix. This estimate of the inverse is updated using the iterative methods we developed for calculating inverse matrices. We perform tests on logistic support vector machine problems that reveal that our preconditioning method coupled with a Newton-CG is very efficient, as compared to Newton-CG without preconditioning, the BFGS and the L-BFGS methods. Further tests on a set of classic test problems reveal that the method is robust. 

\newpage

% % At the University of Edinburgh you need a lay abstract for non science people:
\onehalfspace
\addcontentsline{toc}{chapter}{Lay summary} 
\chapter*{Lay Summary \markboth{Lay summary}{}}
\small

This thesis explores the design and analysis of methods (algorithms) for solving two common problems: solving linear systems of equations and inverting matrices. 
Many engineering and quantitative tasks require the solution of one of these two problems.
In particular, the need to solve linear systems of equations is ubiquitous in essentially all quantitative areas of human endeavour, including industry and science. Specifically, linear systems are a central problem in numerical linear algebra, and play an important role in computer science, mathematical computing, optimization, signal processing, engineering, numerical analysis, computer vision, machine learning,  and many other fields. This thesis proposes new methods for solving large dimensional linear systems and inverting large matrices that use tools and ideas from probability.

The advent of large dimensional linear systems of equations, based on big data sets, poses a challenge. On these large linear systems, the traditional methods for solving linear systems can take an exorbitant amount of time. To address this issue we propose a new class of randomized methods that are capable of quickly obtaining approximate solutions. This thesis lays the foundational work of this new class of randomized methods for solving linear systems and inverting matrices. The main contributions are providing a framework to design and analyze new and existing methods for solving linear systems. In particular, our framework unites many existing methods. For inverting matrices we also provide a framework for designing and analysing methods, but moreover, using this framework we design a highly competitive method for computing an approximate inverse of truly large scale positive definite matrices. Our new method often outperforms previously known methods by several orders of magnitude on large scale matrices. 

%In the setting in which these big data problems arise, one is often only interested in obtaining an approximate solution. Traditional method were designing to obtain high accuracy solution, aimed a regime of medium data sizes.
%
%One of main advantageous of randomized methods is the rate at which they converge to the solution. 
%
% In thesis we presented new randomized methods for solving linear systems that are capable of quickly obtain approximate solutions.
\newpage

%%----------------------------------------------------------------------- % %
%% 	AUTHORS DECLARATION PAGE
%%----------------------------------------------------------------------- % %
\singlespacing
\addcontentsline{toc}{chapter}{Author's Declaration}
\chapter*{Author's Declaration  \markboth{Author's Declaration}{}}
\vspace{.3in}
\begin{spacing}{1.6}
\noindent
\small
 % University of Edinburgh mandatory declaration
I declare that this thesis has been composed solely by myself and that it has not been submitted, either in whole or in part, in any previous application for a degree.
Except where otherwise acknowledged, the work presented is entirely my own.
During the course of the PhD program I co-authored six papers~\cite{Gower2012Sparsity,Gower2014,Gower2014c,Gower2015,Gower2015c,Gower2016}, all of which I was the first author.
 This thesis is based on three of these papers, with the table below indicating which chapters are on which papers.
 I confirm that I contributed to all the results within the papers on which this thesis is based.

% all of which I was the first author and contributed to all the results. This thesis is based on three of these papers, with a one-to-one correspondence between a chapter and a paper according to the following table.
\begin{center}
\begin{tabular}{|c||cccc|} \hline
Chapter & \ref{ch:introduction} & \ref{ch:linear_systems} & \ref{ch:SDA} & \ref{ch:inverse} \\ \hline
Paper & \cite{Gower2015,Gower2015c,Gower2016}  &\cite{Gower2015} & \cite{Gower2015c} & \cite{Gower2016} \\ \hline
\end{tabular}
\end{center}

\vspace{1.75in}
\noindent
Robert Mansel Gower\\ % author's name
29th of February 2016 % date of submission
\end{spacing}
\clearpage

%%----------------------------------------------------------------------- % %
%% 	Acknowledgements
%%----------------------------------------------------------------------- % %
\onehalfspace
%\addcontentsline{toc}{chapter}{Acknowledgments}
\chapter*{Acknowledgments  \markboth{Acknowledgments}{}}
\vspace{-.1in}
\small
\noindent

%
%Well this didn't end as expected..
%It's hard to imagine 
% There are few things more. 

%Where do new ideas come from? New ideas do not form part of math or science. Instead, they sit  in the fog of the terra incognito just outside the crisp and clear plains of science and math. In particular, science only starts with already having a hypothesis.  I believe this explains, in part, why embarking on a PhD is a turbulent ride.
%Tasked with producing new knowledge...who knows where it will take you!?
%% I certainly did not imagine I would writing this thesis 
%%on randomized numerical linear algebra, but hey, here it is. 
%The affect on the passenger is a spot turbulence, with swooping ups and downs that may leave you feeling elated, or just a bit sick in the stomach. To gracefully survive the journey, emotional support is needed, to smoothen out the bumpy ride. I was fortunate enough to have loving support of many, a few of which I mention below, though all of which I am indebted. 

I am immensely grateful to my supervisor Dr. Peter Richt\'arik.  Peter has been a mentor to me on all fronts of being a researcher. From the very process of developing novel research directions, to clear and elegant mathematical writing, delivering the best presentations, and the many facets of applying for a job academia. Peter is my role model for being a researcher and supervisor.
Furthermore, his attempts at besting me at table tennis were also admirable. 

I would like to thank Prof. Jacek Gondzio for numerous research discussions, support and collaboration. Many thanks to my second supervisor Dr. Andreas Grothey for  accompanying the progress of the PhD through the years and career discussions.
I am grateful to my examination committee, Prof. Nicholas J. Higham and Prof. Ben Leimkuhler for their many detailed pointers, suggestions for improvement and for their time and insightful questions. 

I am indebted to the school of Mathematics of the University of Edinburgh, for not only providing me with a wonderful work environment, funding yearly trips to the Firbush sports center, but, most of all, for directly funding my PhD. I am most grateful to the Dr. Laura Wisewell Travel scholarship for funding my conference travels expenses in 2013 and 2015. The chance to participate in international conferences, with all the experts in my area in one place, was invaluable. 

%To Prof. Joel trump I am also grateful for being an excellent editor, and showing

One of the many highlights of my PhD years was a chance encounter Prof. Donald Goldfarb at the  Optimization and Big Data 2015 Workshop in Edinburgh. There we discovered that we both have impeccable taste in research projects, and had independently arrived at the same extension to Don's hailed BFGS method. Prof. Donald Goldfarb has since been an inspiration to me, a great enthusiastic collaborator, and warm person in general. Thanks also to Don for having me over at Columbia University, I look forward to our continued work together.

My brother, Dr. Artur Gower, has been a constant support throughout my PhD studies and throughout my life (he even ventured into the world just before me to check that the coast was clear). Being a year ahead of me on a PhD program in Ireland, Art has given me the heads-up on how to write good science, applying for funding, and finally how to land a job in science. Art went so far as to read draft's of my papers, proposals and co-authored a paper with me. For his ricochet advice, originally directed toward my brother Artur but then finding its way to me, I would like to thank Prof. Michel Destrade.  

Thanks to my many friends in Edinburgh,  without whom life in Edinburgh would have been as grey as the weather and stone that envelopes it. In particular, to my cohort (in order of their appearance) Pablo Gonzalez Brevis, Kimon Fountoulakis, Tim Schultz, Hanyi Chen, Jakub Kone\v{c}n\'y, Dominik Csiba and Nicolas Loizou I extend a heartfelt thanks for many discussions and their friendship.
I would specially like to thank my current and past flatmates Tarek Alabbas and Clara Vergez. Thanks Tarek for your friendship through all the years and being my fellow breakfast musketeer.
%Clara has been an incredible friend and flatmate.
The only reason I made it to work at any reasonable time was down to Clara banging on my bedroom door early in the morning for breakfast.
 Clara you have been an amazing flatmate, friend, and salsa co-star; thanks bro-ster.
%
% It's hard to be a better friend and flatmate  
%
%Clara has played many roles, ranging from ... almost motherly ....

%The end of my PhD marks the beginning of the end of my Scottish experience, so to you Scotland, and your rolling Corbetts, and high Bens, culminating in your Munroes and Cuillins. Wading through your Glens and bog, 

Furthermore, I want to thank my mother Elza Maria, who made my existence possible, and for her unconditional and immeasurable love and support. %, has given me the confidence to pursue my dreams (too much).

Finally, many thanks to Jess (aka leao fofinha), for her loving support, companionship, being my sous-chef, deputy nutritionist and just being awesome in general. 
\vfill
\begin{figure}[!h]
\centering
\includegraphics[width=2cm, natwidth=610,natheight=642]{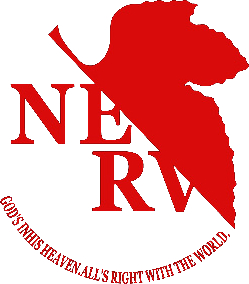} \\
{\small It takes some Nerv to do a PhD}
\end{figure}

% Thank everybody you've ever known here. :)

\clearpage
%%----------------------------------------------------------------------- % %
%% 	Table of contents
%%----------------------------------------------------------------------- % %
\singlespacing
\small
\setcounter{tocdepth}{2}	% Sets the depth of the table of contents: 0 = chapters only, 1  = chapters + sections, 2  = chapters + sections + subsections, etc.
\addcontentsline{toc}{chapter}{Table of Contents}  % Adds the title "Table of Contents"
\tableofcontents
%\clearpage
%%----------------------------------------------------------------------- % %
%% 	List of figures
%%----------------------------------------------------------------------- % %

%\singlespacing
%\addcontentsline{toc}{chapter}{List of Figures}% Adds the title "List of Figures"
%\listoffigures
%\clearpage

%%----------------------------------------------------------------------- % %
%% 	List of tables
%%----------------------------------------------------------------------- % %

%\singlespacing
%\addcontentsline{toc}{chapter}{List of Tables}% Adds the title "List of Tables"
%\listoftables
%\cleardoublepage

%%----------------------------------------------------------------------- % %
%% List ob Abbreviations
%%----------------------------------------------------------------------- % %
\onehalfspace
\addcontentsline{toc}{chapter}{List of Symbols}
\chapter*{List of Symbols \markboth{List of Symbols}{}}
\vspace{-.1in}
\small
\noindent
\begin{tabular}{ll}
%$[n]$ & Shorthand notation for the set $\{1,\ldots,n\}$\\
$\dotprod{x,y}$ & The standard Euclidean inner product of $x,y \in \R^n$, $\dotprod{x,y} = \sum_{i=1}^n x_i y_i$\\
$\norm{x}_2$ & The standard Euclidean norm of $x \in \R^n,$ $\norm{x}_2 = \sqrt{\dotprod{x,x}}$\\
$\norm{x}_B$ & The norm defined by symmetric positive definite $B \in \R^{n\times n}$, $\norm{x}_B = \sqrt{\dotprod{Bx,x}}$\\
$e^i$ & The $i$th coordinate vector in $\R^m$\\
$f_i$ & The $i$th coordinate vector in $\R^n$\\
$\dim (V)$ & The dimension of $V$ where $V$ is a subspace\\
$\Rank{M}$ & The rank of $M$, where $M$ is a matrix\\
%$\mathbf{Domain}(M)$ & The domain space of $M$, e.g, if $M \in \R^{m\times n}$ then $\mathbf{Domain}(M) = \R^{n}$ \\
%$\mathbf{Image}(M)$ & The image space of $M$, e.g, if $M \in \R^{m\times n}$ then $\mathbf{Image}(M) = \R^{m}$ \\
$\Null{M}$  & The nullspace of $M$, $\Null{M}=\{x \, | \, Mx =0 \}$\\
$\myRange{M}$ & The rangespace of $M$, e.g. if $M \in \R^{m\times n}$ then $\myRange{M}=\{Mx \, | \, x \in \R^n\}$\\
$\lambda_{\max}(M)$ & The largest eigenvalue of $M$\\
$\lambda_{\min}(M)$ & The smallest eigenvalue of $M$\\
$\lambda^+_{\min}(M)$ & The smallest nonzero eigenvalue of $M$ (assuming that $M$ is a nonzero matrix) \\
$M^{\dagger}$ & The Moore-Penrose pseudoinverse of $M$.\\
 $\Tr{M}$ & The trace of $M$, e.g. if $M \in \R^{n\times n}$ then $\Tr{M} = \sum_{i=1}^n M_{ii}$ \\
$\norm{M}_F$ & The Frobenius norm of $M$, $\norm{M}_F = \sqrt{\Tr{M^\top M}}$ \\
$\norm{M}_2$ & The spectral norm of $M$, $\norm{M}_2 = \max_{\norm{v}_2=1} \norm{Mv}_2 = \sqrt{\lambda_{\max}(M^\top M)}$ \\
$\norm{M}_B$ & $\eqdef  \norm{B^{1/2}MB^{-1/2}}_2 = \max_{\norm{v}_B=1} \norm{Mv}_B $ \\
$\norm{M}_B^*$ & $\eqdef  \norm{B^{1/2}MB^{1/2}}_2$\\[0.3cm]
%$\E{ \cdot }$ & The expectation operator, e.g. if $X$ is a random variable then $\E{X}$ is the expectation of $X.$
% $Z$ & \\
\end{tabular}
\noindent {\bf \large Special Matrices}\\[0.3cm]
\begin{tabular}{lcl}
$A$ & $\eqdef$ & The $m \times n$ real system matrix (In Chapter~\ref{ch:inverse} we use $m=n$)\\
$S$ & $\eqdef$ & A random $m\times q$ matrix\\
$B$ & $\eqdef$ & A symmetric positive definite $n \times n$ matrix\\
$Z$ & $\eqdef$ & $A^\top  S(S^\top  A B^{-1}A^\top  S)^{\dagger}S^\top A$ \\
%$\E{\cdot}$ &$\eqdef$  & The expectation operator, e.g. $\E{Z}$ is the expected value of $Z$\\ 
$H$ & $\eqdef$ & $\mathbf{E}_{S\sim {\cal D}} \left[ S\left(S^\top AB^{-1}A^\top S\right)^{\dagger}S^\top\right]$
\end{tabular}
%\begin{enumerate}
%\item[$\Rank{A}$] The rank of $A$
%\item[$\mathbf{Domain}(A)$] The domain space of $A$, e.g, if $A \in \R^{m\times n}$ then $\mathbf{Image}(A) = \R^{n}$
%\item[$\mathbf{Image}(A)$] The image space of $A$, e.g, if $A \in \R^{m\times n}$ then $\mathbf{Image}(A) = \R^{m}$
%\item[$\Null{A}$] The nullspace of $A$, i.e., $\Null{A}=\{x \, | \, Ax =0 \}.$
%\item[$\myRange{A}$] The rangespace of $A$, i.e., $\myRange{A}=\{y \, | \,\exists x, \, Ax =y \}.$
%\end{enumerate}

\clearpage
\cleardoublepage
%%----------------------------------------------------------------------- % %
%% 	CHAPTERS
%%----------------------------------------------------------------------- % %
\renewcommand{\baselinestretch}{0.5} 
%\singlespacing
%\pagenumbering{arabic}%

%% % % % % % % % % % % % % % % % % % % % % % % % % 
%%  New Chapter Title Page format. To make things look pretty. % 
%% % % % % % % % % % % % % % % % % % % % % % % % % 
\titleformat{\chapter}[display]
  {\normalfont\Large\raggedleft}
  {\MakeUppercase{\chaptertitlename}%
    \rlap{ \resizebox{!}{1.5cm}{\thechapter} \rule{5cm}{1.5cm}}}
  {10pt}{\Huge}
\titlespacing*{\chapter}{0pt}{30pt}{20pt}
% % % % % % % % % % % % % % % % % % % % % % % % % % 
\onehalfspacing
%\singlespacing
%\clearpage
%---------------------------------------------------------------------------------
%	CHAPTER One: Introduction (C.Stamper template edited by JM)
%---------------------------------------------------------------------------------
%\onehalfspace   %% official UoE spacing
\chapter[Introduction]{Introduction}
\chaptermark{Introduction}
\label{ch:introduction} % label for referring to chapter in other parts of the thesis
%\section{About this Thesis}\label{sec:Intro}
{
\epigraph{\emph{We can only see a short distance ahead, but we can see plenty there that needs to be done.}}{Alan Turing}
\let\clearpage\relax
\section{Introduction: What's to Come}
}
\label{secChOne:whats}
%\begin{flushright} \singlespacing
%\begin{minipage}[r]{5cm}
%``it's not whether you win or lose that counts. In fact, nothing counts, and death is coming for us all.''\\ \hrule \vspace{0.2cm}
%-- Jimmy Kimmel, supposedly an old Norwegian saying
%\end{minipage}
%\end{flushright}\\

This thesis explores the design and analysis of new randomized iterative methods for solving linear systems and inverting matrices. 
All the methods presented in this thesis are globally and linearly convergent. Consequently, the methods are well suited to quickly obtain approximate solutions. 
%which is our focus here.  
% But the methods presented here posses no accelerated local convergence (e.g., local second order convergence) and thus are not suited for calculating highly accurate solutions.  

 The methods are based on a novel sketch-and-project framework. By sketching we mean, to start with a difficult problem and then randomly generate a simple problem that contains all the solutions of the original problem. For instance, consider the linear system
\begin{equation}\label{eq:linsysintro}Ax=b,\end{equation}
where $A\in \R^{m\times n},x\in \R^n$ and $b\in \R^m$. We suppose throughout the thesis that there exists a solution $x^* \in \R^n$ to the linear system, that is, the linear system is \emph{consistent}. Let $S \in \R^{m\times q}$ be a random matrix with the same number of rows as $A$ but far fewer columns ($q \ll n$.) The resulting \emph{sketched} linear system is given by
\[S^\top A x = S^\top b,\]
which has a relatively small number of rows, and is thus easier to solve. There has been a concerted effort into designing the distribution of $S$ with the property that the solution set of the sketched linear system is close to the solution set of the original linear system with high probability, particularly so for solving linear systems that arise from the least-squares problem~\cite{Pilanci2014,Mahoney:2011,Drineas2011}. Determining an $S$ with such a property can be difficult and often depends on properties of $A$ that are  expensive to compute.
Here we take a different approach and use  sketching combined with a projection process. We apply our sketch-and-project technique not only to solve linear systems, but also to find the projection of a vector onto the solution space of a linear system and to invert matrices (by sketching, for example, the inverse equation $AX=I$).

Throughout the entirety of this thesis, we assume that we can access the system matrix $A$ through matrix-vector products. Thus every element $A_{ij}$ of the system matrix may not be explicitly available. The methods presented here are designed with this restriction in mind and are thus compatible with the setting where $A$ can only be accessed as an operator. 

In the remainder of this section we give a summary of each chapter of this thesis. The detailed proofs and careful deductions of any claims made here are left to the chapters.

%\subsection{Chapter~\ref{ch:linear_systems}: Linear Systems with Sketch-and-Project}
\paragraph{Chapter~\ref{ch:linear_systems}: Linear Systems with Sketch-and-Project.}
In Chapter~\ref{ch:linear_systems} we develop an iterative process that gradually refines an approximate solution to~\eqref{eq:linsysintro} using a sequence of sketching matrices, as opposed to the one shot sketching method. To describe our method,  let $x^k\in \R^n$ be our current estimate of the solution to~\eqref{eq:linsysintro}. We obtain an improved estimate by projecting $x^k$ onto the solution space of a sketched system
\begin{equation} \label{C1eq:sketchproj} x^{k+1} = \arg \min_x \norm{x^k-x}_B^2 \quad \mbox{subject to} \quad S^\top Ax=S^\top b,\end{equation}
where $S$ is drawn in each iteration independently from a pre-specified distribution, $B$ is a positive definite matrix and $\norm{x}_B^2 \eqdef \dotprod{Bx,x}.$
This iterative process has a closed form solution given by
\begin{equation} \label{C1eq:xupdate}x^{k+1} = x^k - B^{-1}A^\top S(S^\top A B^{-1}A^\top S)^{\dagger}S^\top(Ax^k-b),\end{equation}
where $\dagger$ denotes the (Moore-Penrose) pseudoinverse.
Using the closed form expression for the update~\eqref{C1eq:xupdate} we show  that the iterates converge if $A$ has full column rank and under mild assumptions on the distribution of $S$. In particular, the convergence analysis will depend heavily on the following random matrix 
\[Z \eqdef A^\top S(S^\top A B^{-1}A^\top S)^{\dagger}S^\top A,\]
which governs the iterative process~\eqref{C1eq:xupdate}. 
Indeed, we will show that when $A$ has full column rank and for any $x_0 \in \R^n,$ the iterates~\eqref{C1eq:xupdate} converge to the unique solution $x^* \in \R^n$ of the linear system exponentially fast according to
\begin{equation} \label{ch:one:Enormconv}
\E {\norm{x^{k} -x^{*} }_B^2 } \leq \rho^k \;\cdot\; \norm{x^{0} - x^{*}}_B^2,
\end{equation}
where \[\rho \eqdef  1- \lambda_{\min}(B^{-1/2}\E{Z}B^{-1/2}).\]
By $B^{-1/2}$ we denote the unique positive definite square root of $B^{-1}.$ Therefore $B^{-1/2}B^{-1/2} =B^{-1}.$ We use $\E{\cdot}$ to denote the expectation operator. For instance, $\E{Z}$ is the expected value of $Z$. Since $Z$ is a function of $S$, which is the only random component of $Z$, we have that $\E{Z}$ is an expectation taken over the distribution of $S$. 
% and $\lambda_{\min}$ returns the smallest eigenvalue
Because of the importance that $Z$ plays in this thesis, later in this chapter in Section~\ref{C1sec:tools}, we prove several properties of $Z.$ 
%that $B^{-1}Z$  is an oblique projection and $B^{-1/2}ZB^{-1/2}$ is an orthogonal projection.

In Section~\ref{C2sec:discrete} we design a discrete distribution for $S$ that yields easily interpretable  convergence rates in terms of a scaled condition number.
Furthermore, we show that by choosing $B$ and the distribution of $S$  appropriately we recover a comprehensive array of well known algorithms as special cases, such as the randomized Kaczmarz method~\cite{Strohmer2009} and randomized Coordinate Descent~\cite{Leventhal2010}, demonstrating the expressive power of the framework. Having established convergence rates for each method defined by $(B,S)$, we explore questions such as: \emph{what distribution of $S$ yields the optimal convergence rates}?
We determine that the optimal distribution of $S$, chosen from a family of discrete distributions, is the solution to a particular semidefinite program. This result determines, for instance, the optimal distribution for selecting the rows of the linear system in the randomized Kaczmarz method.  Furthermore, the optimized randomized Kaczmarz method is shown to converge significantly faster than the randomized Kaczmarz method using the standard distribution for $S$.

Our framework also allows for $S$ to have a continuous distribution, and to give an insight into the possibilities, we present three methods based on a Gaussian sketching matrix $S$ and give convergence rates for each. We then conclude the chapter with further numeric experiments that compare the different methods presented throughout the chapter.

%Full column rank assumption, didactic approach, general approach follows in chapter 3.
%\subsection{Chapter~\ref{ch:SDA}: Stochastic Dual Ascent}
\paragraph{Chapter~\ref{ch:SDA}: Stochastic Dual Ascent for Finding the Projection of a Vector onto a Linear System.}
In Chapter~\ref{ch:SDA} we consider the more general problem of finding the projection of a given vector $c \in \R^n$ onto the solution space of a linear system, that is
\begin{equation}  \min_{x \in \R^n} \quad   \tfrac{1}{2}\|x-c\|_B^2 \quad
\text{subject to} \quad   Ax=b. \label{C1eq:P}
\end{equation}
 To solve this projection problem we develop a new randomized iterative algorithm: {\em stochastic dual ascent (SDA)}. The method is dual in nature, and iteratively solves the dual of the projection problem~\eqref{C1eq:P}. The dual problem is a non-strongly concave quadratic maximization problem without constraints given by
\begin{equation}\label{C1eq:D} \max_{y \in \R^m} \quad (b-Ac)^\top y - \tfrac{1}{2}\|A^\top y\|_{B^{-1}}^2.
\end{equation} 
Each iterate $y^{k+1} \in \R^m$ of the SDA method is carefully chosen from a random affine space that passes through the previous iterate
\begin{equation}\label{C1eq:SDA-compact0}
 y^{k+1} \in  y^k + \myRange{S}, %\Lambda^k 
\end{equation}
where $S$ is a random matrix drawn independently from a fixed distribution.
Specifically, $y^{k+1}$ is the point with least norm that maximizes the dual objective in~\eqref{C1eq:D} constrained within the random affine space~\eqref{C1eq:SDA-compact0}.

% In each iteration of SDA, a dual variable is updated by a carefully chosen point in a subspace spanned by the columns of a random matrix drawn independently from a fixed distribution. 

By mapping our dual iterates~\eqref{C1eq:SDA-compact0} to primal iterates, we uncover that the SDA method is a dual version of the sketch-and-project method~\eqref{C1eq:xupdate}.
We then proceed to strengthen our convergence results established in the previous chapter. 
First, we do away with the assumption that $A$ has full column rank and consider any nonzero matrix $A$ and consistent linear system. In this more general setting, we prove that for
$x_0 \in c + \myRange{B^{-1}A^\top}$  the primal iterates converge to the solution of~\eqref{C1eq:P} exponentially fast in expectation  according to~\eqref{ch:one:Enormconv}  with a convergence rate of 
\[\rho = 1- \lambda_{\min}^+(B^{-1/2}A^THAB^{-1/2}),\]
where
\begin{equation}\label{eq:Hintro}
H = \E{S(S^\top AB^{-1}A^\top S)^{\dagger}S^\top},
\end{equation}
 and $\lambda_{\min}^+(B^{-1/2}A^THAB^{-1/2})$ denotes the smallest nonzero eigenvalue of $B^{-1/2}A^THAB^{-1/2}$. 
% When $A$ has full column rank the convergence rates are equal, thus this result is a proper generalization....
% 
% In agreement... 
  The only condition for this convergence to hold is that $H$ be nonsingular.  We completely characterize the discrete distributions of $S$ for which $H$ is nonsingular in Section~\ref{sec:Hnonsingular}. This shows, for instance, that the Kaczmarz method converges so long as the system matrix has no zero rows.  Thus assuming that $A$ has full column rank, as is required in the convergence theorems in Chapter~\ref{ch:linear_systems}, is an unnecessary assumption for proving convergence of the sketch-and-project method. But the full column rank  assumption makes the proofs of convergence simpler and thus, for pedagogic reasons, we have presented the proofs assuming that $A$ has full column rank earlier on in Chapter~\ref{ch:linear_systems}. We present further improvements in the convergence analysis and give a tight and insightful lower bound for the convergence rate that depends on the rank of $A.$
% as compared with the Theorem~\ref{theo:Enormerror} in Chapter~\ref{ch:linear_systems}. 
%ch:two:theo:Enormconv.

 We also prove that the same rate of convergence $\rho$ governs the convergence of the dual function values, primal function values and the duality gap. Unlike traditional iterative methods, SDA converges under virtually no additional assumptions on the system (e.g., rank, diagonal dominance) beyond consistency. In fact, our lower bound  improves as the rank of the system matrix drops. 
When our method specializes to a known algorithm, we either recover the best known rates, or improve upon them. Finally, we show that the framework can be applied to the distributed average consensus problem to obtain  an array of new algorithms. The randomized gossip algorithm arises as a special case~\cite{Boyd2006,OlshevskyTsitsiklis2009}. 

%\textbf{No assumptions on the system matrix.} 
% In \cite{Gower2015}  the authors only studied the convergence of the primal iterates $\{x^k\}$, establishing a (much) weaker variant of Theorem~\ref{theo:Enormerror}. Indeed,  convergence was only established in the case when $A$ has full column rank. In this work, we lift this assumption completely and hence establish complexity results in the  general case. 
%
%\textbf{Convergence to a shifted point.} As we show in Theorem~\ref{theo:Enormerror},  Algorithm~\ref{alg:SDA-Primal}  converges to $x^*+t$, where $t$ is the projection of $x^0-c$ onto $\Null{A}$. Hence, in general, the method does not converge to the optimal solution $x^*$. This is not an issue if $A$ is of full column rank---an assumption used in the analysis in \cite{Gower2015}---since then $\Null{A}$ is trivial and hence $t=0$. As long as $x^0-c$ lies in $\myRange{B^{-1}A^\top}$, however, we have $x^k\to x^*$. This can be easily enforced (for instance, we can choose $x^0=c$).

%\subsection{Chapter~\ref{ch:inverse}: Randomized Matrix Inversion}
\paragraph{Chapter~\ref{ch:inverse}: Randomized Matrix Inversion.}
In Chapter~\ref{ch:inverse} we extend our method for solving linear systems to inverting matrices, and develop a family of methods with a specialized variant which maintains symmetry or positive definiteness of the iterates.

The initial insight into our matrix inversion methods comes from the simple observation that for a nonsingular matrix $A \in \R^{n\times n}$ the inverse is the solution in $X$ to  one of the \emph{inverse equations} 
\[AX =I \quad \mbox{or}\quad XA =I.\]
Our method for inverting $A$ calculates the new iterate $X_{k+1} \in \R^{n\times n}$ by projecting the previous iterate $X_k \in \R^{n\times n}$ onto the solution space of a sketched version of one of the two inverse equations: either the \emph{row} sketched variant $S^\top AX=S^\top$ or the \emph{column} sketched variant $XAS = S,$ where $S \in \R^{n \times q}$ is a random matrix drawn from a fixed distribution at each iteration.
For example, using the column sketched variant a new iterate is calculated by solving
\begin{equation} \label{C1eq:sketchprojinve} X_{k+1} = \arg \min_{X\in \R^{n\times n}} \norm{X_k-X}_{F(B)}^2 \quad \mbox{subject to} \quad XAS=S,\end{equation}
where the norm is the weighted Frobenius norm.
 When $A$ is symmetric, it can be advantageous to maintain symmetries in the iterates $X_k.$ We propose a sketch-and-project method that maintains symmetry in the iterates by imposing symmetry as a constraint in
 \begin{equation} \label{C1eq:sketchprojinvesym} X_{k+1} = \arg \min_{X\in \R^{n\times n}} \norm{X_k-X}_{F(B)}^2 \quad \mbox{subject to} \quad XAS=S, \quad X = X^\top.\end{equation}
 All the methods we present converge globally and linearly, with the same explicit rate of convergence $\rho= 1- \lambda_{\min}(B^{-1/2}\E{Z}B^{-1/2})$.
 In special cases, we obtain stochastic block variants of several quasi-Newton updates, including  the Broyden-Fletcher-Goldfarb-Shanno (BFGS) update. Ours are the first stochastic quasi-Newton updates shown to converge to an inverse of a fixed matrix. Through a dual viewpoint we uncover a fundamental link between quasi-Newton updates and approximate inverse preconditioning methods, which results in a new interpretation of the quasi-Newton methods. For instance, the BFGS update is the solution in $X$ to 
  \[X_{k+1} = \arg_X  \min_{X\in \R^{n\times n}, Y \in \R^{n\times q}}  \frac{1}{2}\norm{X -A^{-1}}_{F(A)}^2 \quad \mbox{subject to} \quad X = X_k + SY^\top +YS^\top,\]
  for a particular (deterministic) choice of $S.$
This shows that the BFGS udpate can be interpreted as a projection of $A^{-1}$ onto a space of symmetric matrices.
%where the in the classic BFGS update $S$ is a particular deterministic vector. 

  With explicit convergence rates for each method characterized by the distribution of $S$, we again raise the question of selecting a distribution of $S$ that results in a method with a faster convergence rate. Though different from the setting of solving a linear system, the goals of finding the inverse of $A$ and of designing a distribution of $S$ that results in an improved convergence rate are in synchrony. In particular, for many choices of $B$, it will transpire that the covariance of $S$ should be an estimate of the inverse of $A.$ One way to interpret this result is that $S$ should be chosen so that it not only sketches/compresses the inverse equations $AX=I$ or $XA =I$, but also, it should improve the condition number of these equations.
  
    This reasoning leads us to develop an adaptive variant of a randomized block BFGS (AdaRBFGS), where the distribution of $S$ depends on $X_k$. By inverting several matrices from varied applications, we demonstrate that AdaRBFGS is highly competitive when compared with the well established Newton-Schulz~\cite{Schulz1933} and the approximate inverse preconditioning methods~\cite{Chow1998,Saad2003,Gould1998,Benzi1999}. In particular, on large-scale problems our method outperforms the standard methods by orders of magnitude.
 Since the inspiration behind AdaRBFGS method comes from the desire to design an {\em optimal adaptive} distribution for $S$ by examining the convergence rate, this work  also highlights the importance of developing algorithms with explicit convergence rates. 

\paragraph{Organization of Thesis}
Chapters~\ref{ch:linear_systems},~\ref{ch:SDA} and~\ref{ch:inverse} are largely based on the papers~\cite{Gower2015,Gower2015c} and~\cite{Gower2016}, respectively.
Excluding preliminary results in linear algebra presented  in Section~\ref{C1sec:tools}, each of these chapter is mostly self-contained, including the objective, contributions, definitions and notation.

\section{Why Randomized Methods}
Why use randomization in algorithmic design? 
We can answer this question, in practical terms, by  measuring the advantages of using randomized algorithms as compared with existing deterministic methods. The advantages in using randomized methods include: often algorithms that are easier to analyze and implement, better convergence in terms of improved convergence rates or range of convergence (the randomized methods are almost always globally convergent), lower memory requirements, and more scalable and parallelizable methods in practice. 

To substantiate these claims, we now compare the sketch-and-project randomized methods~\eqref{C1eq:sketchproj} for solving the linear system~\eqref{eq:linsysintro} to stationary methods and the Krylov methods. Note that all the iterative methods mentioned here benefit from using a preconditioner, so much so, they are often only used in conjunction with a preconditioner. But to compare the methods on a equal footing, we assume no preconditioning strategy has been applied.

\subsection{A Case Study Comparing to Stationary Methods} 
Iterative methods that fit the simple form
\begin{equation}\label{eq:stationarym}
x^{k+1} = G x^{k-1} + c,
\end{equation}
are known as Stationary Methods, where $G \in \R^{n\times n}$ is the \emph{iteration matrix} and $c \in \R^n$ is the \emph{bias term}. Here neither $G$ nor $c$ depend on the iteration count. Methods that fit the format~\eqref{eq:stationarym} include the Jacobi method, the
Gauss-Seidel method, the Successive Overrelaxation (SOR) method and the Symmetric Successive Overrelaxation (SSOR) method~\cite{barrett1994,Saad1986}.

 The sketch-and-project methods presented in this thesis~\eqref{C1eq:xupdate} would fit the format~\eqref{eq:stationarym} if it were not for the fact that $G$ and $c$ in our methods are randomly generated, and thus, can differ from one iteration to the next. 
 
 Despite this difference, the sketch-and-project methods and the stationary  methods share many similarities.
 
  \paragraph{Low memory requirements.}  Stationary methods and~\eqref{C1eq:xupdate} are both easy to implement and have low memory requirements. Often only the previous iterate $x^k$ needs to be stored to enable the calculation of the next iterate. 
  
  \paragraph{Existence of convergence rates.} A stationary method only converges if the spectral radius of $G$ is less than one~\cite{Saad1986}, that is, if $\rho(G) <1.$ Often $G$ is constructed by splitting the system matrix $A$, e.g., for square systems the Jacobi method iteration matrix is $G = I-D^{-1}(A-D)$ where $D = \mbox{diag}(A_{11}, \ldots, A_{nn})$. Thus for $\rho(G)<1$ to hold one needs strong assumptions on the spectrum of $A.$ Again using the Jacobi method as an example, if $A$ is strictly diagonally dominant, then $\rho(G)<1$ holds. In contrast,  the sketch-and-project methods converge for virtually any matrix $A$ with an explicit convergence rate, as we show in Chapter~\ref{ch:SDA}.
  
  \paragraph{Parallelizable.}  Due to their simple recurrence relationship, the stationary methods lead to straight forward parallel and distributed variants, see  Section 2.5 in~\cite{Bertsekas:1989}. Furthermore, variants of stationary method for solving nonlinear systems have also been adapted to parallel architectures~\cite{BertsekasT91,Robinson2015}. The sketch-and-project methods are similar in this aspect, in that, they are amenable to parallel implementations. 
  For instance, the coordinate descent is member of the sketch-and-project family, and various distributed and parallel variants of coordinate descent have been designed to solve optimization problems~\cite{shotgun,Fercoq2013a,PCDM,TTD}.
  
%  Though there currently does not exist parallel variants of the sketch-and-project methods for solving linear systems,
%   the coordinate descent method for solving optimization methods and nonlinear systems 

\subsection{A Case Study Comparing to Krylov Methods}

The Krylov methods are a well studied and established class of iterative methods for solving linear systems.  In fact, the sketch-and-project methods share a certain similarity to Krylov methods. This can be seen by using the following equivalent dual formulation of~\eqref{C1eq:sketchproj} given by
\begin{equation} x^{k+1} \; = \; \arg\min_{x\in \R^n} \norm{x\phantom{^k}- x^{*}}_B^2 \quad \mbox{subject to} \quad x \in x^{k} + \myRange{B^{-1}A^\top S}.  \label{ch:one:RF}
\end{equation}
We refer to~\eqref{ch:one:RF} as the \emph{constrain and approximate} viewpoint, because a new iterate $x^{k+1}$ is selected as the best possible \emph{approximation} to the solution $x^*$ \emph{constrained} to a randomly generated affine space. Later in Section~\ref{sec:sixviews} we prove that~\eqref{C1eq:sketchproj} and~\eqref{ch:one:RF} are equivalent.

Formulation~\eqref{ch:one:RF} is similar to the framework often used to describe Krylov methods~\cite[Chapter 1]{Liesen2014}, which is
\begin{equation} \label{eq:krylov}
x^{k+1} = \arg\min_{x\in \R^n} \norm{x- x^{*}}_B^2 \quad \mbox{subject to} \quad x \in x^{0} + \K_{k+1},
\end{equation}
where $\K_{k+1}\subset \R^n$ is a $(k+1)$--dimensional subspace.
Note that the constraint $x \in x^{0}+\K_{k+1}$ is an affine space that passes through $x^{0}$, as opposed to passing through $x^{k}$ in the sketch-and-project formulation~\eqref{ch:one:RF}.
The objective $\|x-x^{*}\|^2_B $ is a generalization of the residual, where $B=A^\top A$ is used to characterize minimal residual methods (MINRES and GMRES)~\cite{Paige1975,Saad1986} and $B=A$ is used to describe the Conjugate Gradients (CG) method~\cite{Hestenes1952}. Progress from one iteration to the next is guaranteed by using expanding nested search spaces at each iteration, that is, $\K_k \subset \K_{k+1}.$

\paragraph{Low memory requirements.} To make an efficient Krylov method, the problem~\eqref{eq:krylov} should not be solved from scratch, but rather, one should build upon the knowledge that the previous iterate $x^k$ is the minima restricted to $\K_k$, and calculate $x^{k+1}$ by updating $x^k$ in an inexpensive manner. This inexpensive update is typically achieved through a \emph{short recurrence} update,
that is, an update applied to $x^k$ of the following form
\[x^{k+1} = x^k + \sum_{i=1}^{s} p^{k+2-i},\]
using a small number $s\in \N$ of vectors $p^{k+2-s},\ldots, p^{k+1}$.
To arrive at a short recurrence, one needs to carefully design an orthonormal basis for the spaces $\K_k.$ Such a short recurrence does not always exist. By the Faber-Manteuffel Theorem~\cite{faber1984}, the sufficient and necessary conditions for a short recurrence in a Krylov method are that the system matrix is a nonsingular normal matrix with respect to the geometry defined by the $B$ matrix.
\footnote{Though one cannot guarantee the existence of a short recurrence for a singular matrix in general, there do exist short recurrences for particular singular matrices that arise from deflation techniques~\cite{Gaul2014}.}

In contrast, the sketch-and-project methods  automatically have a short recurrence by design, as can be seen through~\eqref{C1eq:xupdate}, where $x^{k+1}$ is calculated by adding a single random vector to $x^k$ (a very short recurrence). Instead of using ``growing'' search spaces to guarantee convergence, the sketch-and-project methods are guaranteed to converge when the distribution of $S$ is such that the search space in~\eqref{ch:one:RF} ``covers'' the space of interest with a nonzero probability\footnote{As shown later in Chapter~\ref{ch:SDA}, this condition is captured in the requirement that $H$~\eqref{eq:Hintro} be nonsingular.}. Not only can we guarantee convergence on even singular matrices but, our lower bounds indicate that convergence improves for lower rank matrices.
%
% In our setting, progress is enforced by using $x^{k}$ as the displacement term instead of $x^{0}.$  This also allows for a simple recurrence for updating $x^{k}$ to arrive at $x^{k+1}$, which facilitates the analysis of our methods.
 
% 
%The Krylov methods share many of the advantages of the sketch-and-project methods, including, (1) short recurrence; they need only save the previous iterate and a  handful of vectors to compute the next iterate, and (2) have established convergence rates.  Though the Krylov methods have far more stringent requirements to guarantee short recurrence and convergence rates. By the Faber-Manteuffel Theorem~\cite{faber1984}, the sufficient and necessary conditions for a short recurrence in a Krylov method are that $A$ is a special sort of nonsingular normal matrix. In contrast, independent of the system matrix the sketch-and-project methods, by construction, need only store the previous iterate to compute the next iterate (a very short recurrence). 

\paragraph{Existence of convergence analysis.}  As for convergence rates, only certain instantiations of the Krylov methods, such as the CG, MINRES and GMRES methods have well understood convergence rates. These three aforementioned methods are only applicable when the system matrix $A$ is symmetric positive definite, symmetric and nonsingular, respectively. Again in contrast, the sketch-and-project methods converge for virtually any matrix $A$ with an explicit convergence rate, as we show in Chapter~\ref{ch:SDA}.

% it has already been shown that the coordinate descent method, an instance of~\eqref{C1eq:sketchproj}, leads 

\paragraph{Improved convergence rate.} As an example of how the rate of convergence of a Krylov method compares against the rate of convergence of a sketch-and-project type method, let us consider the setting where $A$ is positive definite. In this setting the CG method (a Krylov method) and the Coordinate Descent (\emph{CD}) method (a sketch-and-project method) are the methods of choice. The iteration complexity of the CG method is $O(\sqrt{\lambda_{\max}(A)/\lambda_{\min}(A)})$ while the iteration complexity of the CD method is $O(\Tr{A}/\lambda_{\min}(A))$ (proof in Section~\ref{C2sec:shs98ss}). Thus the CD method requires at least the square of the number of iterations that the CG method requires to reach the same relative accuracy (ignoring the expectation). But this iteration complexity needs to be counter balanced with the very low iteration cost of the CD method. That is, an iteration of the CD method costs $O(n)$ while an iteration of the CG method costs $O(n^2)$ (ignoring possible gains in efficiency due to sparsity for both methods). This is a significant difference in iteration cost when solving large dimensional linear systems. Strohmer and Vershynin~\cite{Strohmer2009} give examples of when this lower cost of the CD method\footnote{In the paper~\cite{Strohmer2009} the authors compare the iteration complexity and iteration cost of the randomized Kaczmarz method against the Conjugate Gradients method applied to the least-squares problem. But their comparisons hold verbatim for the CD method and the CG method applied to a positive definite linear system.} pays off for the higher iteration complexity. Furthermore, accelerated versions of the CD method~\cite{Lee2013,LiuWright-AccKacz-2016,Wright:ABCRRO,Fercoq2013a} improve the iteration complexity of the CD method to $O(\sqrt{\Tr{A}/\lambda_{\min}(A)})$ at the cost of only a constant increase in the iteration cost. Thus when taking iteration cost into account, the accelerated CD methods asymptomatically achieve a solution with higher accuracy than the the CG method for the same computational cost. We do not consider accelerated variants of the CD method, or the sketch-and-project method, in this thesis. For further insight into how accelerated sketch-and-project methods compare with the CG method, see ~\cite{LiuWright-AccKacz-2016} for a comparison to an accelerated randomized Kaczmarz method and~\cite{Lee2013} for a comparison to an accelerated CD method.
%
%was shown under some conditions such as diagonal dominance, to have a faster asymptotic running time than the CG method.

%The conjugate gradient method converges in iterates according to
%\[\|x^k-x^{*}\|_A \leq 2\left(\frac{\sqrt{\kappa} -1 }{\sqrt{\kappa} +1}\right)^k \|x^0-x^{*}\|_A, \]
%\[\frac{\sqrt{\kappa} -1 }{\sqrt{\kappa} +1} = 1  -\frac{2 }{\sqrt{\kappa} +1}\]
%where $\kappa = \lambda_{\min}(A)/\lambda_{\max}(A).$ While the coordinate descent method (as we show in Section~\ref{C2sec:shs98ss}) converges in iterates according to
%\[\E{\|x^k-x^{*}\|_A^2} \leq \left(1-\frac{\lambda_{\min}(A) }{\Tr{A}}\right)^k \E{\|x^0-x^{*}\|_A^2}. \]
% 
%and our sketch-and-project methods have 
%To compare the convergence rates ...
%
%In terms of faster convergence rates, recent accelerated versions of the coordinate descent method~\cite{Lee2013} and the randomized Kaczmarz method~\cite{LiuWright-AccKacz-2016} were shown, under some conditions such diagonal dominance, to have a faster asymptotic running time then the Conjugate Gradients method.

\paragraph{Parallelizable.} Where the  sketch-and-project methods truly diverge from Krylov methods is in how parallelizable they are. Though we do not address this in this thesis, methods based on~\eqref{C1eq:sketchproj} are easy to adapt to a parallel architecture. In contrast, adapting the Krylov methods towards parallel computing is challenging.  This is, in part, because of the delicate orthogonality conditions that need to be enforced on the basis of the search spaces $\K_k$ in order to guarantee short recurrences and convergence.  A few of the current strategies towards adapting the Krylov methods for a parallel implementations are the following.

The conjugate gradients method can be paired with domain decomposition methods for solving linear systems that result from discretizing PDE's  for an overall successful distributed method~\cite{DongarraHL2016}. But in this strategy the gains in parallelism come from the domain decomposition method and not from the conjugate gradient method \emph{per se}. 
 
Efforts towards designing a distributed or communication efficient variant of Krylov methods are focused on two fronts: exploring the parallelism in matrix-vector and vector-vector products performed in the the Krylov methods~\cite{Ballard2014}, and the so called $s$--step Krylov methods.

The $s$--step methods, such as the $s$--step conjugate gradients methods~\cite{Chronopoulos1989}, re-order the computations of standard Krylov methods so that $s$ iterations can be performed simultaneously and in parallel. Though implementing fast and reliable $s$--step methods come with two challenges: achieving numerical stability is more challenging than traditional Krylov methods~\cite{Ballard2014} and correctly identifying the communication bottlenecks is difficult as it depends on the nonzero structure of the system matrix~\cite{Ballard2014}.

%Overall, adapting the Krylov methods towards parallel computing is challenging.
% This is, in part, because of the delicate orthogonality conditions that need to be enforced on the basis of the search spaces $\K_k$ in order to guarantee short recurrences and convergence. 

The sketch and project methods, on the other hand, are relatively easy to adapt to a parallel setting. First, the methods~\eqref{C1eq:sketchproj} make use of shared memory parallelism through their \emph{block} variants. By block variants we refer to the setting when $S$ has more than one column. Through experiments and complexity analysis we see that block variants converge much fast than their ``single column'' counterparts. Furthermore, the   main computational bottleneck, calculating $S^\top A$, can be completely amortized using multi-thread matrix-matrix products. This feature has been explored in implementing parallel coordinate descent methods~\cite{PCDM}. 
%As shown later on, the coordinate descent is an instance of~\eqref{C1eq:sketchproj}.
 In a distributed computing setting, the independence of the sampling of the matrix $S$ allows for distributed implementations, which has already been explored again for coordinate descent methods~\cite{Richtarik2013a,NIPSdistributedSDCA,Maymounkov2010}.

Thus there are a number of clear advantages in using the sketch-and-project methods, as compared with Krylov methods. This discussion has only been a small window into the advantages of randomized methods for solving linear systems. But randomized methods are having a profound impact on other related fields such as optimization methods, in particular, in minimizing partially separable functions, randomized methods are considered the state-of-the-art due, in part, to their fast convergence.  Also in numerical linear algebra, with new randomized methods for solving least squares that, according to the experiments in~\cite{Maymounkov2010}, outperform the long-standing benchmark LAPACK.
%Krylov methods require stringent conditions to arrive at a short recurrence. Our stochastic method will require very lax conditions.
%

%
% While coordinate descent converges with
% 
\section{Tools of the Trade} \label{C1sec:tools}

In this section we present several lemmas concerning pseudoinverses and projections. Though these lemmas are elementary in nature we include them for completeness since they are called upon several times throughout the thesis.

%, which are used several times throughout the thesis, and thus are included for completion. 

\subsection{Pseudoinverse}\label{C1subsec:pseudo}
% We use the pseudoinverse extensively, thus we provide the definition and prove several properties which we will use later on in the thesis.
   The pseudoinverse matrix was introduced by
Moore~\cite{Moore1920} and Penrose~\cite{Penrose1955} in their pioneering work, though our exposition and definition follows that of~\cite{Desoer1963}.

The following lemma is a standard result in linear algebra required in defining the pseudoinverse and projections.
\begin{lemma}\label{lem:WGW}For any matrix $W$ and symmetric positive definite matrix $G$,
\begin{equation}\label{eq:8ys98hs}\Null{W} = \Null{W^\top G W}\end{equation}
and
\begin{equation}\label{eq:8ys98h986ss}\myRange{W^\top } = \myRange{W^\top G W}.\end{equation}
%furthermore
%\begin{equation} \label{eq:nullrangeT}\Null{W} \oplus \myRange{W^\top G W} = \dim{\left(\mathbf{Domain}(W)\right)}.
%\end{equation}
\end{lemma}
\begin{proof} 
In order to establish \eqref{eq:8ys98hs}, it suffices to show the inclusion $\Null{W} \supseteq \Null{W^\top G W}$ since the reverse inclusion trivially holds. Letting $s\in \Null{W^\top G W}$, we see that $\|G^{1/2}Ws\|^2=0$, which implies $G^{1/2}Ws=0$. Therefore, $s\in \Null{W}$. Finally, \eqref{eq:8ys98h986ss} follows from \eqref{eq:8ys98hs} by taking orthogonal complements. Indeed, $\myRange{W^\top}$ is the orthogonal complement of $\Null{W}$ and $\myRange{W^\top G W}$ is the orthogonal complement of $\Null{W^\top G W}$.
% The final result~\eqref{eq:nullrangeT} follows by combining~\eqref{eq:8ys98h986ss} and the fundamental theorem of linear algebra, that is $\Null{W} \oplus \myRange{W^\top} = \dim{\left(\mathbf{Domain}(W)\right)}$.
\end{proof}
  
 The pseudoinverse is a matrix that shares many properties with the inverse matrix.
Given a real matrix $M \in \R^{m\times n},$ when the nullspace of $M$ contains a nonzero vector then $M$, seen as a linear transformation 
 \[M: \R^n \mapsto \myRange{M},\]
  is not injective and thus $M$ has no inverse. But we can construct a \emph{pseudoinverse} of $M$.
The pseudoinverse is constructed by considering a restriction of $M$ that is injective and invertible, and then extending this restriction. Specifically, consider the restriction 
\[M|_{\myRange{M^\top}}:\myRange{M^\top } \mapsto \myRange{M}. \]
 Note that this restriction  is defined on the orthogonal complement of the nullspace of $M$, and thus removes the ``troublesome'' subspace that prevents $M$ from being invertible.
 Indeed, this restriction is invertible since $\myRange{MM^\top }=\myRange{M}$ and $\Null{MM^\top }=\Null{M^\top }$ by Lemma~\ref{lem:WGW}, thus the restriction is surjective and injective.
 The following extension of the inverse of this restriction is what we call the pseudoinverse, see Definition~\ref{def:pseudoinverse}.% This definition was taken from~\cite{Desoer1963}
\begin{definition} \label{def:pseudoinverse} Let $M\in \R^{m\times n}$ be any real matrix. $M^{\dagger}\in \R^{n\times m}$ is said to be the pseudoinverse if
\begin{enumerate}[i)]
\item $M^{\dagger} Mx = x $ for all $x \in \myRange{M^\top }.$
\item $M^{\dagger} x = 0$ for all $x \in \Null{M^\top }.$
%\item If $x \in \Null{M}$ and $y \in \myRange{M^\top }$ then $M^{\dagger}(x+y)=M^{\dagger}x +M^{\dagger}y.$
\end{enumerate}
\end{definition}
Item $i)$ defines $M^{\dagger}$ on
\[ \myRange{MM^\top } \overset{\eqref{eq:8ys98h986ss}}=\myRange{M},\] and item $ii)$ defines $M^{\dagger}$ on $\Null{M^\top }$. Thus the two items together define $M^{\dagger}$ uniquely over $\myRange{M}\oplus \Null{M^\top } = \R^m.$

For the original, and equivalent, definition of pseudoinverse see~\cite{Moore1920} and Penrose~\cite{Penrose1955}. Alternatively, for a definition of pseudoinverse based on the SVD decomposition see Section 5.2.2 in~\cite{Golub2013}.
%and item $3.$ through the decomposition $\R^n = \myRange{M} \oplus \Null{M^\top }$ defines $M^{\dagger}$ \emph{uniquely}. 
%$\myRange{M^\top }$ $\myRange{MM^\top } = \myRange{M}$ (see Lemma~\ref{lem:WGW}).
% Finally $M^{\dagger}$ is a \emph{linear operator}. This can be seen by decomposing every $z \in \R^n$ into $z = My +x$ where $y \in \myRange{M^\top }$ and $x \in \Null{M^\top }.$ With such a decomposition it is easy to verify additivity and homogeneity.  

We now collect the properties of the pseudoinverse that we use through the thesis in a sequence of lemmas.
% We include the proof  
\begin{lemma} \label{lem:pseudo} $ M M^\dagger M = M$
\end{lemma}
\begin{proof}
Let $z \in \R^n$ and consider the decomposition $z =  y +x $ where $y \in \myRange{M^\top }$ and $x \in \Null{M}.$ By item i) of Definition~\ref{def:pseudoinverse} we have 
\[M M^\dagger M z = M y =M(y+x)=Mz.\]
\end{proof}
\begin{lemma} \label{lem:pseudosym} If $M$ is symmetric then $M^{\dagger}$ is symmetric.
\end{lemma}
\begin{proof}
Let $z_1,z_2 \in \R^n$ and consider the decompositions $z_1 = M y_1 +x_1 $ and $z_2 = M y_2 +x_2$ where $y_1, y_2 \in \myRange{M}$ and $x_1,x_2 \in \Null{M},$ which by the symmetry of $M$ always exist.
It follows that
\begin{eqnarray*}
z_1^\top  (M^{\dagger})^\top  z_2 &=&  (M^{\dagger}z_1)^\top  z_2\\
&\overset{\text{Definition~\ref{def:pseudoinverse} item ii)}}{=}& (M^{\dagger}My_1)^\top  z_2\\
&\overset{\text{Definition~\ref{def:pseudoinverse} item i)}}{=}& y_1^\top  z_2\\
& = & y_1^\top My_2.
\end{eqnarray*}
and 
\begin{eqnarray*}
z_1^\top  M^{\dagger} z_2 
&\overset{\text{Definition~\ref{def:pseudoinverse} item ii)}}{=}& z_1^\top  M^{\dagger}My_2\\
&\overset{\text{Definition~\ref{def:pseudoinverse} item i)}}{=}& z_1^\top  y_2\\
& = & y_1^\top My_2,
\end{eqnarray*}
thus $(M^{\dagger})^\top  = M^{\dagger}.$
%If $M$ is symmetric then item ii) of Lemma~\ref{lem:pseudo} states that $M^{\dagger}=(M^{\dagger})^\top .$
\end{proof}

\begin{lemma} \label{lem:pseudoposdef} If $M$ is symmetric positive semidefinite then $M^{\dagger}$ is symmetric positive semidefinite.
\end{lemma}
\begin{proof}
For any $z \in \R^n$, consider the decomposition $z = M y +x $ where $y \in \myRange{M}$ and $x \in \Null{M},$ which by the symmetry of $M$ always exists. It follows that
\[z^\top  M^{\dagger} z \overset{\text{Definition~\ref{def:pseudoinverse} item ii)}}{=} y^\top M M^{\dagger}My \overset{\text{Definition~\ref{def:pseudoinverse} item i)}}{=} y^\top My \geq 0, \]
which shows that $M^{\dagger}$ is positive semidefinite. The symmetry of $M^{\dagger}$ follows from Lemma~\ref{lem:pseudosym}. 
\end{proof}

\begin{lemma}\label{lem:pseudoproj}
 The matrix $M^{\dagger}M$ projects orthogonally onto $ \myRange{M^\top }$ and along $\Null{M}.$ 
%Consequently $M^{\dagger}M = (M^{\dagger}M)^\top  = M^\top  (M^{\dagger})^\top .$
\end{lemma}
\begin{proof} Consider the orthogonal decomposition $z = y +x$ where $y \in \myRange{M^\top }$ and $x \in \Null{M}.$ Then \[M^{\dagger}M z = M^{\dagger}M y = y,\]
where we used  item ii) then item i) of Definition~\ref{def:pseudoinverse}. The result now follows by observing that $\myRange{M^\top }$ and $\Null{M}$ are orthogonal complements.\end{proof}
 
\begin{lemma} \label{lem:pseudoleastnorm}
Consider the consistent linear system $Mx=d$ where $M,x$ and $d$ are of conforming dimensions. It follows that
\begin{equation}\label{eq:pseudoleastnorm}M^{\dagger}d = \arg \min_x \norm{x}_2^2 \quad \mbox{subject to} \quad Mx=d.
\end{equation} 
\end{lemma}
\begin{proof}
As $d \in \myRange{M}$ we have from Lemma~\ref{lem:pseudo} that $MM^{\dagger}d =d$. Using the change of variables $z = x-M^{\dagger}d$ in~\eqref{eq:pseudoleastnorm}
gives
\begin{equation}\label{eq:a3fa3fa}
z^*\eqdef  \arg\min_z \norm{z+M^{\dagger}d}_2^2, \quad \mbox{subject to} \quad z \in \Null{M}. 
\end{equation}
By Lemma~\ref{lem:pseudoproj} we have that $M^{\dagger}d \in \myRange{M^\top } = \Null{M}^{\perp}.$ Consequently  
\[\norm{z+M^{\dagger}d}_2^2 = \norm{z}_2^2 + \norm{M^{\dagger}d}_2^2 \geq \norm{z}_2^2,\]
 thus the minimum $z^*$ of~\eqref{eq:a3fa3fa} is achieved at $z^*=0,$ from which it follows that the minimum of~\eqref{eq:pseudoleastnorm} is achieved at $x=z^*+ M^{\dagger}d = M^{\dagger}d.$
\end{proof}

\subsection{Projection matrices}

The following proposition is a variant of a standard  result of linear algebra (which is often presented in the $B=I$ case). While the results are folklore and easy to establish, in the proof of our main theorems we need certain details which are hard to find in textbooks on linear algebra, and hence hard to refer to. For the benefit of the reader, we include the  detailed statement and proof.

\begin{proposition} [Decomposition and Projection]\label{prop:decomposition} Let $M \in \R^{m\times n}$ by a real matrix and $B \in \R^{n\times n}$ a symmetric positive definite matrix. Each $x\in \R^n$ can be decomposed in a unique way as $x = s(x) + t(x)$, where $s(x)\in \myRange{B^{-1}M^\top}$  and $t(x)\in \Null{M}$. Moreover, the decomposition can be computed explicitly as
\begin{equation}
\label{eq:98hs8hss}s(x) =  \arg \min_{s} \left\{ \|x-s\|_B \;:\; s\in \myRange{B^{-1}M^\top} \right\}=  B^{-1} Z_M x \end{equation}
and
\begin{equation}
\label{eq:98hs8htt}t(x) = \arg \min_{t} \left\{ \|x-t\|_B \;:\; t\in \Null{M} \right\}= (I - B^{-1}Z_M) x,\end{equation}
where
\begin{equation}\label{eq:Z_M}Z_M\eqdef M^\top (MB^{-1}M^\top)^\dagger  M.\end{equation}
Hence, the matrix $B^{-1}Z_M$ is a projection in the $B$-norm onto $\myRange{B^{-1}M^\top}$, and $I-B^{-1}Z_M$ is a projection in the $B$-norm onto $\Null{M}$. Moreover, for all $x\in \R^n$ we have $\|x\|_B^2 = \|s(x)\|_B^2 + \|t(x)\|_B^2$, with
\begin{equation}\label{eq:iuhiuhpp}\|t(x)\|_B^2 = \|(I-B^{-1}Z_{M})x\|_B^2 = x^\top (B-Z_{ M}) x\end{equation}
and
\begin{equation}\label{eq:iuhiuhppss}\|s(x)\|_B^2 = \|B^{-1}Z_{M} x\|_B^2 = x^\top Z_{M} x.\end{equation}
Finally, 
\begin{equation}\label{eq:ugisug7sss}\Rank{M} = \Tr{B^{-1}Z_M}.\end{equation}
\end{proposition}

\begin{proof} Fix arbitrary $x\in \R^n$. We first establish existence of the decomposition. By Lemma~\ref{lem:WGW} applied to $W=M^\top$ and $G=B^{-1}$ we know that there exists $u$ such that $Mx = MB^{-1}M^\top u$. Now let $s = B^{-1}M^\top u$ and $t = x-s$. Clearly, $s\in \myRange{B^{-1}M^\top}$ and $t\in \Null{M}$.  For  uniqueness, consider two decompositions: $x = s_1+ t_1$ and $x=s_2 + t_2$. Let $u_1,u_2$ be vectors such that $s_i = B^{-1}M^\top u_i$, $i=1,2$. Then $MB^{-1}M^\top(u_1-u_2)=0$. Invoking Lemma~\ref{lem:WGW} again, we see that $u_1-u_2\in \Null{M^\top}$, whence $s_1 = B^{-1}M^\top u_1 = B^{-1}M^\top u_2 = s_2$. Therefore, $t_1 = x - s_1 = x-s_2 = t_2$, establishing uniqueness.

Note that $s = B^{-1}M^\top y$, where $y$ is any solution of the optimization problem 
\[\min_y \tfrac{1}{2}\|x-B^{-1}M^\top y\|_B^2.\]
The first order necessary and sufficient optimality conditions are $Mx = MB^{-1}M^\top y$. In particular, we may choose $y$ to be the least norm solution of this system, which by Lemma~\ref{lem:pseudoleastnorm} is given $y=(MB^{-1}M^\top)^\dagger Mx$, from which \eqref{eq:98hs8hss} follows. The variational formulation \eqref{eq:98hs8htt}  can be established in  a similar way, again via first order optimality conditions (note that the closed form formula \eqref{eq:98hs8htt} also directly follows from \eqref{eq:98hs8hss} and the fact that $t = x - s$). 

Next,  since $x=s+t$ and $s^\top B t = 0$,  
\begin{equation}\label{eq:09u0hss}
\|t\|_B^2=(t+s)^\top B t = x^\top B t \overset{\eqref{eq:98hs8htt}}{=} x^\top B (I-B^{-1}Z_{ M})x = x^\top (B - Z_{M}) x
\end{equation}
and
\[ \|s\|_B^2 = \|x\|_B^2 - \|t\|_B^2  \overset{\eqref{eq:09u0hss}}{=} x^\top Z_M x.\]

It only remains to establish \eqref{eq:ugisug7sss}. Since  $B^{-1}Z_M$ projects onto $\myRange{B^{-1}M^\top}$ and since the trace of a projection is equal to the dimension of the space they project onto, we have $\Tr{B^{-1}Z_M} = \dim(\myRange{B^{-1}M^\top}) = \dim(\myRange{M^\top}) = \Rank{M}$.%\qed
\end{proof}
 %\subsection{The $Z$ Matrix}

All the iterative methods presented in the thesis are based on projections. In particular, the projections that govern most of the iterative methods here are constructed from the  
matrix
\begin{equation}\label{eq:Z-C1def}
Z \eqdef Z_{S^\top A} \overset{\eqref{eq:Z_M}}{=} A^\top S(S^\top AB^{-1}A^\top S)^{\dagger}S^\top A.
\end{equation}
%The importance of $Z$ is such that the spectrum of $Z$ (and its spectrum in expectation) will determine the convergence of the methods presented here. 
We now collect in the following lemma several properties pertaining to $Z$ that are repeatedly used throughout the thesis.

%use Proposition~\ref{prop:decomposition} and Lemma~\ref{lem:pseudoposdef} to prove several properties pertaining to $Z$ that are repeatedly used throughout the thesis.

\begin{lemma}\label{ch:one:lem:Z}
The matrix $Z$ defined in~\eqref{eq:Z-C1def} is symmetric positive semidefinite. Furthermore 
$B^{-1}Z$ is a projection with respect to the $B$--norm 
such that
\begin{equation}\label{eq:BZprojdef}
\myRange{B^{-1} Zx}  =  \myRange{B^{-1}A^\top S}  \quad \text{and}  \quad \myRange{I-B^{-1}Z}  = \Null{S^\top A},
\end{equation}
and $B^{-1/2}ZB^{-1/2}$ is a projection with respect to the standard Euclidean geometry, consequently
\begin{equation}\label{eq:B12ZB12proj}
\norm{(I-B^{-1/2} Z B^{-1/2})x}_2^2 = \dotprod{(I-B^{-1/2} Z B^{-1/2})x,x}, \quad \forall x \in \R^n,
\end{equation}
and
\begin{equation}\label{eq:B12ZB12trace}
\Tr{B^{-1/2}ZB^{-1/2}} = \Rank{A^\top S}.
\end{equation} 
%\begin{equation}\label{eq:B12ZB12proj}
%B^{-1/2} Z B^{-1/2}x \, \in \, \myRange{B^{-1/2}A^\top S} \, \quad \text{and} \, \quad (I-B^{-1/2}ZB^{-1/2})x \in \Null{S^\top AB^{-1/2}},
%\end{equation} 
% for all $x \in \R^n.$
\end{lemma}
\begin{proof}
First note that $(B^{-1/2}A^\top S)^\top B^{-1/2}A^\top S =S^\top AB^{-1}A^\top S$ is symmetric positive semidefinite, and consequently by Lemma~\ref{lem:pseudoposdef} the matrix $(S^\top AB^{-1}A^\top S)^{\dagger}$ is also symmetric positive semidefinite. Thus there exists $G$ such that $GG^\top  =(S^\top AB^{-1}A^\top S)^{\dagger} $ and consequently $ (A^\top S G)( A^\top S G)^\top  = Z$ which proves that $Z$ is symmetric positive semidefinite. 

%First note that $(S^\top AB^{-1}A^\top S)^{\dagger} = (B^{-1/2}A^\top S)^{\dagger}$
%\[(S^\top AB^{-1}A^\top S)^{\dagger} = ((B^{-1/2}A^\top S)^\top B^{-1/2}A^\top S)^{\dagger}  = ((B^{-1/2}A^\top S)^\dagger)^\top B^{-1/2}A^\top S\]
%Note that $Z = (A^\top S G)( A^\top S G)^\top$ with $G =$

By Lemma~\ref{prop:decomposition} (with $M=S^\top A$) we have that $B^{-1}Z$ is a projection, and~\eqref{eq:BZprojdef} follows  by~\eqref{eq:98hs8hss} and~\eqref{eq:98hs8htt}. Again by Lemma~\ref{prop:decomposition} (with $M=S^\top AB^{-1/2}$ and $B=I$) we have that $B^{-1/2}ZB^{-1/2}$ projects orthogonally onto $\myRange{B^{-1/2}A^\top S}$, whence~\eqref{eq:B12ZB12proj} and~\eqref{eq:B12ZB12trace} follow from~\eqref{eq:iuhiuhpp} and~\eqref{eq:ugisug7sss}, respectively.
\end{proof}

\subsection{Random variables and the random matrix $S$}
As explained in Section~\ref{secChOne:whats}, the methods proposed in this thesis depend on a random matrix $S\in \R^{m \times q}$. Consequently the iterates of these methods are random variables.
Throughout the thesis we make little to no assumption on the distribution of $S$, and unless explicitly stated, the reader should assume that $S$ is a random matrix in the most general sense.  Here we formalize what is a random variable and what is a random matrix in the most general sense, that is, in the probability measure sense.
 For the reader that is not familiar with probability spaces and measure theory, we suggest the book~\cite{williams1991probability} as quick an enjoyable introduction.

To formalize the notion of a random variable we need the definition of a  \emph{probability space.} A  probability space $(\Omega,\mathcal{F},P)$ is defined by three objects:
\begin{enumerate}
\item The $\Omega$ is a given set known as the \emph{sample space}. It contains all the possible~\emph{outcomes} (elements).
\item The $\mathcal{F}$ is a set of subsets of $\Omega.$ Specifically, it is a $\sigma$--algebra over $\Omega.$ It contains all the possible \emph{events} (subsets) we would like to consider.
\item The $P$ is a function that maps from $\mathcal{F}$ to $[0, \, 1].$  That is, given an event $E \in \mathcal{F}$ it return the probability $P(E) \in [0, \, 1]$ of $E$ occurring. Moreover, $P$ is a probability measure and thus $P(\Omega) =1, P(\emptyset) =0$ and $P$ satisfies the  countable additivity property~\cite{williams1991probability}.
\end{enumerate} 

Often one is not interested in the probability space itself, but in functions over this probability space called random variables. Consider the map $r: \Omega \rightarrow \R.$ We say $r$ is a \emph{random variable} when
\[ \{\omega \, : \, r(\omega)\leq a\} \subset \mathcal{F} \quad \quad \forall a \in \R.\]
An equivalent statement is as follows: The function $r$ is a random variable if the inverse image of $r$ over the interval $(-\infty, a]$ is contained in $\mathcal{F}$ for every $a \in \R.$
 
 A random matrix is simply a matrix valued map where each element is a random variable. That is, consider a map $S: \Omega \rightarrow \R^{m\times q}$ where $m,q\in \N.$ We say that $S$ is a \emph{random matrix} when each element of $S$ is a random variable. For brevity, and as is customary, we use $S \in \R^{m\times q}$ as a shorthand for $S(\omega) \in \R^{m\times q}$ for all $\omega \in \Omega.$  A \emph{random vector} is a random matrix that has only one column or one row. 
We now provide an example of a random matrix. Note that this example, and in fact all the examples in this thesis, are simple enough as to not require this formal probability measure context.
 
% Although when stated in this thesis that $S$ is a random matrix, it is the aforementioned definition we refer to, our examples are simple enough as to not require this formal probability measure construct. For instance, our most common example is when $S$ is a unit coordinate vector selected uniformly at random. 

 \paragraph{Example:} Let $e_i \in \R^{m}$ be the $i$th coordinate vector. Let $S = e_i$ with probability $1/m$ for all $i \in \{1,\ldots,m\}.$ In other words, $P(S =e_i) = 1/m$ for all $i \in \{1,\ldots,m\}.$ We will now show that $S$ is a random matrix by constructing a suitable probability space. Let $\Omega = \{1,\ldots, m\}$, let $\mathcal{F} = 2^\Omega$ be the power set of $\Omega$ and let $P: \mathcal{F} \rightarrow [0, 1] $ be any probability measure that satisfies $P(\{i\}) = 1/m$ for $i=1,\ldots, m.$  Then the map defined by $S(i) = e_i$ is our desired random matrix.  
 
% The level of formalism presented here, in terms of probability spaces, will not be required...
% Throughout this thesis, we will not make explicit use or mention of the probability space. 
%Most of the main theorems in this thesis 
%\[ S^{-1}(A) \in \mathcal{F}, \quad \forall A \in \mathcal{B}(\R^{m\times q}), \]
%%\[ \{ S^{-1}(A)\, : \, A \in \mathcal{B}(\R^{m\times q})  \} \subset \mathcal{F},\]
%where $\mathcal{B}(\R^{m\times q})$ is the Borel $\sigma$--algebra. In other words, we say that $S$ is a random matrix when the inverse image of any ``reasonable'' subset of matrix space (a Borel set) is in the event space.

\subsection{Convergence of a random sequence} \label{sec:convrandseq}

As the methods presented in the thesis depend on a random matrix $S$
 the iterates of our methods are themselves random variables. 
To guarantee that the iterates converge to the desired solution we need to establish the convergence of a sequence of random variables. Throughout the thesis we use two notions of the convergence of random variables; the convergence of the norm of the expectation and the convergence of the expected norm, which we describe here. 

Consider a sequence of random matrices $(Y^k)_k$ on $\R^{m\times q}.$
Let $\dotprod{\cdot, \cdot}$ and $\norm{Y}^2 = \dotprod{Y,Y}$ be an inner product and induced norm, respectively.
  We say that the norm of the expectation of $(Y^k)_k$ converges to zero with rate $\rho \in [0,\, 1)$ if
\begin{equation}\label{ch:one:eq:normexpconv}
\norm{\E{Y^k}} \leq \rho^k \norm{Y^0}.
\end{equation}
 Furthermore, from~\eqref{ch:one:eq:normexpconv} we see that $\E{Y^k} \rightarrow 0,$ and thus the sequence converges in expectation to zero.

We say that the expected norm of $(Y^k)_k$ converges to zero with rate $\rho$ if
\begin{equation}\label{ch:one:eq:expnormconv}
\E{\norm{Y^k}^2} \leq \rho^k \norm{Y^0}^2.
\end{equation}
Note that the order of the expectation operator and the norm are now exchanged in relation to~\eqref{ch:one:eq:normexpconv}.
The convergence of the expected norm implies the convergence of the norm of expectation, as we prove in Lemma~\ref{lem:convrandvar}. In this lemma we also show that the convergence~\eqref{ch:one:eq:expnormconv} implies \emph{convergence in probability}.
We say that $Y^k$ converges in probability to zero if for every $\epsilon>0$ we have that
\begin{equation} \label{eq:convinprob}
 \lim_{k\rightarrow \infty}  \mathbb{P}\left(\norm{Y^k}^2 \geq \epsilon \norm{Y^0}^2 \right) =0. 
\end{equation} 

%We also prove shows that the convergence of expected norm error implies several other notions of convergence of random variables (which we define in the lemma).  
\begin{lemma} \label{lem:convrandvar}
 The convergence of the expected norm~\eqref{ch:one:eq:expnormconv} implies  
 the convergence of the norm of the expectation~\eqref{ch:one:eq:normexpconv} as can be seen through the equality
\begin{equation}\label{eq:convequality}
 \E{\big\| Y^k \big\|^2} = \big\|\E{Y^k}\big\|^2 + \E{\big\| Y^k - \E{Y^k}\big\|^2}.\end{equation}
 Furthermore, the convergence of the expected norm~\eqref{ch:one:eq:expnormconv} implies convergence in probability.
%which shows that the converg  and the convergence of the norm of the expected error.
\end{lemma}
\begin{proof} Let $(Y^k)_k$ be a sequence of random vectors that converges to zero according to~\eqref{ch:one:eq:expnormconv}.
The convergence of the norm of expectation follows from the equality
\begin{align}
\big\|\E{Y^k}\big\|^2 &= \big\|\E{Y^k}\big\|^2  + \E{\norm{Y^k}^2} - \E{\norm{Y^k}^2} \nonumber \\
&=\E{\norm{Y^k}^2} - \left(\E{\norm{Y^k}^2} -2\E{\dotprod{Y^k,\E{Y^k}}} 
+ \big\|\E{Y^k}\big\|^2\right) \nonumber\\
& = \E{\big\|Y^k \big\|^2}  - \E{\big\| Y^k - \E{Y^k}\big\|^2}. \label{C1eq:convcovar}
\end{align}
Indeed, since $\E{\big\| Y^k - \E{Y^k}\big\|^2} \geq 0$ we have that
\[\big\|\E{Y^k}\big\|^2 \leq \E{\big\| Y^k\big\|^2}   \overset{\eqref{ch:one:eq:expnormconv}}{\leq} \rho^k \norm{Y^0}^2. \]
Consequently the norm of the expected error converges with rate $\sqrt{\rho}.$ Finally, let $\epsilon>0.$ Using Markov's inequality we have
\begin{align}\label{eq:probconv}
 \mathbb{P}(\norm{Y^k}^2 \geq \epsilon \norm{Y^0}^2) &\leq \frac{\E{\norm{Y^k}^2}}{\epsilon \norm{Y^0}^2}
  \overset{\eqref{ch:one:eq:expnormconv}}{\leq}  \frac{\rho^k}{\epsilon}.
\end{align}
Thus $ \mathbb{P}(\norm{Y^k}^2 \geq \epsilon \norm{Y^0}^2 ) \rightarrow 0$ as $k\rightarrow \infty.$
\end{proof}
%For every method we present in this thesis, we give the proof of the convergence of the expected norm.  Aside from the implications in Lemma~\ref{lem:convrandvar},
%Thus the convergence of the expected norm is a rather strong notion of convergence of random variables. 
%Our motivation in giving this lemma is to show the strength 

In every chapter that follows, we will present several convergence results of random sequences. In particular, in Chapter~\ref{ch:linear_systems} we prove the convergence of the expected norm~\eqref{ch:one:eq:expnormconv} and the convergence of the norm of the expectation~\eqref{ch:one:eq:normexpconv} of a sequence of random vectors $Y^k =x^k-x^*$. In 
Chapter~\ref{ch:inverse} we prove analogous convergence results of a sequence of random matrices $Y^k = X_k-A^{-1}.$ Thus Lemma~\ref{lem:convrandvar} is important as it sheds light on how these two types of convergence~\eqref{ch:one:eq:expnormconv} and~\eqref{ch:one:eq:normexpconv} are related.

% The convergence of the expected norm~\eqref{ch:one:eq:expnormconv} implies  
% the convergence of the norm of the expectation
% 
% including  
%the convergence of 
%
%Lemma~\ref{lem:convrandvar} sheds lights on the convergence result
%
% 
% 
%Lemma~\ref{lem:convrandvar} can be applied to the case when $Y^k$ are sequences of random vectors or random matrices. Such as in Chapters~\ref{ch:linear_systems} and~\ref{ch:SDA}, where we are interested in the convergence of a sequence of random vectors $Y^k = x^k-x^*$, and in Chapter~\ref{ch:inverse} where we are interested in the convergence of a sequence of random matrices $Y^k = X_k-A^{-1}.$

Convergence according to~\eqref{ch:one:eq:expnormconv} is also commonly referred to as \emph{linear convergence}.  This is because, as explained in the next Section~\ref{secChone:itercompl}, the number of iterations required to reach a certain precision grows linearly and proportionally to $\left.1 \right/ (1-\rho).$ 
Another common synonym to linear convergence, that we sometimes use here, is to say that $Y^k$ converges \emph{exponentially fast} to  zero. This is because the expected norm of $Y^k$ decreases according to the exponential function $\rho^k.$ 
%Which synonym is used often depends on which community. In the applied mathematics and optimization community, the term \emph{linear convergence} is  

 Note that if $Y^k$ is a random vector defined on $\R^m$ with the standard Euclidean inner product then
$\E{\left\| Y^k - \E{Y^k}\right\|_2^2} =\sum_{i=1}^m\E{ (Y^k_i - \E{Y^k_i})^2}  = \sum_{i=1}^m \mathbf{Var}(Y^k_i)$,
where $z_i^k$ denotes the $i$th element of $Y^k.$ Thus, in this case, the equality~\eqref{eq:convequality} also shows that if the norm of expectation converges and the variance of $Y^k_i$  converges to zero for $i=1,\ldots, m$, then the expected norm converges.

\subsection{Iteration complexity} \label{secChone:itercompl}

Both types of convergence~\eqref{ch:one:eq:normexpconv} and~\eqref{ch:one:eq:expnormconv}  can be recast as iteration complexity bounds using the following lemma. With this lemma we can stipulate an lower bound on how many iterations are required to bring the sequence within an $\epsilon>0$ relative distance of its limit point. 

\begin{lemma} \label{lem:itercomplex}
Consider the sequence $(\alpha_k)_k \in \R_+$ of positive scalars that converges to zero according to
\begin{equation}\label{eq:alphaconv} \alpha_k \leq \rho^k \, \alpha_0,\end{equation}
where $\rho \in [0, 1).$
For a given $1>\epsilon >0$ we have that
\begin{equation} \label{C1eq:itercomplex}k\geq \frac{1}{1-\rho} \log\left(\frac{1}{\epsilon}\right)  \quad \Rightarrow \quad \alpha_k \leq \epsilon\, \alpha_0.
\end{equation}
\end{lemma}
\begin{proof}
First note that if $\rho=0$ the result follows trivially. Assuming $\rho \in (0,\,1)$,
rearranging~\eqref{eq:alphaconv} and applying the logarithm to both sides gives
\begin{equation} \label{eq:logalphaconv}
 \log\left(\frac{\alpha_0}{\alpha_k}\right) \geq   k  \log\left(\frac{1}{\rho}\right).
 \end{equation}
Now using that
\begin{equation}\label{eq:logineq}
\frac{1}{1-\rho} \log\left(\frac{1}{\rho}\right) \geq 1,
\end{equation}
for all $\rho \in (0,1)$ and assuming that
\begin{equation}\label{eq:kiterassump}
k\geq \frac{1}{1-\rho} \log\left(\frac{1}{\epsilon}\right) ,
\end{equation}
 we have that
 \begin{eqnarray*}
  \log\left(\frac{\alpha_0}{\alpha_k}\right) & \overset{\eqref{eq:logalphaconv}}{\geq}&
  k \log\left(\frac{1}{\rho}\right)  \\
  & \overset{\eqref{eq:kiterassump}}{\geq} &  \frac{1}{1-\rho} \log\left(\frac{1}{\rho}\right) \log\left(\frac{1}{\epsilon}\right)\\
  & \overset{\eqref{eq:logineq}}{\geq} &
 \log\left(\frac{1}{\epsilon}\right) 
 \end{eqnarray*}
%\begin{eqnarray*}
% \log\left(\frac{1}{\epsilon}\right) & \overset{\eqref{eq:logineq}}{\leq} & \frac{1}{1-\rho} \log\left(\frac{1}{\rho}\right) \log\left(\frac{1}{\epsilon}\right)\\
% & \overset{\eqref{eq:kiterassump}}{\leq} & k \log\left(\frac{1}{\rho}\right) 
% \overset{\eqref{eq:logalphaconv}}{\leq}  \log\left(\frac{\alpha_0}{\alpha_k}\right).
%\end{eqnarray*}
Applying exponentials to the above inequality gives~\eqref{C1eq:itercomplex}.
\end{proof}

As an example of the use this lemma, consider the sequence of random vectors $(Y^k)_k$ for which the expected norm converges to zero according to~\eqref{ch:one:eq:expnormconv}. Then applying Lemma~\ref{lem:itercomplex} with $\alpha_k = \E{\norm{Y^k}^2}$ for a given $1>\epsilon>0$ states that
 \[k \geq \frac{1}{1-\rho} \log\left(\frac{1}{\epsilon}\right) \quad \Rightarrow \quad \E{\norm{Y^k}^2} \leq \epsilon\, \norm{Y^0}^2.\]
 
To give further insight into the implications of the convergence of the expected norm,  for a given $1>\epsilon >0$ consider the sequence $\alpha_0 =1/\epsilon$ and $\alpha^k = 	 \mathbf{P}\left(\norm{Y^k}_2^2 \geq \epsilon \norm{Y^0}_2^2\right)$ for $k \geq 1.$ From~\eqref{eq:probconv} we know that this sequence converges according to $\alpha^k \leq \rho^k \alpha^0.$ We can now use Lemma~\ref{lem:itercomplex} to determine how many iterates are required so that $\norm{Y^k}_2^2 \leq \epsilon \norm{Y^0}_2^2 $ with high probability. Indeed, let $\delta \in (0,\, 1)$ then by Lemma~\ref{lem:itercomplex} we have that 
 \[k \geq \frac{1}{1-\rho} \log\left(\frac{1}{\epsilon\delta}\right)  \quad \Rightarrow \quad \mathbf{P}\left(\norm{Y^k}_2^2 \geq \epsilon  \norm{Y^0}_2^2\right) \leq \delta.\]
  This shows that convergence of the expected norm is almost as good as linear convergence without the expectation, that is, one can guarantee the iterates are relatively close with a high probability  at the cost of only an additional logarithmic growth in the number of iterates. 
  
  %($6/(1\rho)$)  
  
 % (say  $1-\delta =$\%99.9 as opposed to the \%100 of linear convergence)
  
  %As an example, consider the case where we want to guarantee that the solution is within $\epsilon = 10^{-3}$
%   accept the difference being an additional logarithmic growth in the number of iterations.% to achieve  high probability.
  
%   accept the difference being an added term within the logarithm that bounds the number of required iterations. %  That is, excluding logarithmic terms 
%  
%  
%   
   % Thesis Aims and Objectives
%\clearpage
%---------------------------------------------------------------------------------
%	CHAPTER One: Introduction 
%---------------------------------------------------------------------------------
%\onehalfspace   %% official UoE spacing
\chapter[Randomized Iterative Methods for Solving Linear Systems]{Randomized Iterative Methods for Linear Systems}
\chaptermark{Randomized Iterative Methods for Linear Systems}
\label{ch:linear_systems} % label for referring to chapter in other parts of the thesis
{
\epigraph{\emph{Guid gear comes in sma’ bulk}. \\
 %Translation:
  Good things come in small sizes (like sketched linear systems!)
}{Scottish proverb.} 

%It's not whether you win or lose that counts. In fact, nothing counts, and death is coming for us all.}{Jimmy Kimmel, supposedly an old Norwegian saying}
\let\clearpage\relax
\section{Introduction}
}
% Stochastic methods are well suited for efficiently obtaining approximate solutions. Many optimization methods based on  Newton's method only require  quick approximate solutions~\cite{Eisenstat1994b,Dembo1982}, such as interior point methods, augmented Lagrangian methods, Newton-CG and others. Moreover, randomized methods are simple to implement and are often more suited to parallel architectures.

The need to solve linear systems of equations is ubiquitous in essentially all quantitative areas of human endeavour, including industry and science. Linear systems are a central problem in numerical linear algebra, and play an important role in computer science, mathematical computing, optimization, signal processing, engineering, numerical analysis, computer vision, machine learning,  and many other fields.  For instance, in the field of large scale optimization, there is a growing interest in inexact and approximate Newton-type methods for ~\cite{Dembo1982,Eisenstat1994b,Bellavia1998,Zhao2010a,Wang2013,Gondzio2013}, which can benefit from fast subroutines for calculating approximate solutions of linear systems.  In machine learning, applications arise for the problem of finding optimal configurations in Gaussian Markov Random Fields \cite{GMRFbook}, in graph-based semi-supervised learning and other graph-Laplacian problems \cite{Bengio+al-ssl-2006}, least-squares SVMs, Gaussian processes and more. 

In a large scale setting, direct methods can suffer from two shortcomings. First, direct methods often require direct access to individual elements of the system matrix and thus need the system matrix to be stored on RAM. But the dimensions and density of the problem at hand maybe such that the system matrix does not fit on RAM. Second, the complexity of direct methods is of order $O(n^3)$ which can be prohibitively slow when $n$ is large.

%are generally not competitive when compared with iterative approaches.

 While classical iterative methods are deterministic, 
recent breakthroughs suggest that randomization can play a powerful role in the design and analysis of efficient algorithms~\cite{Strohmer2009,Leventhal2010,Needell2010,Drineas2011,Zouzias2013a,Lee2013,Ma2015, Richtarik2015c} which are in many situations competitive  or better than existing deterministic methods.

In this chapter we develop the sketch-and-project family of randomized methods for solving linear systems that are well suited to quickly calculating approximate solutions.

\subsection{Background and related work}

The literature on solving linear systems via iterative methods is vast and has a long history~\cite{Kelley1995,Saad2003}.  For instance, the Kaczmarz method, in which one cycles through the rows of the system and each iteration is formed by projecting the current point to the hyperplane formed by the active row, dates back to the 30's~\cite{Kaczmarz1937}. The Kaczmarz method is just one example of an array of row-action methods for linear systems (and also, more generally, feasibility and optimization problems)  which were studied in the second half of the 20th century~\cite{rowaction1981}. 

Research into the Kaczmarz method was reignited in 2009 by Strohmer and Vershynin~\cite{Strohmer2009}, who gave a brief and elegant proof that a randomized variant thereof enjoys an exponential error decay (also know as ``linear convergence''). This has triggered much research into developing and analyzing randomized linear solvers. 

It should be mentioned at this point that the randomized Kaczmarz (RK) method  arises as a special case (when one considers quadratic objective functions) of the stochastic gradient descent (SGD) method for {\em convex optimization} which can be traced back to the seminal work of  Robbins and Monro's  on stochastic approximation~\cite{RobbinsMonro:1951}. Subsequently, intensive research went into studying various extensions of the SGD method. However, to the best of our knowledge, no complexity results with exponential error decay were established prior to the aforementioned work of  Strohmer and Vershynin~\cite{Strohmer2009}. This is the reason behind our choice of  \cite{Strohmer2009} as the starting point of our discussion.

Motivated by the results of Strohmer and Vershynin~\cite{Strohmer2009}, Leventhal and Lewis~\cite{Leventhal2010} utilize similar techniques to establish the first bounds for {\em randomized coordinate descent} methods for solving systems with positive definite matrices, and systems arising from least squares problems~\cite{Leventhal2010}.  These bounds are similar to those for the RK method. This development was later picked up by the optimization and machine learning communities, and much progress has been made in generalizing these early results in countless ways to various structured  convex optimization problems. For a brief up to date account of the  development in this area, we refer the reader to \cite{Fercoq2013a,ALPHA} and the references therein.

The RK method and its analysis have been further extended to  the least-squares problem~\cite{Needell2010,Zouzias2013a} and the  block setting~\cite{Needell2012,Needell2014}. In~\cite{Ma2015} the authors extend the randomized coordinate descent and the RK methods to the problem of solving underdetermined systems.  The authors of~\cite{Ma2015,Ramdas2014} analyze side-by-side the  randomized coordinate descent  and RK method,  for least-squares, using a convenient notation in order to point out their similarities. Our work takes the next step, by analyzing these, and many other methods, through a genuinely general analysis. Also in the spirit of unifying the analysis of different methods, in~\cite{Oswald2015} the authors provide a unified analysis of iterative Schwarz methods and Kaczmarz methods. 

The use of random Gaussian directions as search directions  in zero-order (derivative-free) minimization algorithm was recently suggested~\cite{Nesterov2011}. Our Gaussian positive definite and Gaussian least-squares in Sections~\ref{C2sec:gaussA} and~\ref{C2sec:gaussATA}, respectively, are special cases of these zero order methods applied to linear systems.
 More recently, Gaussian directions have been combined with exact and inexact line-search into a single {\em random pursuit} framework~\cite{Stich2014a}, and further utilized within a randomized variable metric method~\cite{Stich2014,Stich2015}.
\section{Contributions and Overview}

Given a real matrix $A \in \R^{m \times n}$  and a real vector $b \in \R^m$, in this chapter we consider the linear system
\begin{equation}\label{eq:Axb}Ax =b.\end{equation}
We shall assume throughout that the system is {\em consistent}:  there exists $x^*$ for which $Ax^*=b$.   

We now comment on the main contribution of this chapter.

{\em 1. New method.} We develop a novel, fundamental, and surprisingly simple  {\em randomized iterative method} for solving \eqref{eq:Axb}.

{\em 2. Six equivalent formulations.} Our method allows for several seemingly different but nevertheless equivalent formulations. First, it can be seen as a {\em sketch-and-project} method, in which the system \eqref{eq:Axb} is replaced by its {\em random sketch}, and then the current iterate is projected onto the solution space of the sketched system. We can also view  it as a {\em constrain-and-approximate} method, where we constrain the next iterate to live in a particular random affine space passing through the current iterate, and then pick the point from this subspace which best approximates the optimal solution. Third,  the method can be seen as an   iterative solution of a sequence of random (and simpler) linear equations. The method also allows for a simple geometrical interpretation: the new iterate is defined as the unique intersection of two random affine spaces which are orthogonal complements. The fifth viewpoint gives a closed form formula for the  {\em random update} which needs to be applied to the current iterate in order to arrive at the new one. Finally, the method can be seen as a {\em random fixed point iteration.}

{\em 3. Special cases.} These multiple viewpoints enrich our interpretation of the method, and enable us to draw previously unknown links between several existing algorithms. Our algorithm has two parameters, an $n\times n$ positive definite matrix $B$ defining geometry of the space, and a random matrix $S$.  Through combinations of these two parameters, in special cases our method recovers several well known algorithms.  For instance, we recover the randomized Kaczmarz method of Strohmer and Vershyinin~\cite{Strohmer2009}, randomized coordinate descent method  of Leventhal and Lewis~\cite{Leventhal2010}, random pursuit~\cite{Nesterov2011,Stich2015,Stich2014,Stich2012} (with exact line search), and the stochastic Newton method recently proposed by Qu et al \cite{Qu2015}. However, our method is more general, and leads to i) various generalizations and improvements of the aforementioned methods (e.g., block setup, importance sampling), and ii) completely new methods. Randomness enters our framework in a very general form, which allows us to obtain a {\em Gaussian Kaczmarz method}, {\em Gaussian descent},  and more. 

{\em 4. Complexity: general results.} When $A$ has full column rank, our framework allows us to determine the complexity of these methods using a single analysis. Our main results are summarized in Table~\ref{ch:two:tab:complexity},
where $\{x^k\}$ are the iterates of our method,  $Z$ is a random matrix dependent on the data matrix $A$, parameter matrix $B\in \R^{n\times n}$ and random parameter matrix $S \in \R^{m\times q}$, defined as \begin{equation}\label{eq:Z-first}
Z \eqdef A^\top S(S^\top AB^{-1}A^\top S)^{\dagger}S^\top A,
\end{equation}

where $\dagger$ denotes the (Moore-Penrose) pseudoinverse. For the definition of pseudoinverse, see Section~\ref{C1subsec:pseudo}. Moreover, $\norm{x}_B \eqdef  \sqrt{\dotprod{x,x}_B}$, where  $\dotprod{x,y}_B \eqdef x^\top By$, for all $x,y \in \R^{n}$.

As we shall see later, we will often consider setting $B=I$, $B=A$ (if $A$ is positive definite) or $B=A^\top A$ (if $A$ is of full column rank). In particular, we first
show that the convergence rate  $\rho$ is always bounded between zero and one. We also show that as soon as $\E{Z}$ is invertible (which can only happen if $A$ has full column rank, which then implies that $x^*$ is unique), we have $\rho<1$, and the method converges.
  %, with $\rho$ less than one if and only if $\E{Z}$ is invertible. 
Besides establishing a bound involving the {\em expected norm of the error} (see the last line of Table~\ref{ch:two:tab:complexity}), we also obtain bounds involving the {\em norm of the expected error} (second line of Table~\ref{ch:two:tab:complexity}). Studying the expected sequence of iterates directly is very fruitful, as it allows us to establish an {\em exact characterization} of the evolution of the expected iterates (see the first line of Table~\ref{ch:two:tab:complexity}) through a {\em linear fixed point iteration}. 

% It is known that such an iteration converges for every starting point if and only if the spectral radius of the matrix $I-B^{-1}\E{Z}$ is smaller than 1. That is, the method converges if and only if $\rho<1.$ However, we will show that the spectral radius and the induced $B$-norm of $I-B^{-1}\E{Z}$ are equal. Hence, the rate $\rho$ precisely characterizes the convergence of the method and by studying $\rho$, we can  establish easily interpretable {\em lower bounds}. 

Both of these theorems on the convergence of the method can be recast as iteration complexity bounds by using Lemma~\ref{lem:itercomplex}.
For instance from Theorem~\ref{ch:two:theo:normEconv} in Table~\ref{ch:two:tab:complexity} we observe that for a given $\epsilon >0$ we have that
\begin{equation} \label{ch:two:eq:itercomplex}k \geq \frac{1}{1-\rho} \log\left(\frac{1}{\epsilon}\right) \quad \Rightarrow \quad \norm{\E{x^k-x^*}}_B \leq \epsilon \norm{x^0-x^*}_B.
\end{equation}

% \item  \text{Complexity: lower bounds.} $\left(\Null{A^\top  S}\right) = n-q.$ 

{\em 5. Complexity: special cases.} Besides these generic results, which hold without any major restriction on the sampling matrix $S$ (in particular, it can be either discrete or continuous), we give a specialized result applicable to discrete sampling matrices $S$ (see Theorem~\ref{theo:convsingleS}). In the special cases for which rates are known, our analysis  recovers the existing rates. 

%Furthermore, our complexity bounds hold under very mild assumptions on the way randomness enters the algorithm. With such general bounds, we can now pose questions such as: for a given problem class, what is the {\em optimal distribution} in terms of the complexity bounds?

\begin{table}
\centering
\begin{tabular}{|c|c|}
\hline
& \\
 $\E {x^{k+1} -x^{*}} = \left(I  - B^{-1}\E{Z}\right) \E{x^{k} - x^{*}}  $ & Theorem~\ref{ch:two:theo:normEconv}\\
 & \\
$\norm{\E {x^{k+1} -x^{*}}}_B \leq \rho \; \cdot \; \norm{\E{x^{k} - x^{*}}}_B
 $ & Theorem~\ref{ch:two:theo:normEconv}\\
 & \\
$ \E {\norm{x^{k+1} -x^{*} }_B^2 } \leq \rho \;\cdot\; \E{ \norm{x^{k} - x^{*}}_B^2}$  & Theorem~\ref{ch:two:theo:Enormconv}\\
& \\
 \hline
 \end{tabular}
 \caption{Our main complexity results. The convergence rate is: $\rho = 1- \lambda_{\min}(B^{-1/2}\E{Z}B^{-1/2}).$}
 \label{ch:two:tab:complexity}
 \end{table}

{\em 6. Extensions.} Our approach opens up many avenues for further development and research. For instance, it is possible to extend the results to the case when $A$ is not necessarily of full column rank, which we do in Chapter~\ref{ch:SDA}. Furthermore, as our results hold for a wide range of distributions, new and efficient variants of the general method can be designed for problems of specific structure by fine-tuning the stochasticity to the structure. Similar ideas can be applied to design randomized iterative algorithms for finding the inverse of a very large matrix, which is the focus of Chapter~\ref{ch:inverse}.

\section{One Algorithm in Six Disguises} \label{sec:sixviews}Our method has {\em two parameters}: i) an $n\times n$ positive definite matrix $B$ which is used to define the $B$-inner product and the induced $B$-norm by \begin{equation} \label{eq:B-innerprod}\langle x, y\rangle_{B}\eqdef \langle Bx,y\rangle, \qquad \|x\|_B \eqdef \sqrt{ \langle x, x \rangle_B},\end{equation}
where $\langle \cdot,\cdot \rangle$ is the standard Euclidean inner product,  and ii) a random matrix $S\in \R^{m\times q}$, to be drawn in an i.i.d.\ fashion at each iteration. We stress that we do not restrict the number of columns of $S$; indeed, we even allow $q$ to vary (and hence,  $q$ is a random variable).

\subsection{Six viewpoints}

Starting from $x^k\in\R^n$, our method draws a random matrix $S$ and uses it to generate a new point $x^{k+1}\in\R^n$. As proven at the end of this section, our iterative method can be formulated in {\em six seemingly different but equivalent ways:}

\paragraph{1. Sketching Viewpoint: Sketch-and-Project.} $x^{k+1}$ is the nearest point to $x^k$ which solves a {\em sketched} version of the original linear system:
\begin{equation}
\boxed{\quad x^{k+1} \quad = \quad \arg\min_{x\in \R^n} \norm{x- x^{k}}_B^2 \quad \mbox{subject to} \quad  S^\top  Ax = S^\top  b \quad} \label{ch:two:NF} 
\end{equation}
This viewpoint arises very naturally. Indeed, since the original system \eqref{eq:Axb} is assumed to be complicated, we replace it by a simpler system---a {\em random sketch} of the original system \eqref{eq:Axb}---whose solution set $\{x \;|\; S^\top  Ax = S^\top  b\}$ contains all solutions of the original system. However, this system will typically have many solutions, so in order to define a method, we need a way to select one of them. The idea is to try to preserve as much of the information learned so far as possible, as condensed in the current point $x^k$. Hence, we  pick the solution which is closest to $x^k$. 

\paragraph{2. Optimization Viewpoint: Constrain-and-Approximate.} $x^{k+1}$ is the best approximation of $x^*$ in a random space passing through $x^k$:
\begin{equation}\boxed{\; x^{k+1} \; = \; \arg\min_{x\in \R^n} \norm{x\phantom{^k}- x^{*}}_B^2 \quad \mbox{subject to} \quad x = x^{k} + B^{-1}A^\top S y, \quad y \;\text{is free} \;} \label{ch:two:RF}
\end{equation}
%% Krylov footnote used just after 
%The above step has the following interpretation\footnote{Formulation~\eqref{ch:two:RF} is similar to the framework often used to describe Krylov methods~\cite[Chapter 1]{Liesen2014}, which is
%\[x^{k+1} \eqdef \arg\min_{x\in \R^n} \norm{x- x^{*}}_B^2 \quad \mbox{subject to} \quad x \in x^{0} + \K_{k+1},\]
%where $\K_{k+1}\subset \R^n$ is a $(k+1)$--dimensional subspace.
%Note that the constraint $x \in x^{0}+\K_{k+1}$ is an affine space that contains $x^{0}$, as opposed to $x^{k}$ in our formulation~\eqref{ch:two:RF}.
%The objective $\|x-x^{*}\|^2_B $ is a generalization of the residual, where $B=A^\top A$ is used to characterize minimal residual methods~\cite{Paige1975,Saad1986} and $B=A$ is used to describe the Conjugate Gradients method~\cite{Hestenes1952}. Progress from iteration to the next is guaranteed by using expanding nested search spaces at each iteration, that is, $\K_k \subset \K_{k+1}.$ In our setting, progress is enforced by using $x^{k}$ as the displacement term instead of $x^{0}.$  This also allows for a simple recurrence for updating $x^{k}$ to arrive at $x^{k+1}$, which facilitates the analysis of the method. In the Krylov setting, to arrive at an explicit recurrence, one needs to carefully select a basis for the nested spaces that allows for short recurrence. }

The above step has the following interpretation. We  choose a random affine space containing $x^k$, and constrain our method to choose the next iterate from this space. We then do as well as we can on this space; that is, we pick $x^{k+1}$ as the point which best approximates $x^*$. Note that $x^{k+1}$ does not depend on which solution $x^*$ is used in \eqref{ch:two:RF} (this can be best seen by considering the geometric viewpoint, discussed next).

 \begin{figure}[!h] \centering
\includegraphics[width =7cm]{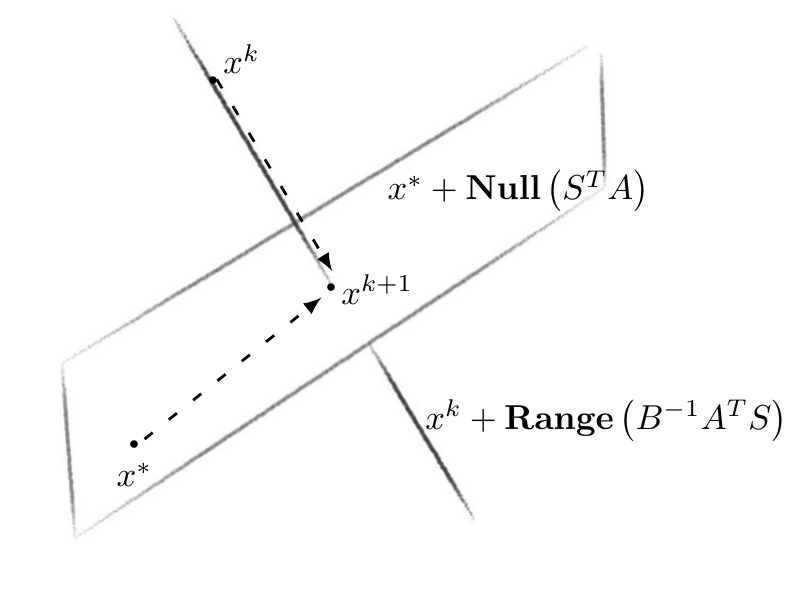}
\caption{\footnotesize The geometry of our algorithm. The next iterate, $x^{k+1}$, arises as the intersection of two random affine spaces: $x^k + \myRange{B^{-1}A^\top  S}$ and $x^* + \Null{S^\top A}$ (see \eqref{eq:geometry}). The spaces are orthogonal complements of each other with respect to the $B$-inner product, and hence $x^{k+1}$ can equivalently be written as the projection, in the $B$-norm, of $x^k$ onto $x^* +\Null{S^\top  A}$ (see \eqref{ch:two:NF}), or the projection of $x^*$  onto $x^k +\myRange{B^{-1}A^\top  S}$ (see \eqref{ch:two:RF}). The intersection $x^{k+1}$ can also be expressed as the solution of a system of linear equations (see \eqref{eq:2systems}). Finally, the new error $x^{k+1}-x^*$ is the projection, with respect to the $B$-inner product, of the current error $x^k-x^*$ onto $\Null{S^\top A}$. This gives rise to a random fixed point formulation (see \eqref{eq:xZupdate}).}
\label{ch:two:fig:proj}
\end{figure}

\paragraph{3. Geometric viewpoint: Random Intersect.} $x^{k+1}$ is the (unique) intersection of two affine spaces:
\begin{equation} \label{eq:geometry}
\boxed{ \quad \{x^{k+1}\} \quad =\quad  \left( x^* + \Null{S^\top  A}\right) \quad \bigcap \quad \left(x^k + \myRange{B^{-1}A^\top  S}\right) \quad}
\end{equation}
First, note that the first affine space above does not depend on the choice of $x^*$ from the set of optimal solutions of \eqref{eq:Axb}. A basic result of linear algebra says that the nullspace of an arbitrary matrix is the orthogonal complement of the range space of its transpose. Hence, whenever we have $h\in \Null{S^\top A}$ and $y\in \R^q$, where $q$ is the number of rows of $S$, then $\dotprod{h, B^{-1}A^\top  S y}_B =\dotprod{h, A^\top  S y} = 0$. It follows that the two spaces in \eqref{eq:geometry} are orthogonal complements with respect to the $B$-inner product  and as such, they intersect at a unique point (see Figure~\ref{ch:two:fig:proj}).
% The orthogonality of these two subspaces is also formally proven in Proposition~\ref{prop:decomposition}.

 \paragraph{4. Algebraic viewpoint: Random Linear Solve.} Note that $x^{k+1}$ is the (unique) solution (in $x$) of a linear system (with variables $x$ and $y$):
\begin{equation}\label{eq:2systems}
\boxed{\quad x^{k+1} \quad = \quad \text{solution of }\quad  S^\top  A x = S^\top  b, \quad x = x^k + B^{-1}A^\top  S y\quad } \end{equation}
This system is clearly equivalent to \eqref{eq:geometry}, and can alternatively be written as:
\begin{equation}\label{eq:view-algebraic}\begin{pmatrix}S^\top  A & 0  \\
B & -A^\top  S\end{pmatrix} \begin{pmatrix}x \\y\end{pmatrix} = \begin{pmatrix}S^\top  b \\ B x^k\end{pmatrix}.\end{equation}
Hence, our method reduces the solution of the (complicated) linear  system \eqref{eq:Axb} into a sequence of (hopefully simpler) random systems of the form \eqref{eq:view-algebraic}.

 \paragraph{5. Algebraic viewpoint: Random Update.} 
 By plugging the second equation in \eqref{eq:2systems} into the first, we eliminate $x$ and obtain the system $(S^\top  A B^{-1}A^\top  S) y = S^\top (b-Ax^k)$. Note that for all solutions $y$ of this system we must have $x^{k+1} = x^k + B^{-1}A^\top  S y$. In particular,  we can choose the solution $y=y^k$ of minimal Euclidean norm, which by Lemma~\ref{lem:pseudoleastnorm} is given by $y^k = (S^\top  A B^{-1}A^\top  S)^{\dagger}S^\top (b-Ax^k)$. This leads to an expression for $x^{k+1}$ with an explicit form of the {\em random update} which must be applied to $x^k$ in order to obtain $x^{k+1}$:
\begin{equation}\label{eq:MP} 
\boxed{\quad x^{k+1} = x^k - B^{-1}A^\top  S(S^\top  A B^{-1}A^\top  S)^{\dagger}S^\top (Ax^k-b) \quad }\end{equation}
In some sense, this form is the standard: it is customary for iterative techniques to be written in the form $x^{k+1}=x^k + d^k$, which is precisely what \eqref{eq:MP}  does.

\paragraph{6. Analytic viewpoint: Random Fixed Point.} 
Note that  iteration~\eqref{eq:MP} can be written as 
\begin{equation}\label{eq:xZupdate}
\boxed{\quad x^{k+1} -  x^* \quad =\quad (I-  B^{-1}Z)(x^{k} -x^*)\quad}
\end{equation} 
where $Z$ is defined in \eqref{eq:Z-first} and where we used the fact that $Ax^{*} =b$. Matrix $Z$ plays a central role in our analysis, and can be used to construct explicit projection matrices of the two projections depicted in Figure~\ref{ch:two:fig:proj}.

The equivalence between these six viewpoints is formally captured in the next statement.

\begin{theorem}[Equivalence]\label{thm:equivalence} The six viewpoints are equivalent: they all produce the same (unique) point $x^{k+1}$.
\end{theorem}
\begin{proof} The proof is simple, and follows  directly from the above discussion. In particular, see the caption of Figure~\ref{ch:two:fig:proj}.
\end{proof}

\subsection{Projection matrices}

The explicit projection matrices of the projections depicted in Figure~\ref{ch:two:fig:proj} can be constructed using the $Z$ matrix.
% 
%In this section we state a few key properties of matrix $Z$. This will shed light on the previous discussion and will also be useful later in the convergence proofs. 
Indeed, recall that $S$ is a $m\times q$ random matrix (with $q$ possibly being random),  and that $A$ is an $m\times n$ matrix. Let us define the random quantity \begin{equation}\dd \eqdef \Rank{S^\top  A}\end{equation} and notice that  $d\leq \min\{q,n\}$,
\begin{equation} \label{eq:dimproj}
\dim \left(\myRange{B^{-1}A^\top  S}\right) = \dd, \qquad \text{and}\qquad \dim\left(\Null{S^\top A}\right) = n-\dd.
\end{equation}

Recall that~\eqref{eq:BZprojdef} shows that $B^{-1}Z$ is a projection onto $\myRange{B^{-1}A^\top  S}$ and along $\Null{S^\top A}.$
This sheds additional light on Figure~\ref{ch:two:fig:proj} as it gives explicit expressions for the associated projection matrices. This also shows that $I-B^{-1}Z$ is a projection and thus implies that $I-B^{-1}Z$ is a {\em contraction} with respect to the $B$-norm, which means that the random fixed point iteration~\eqref{eq:xZupdate} has only very little room not to work. While $I-B^{-1}Z$ is not a strict contraction, under some reasonably weak assumptions on $S$ it will be a strict contraction in expectation, which  ensures convergence. We shall state these assumptions and develop the associated convergence theory for our method in Section~\ref{C2sec:convergence} and Section~\ref{C2sec:discrete}.

\section{Special Cases: Examples} \label{C2sec:examples}

%In order to build intuition about the expressive power of our framework prior to proceeding to the complexity analysis in Section~\ref{C2sec:convergence}, it will be instructive to enumerate several of the many special cases the method specializes to based on the choice of the parameters $B$ and $S$.

In this section we briefly mention how by selecting the parameters $S$ and $B$ of our method we recover several existing methods. The list is by no means comprehensive and merely serves the purpose of an illustration of the flexibility of our algorithm. All the associated complexity results we present in this section, can be recovered from Theorem~\ref{theo:convsingleS}, presented later in Section~\ref{C2sec:discrete}. 

\subsection{The one step method} 

When $S$ is an $m\times m$ invertible matrix with probability one, then the system $S^\top Ax=S^\top b$ is equivalent to solving $Ax=b,$ thus the  solution to~\eqref{ch:two:NF} must be $x^{k+1}=x^*$, independently of matrix $B.$
% When $A$ is invertible,  we can also see this more explicitly through \eqref{eq:MP}: \begin{align*} x^{k+1} & = x^k -B^{-1}A^\top S S^{-1}A^{-T}BA^{-1}S^{-T}S^\top (A x^k -b) \\ & = x^k -A^{-1}(A x^* -b) = x^k -(x^k-x^*)  = x^*. \end{align*}
Our convergence theorems also predict this one step behaviour, since $\rho =0$ (see Table~\ref{ch:two:tab:complexity}).

% When $A$ is symmetric, this is equivalent to a Newton step applied to the optimization problem $\min_{x\in \R^n} f(x)$, where $f(x) = \tfrac{1}{2}x^\top Ax - b^\top  x$.

\subsection{Random vector sketch}
When $S = s \in \R^{m}$ is restricted to being a random column vector, then from~\eqref{eq:MP} a step of our method is given by
\begin{equation}\label{eq:vecsketch}x^{k+1} = x^{k}  - \frac{s^\top (A x^{k}-b)}{ s^\top AB^{-1}A^\top s} B^{-1}A^\top s,
\end{equation}
if $A^\top s \neq 0$ and $x^{k+1} =x^k$ otherwise. This is because the pseudoinverse of a scalar $\alpha \in \R$ is given by \[\alpha^{\dagger} = 
\begin{cases}
1/\alpha  & \mbox{if } \alpha \neq 0\\
0 &  \mbox{if } \alpha = 0.
\end{cases}
\]
Next we describe several well known specializations of the random vector sketch and for brevity, we write the updates in the form of~\eqref{eq:vecsketch} and leave implicit that when the denominator is zero, no step is taken.
\subsection{Randomized Kaczmarz} \label{ch:two:sec:RK}

If we choose $S=e^i$ (unit coordinate vector in $\R^m$) and $B=I$ (the identity matrix), in view of \eqref{ch:two:NF} we obtain the method:
\begin{equation} \label{eq:RKintro} x^{k+1} = \arg\min_{x\in \R^n} \norm{x- x^{k}}_2^2 \quad \mbox{ subject to } \quad  A_{i:}x =b_{i}.
\end{equation}
Using~\eqref{eq:MP}, these iterations can be calculated with
\begin{equation}\label{eq:RKiterate} \boxed{x^{k+1} = x^{k} - \frac{A_{i:} x^{k}-b_{i}}{\norm{A_{i:}}_2^2}(A_{i:})^\top } \end{equation}

\paragraph{Complexity.} When $i$ is selected at random, this is the randomized Kaczmarz (\emph{RK}) method~\cite{Strohmer2009}. A specific non-uniform probability distribution for $S$ yields simple and easily interpretable (but not necessarily optimal) complexity bound. In particular, by selecting $i$ with probability proportional to the magnitude of row $i$ of $A$, that is $p_i = \norm{ A_{i:} }_2^2/\norm{A}_F^2$,  it follows from Theorem~\ref{theo:convsingleS} that RK enjoys the following complexity bound:
\begin{equation}\label{eq:RKconv}\E {\norm{x^{k} -x^{*} }_2^2 } \leq  \left(1  -  \frac{\lambda_{\min}\left(A^\top  A \right)}{\norm{A}_F^2} \right)^k \norm{x^{0} - x^{*}}_2^2.\end{equation}
This result was first established by Strohmer and Vershynin \cite{Strohmer2009}. 
We also provide new convergence results in Theorem~\ref{ch:two:theo:normEconv}, based on the convergence of the norm of the expected error. Theorem~\ref{ch:two:theo:normEconv} applied to the RK method gives
\begin{equation}\label{eq:RKconv2}\norm{\E {x^{k} -x^{*} } }_2^2 \leq  \left(1  -  \frac{\lambda_{\min}\left(A^\top  A \right)}{\norm{A}_F^2} \right)^{2k} \norm{x^{0} - x^{*}}_2^2.\end{equation}
Now the convergence rate appears squared, which is a better rate, though, the expectation has moved inside the norm, which is a weaker form of convergence as proven in Lemma~\ref{lem:convrandvar}.
  
Analogous results for the convergence of the norm of the expected error holds  for all the methods we present, though we only illustrate this with the RK method.

\paragraph{Re-interpretation as SGD with exact line search.} Using the ``Constrain and Approximate'' formulation~\eqref{ch:two:RF}, randomized Kaczmarz method can also be written as
\[x^{k+1} =\arg \min_{x\in \R^n} \norm{x- x^*}_2^2 \quad \mbox{ subject to } \quad  x = x^k + y(A_{i:})^\top , \quad y \in \R, \]
with probability $p_i$. Writing the least squares function $f(x) = \tfrac{1}{2}\|Ax-b\|_2^2$ as
\[f(x) = \sum_{i=1}^m p_i f_i(x), \qquad f_i(x) = \frac{1}{2 p_i}(A_{i:}x - b_i)^2,\]
we see that the random vector $\nabla f_i(x) = \tfrac{1}{p_i}(A_{i:}x-b_i) (A_{i:})^\top $ is an unbiased estimator of the  gradient of $f$ at $x$. That is,  $\E{\nabla f_i(x)} = \nabla f(x)$. Notice that RK takes a step in the direction $-\nabla f_i(x)$. This is true even when $A_{i:}x-b_i = 0$, in which case, the RK does not take any step. Hence, RK takes a step in the direction of the negative stochastic gradient. This means that it is equivalent to the Stochastic Gradient Descent (SGD) method. However,  the stepsize choice is very special: RK chooses the stepsize which leads to the point which is closest to $x^*$ in the Euclidean norm. 
   
Later in Section~\ref{subsec:RKvsRCA} in Chapter~\ref{ch:SDA} we give yet another interpretation of the RK method, namely, that the RK method is the equivalent to applying the randomized coordinate descent method to the dual of the least-norm problem.

\subsection{Randomized Coordinate Descent: positive definite case}  \label{C2sec:shs98ss}
 
 If $A$ is symmetric positive definite, then we can choose $B= A$ and $S = e^i$  in~\eqref{ch:two:NF}, which results in 
\begin{equation} \label{eq:CDpdintro}
x^{k+1} \eqdef \arg\min_{x\in \R^n} \norm{x- x^{k}}_A^2 \quad \mbox{subject to} \quad (A_{i:})^\top  x =b_{i},
\end{equation}
where we used the symmetry of $A$ to get $(e^i)^\top A = A_{i:}=(A_{:i})^\top .$
The solution to the above, given by~\eqref{eq:MP}, is 
\begin{equation}\label{eq:09j0s9jsss}
\boxed{x^{k+1} = x^{k} -  \frac{(A_{i:})^\top x^{k}-b_i}{A_{ii}}e^{i}}\end{equation}
\paragraph{Complexity.} When $i$ is chosen randomly, this is the  \emph{Randomized CD} method (CD-pd). Applying Theorem~\ref{theo:convsingleS}, we see the probability distribution $p_i = A_{ii}/\Tr{A}$
results in a convergence with
\begin{equation}\label{eq:CDposconv}
\E {\norm{x^{k} -x^{*} }_{A}^2 } \leq \left(1-  \frac{\lambda_{\min}\left(A \right)}{\Tr{A}}\right)^k \norm{x^{0} - x^{*}}_{A}^2.\end{equation}
  This result was first established by Leventhal and Lewis~\cite{Leventhal2010}.
  
\paragraph{Interpretation.} Using the Constrain-and-Approximate formulation~\eqref{ch:two:RF}, this method can be interpreted as
\begin{equation} \label{eq:CDpdRF}
x^{k+1} =\arg \min \|x-x^*\|_A^2\quad \mbox{ subject to } \quad  x = x^k + y e^i, \quad y \in \R, \end{equation}
with probability $p_i$. Using the identity $Ax^* = b$, it is easy to check that the function $f(x)=\tfrac{1}{2}x^\top  A x - b^\top  x$ satisfies: $\|x-x^*\|_A^2 = 2f(x) + b^\top  x^*$. Therefore,  \eqref{eq:CDpdRF} is equivalent to
\begin{equation} \label{eq:CDpdRFxx}
x^{k+1} =\arg \min f(x) \quad \mbox{ subject to } \quad  x = x^k + y e^i, \quad y\in \R. \end{equation}
The iterates~\eqref{eq:09j0s9jsss} can also be written as
\[x^{k+1} = x^k - \frac{1}{L_i}\nabla_{i} f(x^k) e^i,\]
where $L_i = A_{ii}$ is the Lipschitz constant of the gradient of $f$ corresponding to coordinate $i$ and $\nabla_{i} f(x^k)$ is the $i$th partial derivative of $f$ at $x^k$.

\subsection{Randomized block Kaczmarz}

Our framework also extends to new block formulations of the randomized Kaczmarz method. Let   $R$ be a random subset of $[m]$ and let $S =I_{:R}$ be a column concatenation of the columns of the $m\times m$ identity matrix $I$ indexed by $R$. Further, let $B=I$. Then \eqref{ch:two:NF} specializes to

\[
 x^{k+1}  =  \arg\min_{x\in \R^n} \norm{x- x^{k}}_2^2 \quad \mbox{subject to} \quad  A_{R:}x =  b_{R} .
\]
In view of~\eqref{eq:MP}, this can be equivalently written as 
\begin{equation}\label{eq:blockRK}\boxed{x^{k+1}=x^k - (A_{R:})^\top  (A_{R:} (A_{R:})^\top )^{\dagger}(A_{R:}x^k - b_{R})} \end{equation}
%{If $A_{R:}$ has full row rank with probability 1, then $A_{R:} (A_{R:})^\top $ is invertible the pseudoinverse in~\eqref{eq:blockRK} can be replaced by regular inverse.}

\paragraph{Complexity.} From Theorem~\ref{ch:two:theo:Enormconv} we obtain the following new complexity result:

\[\E{\|x^{k}-x^*\|_2^2} \leq  \left(1-\lambda_{\min}\left(\E{(A_{R:})^\top  (A_{R:} (A_{R:})^\top )^{\dagger}A_{R:}}\right)\right)^k\|x^0-x^*\|^2_2.\]

To obtain a more meaningful convergence rate, we would need to bound the smallest eigenvalue of $\E{(A_{R:})^\top  (A_{R:} (A_{R:})^\top )^{\dagger}A_{R:}}.$ This has been done in~\cite{Needell2012,Needell2014} when the image of $R$ defines a row paving of $A$.  Our framework paves the way for analysing the convergence of new block methods for a large set of possible random subsets $R,$ including, for example, overlapping partitions.

 \subsection{Randomized Newton: positive definite case}
  
 If $A$ is symmetric positive definite, then we can choose $B= A$ and $S = I_{:C}$, a column concatenation of the columns of $I$ indexed by $C$, which is a random subset  of $\{1,\ldots, n\}$.  In view of~\eqref{ch:two:NF}, this results in 
\begin{equation} \label{eq:CDpdblock}
x^{k+1} \eqdef \arg\min_{x\in \R^n} \norm{x- x^{k}}_A^2 \quad \mbox{subject to} \quad (A_{:C})^\top  x =b_{C}.
\end{equation}
In view of \eqref{eq:MP}, we can equivalently write the method as
  \begin{equation} \boxed{ \quad x^{k+1}  \quad = \quad x^k -  I_{:C} ((I_{:C})^\top  A I_{:C})^{-1}  (I_{:C})^\top  (Ax^k - b) \quad }\label{eq:Method1}
 \end{equation}

\paragraph{Complexity.} Clearly, iteration \eqref{eq:Method1} is well defined  as long as $C$ is nonempty with probability 1. Such $C$ is referred to in~\cite{Qu2015} as a ``non-vacuous'' sampling. From Theorem \ref{ch:two:theo:Enormconv} we obtain the following convergence rate:
\begin{align}\label{eq:Method1xxx}
\E { \norm{x^{k} -x^{*}}_{A}^2 } & \leq  \rho^k \|x^0 - x^*\|_A^2 \nonumber  \\
&=
\left(1-  \lambda_{\min}\left( \E{ I_{:C}  ((I_{:C})^\top  A I_{:C})^{-1} (I_{:C})^\top  A} \right)\right)^k \norm{x^{0} - x^{*}}_{A}^2.
\end{align}

The convergence rate of this particular method was first established and studied in \cite{Qu2015}.  Moreover, it was shown in \cite{Qu2015} that $\rho<1$ if one additionally assumes that the probability that $i \in C$ is positive for each column $i\in \{1, \ldots, n\}$, i.e., that $C$ is a ``proper'' sampling.

%We additionally prove in Theorem~\ref{theo:convsingleS} under what conditions the iterates~\eqref{eq:Method1} converge according to~\eqref{eq:CDLSconv}.

\paragraph{Interpretation.}   Using  formulation~\eqref{ch:two:RF}, and in view of the equivalence between $f(x)$ and $\|x-x^*\|_A^2$ discussed in Section~\ref{C2sec:shs98ss},  the Randomized Newton method can be equivalently written as
\[x^{k+1} =\arg \min_{x\in \R^n} f(x) \quad \mbox{ subject to } \quad  x = x^k + I_{:C}\, y, \quad y \in \R^{|C|}. \]
The next iterate is determined by advancing from the previous iterate over a subset of coordinates such that $f$ is minimized. Hence, an exact line search is performed in a random $|C|$ dimensional subspace.

Method~\eqref{eq:Method1} was first studied by Qu et al~\cite{Qu2015}, and referred therein as  ``Method~1'', or {\em Randomized Newton Method}. The name comes from the observation that the method inverts random principal submatrices of $A$ and that in the special case when $C=\{1, \ldots, n\}$ with probability 1, it specializes to the Newton method (which in this case converges in a single step). 
The expression $\rho$ defining the convergence rate of this method is rather involved and it is not immediately obvious what is gained by performing a search in a higher dimensional subspace ($|C|>1$) rather than in the one-dimensional subspaces ($|C|=1$), as is standard in the optimization literature. Let us write  $\rho = 1-\sigma_\tau$ in the case when the $C$ is chosen to be a subset of $\{1,\ldots, n\}$ of size $\tau$, uniformly at random. In view of Lemma~\ref{lem:itercomplex}, the method takes $\tilde{O}(1/\sigma_\tau)$ iterations to converge, where the tilde notation suppresses logarithmic terms. It was shown in \cite{Qu2015}  that $1/\sigma_\tau \leq 1/(\tau \sigma_1)$. That is, one can expect to obtain at least {\em superlinear speedup} in $\tau$ --- this is what is gained by moving to blocks / higher dimensional subspaces. For further details and additional properties of the method we refer the reader to \cite{Qu2015}.

\subsection{Randomized Coordinate Descent: least-squares version} 

By choosing $S=Ae^{i} =:A_{:i}$ as the $i$th column of $A$ and $B=A^\top A$, the resulting iterates~\eqref{ch:two:RF} are given by
\begin{equation} \label{eq:CDLSintro}
x^{k+1} = \arg\min_{x\in \R^n} \norm{Ax-b}_2^2 \quad \mbox{ subject to } \quad  x = x^{k} + y \, e^{i}, \quad y \in \R.
\end{equation}
When $i$ is selected at random, this is the Randomized Coordinate Descent method (\emph{CD-LS})  applied to the least-squares problem: $\min_x \|Ax-b\|_2^2$.  Using~\eqref{eq:MP}, these iterations can be calculated with
\begin{equation}\label{eq:098sh98hs} \boxed{x^{k+1} = x^{k} - \frac{(A_{:i})^\top (A x^{k} -b)}{\norm{A_{:i}}_2^2} e^{i} } \end{equation}

\paragraph{Complexity.}
Applying Theorem~\ref{theo:convsingleS}, we see that by selecting $i$ with probability proportional to magnitude of column $i$ of $A$, that is $p_i = \norm{ A_{:i} }_2^2/\norm{A}_F^2$, 
results in a convergence with
\begin{equation}\label{eq:CDLSconv}
\E {\norm{x^{k} -x^{*} }_{A^\top A}^2 } \leq \rho^k \|x^0-x^*\|^2_{A^\top  A} =   \left(1  -  \frac{\lambda_{\min}\left(A^\top  A \right)}{\norm{A}_F^2} \right)^k \norm{x^{0} - x^{*}}_{A^\top A}^2.\end{equation}
  This result was first established by Leventhal and Lewis~\cite{Leventhal2010}.
  
\paragraph{Interpretation.} 
  Using the Constrain-and-Approximate formulation~\eqref{ch:two:RF}, the CD-LS method can be interpreted as
\begin{equation} \label{eq:CDLSRF}
x^{k+1} =\arg \min_{x\in \R^n} \norm{x- x^*}_{A^\top A}^2 \quad \mbox{ subject to } \quad  x = x^k + y e^i , \quad y \in \R.\end{equation}
The  CD-LS method selects a coordinate to advance from the previous iterate $x^k$, then performs an exact minimization of the least squares function over this line.
This is equivalent to applying coordinate descent to the least squares problem $\min_{x\in \R^n} f(x) \eqdef \tfrac{1}{2}\|Ax-b\|_2^2.$ The iterates~\eqref{eq:CDLSintro} can be written as
\[x^{k+1} =x^k -\frac{1}{L_i}\nabla_i f(x^k) e^i,\] 
 where $L_i \eqdef \norm{A_{:i}}_2^2$ is the Lipschitz constant of the gradient 
 corresponding to coordinate $i$ and $\nabla_{i} f(x^k)$ is the $i$th partial derivative of $f$ at $x^k$.

\section{Convergence: General Theory} \label{C2sec:convergence}

We shall present two complexity theorems: we first study the convergence of $\norm{\E{x^{k}-x^*}}_B$ , and then move on to analysing the convergence of $\E{\norm{x^{k}-x^*}}_B$. 
Both theorems depend on the same convergence rate $\rho \in [0,\, 1]$, which we examine in the next section. In particular, we show that $\rho <1$ if and only if $A$ has full column rank. Thus the convergence results in this section only prove that the method convergence when $A$ has full column rank. Later in Chapter~\ref{ch:SDA} we extend these convergence results, and show that $A$ need not have full column rank.

\subsection{The rate of convergence}

All of our convergence theorems (see Table~\ref{ch:two:tab:complexity}) depend on the  convergence rate
\begin{equation}\label{eq:rho}\rho \eqdef 1 - \lambda_{\min}(B^{-1/2}\E{Z}B^{-1/2}).\end{equation} 
To show that the rate is meaningful, in Lemma~\ref{lem:rho1} we prove that $0\leq \rho \leq 1$. We also give an alternative expression and provide a meaningful lower bound for $\rho$.

\begin{lemma}\label{lem:rho1} The quantity $\rho$ defined in~\eqref{eq:rho} satisfies:
\begin{equation} \label{eq:rho_lower}0 \leq 1-\dfrac{\E{d}}{n}\leq \rho \leq 1,\end{equation}
where $\dd=\Rank{S^\top A}$. Furthermore
\begin{equation} \label{eq:rhoopnorm}\rho = \norm{I - B^{-1/2}\E{Z}B^{-1/2}}_2.\end{equation} 
\end{lemma}
\begin{proof}
Recall from Lemma~\ref{ch:one:lem:Z} that $B^{-1/2}ZB^{-1/2}$ is a projection,
whence the spectrum of $B^{-1/2}ZB^{-1/2}$ is contained in $\{0,1\}$. Using this,  combined with the fact that the mapping $A \mapsto \lambda_{\max}(A)$ is convex on the set of symmetric matrices  and  Jensen's inequality, we get
\begin{equation}\label{eq:a89h93a8} \lambda_{\max} (B^{-1/2} \E{Z}B^{-1/2}) \leq\E{\lambda_{\max}(B^{-1/2}ZB^{-1/2})} \leq 1.
\end{equation}
The inequality $\lambda_{\min}(B^{-1/2}\E{Z}B^{-1/2})~\geq~0$ can be shown analogously using convexity of the mapping $A\mapsto -\lambda_{\min}(A)$. Thus, $ \lambda_{\min}(B^{-1/2}\E{Z}B^{-1/2}) \in [0, 1]$, which implies   $0 \leq \rho \leq 1.$
We now refine the lower bound. As the trace of a matrix is equal to the sum of its eigenvalues, we have
\begin{equation}\label{ch:one:eq:traceb1} \E{\Tr{B^{-1/2}ZB^{-1/2}}} =\Tr{\E{B^{-1/2}ZB^{-1/2}}} \geq n \, \lambda_{\min}(\E{B^{-1/2}ZB^{-1/2}}). \end{equation}
From~\eqref{eq:B12ZB12trace} we that have $\Tr{B^{-1/2}ZB^{-1/2}} = \dd.$ 
Thus rewriting~\eqref{ch:one:eq:traceb1} gives $1-\E{\dd}/n \leq \rho.$ Finally, from the symmetry of $Z$ it follows that 
$(I-B^{-1/2}\E{Z}B^{-1/2})$ is symmetric, and consequently
\begin{align*}
\norm{(I-B^{-1/2}\E{Z}B^{-1/2})}_2
&= \lambda_{\max} (I-B^{-1/2}\E{Z}B^{-1/2})\\
&=  \left(1-\lambda_{\min}(B^{-1/2}\E{Z}B^{-1/2})\right)  = \rho.\nonumber
\end{align*}
\end{proof}

The lower bound on $\rho$ in~\eqref{eq:rho_lower} has a natural interpretation which makes intuitive sense. We shall present it from the perspective of the Constrain-and-Approximate formulation~\eqref{ch:two:RF}. As the dimension ($\dd$) of the search space $B^{-1}A^\top S$ increases (see \eqref{eq:dimproj}), the lower bound on $\rho$ decreases, and a faster convergence is possible. For instance, when $S$ is restricted to being a random column vector, as it is in the RK~\eqref{eq:RKiterate}, CD-LS~\eqref{eq:098sh98hs} and CD-pd~\eqref{eq:CDposconv} methods, the convergence rate is bounded with $1 -1/n\leq \rho.$ Using Lemma~\ref{lem:itercomplex}, this translates into the simple iteration complexity bound of $k \geq n \log(1/\epsilon)$. On the other extreme, when the search space is large, then the lower bound  is close to zero, allowing room for the method to be faster.

We now characterize circumstances under which $\rho$ is strictly smaller than one.
\begin{lemma}\label{lem:rho} If $\E{Z}$ is invertible, then $\rho <1$, $A$ has full column rank and $x^*$ is unique.
\end{lemma}
\begin{proof}
Assume that $\E{Z}$ is invertible. First, this means that $B^{-1/2}\E{Z}B^{-1/2}$ is positive definite, which in view of \eqref{eq:rho} means that $\rho <1.$ If $A$ did not have full column rank, then there would be $0\neq x\in \R^n$ such that $Ax=0$. However, we then have $Zx = 0$ and also $\E{Z}x=0$, contradicting the assumption that $\E{Z}$ is invertible. Finally, since $A$ has full column rank, $x^*$ must be unique (recall that we assume throughout this chapter that the system $Ax=b$ is consistent).
\end{proof}

\subsection{Exact characterization and norm of  expectation}

We now state a theorem which exactly characterizes the evolution of the expected iterates through a linear fixed point iteration. As a consequence, we obtain a convergence result for the norm of the expected error.  While we do not highlight this in the text, this theorem can be applied to all the particular instances of our general method we detail throughout this chapter.

%We overload the symbol $\norm{\cdot}_B$ to denote the induced operator norm for matrices. That is, for any $M\in \R^{n\times n}$ we write 
%\begin{equation}\label{eq:specnorm}
%\|M\|_B \eqdef \max_{\norm{x}_B=1} \|Mx\|_B.
%\end{equation}

For any $M\in \R^{n\times n}$ let us define
\begin{equation}\label{eq:specnorm}
\|M\|_B \eqdef \max_{\norm{x}_B=1} \|Mx\|_B.
\end{equation}

\begin{theorem}[Norm of expectation]\label{ch:two:theo:normEconv} For every  $x^*\in \R^n$ satisfying $Ax^*=b$ we have
\begin{equation} \label{eq:Eerror}
\E {x^{k+1} -x^{*}} = \left(I  - B^{-1}\E{Z}\right) \E{x^{k} - x^{*}}.
 \end{equation}
Moreover, the induced $B$-norm of the iteration matrix $I-B^{-1}\E{Z}$ is equal to $\rho$:
\begin{equation}\label{eq:rhofixedpoint} 
\|I-B^{-1}\E{Z}\|_B =  1 - \lambda_{\min}(B^{-1/2}\E{Z}B^{-1/2}) = \rho.\end{equation}
Therefore,
 \begin{equation} \label{ch:two:eq:normEconv}
\norm{\E {x^{k} -x^{*}} }_B \leq\rho^{k} \norm{x^{0} - x^{*}}_B.
 \end{equation}
\end{theorem}

\begin{proof}
Taking expectations conditioned on $x^k$ in~\eqref{eq:xZupdate}, we get
\begin{equation}\label{eq:0suj9sj}\E{x^{k+1}-x^* \;|\; x^k} = (I-B^{-1}\E{Z})(x^k-x^*).\end{equation}
Taking expectation again gives
\begin{eqnarray*}
\E{x^{k+1}-x^*} &=& \E{\E{x^{k+1}-x^* \;|\; x^k}} \\
&\overset{\eqref{eq:0suj9sj}}{=}& \E{(I-B^{-1}\E{Z})(x^k-x^*)} \\ &=& (I-B^{-1}\E{Z})\E{x^k-x^*},\end{eqnarray*}
and thus~\eqref{eq:Eerror} holds. The equivalence~\eqref{eq:rhofixedpoint} follows by 
\begin{eqnarray*}
\rho &\overset{\eqref{eq:rhoopnorm}}{=} &
\norm{I - B^{-1/2}\E{Z}B^{-1/2}}_2 \\
 &=& \max_{\norm{v}_2=1} \norm{(I - B^{-1/2}\E{Z}B^{-1/2})v}_2\\
&\overset{(v=B^{1/2}w)}{=}& \max_{\norm{B^{1/2}w}_2=1}
 \norm{B^{1/2}B^{-1/2}(I - B^{-1/2}\E{Z}B^{-1/2})B^{1/2}w}_2\\
&=&\max_{\norm{w}_B=1} \norm{(I - B^{-1}\E{Z})w}_B = \norm{I - B^{-1}\E{Z}}_B.
\end{eqnarray*}
 For all $k$, define $r^{k} \eqdef B^{1/2}(x^{k}-x^*).$ Left multiplying~\eqref{eq:Eerror} by $B^{1/2}$ gives
\[\E{r^{k+1}} = (I-B^{-1/2}\E{Z}B^{-1/2})\E{r^{k}}.\]
Applying the norms to both sides we obtain the estimate
\begin{equation}\label{eq:normEr}
\norm{\E {r^{k+1}} }_2 \leq \norm{I-B^{-1/2}\E{Z}B^{-1/2}}_2 \, \norm{\E {r^k} }_2. 
\end{equation}
The claim~\eqref{ch:two:eq:normEconv} now follows by observing~\eqref{eq:rhoopnorm},
that $\norm{r_k}_2 = \norm{x^{k}-x^*}_B$ and unrolling the recurrence in~\eqref{eq:normEr}.

%
%It remains to prove that $\rho = \norm{I-B^{-1}\E{Z}}_B$ and then unroll the recurrence.
%According to the definition of operator norm~\eqref{eq:specnorm}, we have
%\begin{align*}
%\norm{I-B^{-1}\E{Z}}_B^2 &=  \max_{\norm{B^{1/2}x}_2=1} \norm{B^{1/2}(I-B^{-1}\E{Z})x }_2^2.
%\end{align*}																		
%Substituting $B^{1/2}x =y$ in the above gives
%\begin{align*}
%\norm{I-B^{-1}\E{Z}}_B^2 &= \max_{\norm{y}_2=1} \norm{B^{1/2}(I-B^{-1}\E{Z})B^{-1/2}y}_2^2 \nonumber \\ 
%&= \max_{\norm{y}_2=1} \norm{(I-B^{-1/2}\E{Z}B^{-1/2})y}_2^2 \nonumber\\
%&= \lambda_{\max}^2 (I-B^{-1/2}\E{Z}B^{-1/2})\\
%&=  \left(1-\lambda_{\min}(B^{-1/2}\E{Z}B^{-1/2})\right)^2  = \rho^2,\nonumber
%\end{align*}
%where in the third equality we used the symmetry of $(I-B^{-1/2}\E{Z}B^{-1/2})$ when passing from the operator norm
%to the spectral radius. Note that the symmetry of $\E{Z}$ derives from the symmetry of $Z$. 
\end{proof}

\subsection{Expectation of norm}
We now turn to analysing the convergence of the expected norm of the error, for which we need the following technical lemma.
  
\begin{theorem}[Expectation of norm]\label{ch:two:theo:Enormconv} If $\E{Z}$ is positive definite, then
\begin{equation} \label{ch:two:eq:Enormconv}
 \E {\|x^{k} -x^{*} \|_B^2 } \leq \rho^k \norm{x^{0} - x^{*}}_B^2,
 \end{equation}
 where $\rho<1$ is given in~\eqref{eq:rho}.
\end{theorem}
\begin{proof}
Let $r^k = B^{1/2}(x^k-x^*)B^{1/2}$. Taking expectation in~\eqref{eq:xZupdate} conditioned on $r^k$ gives
\begin{eqnarray*}
\E{\|r^{k+1}\|_2^2 \,\ | \,\ r^k} &\overset{\eqref{eq:xZupdate}}{=} & \E{\|(I-B^{-1/2}ZB^{-1/2})r^k\|_2^2 \,\ | \,\ r^k }  \\
&\overset{\eqref{eq:B12ZB12proj}}{=}&  \left< (I-B^{-1/2}\E{Z}B^{-1/2})r^k,r^k\right>  \nonumber\\
& \leq & \norm{I-B^{-1/2}\E{Z}B^{-1/2}}_2 \, \|r^k\|_2^2.
\end{eqnarray*}
Using~\eqref{eq:rhoopnorm}, taking expectation again and unrolling the recurrence gives the result. 
\end{proof}

The convergence rate $\rho$ of the expected norm of the error is ``worse'' than the $\rho^2$ rate of convergence of the norm of the expected error in Theorem~\ref{ch:two:theo:normEconv}.
 This should not be misconstrued as Theorem~\ref{ch:two:theo:normEconv} offering a ``better'' convergence rate than Theorem~\ref{ch:two:theo:Enormconv}, because, as explained in Lemma~\ref{lem:convrandvar}, convergence of the expected norm of the error is a stronger type of convergence. More importantly, the exponent is not of any crucial importance; clearly, an exponent of $2$ manifests itself only in halving the number of iterations (see Lemma~\ref{lem:itercomplex}).

\section{Methods Based on Discrete Sampling}\label{C2sec:discrete}
When $S$ has a discrete distribution, we can establish under reasonable assumptions when $\E{Z}$ is positive definite (Proposition~\ref{pro:Ediscrete}),  we can optimize the convergence rate in terms of the chosen probability distribution, and finally, determine a  probability distribution for which the convergence rate is expressed in terms of the scaled condition number (Theorem~\ref{theo:convsingleS}).

%theorem tailored to discrete sampling matrices $S$. 
%First we show under the following reasonable assumptions on $S$ that $\E{Z}$ is positive %definite.

\begin{assumption}\label{ass:complete}
The random matrix $S$ has a discrete distribution. In particular,  $S= S_i \in \R^{m \times q_i}$ with probability  $p_i>0$, $\sum_{i=1}^r p_i =1$, where $S_i^\top A$ has full row rank and $q_i \in \N,$ for $i=1,\ldots, r$. Furthermore $\mathbf{S} \eqdef [S_1, \ldots, S_r] \in \R^{m\times \sum_{i=1}^r q_i}$
is such that $A^\top \mathbf{S}$ has full row rank.
\end{assumption}

For simplicity, sampling $S$ satisfying the above assumption will be called a {\em complete discrete sampling}. We now give an example of such a sampling. If $A$ has full column rank  and each row of $A$ is not strictly zero, $S =e^i$ with probability $p_i =1/n$, for $i =1,\ldots, n,$ then $\mathbf{S} =I$ and $S$ is a complete discrete sampling. In fact, from any basis of $\R^n$ we can construct a complete discrete sampling in an analogous way. 

When $S$ is a complete discrete sampling, then $S^\top A$ has full row rank and 
\[(S^\top  A B^{-1}A^\top  S)^{\dagger}=(S^\top  A B^{-1}A^\top  S)^{-1}.\] Therefore we replace the pseudoinverse in~\eqref{eq:MP} and~\eqref{eq:xZupdate} by the inverse.
 Furthermore, using a complete discrete sampling guarantees convergence of the resulting method.

\begin{proposition}\label{pro:Ediscrete} 
If $S$ is a complete discrete sampling,   $\E{Z}$ is positive definite.
\end{proposition}
\begin{proof}
Let 
\begin{equation}\label{eq:defD}
D \eqdef	 \mbox{diag}\left( \sqrt{p_1}((S_1)^\top  A B^{-1}A^\top  S_1)^{-1/2}, \ldots, 
\sqrt{p_r}((S_r)^\top  A B^{-1}A^\top  S_r)^{-1/2}\right)
\end{equation}
which is a block diagonal matrix, and is well defined and invertible as $S_i^\top  A$  has full row rank for $i=1,\ldots, r$.
 Taking the expectation of $Z$~\eqref{eq:Z-first} gives
\begin{align}
\E{Z} &= \sum_{i=1}^r A^\top  S_i (S_i^\top  A B^{-1}A^\top  S_i)^{-1}S_i^\top  A p_i \nonumber \\
&=  A^\top \left(\sum_{i=1}^r   S_i \sqrt{p_i}(S_i^\top  A B^{-1}A^\top  S_i)^{-1/2} (S_i^\top  A B^{-1}A^\top  S_i)^{-1/2}  \sqrt{p_i}S_i^\top  \right) A \nonumber \\
&= \left(  A^\top  \mathbf{S} D\right)  \left( D \mathbf{S}^\top   A\right), \label{ch:one:eq:EZdiscrete}
\end{align} which is positive definite because $A^\top \mathbf{S}$ has full row rank and $D$ is invertible. 
\end{proof}

 With $\E{Z}$ positive definite, we can apply the convergence Theorem~\ref{ch:two:theo:normEconv} and~\ref{ch:two:theo:Enormconv}, and the resulting method converges.
 
%  thus we focus solely on the convergence of the expected norm error in %Theorem~\ref{ch:two:theo:Enormconv}, as this is the stronger of the two types of convergence.
 
 \subsection{Optimal probabilities}\label{ch:two:sec:optprob}
 
We can choose the discrete probability distribution that
 optimizes the convergence rate. For this, according to Theorems~\ref{ch:two:theo:Enormconv} and~\ref{ch:two:theo:normEconv} we need to find $p=(p_1,\dots,p_r)$ that maximizes the minimal eigenvalue of $B^{-1/2}\E{Z}B^{-1/2}$.
Let $S$ be a complete discrete sampling and fix the sample matrices $S_1,\dots, S_r$. Let us denote $Z=Z(p)$ as a function of $p=(p_1,\dots,p_r)$.  Then we can also think of the spectral radius as a function of $p$ where \[\rho(p) = 1 - \lambda_{\min}(B^{-1/2}\E{Z(p)}B^{-1/2}).\]

If we let $\Delta_r =\left\{p = (p_1,\dots,p_r) \in \R^r \;:\; \sum_{i=1}^r p_i =1, \; p\geq 0\right\}$, the problem of minimizing the spectral radius (i.e., optimizing the convergence rate) can be written as
 \[\rho^* \quad \eqdef\quad \min_{p\in \Delta_r} \rho(p) \quad = \quad 1 - \max_{p\in \Delta_r} \lambda_{\min} (B^{-1/2}\E{Z(p)}B^{-1/2}).\]
 This can be cast as a convex optimization problem, by first re-writing 
\begin{align*}
B^{-1/2}\E{Z(p)}B^{-1/2} &= \sum_{i=1}^r p_i \left(B^{-1/2}A^\top  S_i (S_i^\top  A B^{-1}A^\top  S_i)^{-1}S_i^\top  AB^{-1/2}  \right)\\
	 &= \sum_{i=1}^r p_i \left(V_i (V_i^\top  V_i)^{-1}V_i^\top   \right),
\end{align*} 
 where $V_i = B^{-1/2}A^\top  S_i.$ Thus
\begin{equation}\label{eq:opt_sampling}\rho^* \quad = \quad 1 - \max_{p\in \Delta_r} \lambda_{\min}  \left( \sum_{i=1}^r p_i V_i (V_i^\top V_i)^{-1} V_i^\top   \right).\end{equation} 
 To obtain $p$ that maximizes the smallest eigenvalue, we solve 
 \begin{align}
 \max_{p,t} \,\, &\quad t  \nonumber \\
 \mbox{subject to}& \quad \sum_{i=1}^r p_i \left(V_i (V_i^\top  V_i)^{-1}V_i^\top   \right) \succeq t\cdot  I, \label{eq:optconv}\\
 & \quad p \in \Delta_r. \nonumber
 \end{align}
 Despite~\eqref{eq:optconv} being a convex semi-definite program{\footnote{When preparing a revision of the paper on which this chapter is based, we have learned about the existence of prior work~\cite{Dai2014} where the authors have also characterized the probability distribution that optimizes the convergences rate of the RK method as the solution to an SDP.}, which is apparently a harder problem than solving the original linear system, investing the time into solving~\eqref{eq:optconv} using a solver for convex conic programming such as \texttt{cvx}~\cite{cvx} can pay off, as we show in Section~\ref{C2sec:numopt}. Though for a practical method based on this, we would need to develop an approximate solution to~\eqref{eq:optconv} which can be efficiently calculated.

 \subsection{Convenient probabilities} \label{ch:two:sec:convprob}

Next we develop a choice of probability distribution that yields a convergence rate that is easy to interpret. This result is an extension 
of Strohmer and Vershynin's~\cite{Strohmer2009} non-uniform probability distribution for the Kaczmarz method. Our extension includes a wide range of methods, such as the randomized Kaczmarz, randomized coordinate descent, as well as their block variants. However, it is more general, and covers many other possible particular algorithms, which arise by choosing a particular set of sample matrices $S_i$, for $i=1,\ldots, r.$

\begin{theorem} \label{theo:convsingleS}  Let $S$ be a complete discrete sampling such that $S = S_i \in \R^{m}$ with probability 
\begin{equation}\label{ch:one:eq:convprob}p_i~=~\dfrac{\Tr{S_i^\top  AB^{-1}A^\top S_i}}{\norm{B^{-1/2}A^\top \mathbf{S}}_F^2},\quad \mbox{for } \quad i=1,\ldots, { r}.
\end{equation}
Then the iterates~\eqref{eq:MP} satisfy 
\begin{equation} \label{eq:expnormcon}
\E{\norm{x^{k} -x^{*} }_B^2} \leq \rho_c^k \,\norm{x^{0} - x^{*}}_B^2,\end{equation}
where
\begin{equation}\label{ch:one:eq:rhoconv} \rho_c = 1- \frac{\lambda_{\min}\left(\mathbf{S}^\top  A B^{-1}A^\top  \mathbf{S} \right)}{\norm{B^{-1/2}A^\top \mathbf{S}}_F^2}.
\end{equation}
\end{theorem}
\begin{proof}
Let $t_i = \Tr{S_i^\top  A B^{-1}A^\top  S_i}$, and with~\eqref{ch:one:eq:convprob} in~\eqref{eq:defD} we have
\[ D^2 =\frac{1}{\norm{B^{-1/2}A^\top \mathbf{S}}_F^2}\mbox{diag}\left(t_1 (S_1^\top  A B^{-1}A^\top  S_1)^{-1}, \ldots,
t_r (S_r^\top  A B^{-1}A^\top  S_r)^{-1}\right),
 \]
thus
\begin{equation}\label{eq:Dlambda} \lambda_{\min}(D^2) = \frac{1}{\norm{B^{-1/2}A^\top \mathbf{S}}_F^2}\min_i\left\{ \frac{t_i}{\lambda_{\max}(S_i^\top  A B^{-1}A^\top  S_i)} \right\} \geq \frac{1}{\norm{B^{-1/2}A^\top \mathbf{S}}_F^2}. 
\end{equation}
Applying the above in~\eqref{ch:one:eq:EZdiscrete} gives
\begin{align}
\lambda_{\min}\left(B^{-1/2}\E{Z}B^{-1/2} \right)
&= \lambda_{\min}\left(B^{-1/2}  A^\top  \mathbf{S} D^2 \mathbf{S}^\top   A B^{-1/2}\right) \nonumber\\
&= \lambda_{\min}\left(  \mathbf{S}^\top   A B^{-1} A^\top  \mathbf{S} D^2\right) \nonumber\\
& \geq  \lambda_{\min}\left(\mathbf{S}^\top   A B^{-1} A^\top  \mathbf{S}  \right)\lambda_{\min}(D^2) \label{eq:prodpdeig}\\
& \geq  \frac{\lambda_{\min}\left(\mathbf{S}^\top   A B^{-1} A^\top  \mathbf{S}  \right)}{\norm{B^{-1/2}A^\top \mathbf{S}}_F^2}, \nonumber
\end{align} 
where in the first step we used the fact that for arbitrary matrices $B,C$ of appropriate sizes, $\lambda_{\min}(BC)=\lambda_{\min}(CB)$, and in the first inequality the fact that if $B,C \in \R^{n\times n}$ are positive definite, then $\lambda_{\min}(BC) \geq \lambda_{\min}(B)\lambda_{\min}(C)$. Finally
\begin{equation}\label{eq:convrhobound}
1 - \lambda_{\min}\left(B^{-1/2}\E{Z}B^{-1/2} \right) \leq 1 - \frac{\lambda_{\min}\left(\mathbf{S}^\top   A B^{-1} A^\top  \mathbf{S}  \right)}{\norm{B^{-1/2}A^\top \mathbf{S}}_F^2}.\end{equation}
The result~\eqref{eq:expnormcon} follows by applying Theorem~\ref{ch:two:theo:Enormconv}.
\end{proof}

The convergence rate $\lambda_{\min}\left(\mathbf{S}^\top   A B^{-1} A^\top  \mathbf{S}\right)/ \norm{B^{-1/2}A^\top \mathbf{S}}_F^2$ is known as the scaled condition number, and naturally appears in other numerical schemes, such as matrix inversion~\cite{Edelman1992,Demmel1988}. When $S_i =s_i \in \R^n$ is a column vector then \[p_i~=~\left((s_i)^\top  AB^{-1}A^\top  s_i\right)/\norm{B^{-1/2}A^\top \mathbf{S}}_F^2,\] for $i=1,\ldots r.$ In this case, the bound~\eqref{eq:Dlambda} is an equality and $D^2$ is a scaled identity, so~\eqref{eq:prodpdeig} and consequently~\eqref{eq:convrhobound} are equalities. For block methods, it is different story, and there is much more slack in the inequality~\eqref{eq:convrhobound}. So much so, 
the convergence rate~\eqref{ch:one:eq:rhoconv} does not indicate any advantage of using a block method (contrary to numerical experiments).
  To see the advantage of a block method, we need to use the exact expression for $\lambda_{\min}(D^2)$ given in~\eqref{eq:Dlambda}. Though this results in a somewhat harder to interpret convergence rate, one could use a so called \emph{matrix paving} to explore this convergence rate,  as was done for the block Kaczmarz method (see~\cite{Needell2014,Needell2012} for more details).

By appropriately choosing $B$ and $S$, this theorem applied to RK method~\eqref{eq:RKintro}, the CD-LS method~\eqref{eq:CDLSintro} and the CD-pd method~\eqref{eq:CDpdintro},  yields the convergence results~\eqref{eq:RKconv},~\eqref{eq:CDLSconv} and~\eqref{eq:CDposconv}, respectively, for single column sampling or block methods alike.

This theorem also suggests a preconditioning strategy, in that, a faster convergence rate will be attained if $\mathbf{S}$ is an approximate inverse of $B^{-1/2}A^\top .$ For instance, in the RK method where $B=I$, this suggests that an accelerated convergence can be attained if $S$ is a random sampling of the rows of a preconditioner (approximate inverse) of $A.$

 \section{Methods Based on Gaussian Sampling} \label{C2sec:gauss}

In this section we shall describe variants of our method in the case when $S$ is a Gaussian  vector with mean $0\in \R^m$ and  a positive definite covariance matrix $\Sigma\in \R^{m \times m}$.  That is, $S =\zeta \sim N(0,\Sigma)$. This applied to~\eqref{eq:MP} results in iterations of the form
\begin{equation}\label{eq:gaussupdate}
\boxed{x^{k+1} = x^{k}  - \frac{\zeta ^\top (A x^{k}-b)}{\zeta^\top AB^{-1}A^\top \zeta } B^{-1}A^\top \zeta}
\end{equation}

Unlike the discrete methods in Section~\ref{C2sec:examples},  to calculate an iteration of~\eqref{eq:gaussupdate} we need to compute the product of a matrix with a dense vector $\zeta$.
This significantly raises the cost of an iteration. Though in our numeric tests in Section~\ref{C2sec:numerics}, the faster convergence of the Gaussian method often pays off for their high iteration cost. 

To analyze the complexity of the resulting method let $\xi \eqdef B^{-1/2}A^\top  S,$
which is also Gaussian, distributed as $\xi\sim N(0, \Omega)$,
where $\Omega\eqdef B^{-1/2}A^\top  \Sigma A B^{-1/2}.$
In this section we  assume {{\em $A$ has full column rank}, so that $\Omega$ is always positive definite. The complexity of the method can be established through
\begin{eqnarray} %\rho &=&
\rho &=& 1 - \lambda_{\min}\left(\E{ B^{-1/2} ZB^{-1/2}}\right) 
 =  1- \lambda_{\min}\left(\E{\frac{\xi\xi^\top }{\norm{\xi}^2_2}}\right).
 \end{eqnarray}
 We can simplify the above by using the lower bound
 \[ \E{\frac{\xi\xi^\top }{\norm{\xi}^2_2}} \succeq \frac{2}{\pi}\frac{\Omega}{\Tr{\Omega}},\] 
which is proven in Lemma~\ref{lem:gaussdiag} in the Appendix of this chapter. Thus  
 \begin{equation}\label{eq:rhoboundgauss} 1- \frac{1}{n} \leq \rho \leq 1 -\frac{2}{\pi}\frac{\lambda_{\min}(\Omega)}{\Tr{\Omega}},
 \end{equation}
 where we used the general lower bound in~\eqref{eq:rho_lower}.
 Lemma~\ref{lem:gaussdiag} also shows that $\E{\xi\xi^\top /\norm{\xi}^2_2}$ is positive definite,  thus Theorem~\ref{ch:two:theo:Enormconv} guarantees that the expected norm of the error of all Gaussian methods converges exponentially to zero. This bound is tight upto a constant factor. Indeed,  if $A=I=\Sigma$ then $\xi \sim N(0,I)$ and $\E{\xi\xi^\top /\norm{\xi}^2_2} = \tfrac{1}{n} I,$  which yields \[1-\dfrac{1}{n}\leq \rho  \leq 1- \dfrac{2}{\pi }\cdot \dfrac{1}{n}.\]

 When $n=2$, then in  
Lemma~\ref{lem:2Dgausscov} of the Appendix of this chapter we prove that
\[\E{\frac{\xi\xi^\top }{\norm{\xi}^2_2}} = \frac{\Omega^{1/2}}{\Tr{\Omega^{1/2}}},\]
which yields a very favourable convergence rate. 

%This expression does not hold for $n>2,$ and instead, we conjecture that
%\[\E{M_\xi} \succeq \frac{\Omega}{\Tr{\Omega}},\]
%for all $n$ and perform numeric tests in Section~\ref{C2sec:numgaussbound} to support this.

\subsection{Gaussian Kaczmarz}\label{C2sec:gaussI}
Let $B=I$ and choose $\Sigma =I$ so that $S=\eta\sim N(0,I)$. Then~\eqref{eq:gaussupdate}  has the form
\begin{equation} \label{eq:gaussrkupdate}
\boxed{ x^{k+1} = x^{k} 
- \frac{\eta^\top  (Ax^{k}-b)}{\|A^\top  \eta\|_2^2} A^\top  \eta } \end{equation}
which we call the \emph{Gaussian Kaczmarz} (GK) method, for it is the analogous method to the Randomized Kaczmarz method in the discrete setting.  Using the formulation~\eqref{ch:two:RF},  for instance, the GK method can be interpreted as
\[x^{k+1} =\arg \min_{x\in \R^n} \norm{x- x^{*}}^2 \quad \mbox{ subject to } \quad  x = x^{k} + A^\top \eta y, \quad y\in \R. \]
Thus at each iteration, a random normal Gaussian vector $\eta$ is drawn and a search direction is formed by $A^\top \eta.$
 Then, starting from the previous iterate $x^{k}$,  an exact line search is performed over this search direction so that the euclidean distance from the optimal is minimized.

\subsection{Gaussian least-squares} \label{C2sec:gaussATA}
Let $B=A^\top A$ and choose $S\sim N(0,\Sigma)$ with $\Sigma=AA^\top $. It will be convenient to write $S=A\eta$, where $\eta\sim N(0,I)$. Then method \eqref{eq:gaussupdate} then has the form
\begin{equation} \label{eq:gausslsupdate}
\boxed{x^{k+1}  = x^{k} - \frac{\eta^\top  A^\top (Ax^{k} -b)}{\|A\eta\|_2^2} \eta}
 \end{equation}
 which we call the \emph{Gauss-LS} method. This method has a natural interpretation through formulation~\eqref{ch:two:RF} as
\[x^{k+1} =\arg \min_{x\in \R^n} \frac{1}{2}\|Ax-b\|_2^2 \quad \mbox{ subject to } \quad  x = x^{k} + y\eta, \quad y \in \R. \]
 That is, starting from $x^{k}$, we take a step in a random (Gaussian) direction, then perform an exact line search over this direction that minimizes the least squares error.
Thus the Gauss-LS method is the same as applying the Random Pursuit method~\cite{Stich2014} with exact line search to the Least-squares function.

\subsection{Gaussian positive definite}\label{C2sec:gaussA}
When $A$ is positive definite, we achieve an accelerated Gaussian method. 
Let $B =A$ and choose $S= \eta \sim N(0,I)$.  Method~\eqref{eq:gaussupdate} then has the form
\begin{equation}\label{eq:gausspd}\boxed{x^{k+1} 
= x^{k} - \frac{ \eta^\top (A x^{k}-b)}{\norm{\eta}_A^2} \eta}
 \end{equation}
 which we call the \emph{Gauss-pd} method.
 
% \paragraph{Interpretation.} 
Using  formulation~\eqref{ch:two:RF}, the  method can be interpreted as
\[x^{k+1} =\arg \min_{x\in \R^n} \left\{f(x) \eqdef \tfrac{1}{2} x^\top Ax -b^\top  x\right\} \quad \mbox{ subject to } \quad  x = x^{k} + y\eta, \quad y \in \R. \]
 That is, starting from $x^{k}$, we take a step in a random (Gaussian) direction, then perform an exact line search over this direction. Thus the Gauss-pd method is equivalent to applying the Random Pursuit method~\cite{Stich2014} with exact line search to $f(x).$
 
      All the Gaussian methods can be extended to  block versions. We illustrate this by designing a Block Gauss-pd method where $S \in \R^{n \times q}$ has   i.i.d.\ Gaussian normal entries and $B=A.$ This results in the  iterates
 \begin{equation} \label{eq:Bgausspd} x^{k+1} 
= x^{k} - S(S^\top AS)^{-1}S^\top (A x^{k}-b).
 \end{equation}

\section{Numerical Experiments} \label{C2sec:numerics}
 We perform some preliminary numeric tests on consistent overdetermined linear systems  and positive definite systems. 
   Everything was coded and run in MATLAB R2014b. 
Let $\kappa_2 = \norm{A}\norm{A^{\dagger}}$ be the $2-$norm condition number. In comparing different methods for solving overdetermined systems, we use the relative residual $\norm{Ax^k-b}_2/\norm{b}_2,$ while for positive definite systems we use $\norm{x^k-x^*}_A/\norm{x^*}_A$ as a relative residual measure. We run each method until the relative residual is below $10^{-4} =0.01 \%$ or until $400$ seconds in time is exceeded.
We test this relatively low precision of $0.01 \%$ because our applications of interest (e.g. ridge regression in machine learning) only require a low precision.

Note that when $A$ has full column rank, the convergence of the relative residual implies the convergence of norm error $\norm{x-x^*}$. Indeed, this follows from
\[\norm{Ax^k -b}_2^2 = \dotprod{A^\top A(x^k-x^*),x^k-x^*} \geq \lambda_{\min}(A^\top A) \norm{x^k-x^*}_2^2.\]
Consequently bringing the relative residual $\norm{Ax^k-b}_2/\norm{b}_2$ below $0.01 \%$ implies that
\[ \norm{x^k-x^*}_2 \leq \frac{\norm{b}_2}{\sqrt{\lambda_{\min}(A^\top A)}}10^{-4}.\]
  
When the solution $x^*$ and the right hand $b$ were not supplied by the data, we generated them as follows. The solution was generated using $x^*=$\texttt{rand}$(n,1)$, that is, each entry of $x^*$ is selected uniformly at random from the interval $[0, 1]$. The right hand side $b$ was set to $b=Ax^*.$
As the starting point we used $x_0=0 \in \R^n$ in all experiments.
   
In each figure we plot the relative residual in percentage on the vertical axis, starting with $100\%$. For the horizontal axis, we use either wall-clock time measured using the \texttt{tic-toc} MATLAB function or the total number of floating point operations (\emph{flops}). Specifically, we use an upper bound  on the number of flops performed in each iteration as a proxy. 
For example, the number of flops required to compute the matrix-vector product $A v$ is bounded by $O(nnz(A))$ from above. This bound is tight when $v$ is dense. When $v$ is not dense, as is the case when $v= e_i,$  we use an appropriately tight upper bound, such as $O(nnz(A_{i:}))$ when $v=e_i$.

In implementing the discrete sampling methods we used the convenient probability distributions~\eqref{ch:one:eq:convprob}.
  
All tests were performed on a Desktop with 64bit quad-core Intel(R) Core(TM) i5-2400S CPU @2.50GHz with 6MB cache size with a Scientific Linux release 6.4 (Carbon) operating system.

 %number, where $A^{\dagger}$ is a pseudo-inverse of $A$.
 % We are interested in positive definite systems
 %arises from applications in optimization methods 
  
\subsection{Overdetermined linear systems}
First we compare the methods Gauss-LS~\eqref{eq:gausslsupdate} , CD-LS~\eqref{eq:CDLSintro} , Gauss-Kaczmarz~\eqref{eq:gaussrkupdate} and RK~\eqref{eq:RKiterate} methods on synthetic linear systems generated with the matrix functions {\tt rand} and {\tt sprandn}, see Figure~\ref{fig:oversynth}. The {\tt rand}$(m,n)$ function returns a  $m$-by-$n$ matrix  where each entry is a random variable selected uniformly at random from the interval $[0, 1]$. 
The high iteration cost of the Gaussian methods resulted in poor performance on the dense problem generated using \texttt{rand} in Figure~\ref{fig:randch1}. 
In Figure~\ref{fig:sprandn} we compare the methods on a sparse linear system generated  using the MATLAB sparse random matrix function \texttt{sprandn}($m,n$,{\tt density,rc}), where {\tt density} is the percentage of nonzero entries and {\tt rc} is the reciprocal of the condition number. The {\tt sprandn}($m,n$,{\tt density,rc}) returns a  random $m$-by-$n$ sparse matrix with approximately {\tt density}$\,\times m \times n$ normally distributed nonzero entries. Please consult \url{http://uk.mathworks.com/help/matlab/ref/sprandn.html} for more details on {\tt sprandn}.  On this sparse problem the Gaussian methods are more efficient, and converge at a similar rate in time to the discrete sampling methods.
% converging with a similar rate to their ``non-Gaussian'' counterparts.  

\begin{figure}
    \centering
    \begin{subfigure}[t]{0.7\textwidth}
     %   \centering
\includegraphics[width =  0.85\textwidth, height =4.5cm, trim= 32 270 47 285, clip ]{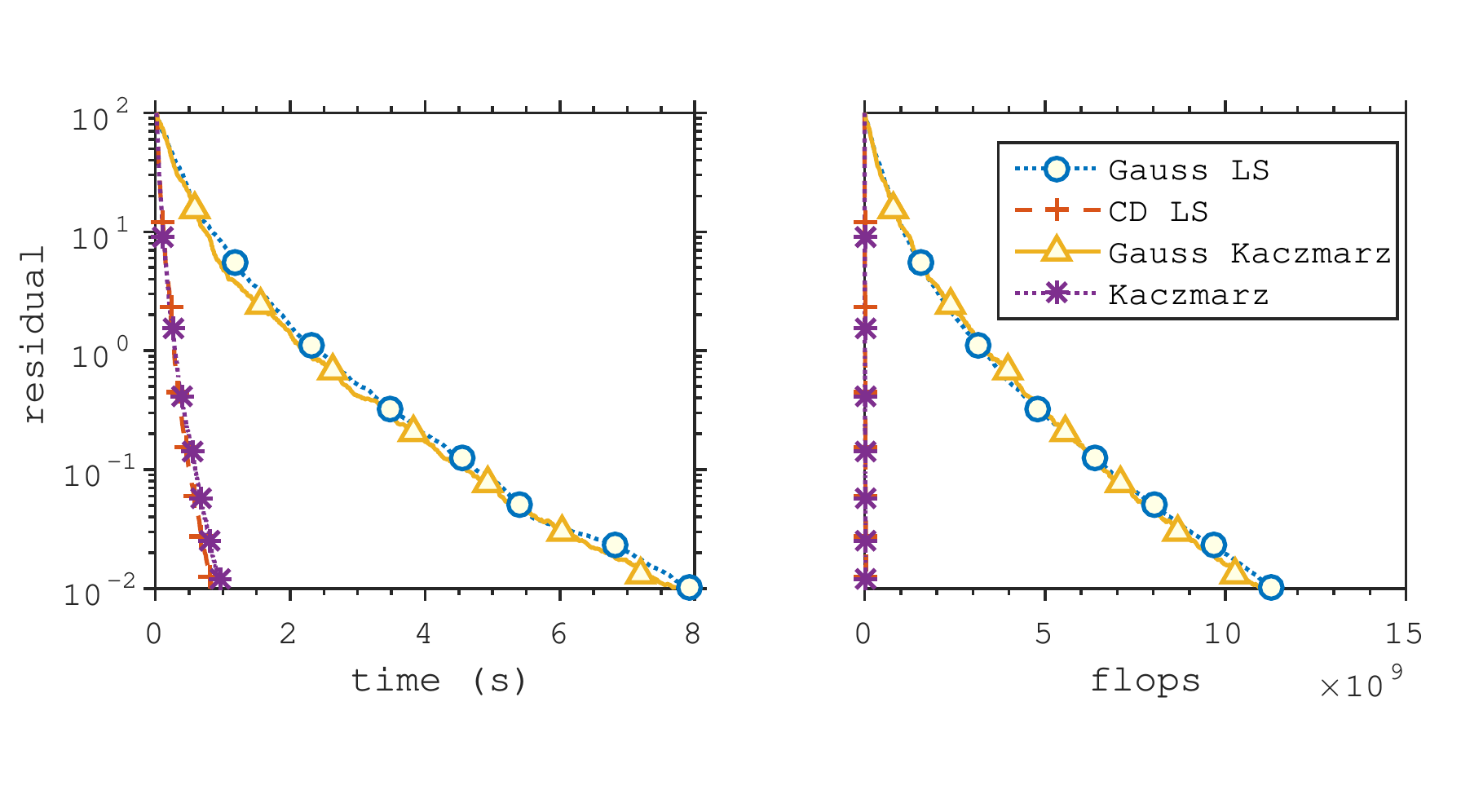}
        \caption{\texttt{rand}}\label{fig:randch1}
    \end{subfigure}%
    \hspace{0.05\textwidth}
    \begin{subfigure}[t]{0.7\textwidth}
       % \centering
\includegraphics[width =  0.85\textwidth, height =4.5cm, trim= 32 270 47 285, clip ]{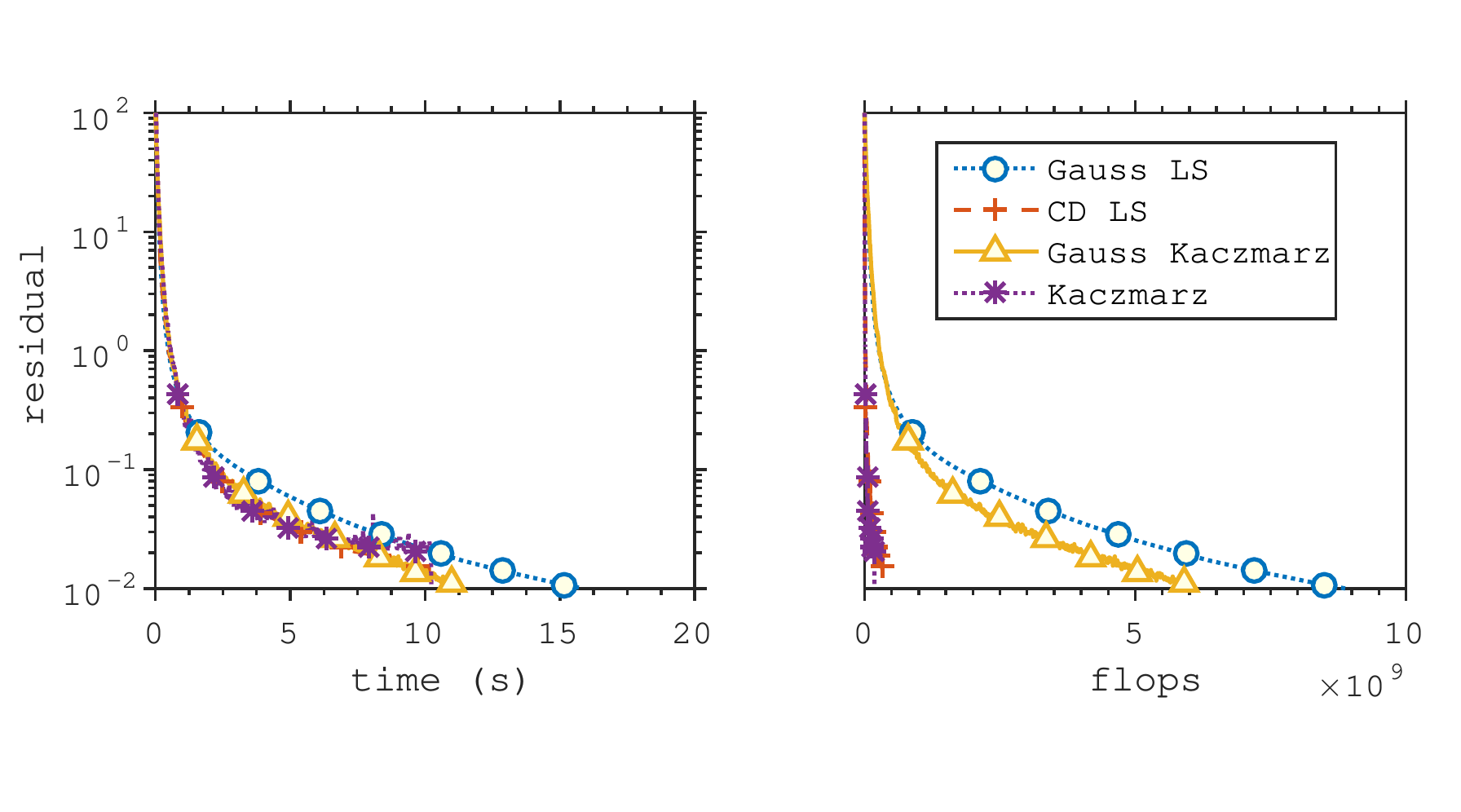}
        \caption{\texttt{sprandn}} \label{fig:sprandn}
    \end{subfigure}
    \caption{
The performance of the Gauss-LS, CD-LS, Gauss-Kaczmarz and RK methods on
    synthetic MATLAB generated problems (a) \texttt{rand}$(n,m)$ with $(m;n)=(1000,500)$ (b) \texttt{sprandn}($m,n$,{\tt density,rc}) with $(m;n) =(1000,500)$, {\tt density}$=1/\log(nm)$ and {\tt rc}$= 1/\sqrt{mn}$. In both experiments dense solutions were generated with $x^*=$\texttt{rand}$(n,1)$ and $b=Ax^*.$ }\label{fig:oversynth}
\end{figure}

In Figure~\ref{fig:overMM} we test two overdetermined linear systems taken from the the Matrix Market collection~\cite{Boisvert1997}. The collection also provides the right-hand side of the linear system.  Both of these systems are very well conditioned, but do not have full column rank, thus Theorem~\ref{ch:two:theo:Enormconv} does not apply. The four methods have a similar performance on Figure~\ref{fig:illc1033}, while the  Gauss-LS and  CD-LS method converge faster on~\ref{fig:well1033} as compared with the Gauss-Kaczmarz and Kaczmarz methods in terms of time taken. In terms of number of flops, the CD-LS and Kaczmarz method outperform the  Gauss-LS and Gauss-Kaczmarz methods.

\begin{figure}
    \centering
    \begin{subfigure}[t]{0.70\textwidth}
     %   \centering
\includegraphics[width =  0.85\textwidth, height =4.5cm, trim= 32 270 47 285, clip ]{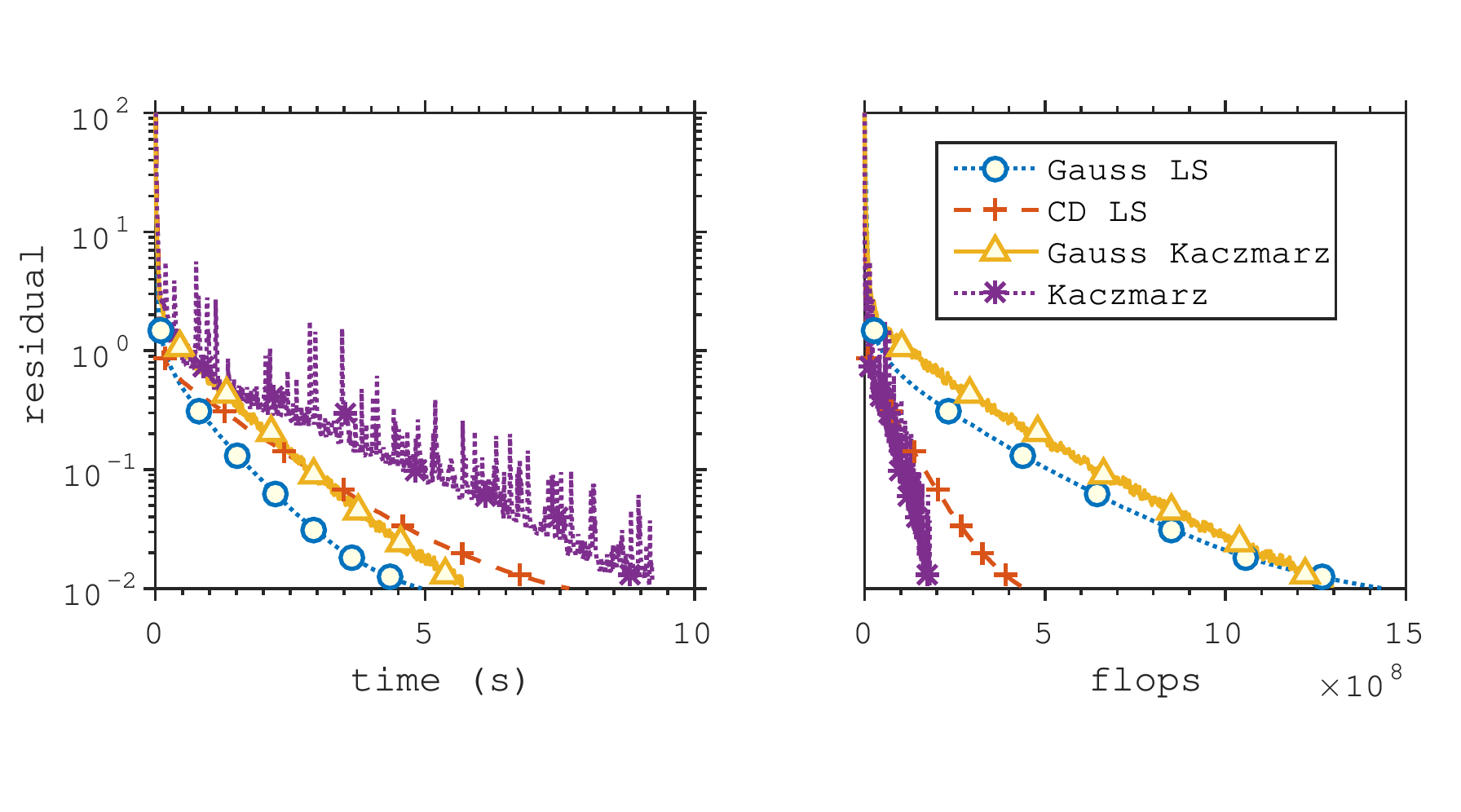}
        \caption{\texttt{illc1033}}\label{fig:illc1033}
    \end{subfigure}%
    \hspace{0.05\textwidth}
    \begin{subfigure}[t]{0.70\textwidth}
       % \centering
\includegraphics[width =  0.85\textwidth, height =4.5cm, trim= 32 270 47 285, clip ]{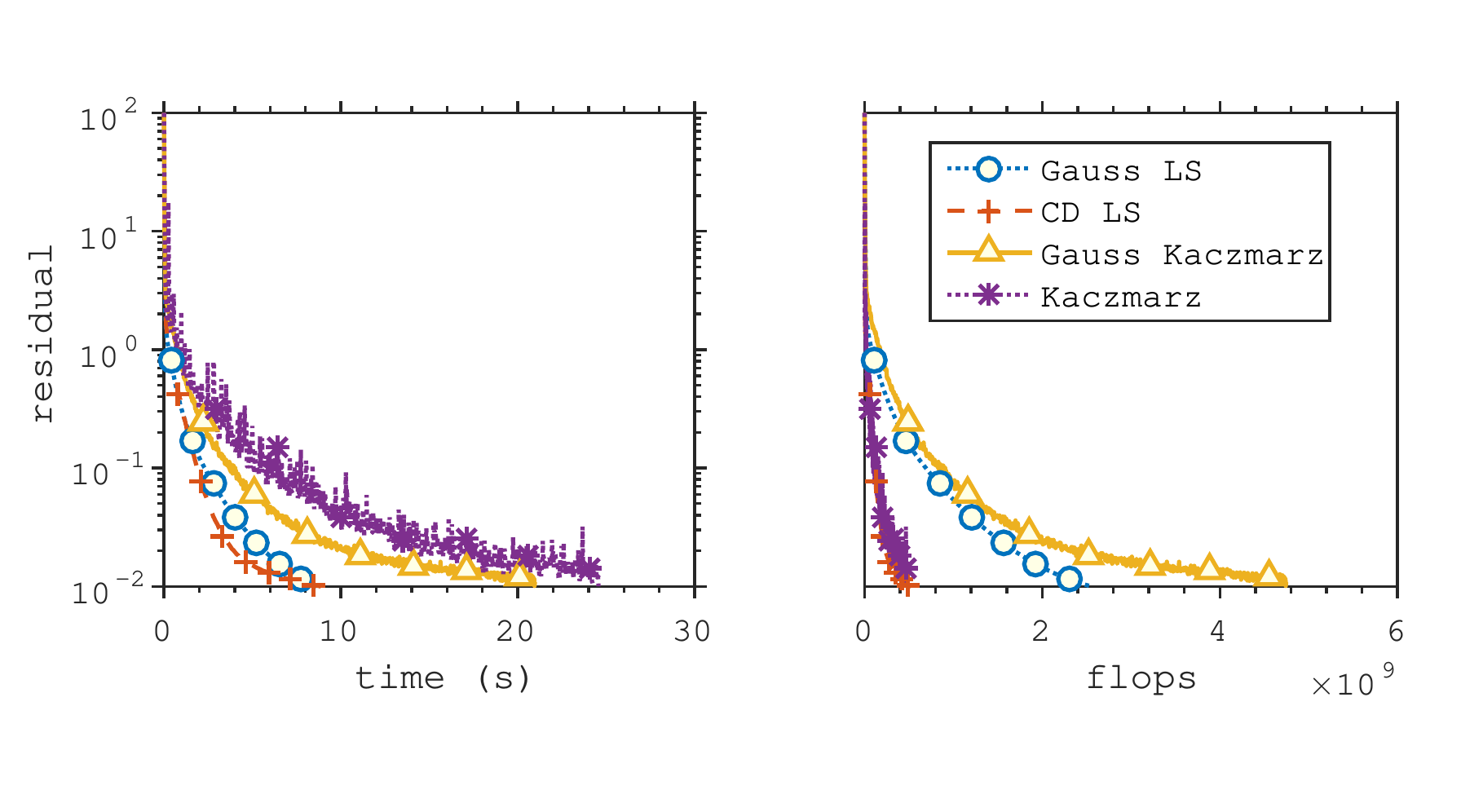}
        \caption{\texttt{well1033} } \label{fig:well1033}
    \end{subfigure}
    \caption{The performance of the Gauss-LS, CD-LS, Gauss-Kaczmarz and RK methods on
        linear systems (a) \texttt{well1033} where $(m;n) = (1850,750)$, $nnz = 8758$ and $\kappa_2 = 1.8$ (b) \texttt{illc1033} where $(m;n) =(1033;320)$, $nnz= 4732$ and $\kappa_2 = 2.1$, from the Matrix Market~\cite{Boisvert1997}.}\label{fig:overMM}
\end{figure}

Finally, we test  two problems, the \texttt{SUSY} problem  and the \texttt{covtype.binary} problem, from the library of support vector machine problems LIBSVM~\cite{Chang2011}. These problems do not form consistent linear systems, thus only the Gauss-LS and CD-LS methods are applicable, see Figure~\ref{ch:one:fig:LIBSVM}.  This is equivalent to applying the Gauss-pd and CD-pd to the least squares system $A^\top Ax=A^\top b,$ which is always consistent.

\begin{figure}
    \centering
    \begin{subfigure}[t]{0.7\textwidth}
        \centering
\includegraphics[width =  0.85\textwidth, height =4.5cm, trim= 28 270 47 285, clip ]{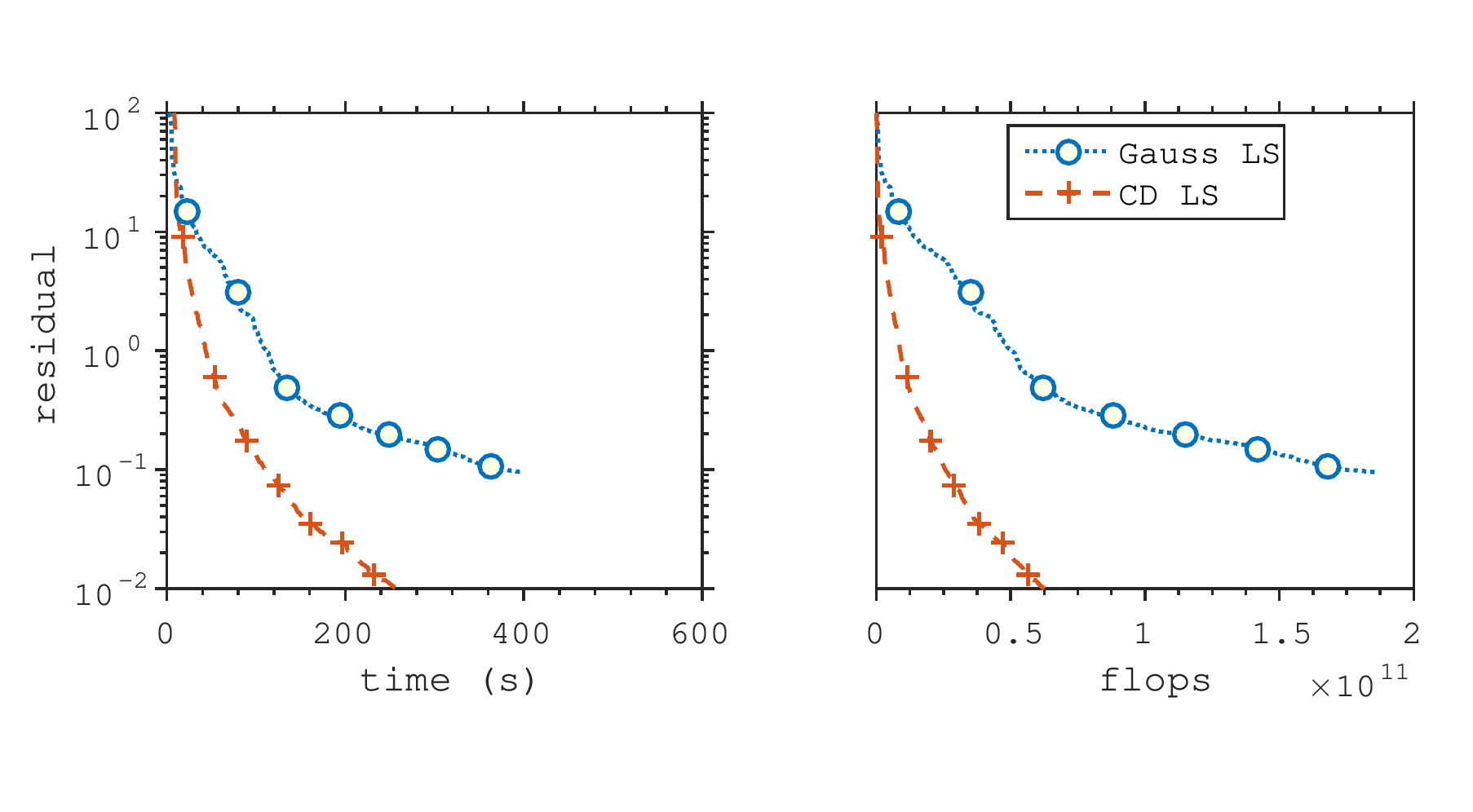}
        \caption{\texttt{SUSY}}
    \end{subfigure}%
    \hspace{0.05\textwidth}
    \begin{subfigure}[t]{0.7\textwidth}
        \centering
\includegraphics[width =  0.85\textwidth, height =4.5cm, trim= 30 270 47 285, clip ]{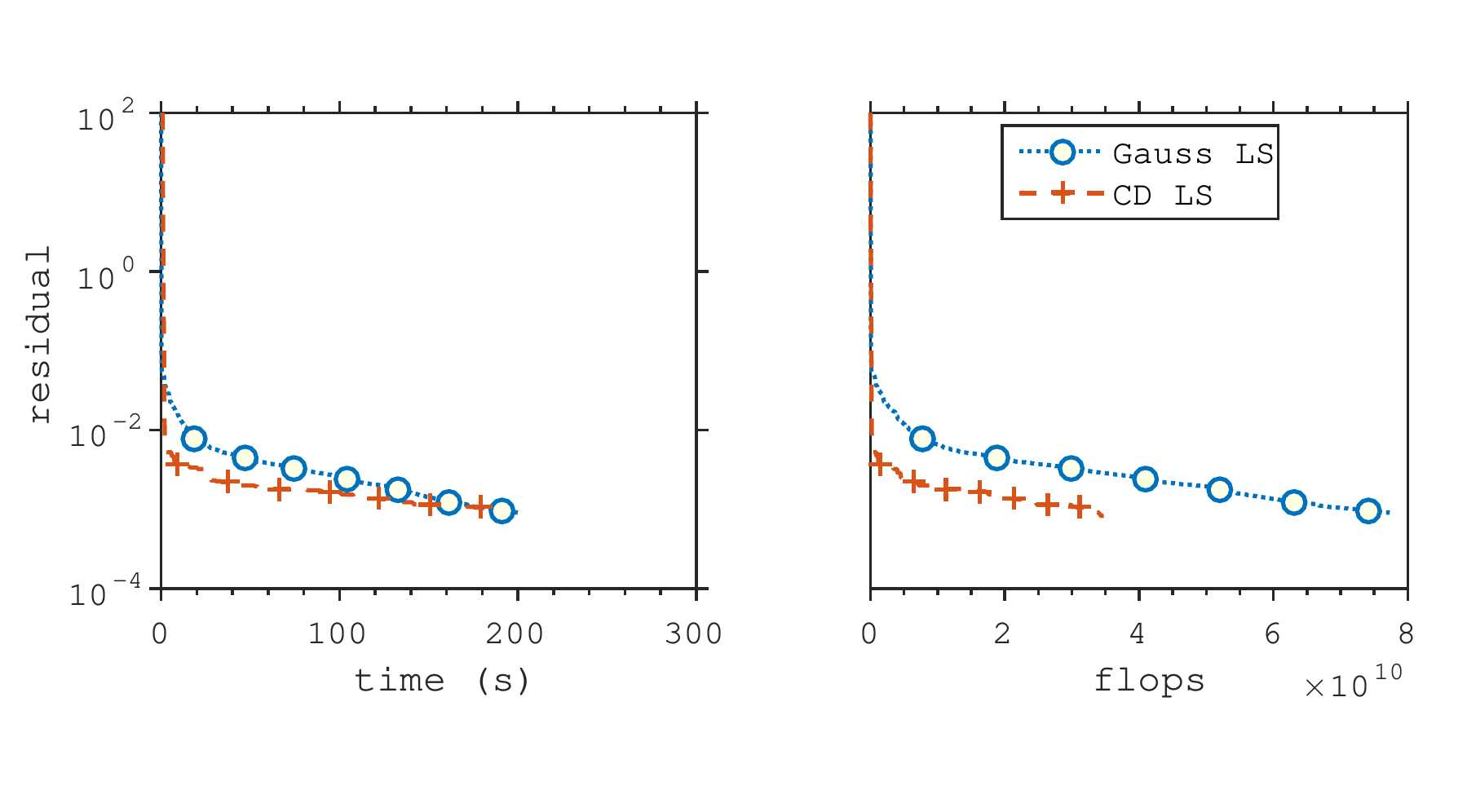}
        \caption{\texttt{covtype-libsvm-binary}}
    \end{subfigure}
    \caption{The performance of Gauss-LS and CD-LS methods on two LIBSVM test problems: (a)  \texttt{SUSY}: $(m;n)=(5\times 10^6; 18)$ (b) \texttt{covtype.binary}:  $(m;n)=(581,012; 54)$.} \label{ch:one:fig:LIBSVM}
\end{figure}
       
Despite the higher iteration cost of the Gaussian methods, their performance, in terms of the wall-clock time, is comparable to performance of the discrete methods when the system matrix is sparse.

\subsection{Bound for Gaussian convergence} \label{C2sec:numgaussbound}

Now we compare the error over the number iterations of the Gauss-LS method to theoretical rate of convergence given by the bound~\eqref{eq:rhoboundgauss}. For the Gauss-LS method~\eqref{eq:rhoboundgauss} becomes \[1-\frac{1}{n} \leq \rho \leq 1-\frac{2}{\pi}\lambda_{\min}\left( \frac{A^\top A}{\norm{A}_F^2}\right).\]
 In Figures~\ref{fig:overtheouniform} and~\ref{fig:overtheoliver} we compare the empirical and theoretical bound on a random Gaussian matrix and the \texttt{liver-disorders} problem~\cite{Chang2011}. Furthermore, we ran the Gauss-LS method 100 times and plot as a shaded region the outcomes within the 95\% and 5\% quantiles. These tests indicate that the bound is tight for well conditioned problems, such as Figure~\ref{fig:overtheouniform} in which the system matrix has a condition number equal to $1.94$. While in Figure~\ref{fig:overtheoliver} the system matrix has a condition number of $41.70$ and there is some much more slack between the empirical convergence and the theoretical bound. 
\begin{figure}
    \centering
    \begin{subfigure}[t]{0.47\textwidth}
     %   \centering
\includegraphics[width =  \textwidth, trim= 110 280 100 285, clip ]{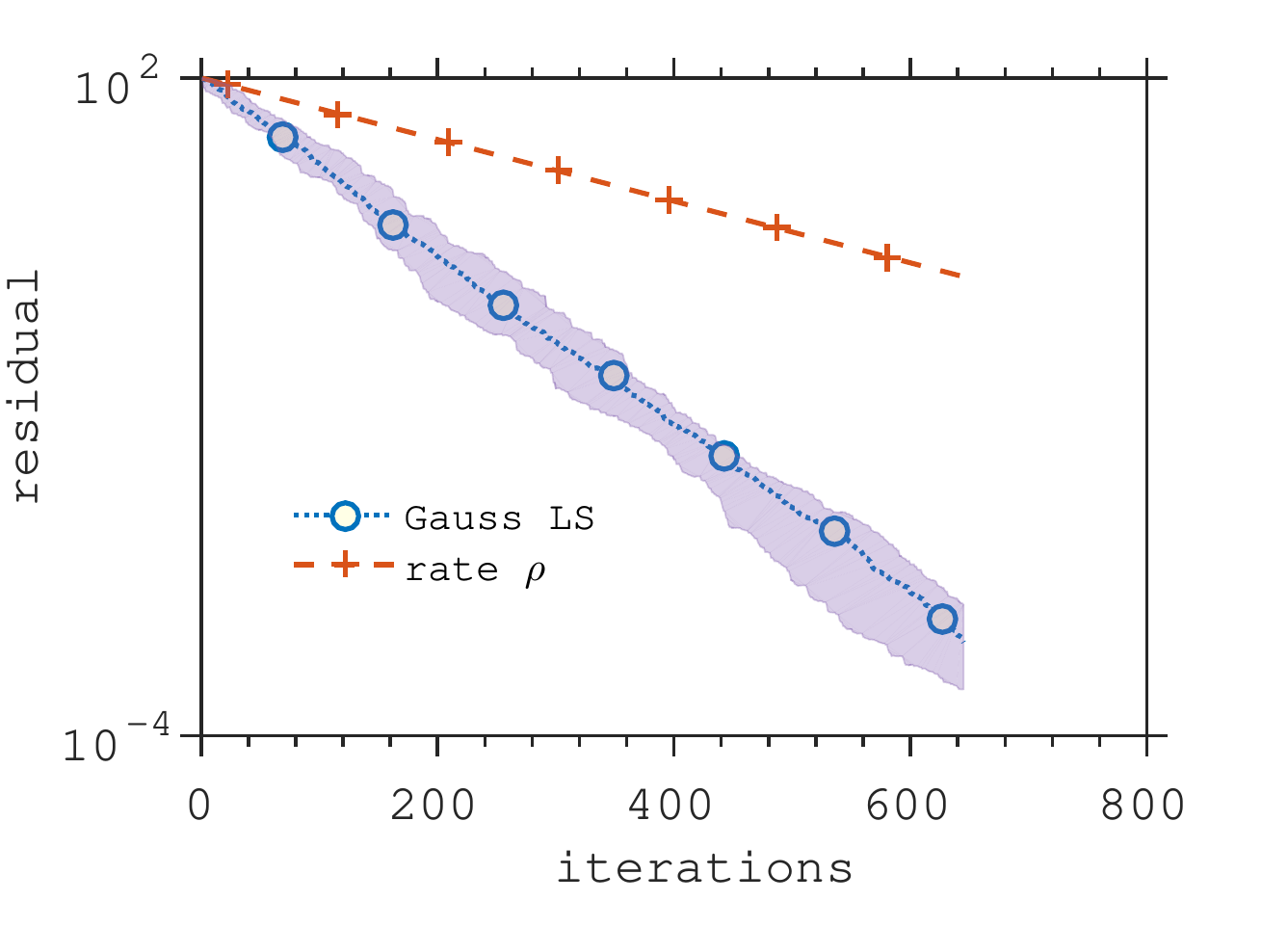}
        \caption{\texttt{rand}$(n,m)$}\label{fig:overtheouniform}
    \end{subfigure}%
    \hspace{0.05\textwidth}
    \begin{subfigure}[t]{0.47\textwidth}
       % \centering
\includegraphics[width =  \textwidth, trim= 110 280 100 285, clip ]{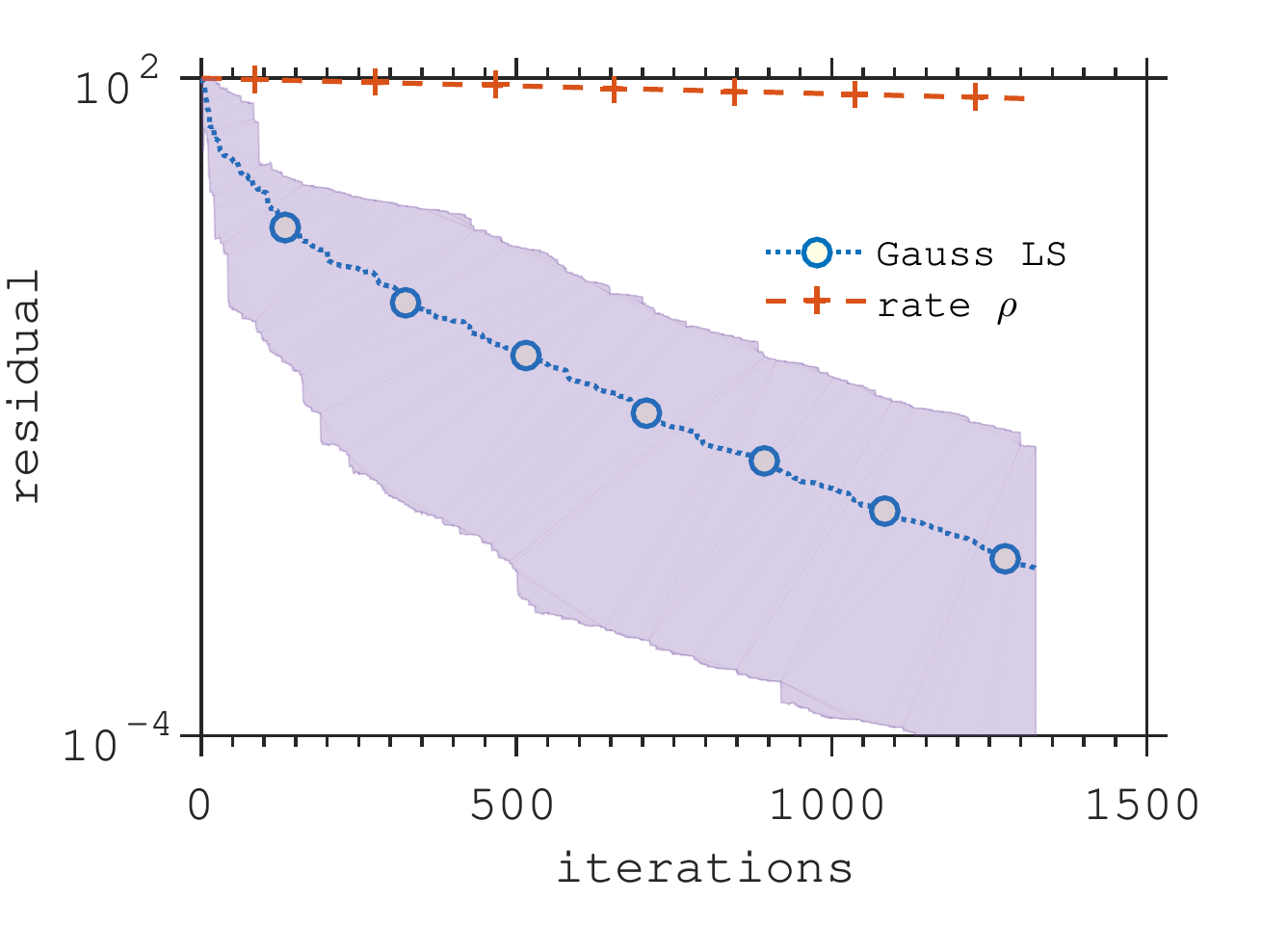}
        \caption{\texttt{liver-disorders}} \label{fig:overtheoliver}
    \end{subfigure}
    \caption{A comparison between the Gauss-LS method and the theoretical bound $\rho_{theo} \eqdef 1-\lambda_{\min}(A^\top A)/\norm{A}_F^2$ on (a) \texttt{rand}$(n,m)$ with $(m;n)=(500,50), \kappa_2 =1.94$ and a dense solution generated with $x^*=$\texttt{ rand}$(n,1)$ (b) \texttt{liver-disorders} with $(m;n) =(345,6)$ and $\kappa_2 = 41.70.$ }\label{fig:overtheo}
\end{figure}       

\subsection{Positive definite}
First we compare the two methods Gauss-pd~\eqref{eq:gausspd} and CD-pd~\eqref{eq:09j0s9jsss} on synthetic data in Figure~\ref{fig:2}.
\begin{figure}
    \centering
    \begin{subfigure}[t]{0.47\textwidth}
     %   \centering trim= 28 270 47 285,   left, bottom,right, top
\includegraphics[width =  \textwidth, trim= 40 310 60 310, clip ]{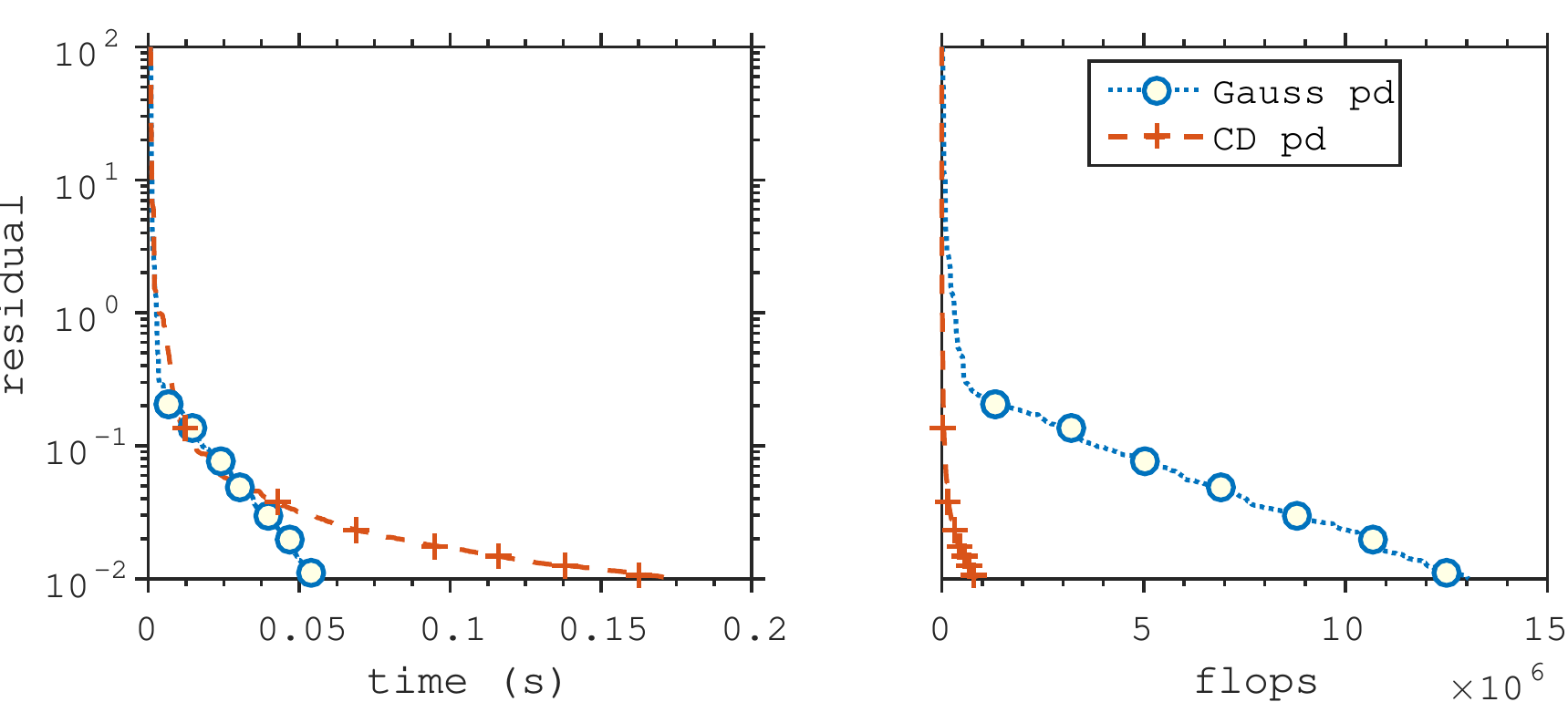}
    \end{subfigure}%
    \hspace{0.05\textwidth}
    \begin{subfigure}[t]{0.47\textwidth}
\includegraphics[width =  \textwidth, trim= 40 310 60 310, clip ]{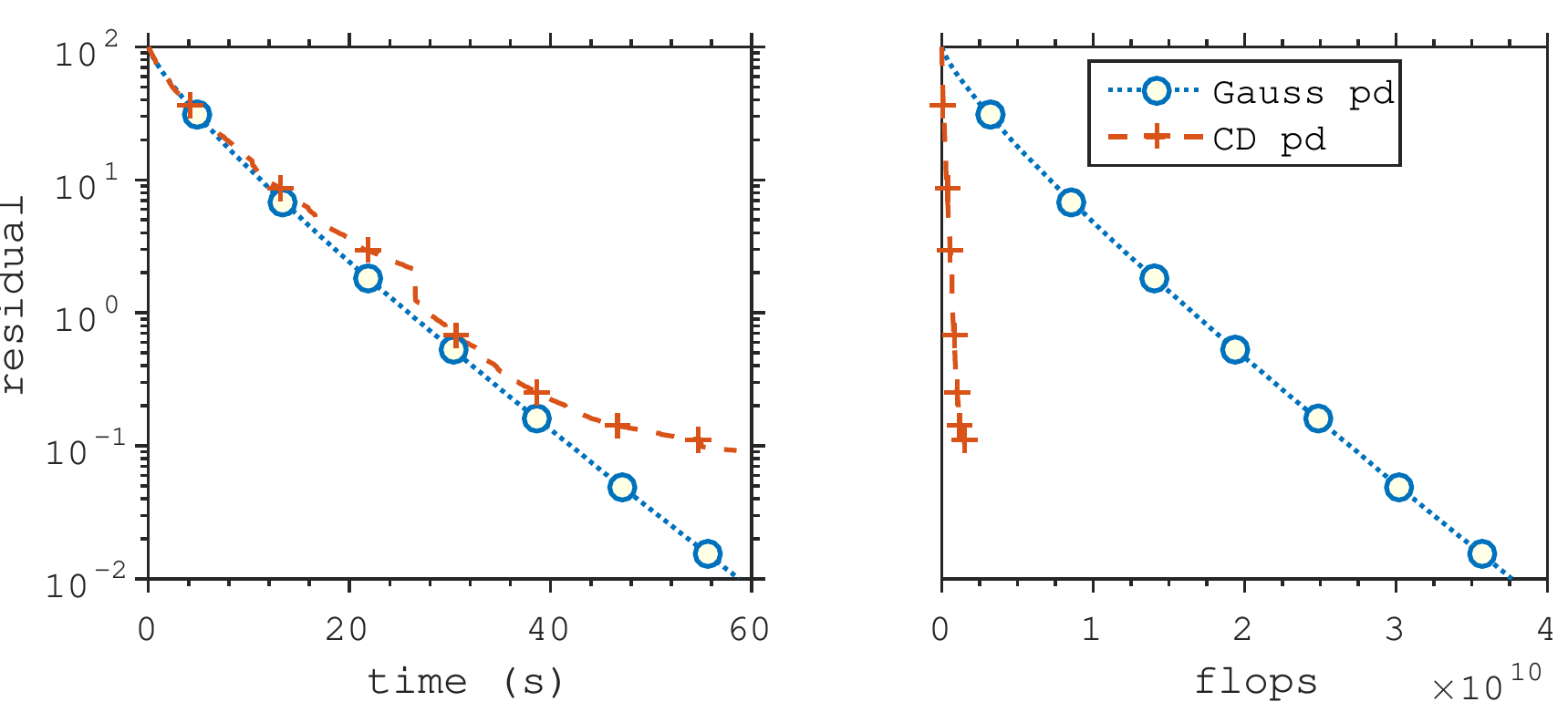}
    \end{subfigure}
    
\caption{Synthetic MATLAB generated problem. The Gaussian methods are more efficient on sparse matrices. LEFT: The \texttt{Hilbert Matrix} with $n=100$ and  condition number $\norm{A}\norm{A^{-1}} \approx 0.001133\times e^{349}$. RIGHT: Sparse random matrix $A=$ \texttt{sprandsym} ($n$, {\tt density}, {\tt rc}, {\tt type}) with $n=1000$, {\tt density}$ =1/\log(n^2)$ and ${\tt rc}= 1/n=0.001$. Dense solution generated with $x^{*}=$\texttt{rand}$(n,1).$}\label{fig:2}
\end{figure} 
Using the MATLAB function {\tt hilbert} we generate the positive definite Hilbert matrix which has a very high condition number, see Figure~\ref{fig:2}(LEFT). Indeed,  the $100\times 100$ Hilbert matrix we tested has a condition number of approximately $0.001133\times e^{349}$! Both methods converge slowly and, despite the dense system matrix, the Gauss-pd method has a similar performance to CD-pd. In Figure~\eqref{fig:2}(RIGHT) we compare the two methods on a system generated by the  MATLAB function \texttt{sprandsym} ($m$, $n$, {\tt density}, {\tt rc}, {\tt type}), where {\tt density} is the percentage of nonzero entries, {\tt rc} is the reciprocal of the condition number and {\tt type=1} returns a positive definite matrix. The Gauss-pd  and the CD-pd method have a similar performance in terms of wall clock time on this sparse problem.
\begin{figure}
    \centering
    \begin{subfigure}[t]{0.475\textwidth}
        \centering %30 270 40 285
\includegraphics[width =  \textwidth, trim= 40 310 60 310, clip ]{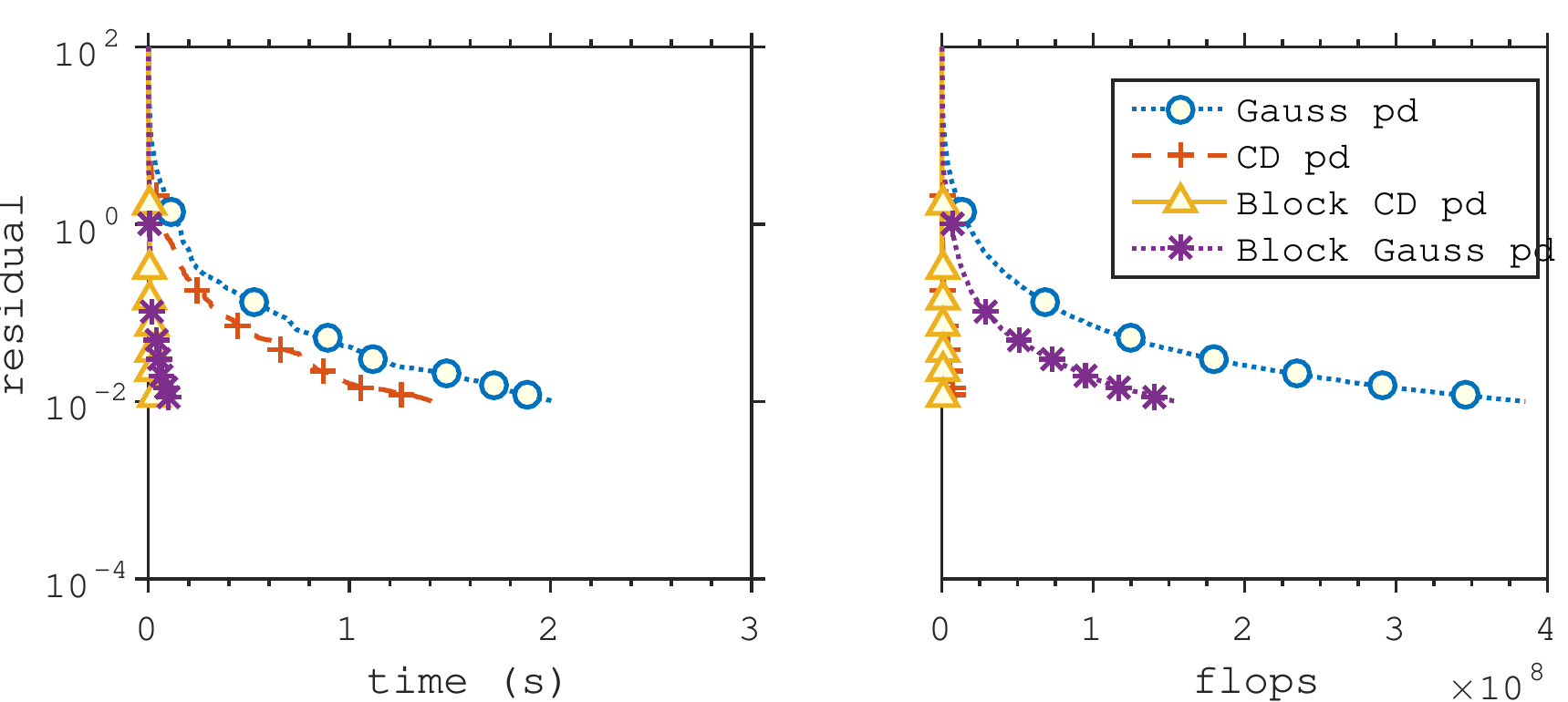}
        \caption{\texttt{aloi}}
    \end{subfigure}%
  \hspace{0.04\textwidth}
%%%%%%%%%%%%%%%%%%%%%%%%%%%%%%%%%   
    \begin{subfigure}[t]{0.475\textwidth}
        \centering
\includegraphics[width =  \textwidth, trim= 40 310 60 310, clip ]{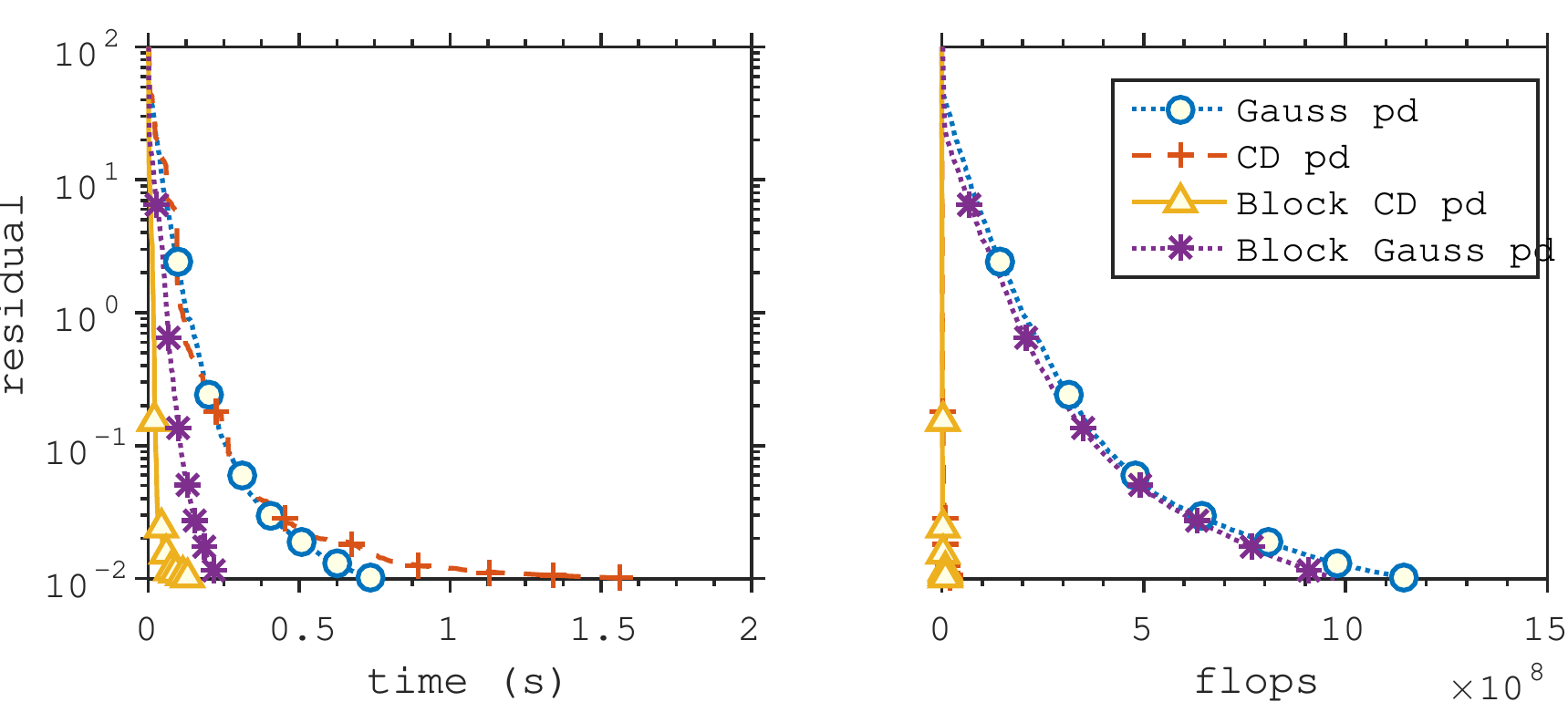}
        \caption{\texttt{protein}}
    \end{subfigure}
 % \hspace{0.05\textwidth}
%%%%%%%%%%%%%%%%%%%%%%%%%%%%%%%%%   
        \begin{subfigure}[t]{0.475\textwidth}
        \centering
\includegraphics[width =  \textwidth, trim=  40 310 60 310, clip ]{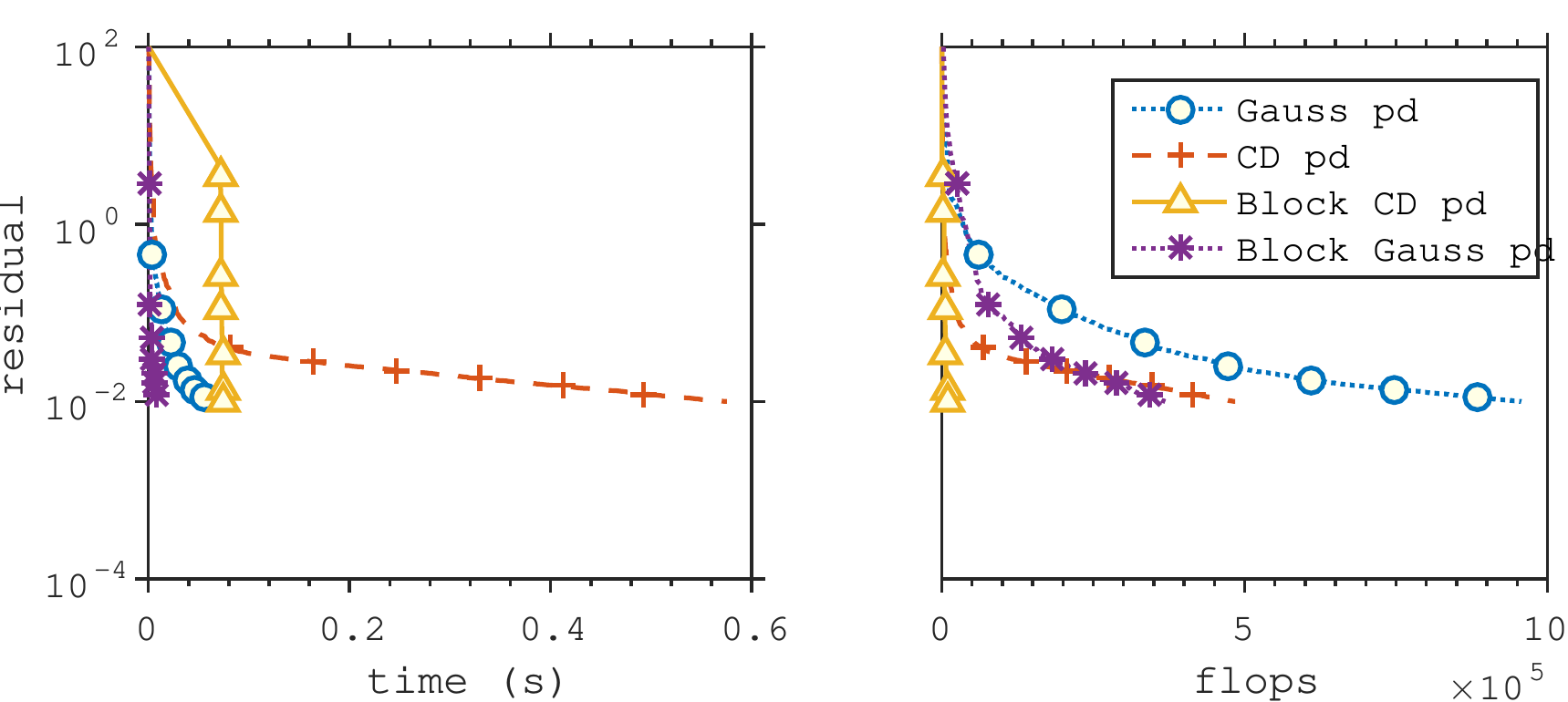}
        \caption{\texttt{SUSY}}
    \end{subfigure}%
  \hspace{0.04\textwidth}
%%%%%%%%%%%%%%%%%%%%%%%%%%%%%%%%%   
    \begin{subfigure}[t]{0.475\textwidth}
        \centering
\includegraphics[width =  \textwidth, trim=  40 310 60 310, clip ]{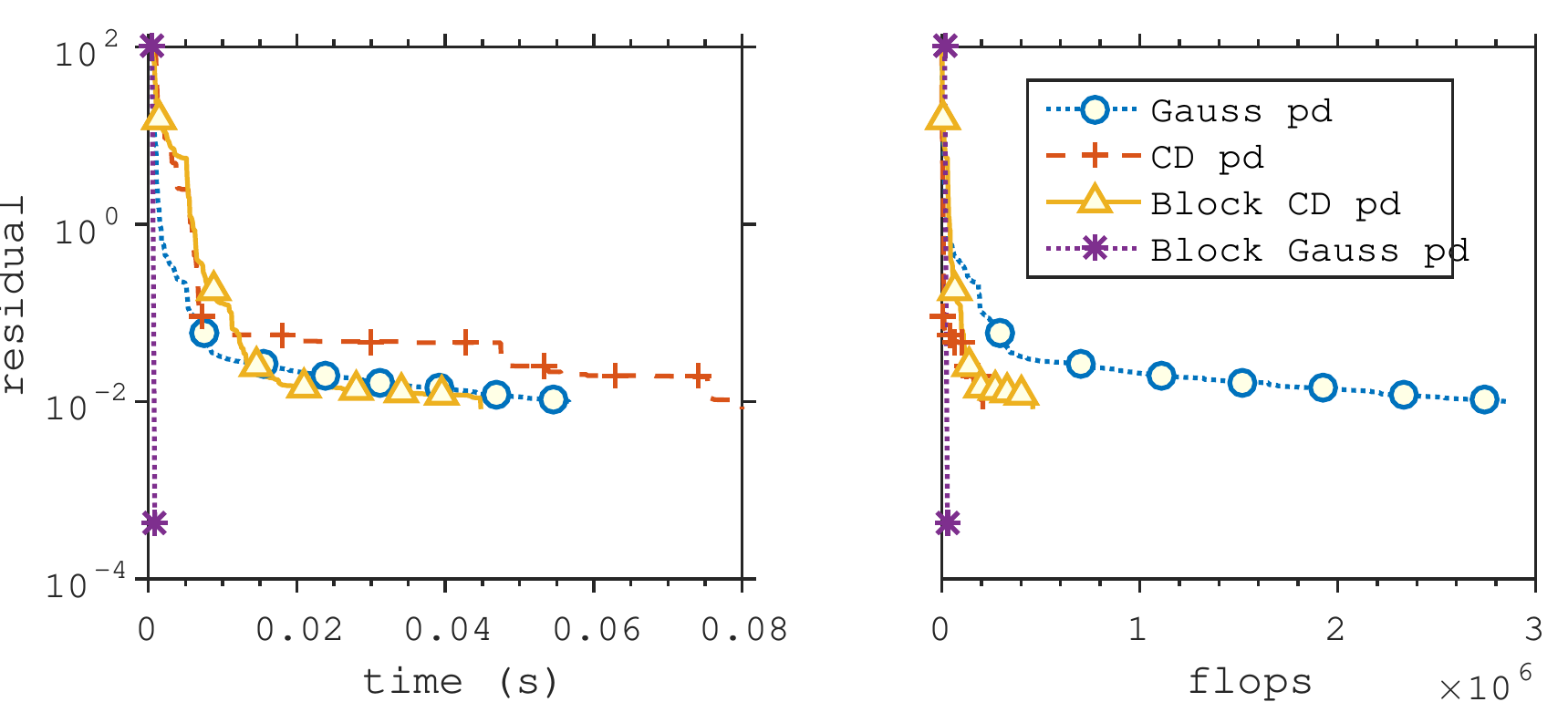}
        \caption{\texttt{covtype.binary}}
    \end{subfigure}
    \caption{The performance of Gauss-pd and CD-pd methods on four ridge regression problems: (a) \texttt{aloi}: $(m;n)=(108,000;128)$ (b) \texttt{protein}: $(m; n)=(17,766; 357)$ (c)  \texttt{SUSY}: $(m;n)=(5\times 10^6; 18)$ (d) \texttt{covtype.binary}:  $(m;n)=(581,012; 54)$.} \label{ch:one:fig:LIBSVMridge}
\end{figure}

 To appraise the performance gain in using block variants,
we perform tests using two block variants: the Randomized Newton method~\eqref{eq:CDpdblock},  which we will now refer to as the Block CD-pd method, and the Block Gauss-pd method~\eqref{eq:Bgausspd}.
We set the block size to $q= \sqrt{n}$ in both methods. To solve the $q\times q$ system required in the block methods, we use MATLAB's built-in direct solver, sometimes referred to as ``back-slash''.

Next we test the Newton system $\nabla^2 f(w_0) x = - \nabla f(w_0)$,  arising from four ridge-regression problems of the form 
 \begin{equation}\label{eq:ridgeMatrix}
\min_{w\in \R^n}f(w)\eqdef \tfrac{1}{2} \norm{Aw-b}_2^2 + \tfrac{\lambda}{2} \norm{w}_2^2,
\end{equation}
using data from LIBSVM~\cite{Chang2011}. In particular, we set
$w_0=0$ and use $\lambda =1$ as the regularization parameter, whence $\nabla f(w_0) = A^\top b$ and $\nabla^2 f(w_0) = A^\top A+ I$.

% In reaching a low precision solution with $1\%$ error, the  CD-pd method and Gauss-pd method have a comparable performance, see Figure~\ref{ch:one:fig:LIBSVMridge}.
 In terms of wall clock time, the Gauss-pd method converged faster on all problems accept the \texttt{aloi} problem as compared with CD-pd. 
The two Block methods had a comparable performance on the \texttt{aloi} and the \texttt{protein} problem. 
The Block Gauss-pd method converged in one iteration on \texttt{covtype.binary} and was the fastest method on the \texttt{SUSY} problem.

We now compare the methods on two positive definite matrices from the Matrix Market collection~\cite{Boisvert1997}, see Figure~\ref{fig:pdMM}. The right-hand side was not supplied by the data set, and thus we generated $b$ using {\tt  rand(n,1)}. 
 The  Block CD-pd method converged much faster on both problems. 
The lower condition number  ($\kappa_2=12$) of the \texttt{gr\_30\_30-rsa} problem resulted in fast convergence of all methods, see Figure~\ref{fig:gr_30_30-rsa}. While the high condition number  ($\kappa_2=4.3 \cdot 10^4$)  of the \texttt{bcsstk18} problem, resulted in a slow convergence for all methods, see Figure~\ref{fig:bcsstk18-rsa}.

\begin{figure}
    \centering
    \begin{subfigure}[t]{0.47\textwidth}
     %   \centering
\includegraphics[width =  \textwidth, trim= 40 310 60 310, clip ]{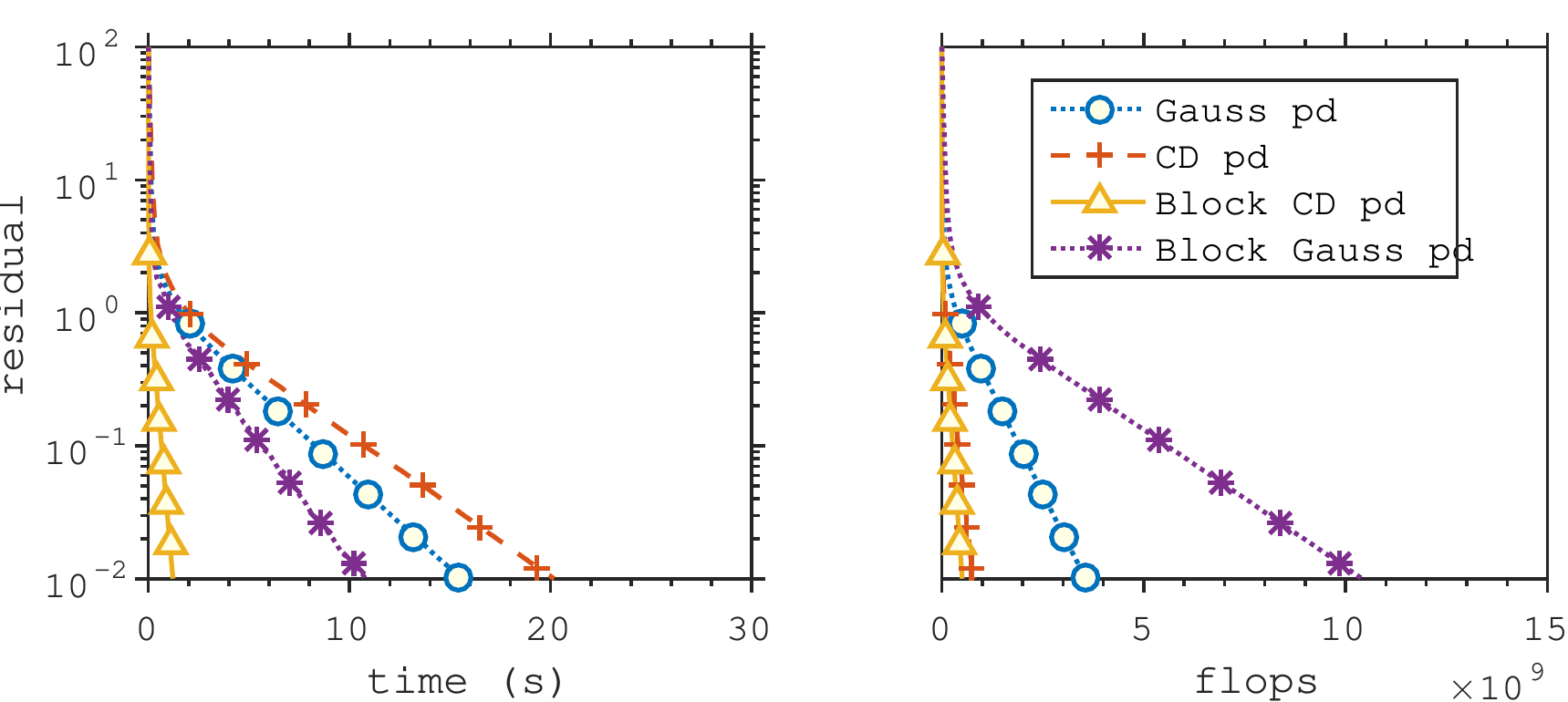}
        \caption{ \texttt{gr\_30\_30-rsa}}\label{fig:gr_30_30-rsa}
    \end{subfigure}%
    \hspace{0.05\textwidth}
    \begin{subfigure}[t]{0.47\textwidth}
       % \centering
\includegraphics[width =  \textwidth, trim= 40 310 60 310, clip ]{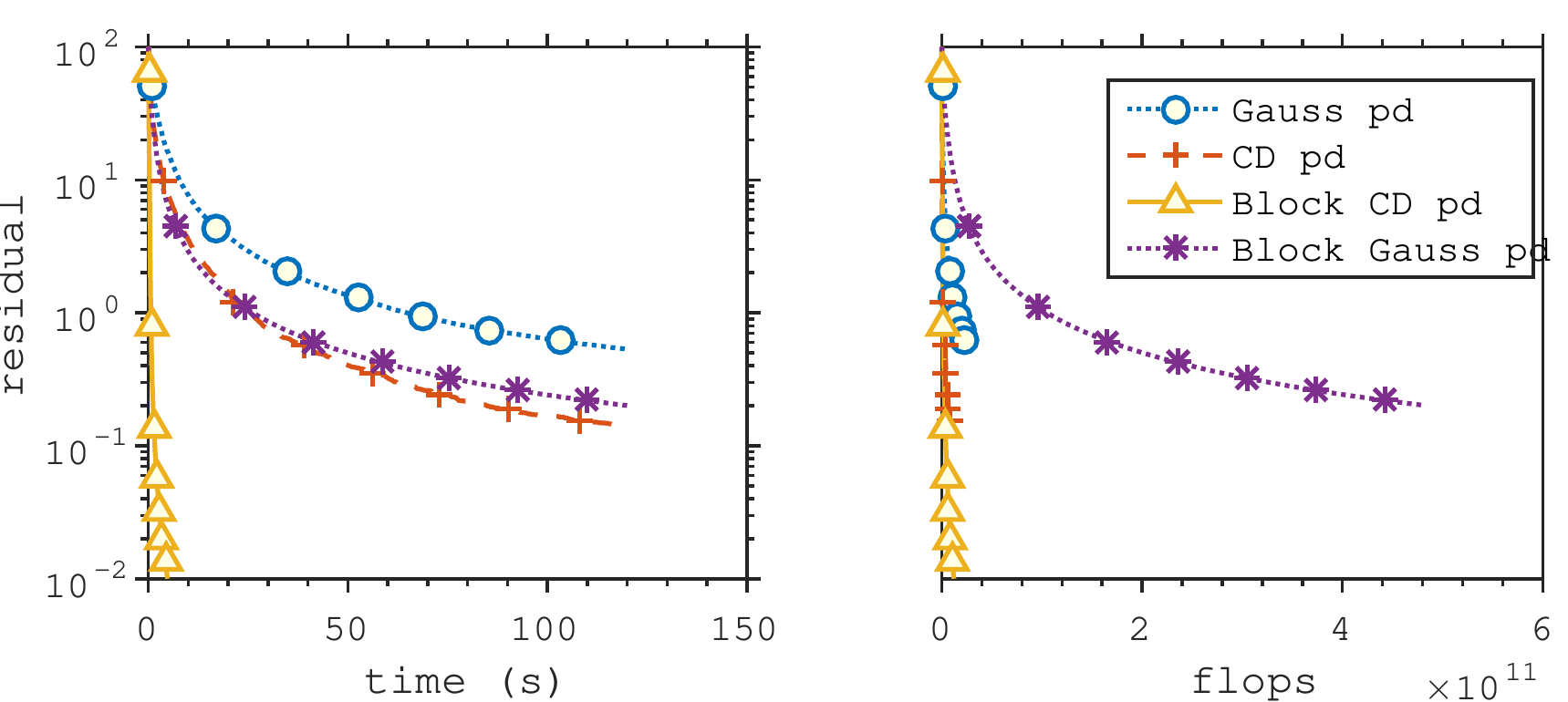}
        \caption{\texttt{bcsstk18}} \label{fig:bcsstk18-rsa}
    \end{subfigure}
    \caption{The performance of the Gauss-pd, CD-pd and the Block CD-pd methods on two linear systems from the MatrixMarket (a) \texttt{gr\_30\_30-rsa} with $n = 900$, $nnz = 4322$ ({\tt density}$=0.53\%$) and $\kappa_2 =12.$ (b) \texttt{bcsstk18} with $n = 11948$, $nnz=80519$ ({\tt density}$=0.1\%$) and  $\kappa_2 = 4.3 \cdot 10^{10} $.}\label{fig:pdMM}
\end{figure}

Despite the clear advantage of using a block variant, applying a block method that uses a direct solver can be infeasible on very ill-conditioned problems. As an example, applying the Block CD-pd to the Hilbert system, and using MATLAB back-slash solver to solve the inner $q\times q$ systems, resulted in large numerical inaccuracies, and ultimately, prevented the method from converging. This occurred because the submatrices of the Hilbert matrix are also very ill-conditioned.

\subsection{Comparison between optimized and convenient probabilities}\label{C2sec:numopt}
 We compare the practical performance of using the convenient probabilities~\eqref{ch:one:eq:convprob} against using the optimized probabilities by solving~\eqref{eq:optconv}. We solved~\eqref{eq:optconv} using the disciplined convex  programming solver \texttt{cvx}~\cite{cvx} for MATLAB.
%Though the convenient probability distribution~\eqref{ch:one:eq:convprob} yields an interpretable convergence rate, this rate may not be a good one. 

In Table~\ref{tab:CD-pd-opt} we compare the different convergence rates for the CD-pd method, where $\rho_c$ is the convenient convergence rate~\eqref{ch:one:eq:rhoconv}, $\rho^*$ the optimized convergence rate, $(1-1/n)$ is the lower bound, and in the final ``optimized time(s)'' column the time taken to compute $\rho^*$. In Figure~\ref{fig:CD-pd-opt}, we compare the empirical convergence of the CD-pd method when using the convenient probabilities~\eqref{ch:one:eq:convprob} and CD-pd-opt, the CD-pd method with the optimized probabilities. We tested the two methods on four ridge regression problems and a synthetic positive definite system which is the square of a uniform random matrix:  $A~=~\bar{A}^\top \bar{A}$ where $\bar{A}=$\texttt{rand}$(50)$.

We ran each method for $60$ seconds. 

 In most cases using the optimized probabilities results in a much faster convergence, see Figures~\ref{fig:aloi-opt},~\ref{fig:liver-opt},~\ref{fig:mushrooms-opt} and~\ref{fig:uniform-opt}. In particular, the $7.401$ seconds spent calculating the optimal probabilities for \texttt{aloi} paid off with a convergence that was $55$ seconds faster. The \texttt{mushrooms} problem was insensitive to the choice of probabilities~\ref{fig:mushrooms-opt}. Finally despite $\rho^*$ being much less than $\rho_c$ on \texttt{covtype}, see Table~\ref{tab:CD-pd-opt}, using optimized probabilities  resulted in an initially slower method, though CD-pd-opt eventually catches up as CD-pd stagnates, see Figure~\ref{fig:covtype-opt}.
% goes as warning, that optimizing an upper bound on the rate of convergence does not %always result in a faster method in practice.
 
 % guarantee that the resulting method will be faster in practice.
\begin{table} \centering
\footnotesize
\begin{tabular}{c|lll|c} \hline
data set & $\rho_c$ &$\rho^* $ & $1-1/n$ & optimized time(s) \\ \hline
\texttt{rand}(50,50) & $1-2\cdot 10^{-6}$ & $1-3.05\cdot 10^{-6}$ & $1-2.10^{-2}$ & 1.076\\
 {\tt mushrooms-ridge} & $1-5.86\cdot 10^{-6}$ & $1-7.15\cdot 10^{-6}$ & $1-8.93\cdot 10^{-3}$ &  4.632\\
  {\tt aloi-ridge} & $1-2.17\cdot 10^{-7}$ & $1-1.26\cdot 10^{-4}$ & $1-7.81\cdot 10^{-3}$ &  7.401\\
  {\tt liver-disorders-ridge} & $1-5.16\cdot 10^{-4}$ & $1-8.25\cdot 10^{-3}$ & $1-1.67\cdot 10^{-1}$ &  0.413\\  
   {\tt covtype.binary-ridge} & $1-7.57\cdot 10^{-14}$ & $1-1.48\cdot 10^{-6}$ & $1-1.85\cdot 10^{-2}$ &  1.449\\  	
\end{tabular}
\caption{Optimizing the convergence rate for CD-pd.}
\label{tab:CD-pd-opt}
\end{table}

\begin{figure}[!h]
    \centering
    \begin{subfigure}[t]{0.45\textwidth}
        \centering %left, bottom,right, top 120 295 120 300
\includegraphics[width =  \textwidth, trim=100 280 100 290, clip ]{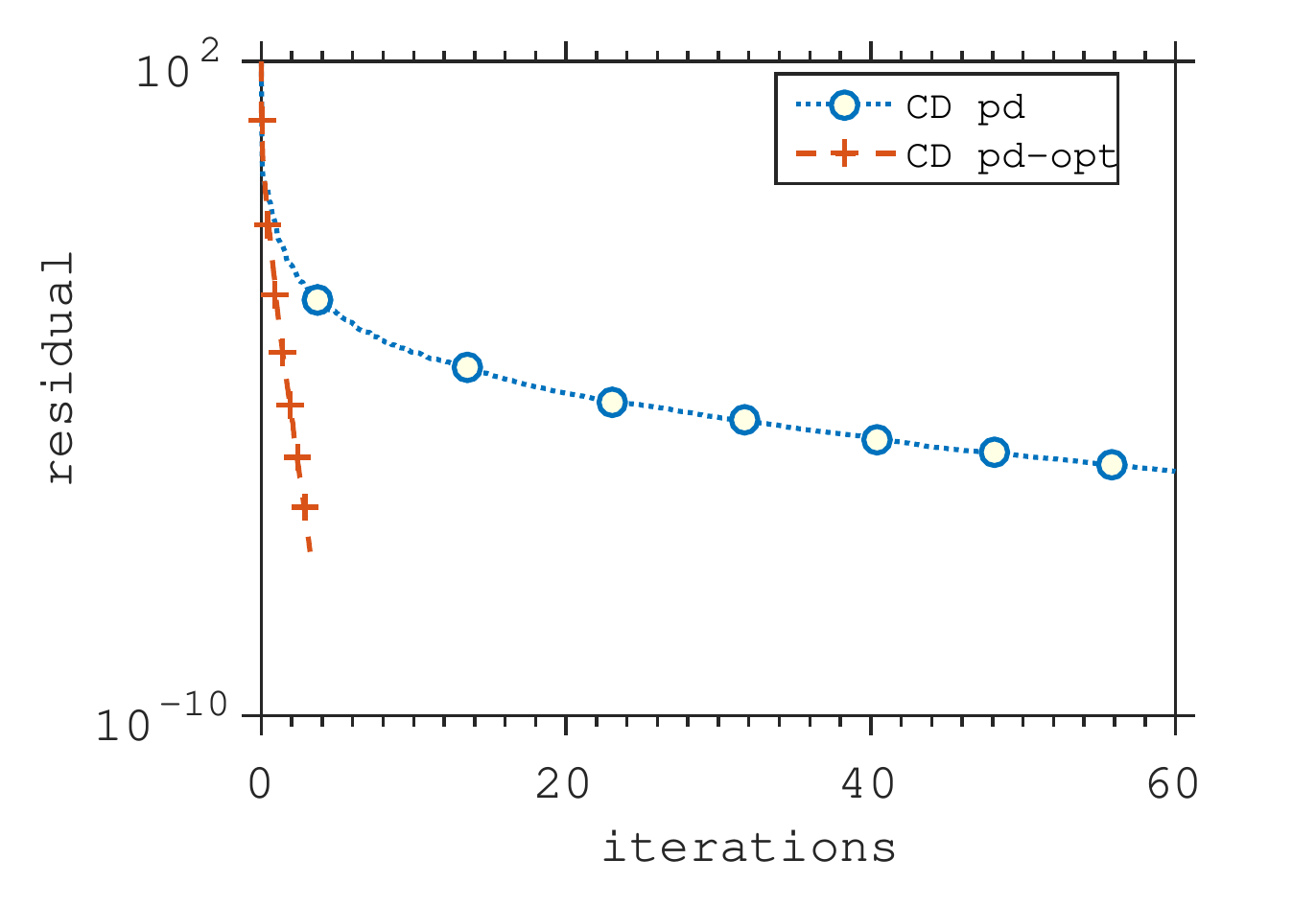}
        \caption{\texttt{aloi}}
        \label{fig:aloi-opt}
    \end{subfigure}%
  \hspace{0.02\textwidth}
%%%%%%%%%%%%%%%%%%%%%%%%%%%%%%%%%   
    \begin{subfigure}[t]{0.45\textwidth}
        \centering
\includegraphics[width =  \textwidth, trim=100 280 100 290, clip ]{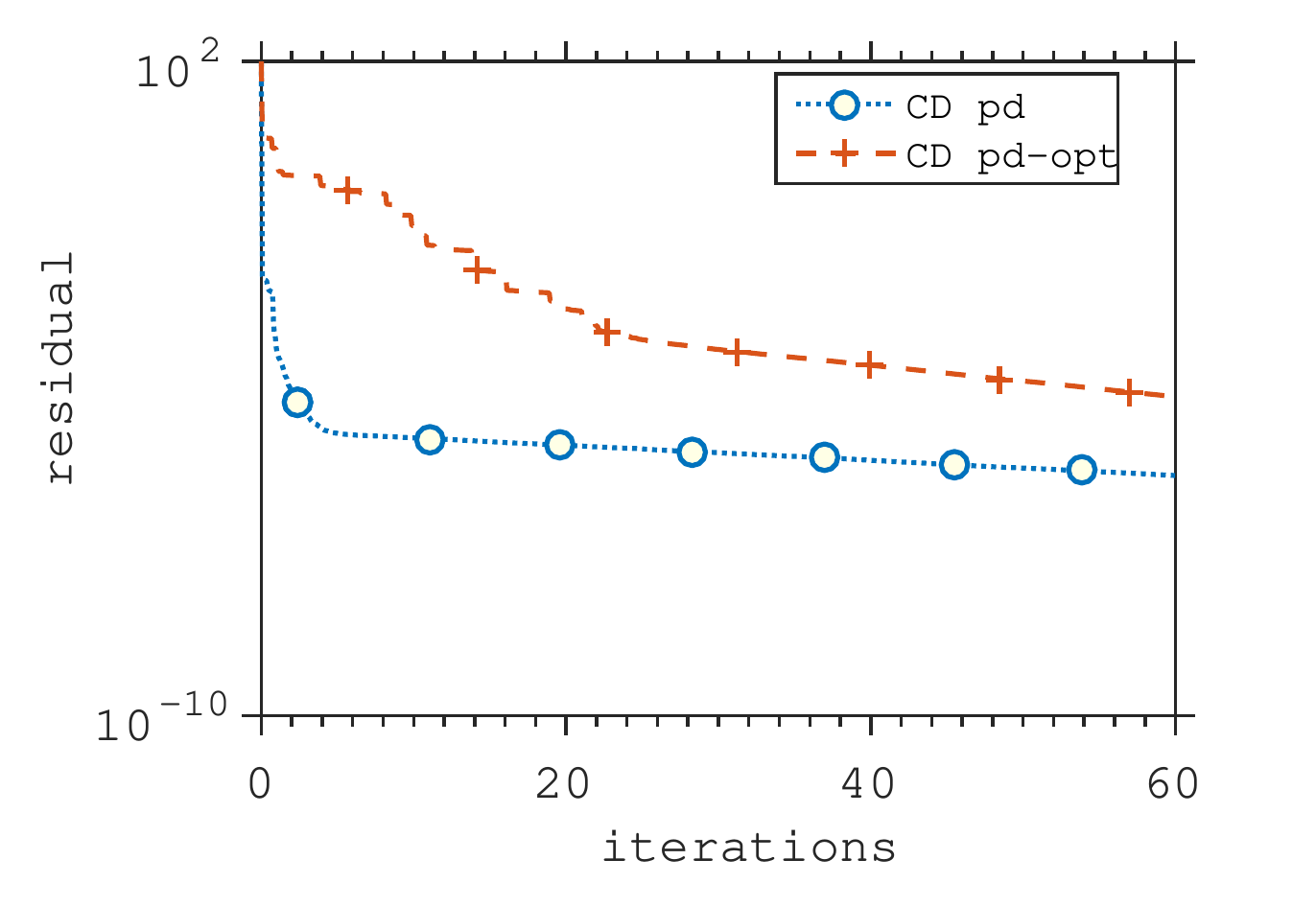}
        \caption{\texttt{covtype.libsvm.binary}}
        \label{fig:covtype-opt}
    \end{subfigure}
    %%%%%%%%%%%%%%%%%%%%%%%%%%%%%%%%%   
    \begin{subfigure}[t]{0.45\textwidth}
        \centering
\includegraphics[width =  \textwidth, trim=100 280 100 290, clip ]{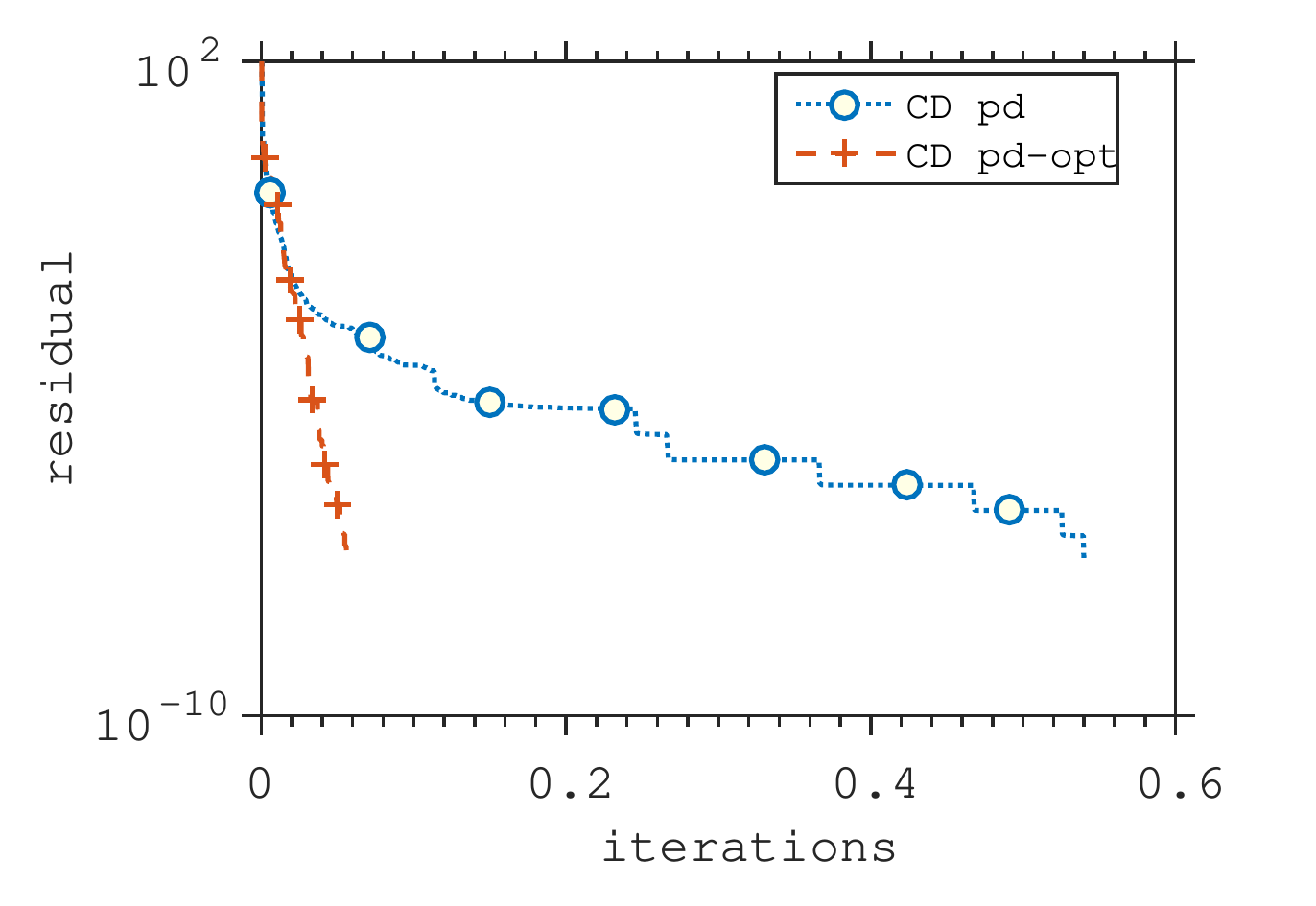}
        \caption{\texttt{liver-disorders-ridge}}
         \label{fig:liver-opt}
    \end{subfigure}
    %%%%%%%%%%%%%%%%%%%%%%%%%%%%%%%%%   
    \begin{subfigure}[t]{0.45\textwidth}
        \centering
\includegraphics[width =  \textwidth, trim=100 280 110 270, clip ]{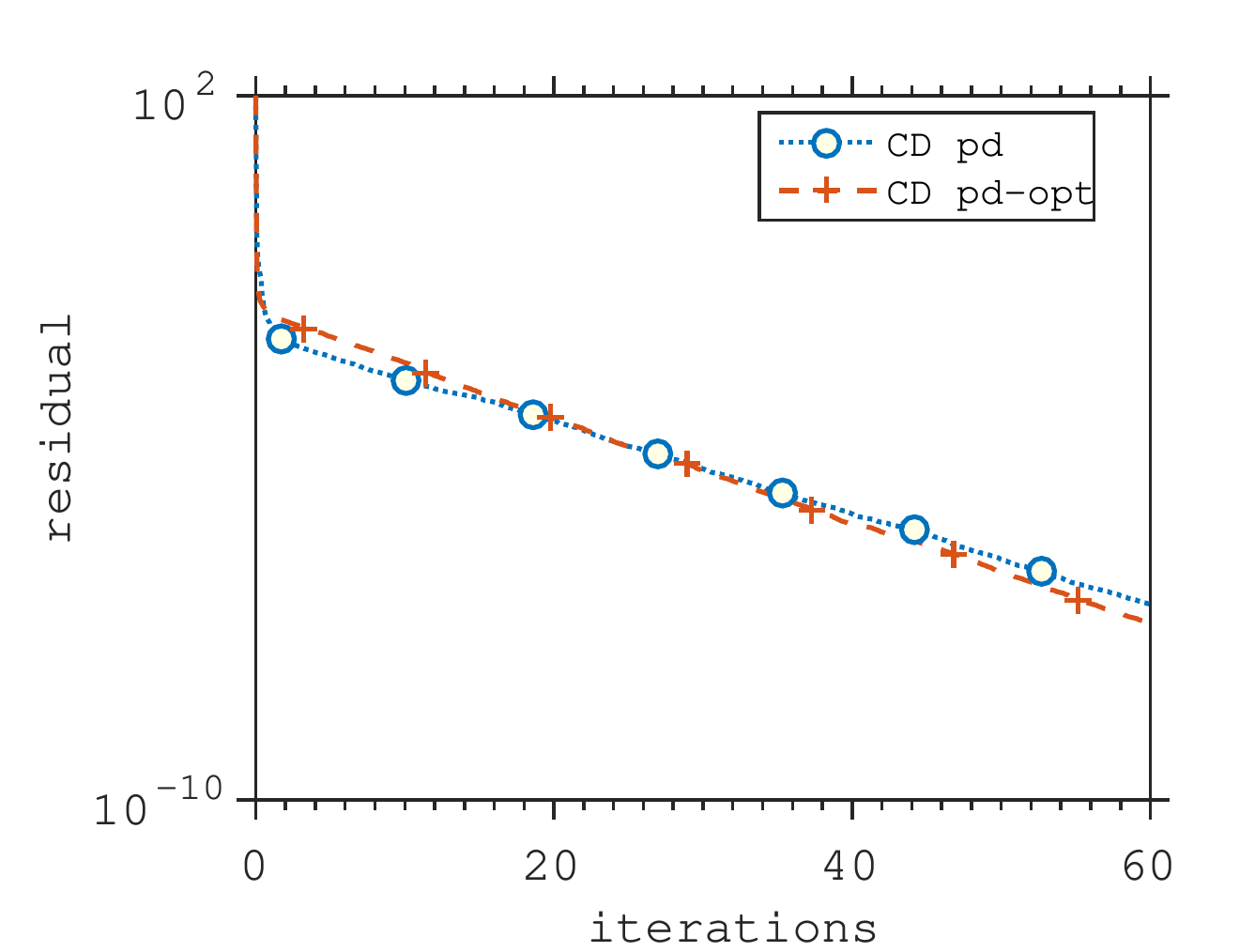}
        \caption{\texttt{mushrooms-ridge}}
         \label{fig:mushrooms-opt}
    \end{subfigure} 
    %%%%%%%%%%%%%%%%%%%%%%%%%%%%%%%%%   
    \begin{subfigure}[t]{0.45\textwidth}
        \centering
\includegraphics[width =  \textwidth, trim=100 280 100 290, clip ]{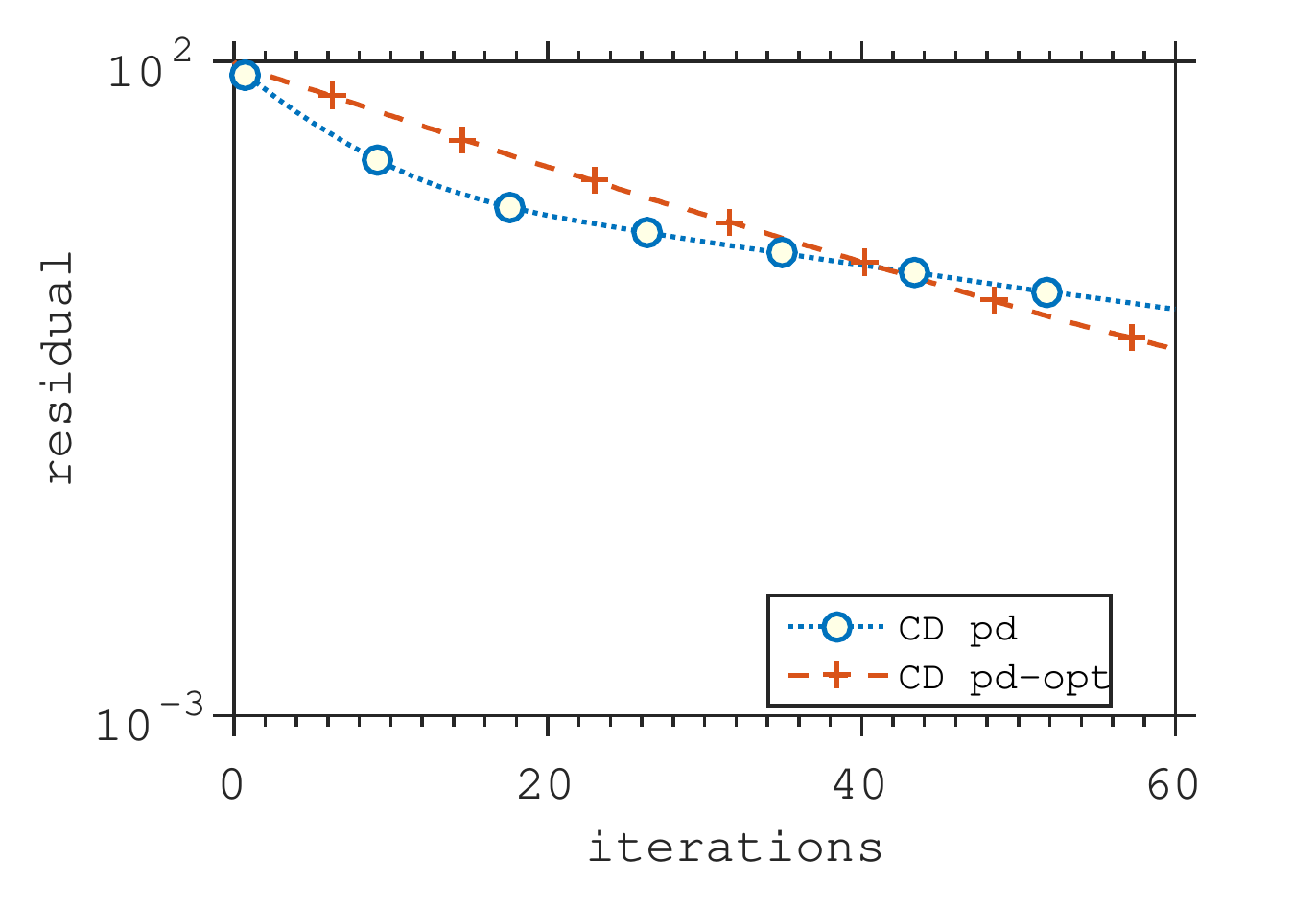}
        \caption{$50X50$ synthetic positive definite}
        \label{fig:uniform-opt}
    \end{subfigure}    
    \caption{The performance of CD-pd and optimized CD-pd methods on (a) \texttt{aloi}: $(m;n)=(108,000;128)$ (b) \texttt{covtype.binary}:  $(m;n)=(581,012; 54)$ (c) \texttt{liver-disorders}: $(m;n)=(345,6)$ (c)\texttt{mushrooms}: $(m;n) = (8124,112)$  (d) $A~=~\bar{A}^\top \bar{A}$ where $\bar{A}=$\texttt{rand}$(50)$.} \label{fig:CD-pd-opt}
\end{figure}

In Table~\ref{tab:RK-pd-opt} we compare the different convergence rates for the RK method. In Figure~\ref{fig:kaczmacz-opt}, we then compare the empirical convergence of the RK method when using the convenient probabilities~\eqref{ch:one:eq:convprob} and RK-opt, the RK method with the optimized probabilities by solving~\eqref{eq:optconv}. The rates $\rho^*$ and $\rho_c$ for the  \texttt{rand}(500,100) problem are similar, and accordingly, both the convenient and optimized variant converge at a similar rate in practice, see Figure~\ref{fig:kaczmacz-opt}b. While the difference in the rates  $\rho^*$ and $\rho_c$ for the  {\tt liver-disorders} is more pronounced, and in this case, the $0.83$ seconds invested in obtaining the optimized probability distribution paid off in practice, as the optimized method converged $1.25$ seconds before the RK method with the convenient probability distribution, see Figure~\ref{fig:kaczmacz-opt}a.

\begin{table} \centering
\footnotesize
\begin{tabular}{c|lll|c} \hline
data set & $\rho_c$ &$\rho^* $ & $1-1/n$ & optimized time(s) \\ \hline
\texttt{rand}(500,100) & $1-3.37\cdot 10^{-3}$ & $1-4.27\cdot 10^{-3}$ & $1-1\cdot 10^{-2}$ & 33.121\\
  {\tt liver-disorders} & $1-5.16\cdot 10^{-4}$ & $1-4.04\cdot 10^{-3}$ & $1-1.67 \cdot 10^{-1}$ &  0.8316\\
\end{tabular}
\caption{Optimizing the convergence rate for randomized Kaczmarz.}
\label{tab:RK-pd-opt}
\end{table}

\begin{figure}[!h]
    \centering
    \begin{subfigure}[t]{0.45\textwidth}
        \centering
\includegraphics[width =  \textwidth, trim=100 280 100 290, clip ]{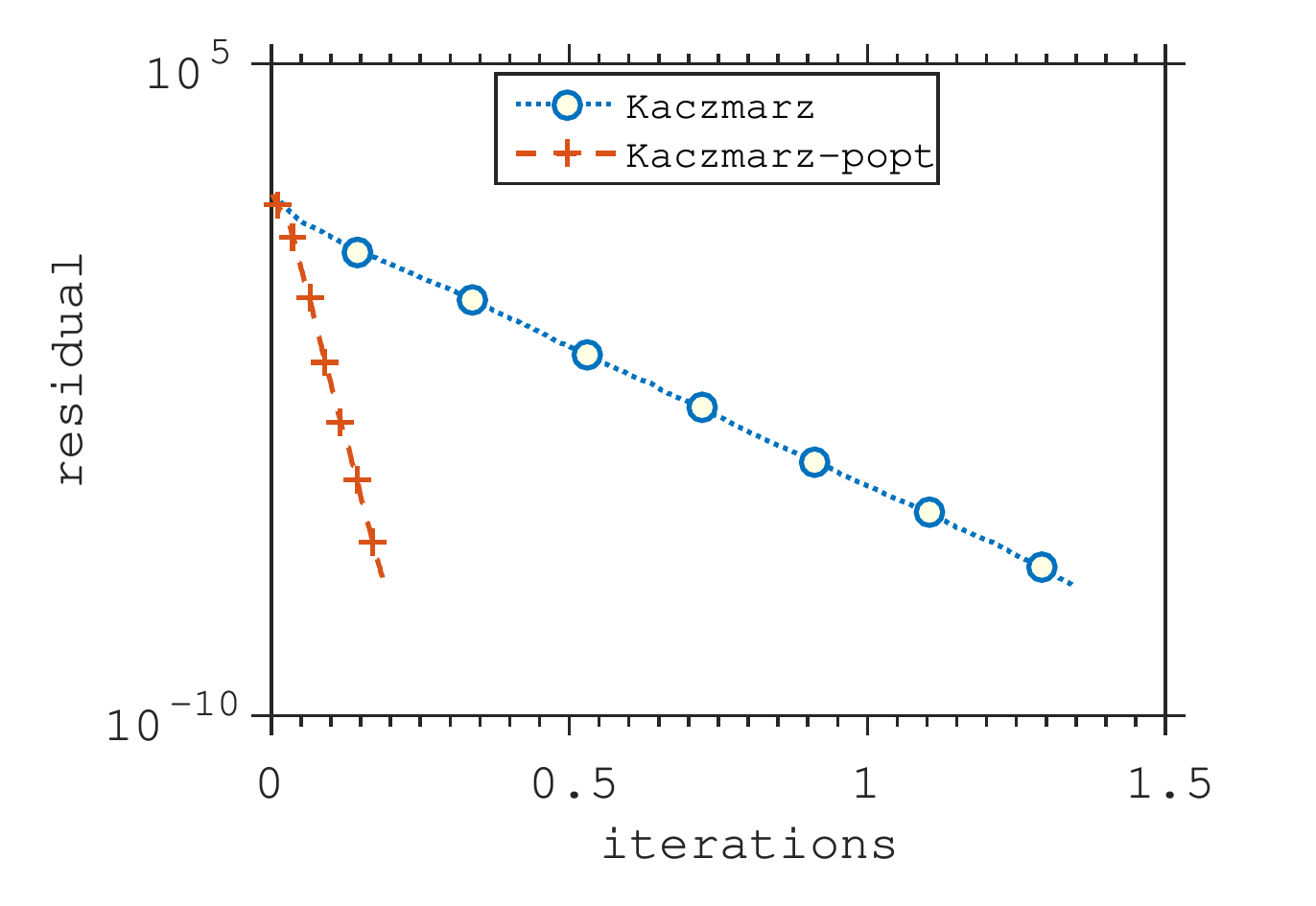}
        \caption{\texttt{liver-disorders-popt-k}}
    \end{subfigure}%
  \hspace{0.02\textwidth}
%%%%%%%%%%%%%%%%%%%%%%%%%%%%%%%%%   
    \begin{subfigure}[t]{0.45\textwidth}
        \centering
\includegraphics[width =  \textwidth, trim=100 280 100 290, clip ]{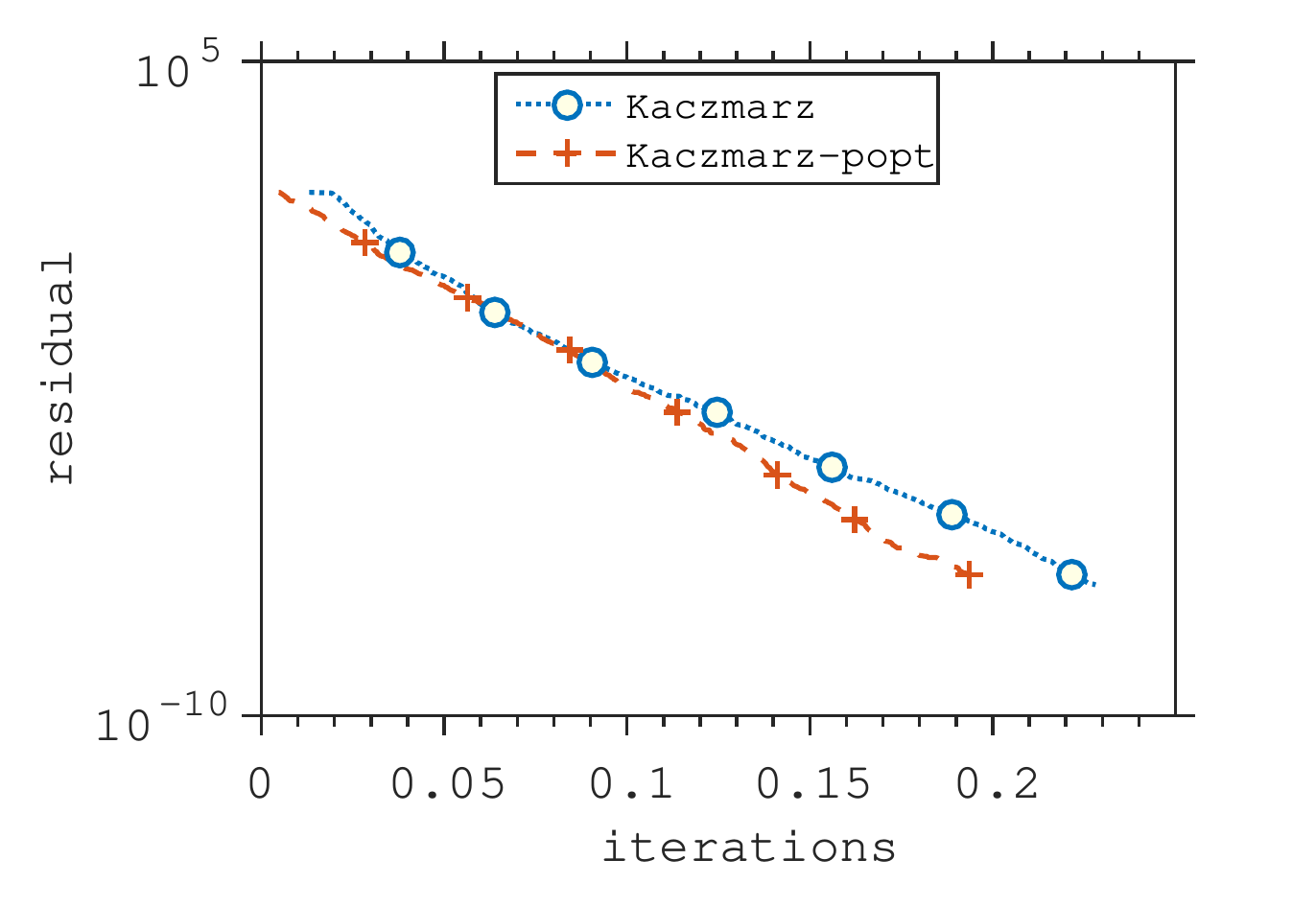}
        \caption{\texttt{rand}(500,100)}
    \end{subfigure}
    \caption{The performance of the Kaczmarz and optimized Kaczmarz methods on  (a) \texttt{liver-disorders}: $(m;n)=(345,6)$  (b) \texttt{rand}(500,100)}
\label{fig:kaczmacz-opt}
\end{figure} 

We conclude from these tests that the choice of the probability distribution can greatly affect the performance of the method. Hence, it is worthwhile to develop approximate solutions to~\eqref{eq:opt_sampling}.

 \subsection{Conclusion of numeric experiments}
 We now summarize the findings of our numeric experiments.
 \begin{itemize}
  \item Consistently across our experiments, in terms of number of flops taken to reach a desired precision, the three discrete sampling methods  CD-LS, CD-pd and Kaczmarz are the most efficient. 
That is, the Gaussian methods almost always require more flops to reach a solution with the same precision as their discrete sampling counterparts. This is due to the expensive matrix-vector product required by the Gaussian methods.
%  This is due to the very low iteration cost of these methods. This indicates that for 
%While the results are more mixed when measured in terms of wall clock time. This is 
 \item In terms of wall-clock time, the Gaussian methods Guass-LS, Gauss-pd and Gauss  Kaczmarz are competitive as compared to the discrete sampling methods. This occurred because MATLAB  performs automatic multi-threading when calculating matrix-vector products, which was the bottleneck cost in the Gaussian methods. As our machine has four cores, this explains some of the difference observed when measuring performance in terms of  number of flops and wall clock time.
 \item  In terms of both time taken and flops, the block variants proved to be significantly more efficient.  
 \item Using the optimized probabilities~\eqref{eq:optconv} for the discrete sampling methods RK and CD-pd can result in significant speed-ups, as compared to RK and CD-pd using the convenient probabilities~\eqref{ch:one:eq:convprob}. So much so, that the time spent solving the SDP~\eqref{eq:optconv} using {\tt cvx}~\cite{cvx} often paid off. We can draw two interesting conclusions from this: (1) using convenient probabilities~\eqref{ch:one:eq:convprob} is not the best choice, but simply, the choice that provides easily interpretable convergence rates, (2) it is worth further investigating the use of optimization probabilities, for instance, one should investigate it is possible to obtain affordable approximate solutions to~\eqref{eq:optconv}. 
 \end{itemize}
% \clearpage
\section{Summary} \label{C2sec:conclusion}

In this chapter we presented a unifying framework for the randomized Kaczmarz method, randomized Newton method, randomized coordinate descent method and  random Gaussian pursuit. Not only can we recover these methods by selecting appropriately the parameters $S$ and $B$, but also, we can analyze them and their block variants through a single Theorem~\ref{ch:two:theo:Enormconv}. Furthermore, we obtain a new lower bound for all these methods in Theorem~\ref{ch:two:theo:normEconv}, and in the discrete case, recover all known convergence rates expressed in terms of the scaled condition number in Theorem~\ref{theo:convsingleS}. 

Theorem~\ref{theo:convsingleS} also suggests a preconditioning strategy. Developing preconditioning methods are important for reaching a higher precision solution on ill-conditioned problems. For as we have seen in the numerical experiments, the randomized methods struggle to bring the solution within $10^{-2}\%$ relative residual when the matrix is ill-conditioned.

This is also a framework on which randomized methods for linear systems can be designed. As an example, we have designed a new block variant of RK, a new Gaussian Kaczmarz method and a new Gaussian block method for positive definite systems. Furthermore, the flexibility of our framework and the general convergence Theorems~\ref{ch:two:theo:Enormconv} and~\ref{ch:two:theo:normEconv} allows one to tailor the probability distribution of $S$ to a particular problem class. For instance, other continuous distributions such uniform, or other discrete distributions such Poisson might be more suited to a particular class of problems. 

Numeric tests reveal that the new Gaussian methods designed for overdetermined systems are competitive on sparse problems, as compared with the Kaczmarz and CD-LS methods. The Gauss-pd also proved competitive as compared with  CD-pd on all tests. Though, when applicable, the combined efficiency of using a direct solver and an iterative procedure, such as in  Block CD-pd method, proved the most efficient.

The work opens up many possible future venues of research.  Including investigating accelerated convergence rates through preconditioning strategies based on Theorem~\ref{theo:convsingleS} or by obtaining approximate optimized probability distributions~\eqref{eq:optconv}.  

%In the next chapter, we present a dual perspective of the sketch-and-project method, and  establish the convergence of the method for \emph{any} matrix $A$, as opposed to only matrices with $A$ with full column rank as we addressed here.  

\subsubsection*{Acknowledgments}
I would like to thank Prof. Sandy Davie for useful discussions relating to  Lemma~\ref{lem:2Dgausscov}, and Prof. Joel Tropp for help with formulating and proving  Lemma~\ref{lem:gaussdiag}. 

\section{Appendix: A Bound on the Expected Gaussian Projection Matrix}

We now bound the covariance of a random Gaussian vector projected onto the sphere. This bound is used to study the complexity of Gaussian methods in Section~\ref{C2sec:gauss}.
\begin{lemma} \label{lem:gaussdiag}
Let $D \in \R^{n\times n}$ be a positive definite diagonal matrix, $U \in \R^{n\times n}$ an orthogonal matrix and $\Omega  =UDU^\top $. If $u \sim N(0,D)$ and $\xi \sim N(0,\Omega)$  then 
\begin{equation}\label{eq:gaussdiag}\E{\frac{\xi \xi^\top }{\xi^\top \xi}} =U\E{\frac{u u^\top }{u^\top u}}U^\top  ,\end{equation}
and
\begin{equation}\label{eq:gaussupper}
\E{\frac{\xi \xi^\top }{\xi^\top \xi}} \succeq \frac{2}{\pi}\frac{\Omega}{\Tr{\Omega}}.
\end{equation}
\end{lemma}
\begin{proof}
Let us write $S(\xi)$ for the random vector $\xi/\|\xi\|_2$ (if $\xi=0$, we set $S(\xi)=0$).  Using this notation, we can write \[ \E{ \xi (\xi^\top  \xi)^{-1} \xi^\top   }  = \E{S(\xi)(S(\xi))^\top } = \COV{S(\xi)},\]
where the last identity follows since $\E{S(\xi)}=0$, which in turn holds as the Gaussian distribution is centrally symmetric. As $\xi = U u$,  note that  
\[S(u) = \frac{U^\top  \xi}{\|U^\top  \xi\|_2} = \frac{U^\top  \xi}{\|\xi\|_2} = U^\top  S(\xi).\]
Left multiplying both sides by $U$ we obtain
$U S(u) = S(\xi)$, from which we obtain
\[\COV{S(\xi)} = U \COV{ S(u) } U^\top ,  %= U \COV{ S(u) } U^\top ,
\]
which is equivalent to~\eqref{eq:gaussdiag}.

  To prove\footnote{A version of Lemma~\ref{lem:gaussdiag} was conjectured in the original draft of the paper on which this chapter is based. Prof. Joel Tropp provided this formulation and the remainder of this proof.} \eqref{eq:gaussupper}, note first that  $M \eqdef \E{u u^\top /u^\top u}$ is a diagonal matrix.  One can verify this by direct calculation (informally, this holds because the entries of $u$ are independent and centrally symmetric). The $i$th diagonal entry is given by
\[M_{ii} = \E{\frac{u_{i}^2}{\sum_{j=1}^n u_j^2}}.\]

%
%***{\bf Alternative proof off-diagonals are zero}***\\
%\[
%\int_{\R^n} \frac{u_1u_2}{\sum_{t=1}^n u_t^2} e^{-\frac{1}{2}\left( u^\top Du \right) } du
%\eqdef \int_{\R^n} h(u) du =0,   
%\]
%as  $-h(u_1,u_2,\ldots,u_n) =h(-u_1,u_2,\ldots,u_n)$ and $-h(u_1,u_2,\ldots,u_n)=h(u_1,-u_2,\ldots,u_n).$***\\
As the map $(x,y) \rightarrow x^2/y$ is convex on the positive orthant, we can apply Jensen's inequality, which gives
\[\E{\frac{u_{i}^2}{\sum_{j=1}^n u_j^2}} \geq \frac{\left(\E{|u_{i}|}\right)^2}{\sum_{j=1}^n \E{u_j^2}}
= \frac{2}{\pi}\frac{D_{ii}}{\Tr{D}},\]
which concludes the proof.   
\end{proof}

\section{Appendix: Expected Gaussian Projection Matrix in 2D}
\begin{lemma} \label{lem:2Dgausscov}
Let $\xi \sim N(0,\Omega)$ and $\Omega \in \R^{2\times 2}$ be a positive definite matrix, then
\begin{equation}\label{eq:2Dgausscov}\E{\frac{\xi \xi^\top }{\xi^\top \xi}} = \frac{\Omega^{1/2}}{\Tr{\Omega^{1/2}}}.\end{equation}
\end{lemma}

\begin{proof}
Let $\Sigma = U D U^\top $ and $u \sim N(0,D).$ Given~\eqref{eq:gaussdiag} it suffices to show that 
\begin{equation}\label{eq:h98hs98hs9ss}\COV{S(u)} = \frac{D^{1/2}}{\Tr{D^{1/2}}},\end{equation}
which we will now prove.

Let $\sigma_x^2$ and $\sigma_y^2$ be the two diagonal elements of $D.$ First, suppose that $\sigma_x = \sigma_y.$ Then $u = \sigma_x \eta$ where $\eta \sim N(0,I)$ and
\[\E{\frac{u u^\top }{u^\top u}} = \frac{\sigma_x^2}{\sigma_x^2}\E{\frac{\eta \eta^\top }{\eta^\top \eta}} = \frac{1}{n}I = \frac{D^{1/2}}{\Tr{D^{1/2}}}. \]
Now suppose that $\sigma_x \neq \sigma_y.$ We calculate the diagonal terms of the covariance matrix by integrating
 \[\E{\frac{u_1^2}{u_1^2+u_2^2}}=\frac{1}{2\pi \sigma_x\sigma_y}\int_{\R^2} \frac{x^2}{x^2+y^2} e^{-\frac{1}{2}\left( x^2/\sigma_x^2 +y^2/\sigma_y^2 \right) } dxdy. \]

 Using polar coordinates $x= R \cos(\theta)$ and $y =R \sin(\theta)$ we have
\begin{equation}\int_{\R^2} \frac{x^2}{x^2+y^2} e^{-\frac{1}{2}\left( x^2/\sigma_x^2 +y^2/\sigma_y^2 \right) } dxdy =
\int_0^{2\pi}\int_{0}^{\infty} R\cos^2(\theta) e^{-\frac{R^2}{2} C(\theta) } dRd\theta,\label{eq:ondstep}
\end{equation}
where $C(\theta) \eqdef \left( \cos(\theta)^2/\sigma_x^2 +\sin(\theta)^2/\sigma_y^2 \right). $ Note that
\begin{equation}\label{eq:Rintfirst}
\int_{0}^{\infty} R e^{-\frac{C(\theta)R^2}{2} } dR = \left.-\frac{1}{C(\theta)}e^{-\frac{C(\theta)R^2}{2}} \right|_{0}^{\infty} =\frac{1}{C(\theta)}. 
\end{equation}
This applied in~\eqref{eq:ondstep} gives
\begin{align*} \E{\frac{u_1^2}{u_1^2+u_2^2}} &=\frac{1}{2\pi \sigma_x \sigma_y}  \int_0^{2\pi}\frac{\cos^2(\theta)}{ \cos(\theta)^2/\sigma_x^2 +\sin(\theta)^2/\sigma_y^2  }d\theta = \frac{b}{\pi}  \int_0^{\pi}\frac{\cos^2(\theta)}{ \cos^2(\theta) +b^2\sin^2(\theta)}d\theta, \end{align*}
where $b = \sigma_x/\sigma_y.$
Multiplying the numerator and denominator of the integrand by $\sec^4(x)$ gives the integral
\[\E{\frac{u_1^2}{u_1^2+u_2^2}} = \frac{b}{\pi} \int_0^{\pi}\frac{\sec^2(\theta)}{ \sec(\theta)^2\left( 1 +b^2\tan^2(\theta) \right) }d\theta.\]
Substituting $v=\tan(\theta)$ so that $v^2 +1=\sec^2(\theta)$, $dv =\sec^2(\theta)d \theta$ and using the partial fractions
\[\frac{1}{(v^2+1)\left( 1 +b^2v^2 \right)} = \frac{1}{1-b^2}\left( \frac{1}{v^2+1} -\frac{b^2}{b^2v^2+1} \right), \]
gives the integral
\begin{align}
\int \frac{dv}{ (v^2+1)\left( 1 +b^2 v^2 \right) } &=
\frac{1}{1-b^2}\left(\arctan(v)-b\arctan(b v) \right) \nonumber\\
&= \frac{1}{1-b^2}\left(\theta-b\arctan(b \tan(\theta))\right). \label{eq:A121}
\end{align} 
To apply the limits of integration, we must take care because of the singularity at $\theta=\pi/2$. For this, consider the limits 
\[\lim_{\theta \rightarrow (\pi/2)^-} \arctan(b\tan(\theta)) = \frac{\pi}{2}, \qquad \lim_{\theta \rightarrow (\pi/2)^+} \arctan(b\tan(\theta)) = -\frac{\pi}{2}.\]
Using this to evaluate~\eqref{eq:A121} on the limits of the interval $[0,\, \pi/2]$  gives 
\[\lim_{t \rightarrow (\pi/2)^-}\left.\frac{1}{1-b^2}\left(\theta-b\arctan(b \tan(\theta))\right)\right|_0^{t}=
 \frac{1}{1-b^2}\frac{\pi}{2}(1-b) = \frac{\pi}{2(1+b)}.\]
 Applying a similar argument for calculating the limits from  $\pi/2^+$ to $\pi$, we find
\[ \E{\frac{u_1^2}{u_1^2+u_2^2}} =\frac{2b}{\pi}  \frac{ \pi}{2(1+b)} =\frac{\sigma_x}{\sigma_y+\sigma_x}.\]
Repeating the same steps with $x$ swapped for $y$ we obtain the other diagonal element, which concludes the proof of~\eqref{eq:h98hs98hs9ss}.
%\[ \E{\frac{u_2^2}{u_1^2+u_2^2}}  =\frac{\sigma_y}{\sigma_y+\sigma_x}. \qed\]
%Plugging this into \eqref{eq:98h98hsss}, we get
%\[\COV{S(u)} = \frac{U D^{1/2}U^\top }{\Tr{D^{1/2}}} = \frac{\Omega^{1/2}}{\Tr{\Omega^{1/2}}},\]
%as desired.
\end{proof}  % introduction to linear systems
%\clearpage
%---------------------------------------------------------------------------------
%	CHAPTER three: Stochastic Dual Coordinate Ascent
%---------------------------------------------------------------------------------
%\onehalfspace   %% official UoE spacing
\chapter[Stochastic Dual Ascent for Finding the Projection of a Vector onto a Linear System]{Stochastic Dual Ascent for Finding the Projection of a Vector onto a Linear System}
\chaptermark{SDA for Finding the Projection of a Vector onto a Linear System}
%Solving Least Norm]{Stochastic Dual Ascent for Solving the Least Norm Problem}
%\chaptermark{SDA for Solving Least Norm}
\label{ch:SDA}

{
\epigraph{\emph{Dyfal donc a dyr y garreg.}\\
%Translation:~
Tapping~persistently~breaks~the~stone.}{Welsh proverb}
\let\clearpage\relax
\section{Introduction}
}

In this chapter we consider the more general problem of finding the projection of a given vector onto the solution space of a linear system. This projection problem includes the problem of determining the least norm solution of a linear system (when the given vector is the zero vector). To solve this projection problem, we develop a new randomized iterative algorithm---{\em stochastic dual ascent (SDA)}. The method is dual in nature: with the dual being a non-strongly concave quadratic maximization problem without constraints.

By mapping our dual iterates to primal iterates, we uncover that the SDA method is a dual version of the sketch-and-project method~\eqref{C1eq:xupdate}.
We then proceed to strengthen our convergence results established in Chapter~\ref{ch:linear_systems}. First, we do away with the assumption that the system matrix has full column rank that was required to establish convergence through Theorem~\ref{ch:two:theo:normEconv} and~\ref{ch:two:theo:Enormconv} and consider any matrix and consistent linear system. In this more general setting we show that the primal iterates still converge linearly with a convergence rate that is at least as small as the  convergence rate established in Chapter~\ref{ch:linear_systems}. 
%This shows that assuming that $A$ has full column rank is an unnecessary assumption for proving convergence of the sketch-and-project method. 

Furthermore we give a formula and a tighter lower bound for the convergence rate. We also prove that the same rate of convergence applies to dual function values, primal function values and the duality gap. Unlike traditional iterative methods, SDA converges under virtually no additional assumptions on the system (e.g., rank, diagonal dominance) beyond consistency. In fact, our lower bound  improves as the rank of the system matrix drops. 

When our method specializes to a known algorithm, we either recover the best known rates, or improve upon them. Finally, we show that the framework can be applied to the distributed average consensus problem to obtain  an array of new algorithms. The randomized gossip algorithm arises as a special case~\cite{Boyd2006,OlshevskyTsitsiklis2009}. 

%In each iteration of SDA, a dual variable is updated by a carefully chosen point in a subspace  spanned by the columns of a random matrix drawn independently from a fixed distribution. The distribution plays the role of a parameter of the method.  Our complexity results hold for a wide family of distributions of random matrices, which opens the possibility to fine-tune the stochasticity of the method to particular applications. 

% We prove that  primal iterates associated with the dual process converge to the projection exponentially fast in expectation, and give a formula and an insightful lower bound for the convergence rate. We also prove that the same rate applies to dual function values, primal function values and the duality gap. Unlike traditional iterative methods, SDA converges under no additional assumptions on the system (e.g., rank, diagonal dominance) beyond consistency. In fact, our lower bound  improves as the rank of the system matrix drops.   Many existing randomized methods for linear systems arise as special cases of SDA, including randomized Kaczmarz, randomized Newton, randomized coordinate descent, Gaussian descent, and their variants. In special cases where our method specializes to a known algorithm, we either recover the best known rates, or improve upon them. Finally, we show that the framework can be applied to the distributed average consensus problem to obtain  an array of new algorithms. The randomized gossip algorithm arises as a special case. 

\section{Contributions and Overview}

\subsection{The problem:}

%In this paper we  consider a key problem in linear algebra, that of finding a solution of  a system of linear equations
\begin{equation}\label{eq:Axbx}Ax =b,\end{equation}
where $A \in \R^{m \times n}$  and $b \in \R^m$. We shall only assume that the system is {\em consistent}, that is,  that  there exists $x^*$ for which $Ax^*=b$.  Note that we make no assumptions on $n$ or $m$ and all the configurations $m<n,$ $m=n$ and $m>n$ are allowed. While we assume the existence of a solution, we do not assume uniqueness. In situations with multiple solutions, one is often interested in finding a solution with specific properties. For instance, in compressed sensing and sparse optimization, one is interested in finding the least $\ell_1$-norm, or the least $\ell_0$-norm (sparsest) solution. 

In this chapter we shall focus on the canonical problem of finding the solution  of \eqref{eq:Axbx} closest, with respect to a Euclidean distance,  to a given vector $c\in \R^n$:
\begin{eqnarray}  \text{minimize} \quad \  && P(x)\eqdef \tfrac{1}{2}\|x-c\|_B^2 \notag\\
\text{subject to} \quad \  && Ax=b \label{eq:P}\\
&& x\in \R^n.\notag
\end{eqnarray}
where $B$ is an $n\times n$ symmetric positive definite matrix and 
$\|x\|_B \eqdef \sqrt{x^\top B x}$. By $x^*$ we denote the (necessarily) unique solution of \eqref{eq:P}. Of key importance in this chapter is the {\em dual problem}\footnote{Technically, this  is both the Lagrangian and Fenchel dual of \eqref{eq:P}.} to \eqref{eq:P}, namely
\begin{eqnarray}\label{eq:Dualfunc} \text{maximize}\quad \  && D(y)\eqdef (b-Ac)^\top y - \tfrac{1}{2}\|A^\top y\|_{B^{-1}}^2\\
\text{subject to} \quad \ && y \in \R^m. \notag
\end{eqnarray}

Due to the consistency assumption, strong duality holds and we have $P(x^*) = D(y^*)$, where $y^*$ is any dual optimal solution.

\subsection{A new family of stochastic optimization algorithms}

We propose to solve \eqref{eq:P} via a new method operating in the dual \eqref{eq:Dualfunc}, which we call {\em stochastic  dual ascent} (SDA). The iterates of SDA are of the form
\begin{equation}\label{eq:methoddual0} y^{k+1} = y^k + S \lambda^k,\end{equation}
where $S$ is a random matrix with $m$ rows  drawn in each iteration independently from a pre-specified distribution ${\cal D}$, which should be seen as a parameter of the method. In fact,  by varying ${\cal D}$, SDA should be seen as a family of algorithms indexed by $\cal D$, the choice of which leads to specific algorithms in this family.   By performing steps of the form \eqref{eq:methoddual0}, we are moving in the range space of the random matrix $S$. A key feature of SDA enabling us to prove strong  convergence results despite the fact that the dual objective is in general not strongly concave is the way in which  the ``stepsize'' parameter $\lambda^k$ is chosen: we choose $\lambda^k$ to be the {\em least-norm} vector for which $D(y^k + S\lambda)$ is maximized in $\lambda$. Plugging this $\lambda^k$ into~\eqref{eq:methoddual0}, we obtain the SDA method:
\begin{equation}\label{eq:SDA-compact0}
\boxed{\quad y^{k+1} =  y^k + S \left(S^\top A B^{-1} A^\top S\right)^\dagger S^\top \left(b - A \left( c+B^{-1}A^\top y^k \right) \right) \quad }
\end{equation}

To the best of our knowledge, a randomized optimization algorithm  with iterates of the {\em general} form \eqref{eq:methoddual0}  was not considered nor analyzed before. In the special case when $S$ is chosen to be a random unit coordinate vector, SDA specializes to  the {\em randomized coordinate descent method}, first analyzed by Leventhal and Lewis \cite{Leventhal2010}. 
In the special case when $S$ is chosen as a random column submatrix of the $m\times m$ identity matrix, SDA specializes to the {\em randomized Newton method} of Qu, Fercoq, Richt\'{a}rik and Tak\'{a}\v{c}~\cite{Qu2015}.

With the dual iterates $\{y^k\}$ we associate a sequence of primal iterates $\{x^k\}$ as follows:
\begin{equation} \label{eq:primaliterates0} x^k \eqdef c + B^{-1}A^\top y^k.\end{equation}
In combination with \eqref{eq:SDA-compact0}, this yields the primal iterative process
 \begin{equation}\label{eq:SDA-primal0}
\boxed{\quad x^{k+1} =  x^k - B^{-1}A^\top S \left(S^\top A B^{-1} A^\top S\right)^\dagger S^\top \left(A x^k  -b \right) \quad }
\end{equation}

Optimality conditions (see Section~\ref{subsec:OptCond}) imply that if $y^*$ is any dual optimal point, then $c+B^{-1}A^\top y^*$ is necessarily primal optimal and hence equal to $x^*$, the optimal solution of \eqref{eq:P}. Moreover,  we have the following useful and insightful correspondence between the quality of the primal and dual iterates (see Proposition~\ref{lem:correspondence}):
\begin{equation}\label{eq:iugs8gs}D(y^*) - D(y^k) = \tfrac{1}{2}\|x^k-x^*\|_B^2.\end{equation}
Hence, {\em dual convergence in  function values is equivalent to primal convergence in iterates.} 

This work belongs to a growing literature on randomized methods for various problems appearing in linear algebra, optimization and computer science. In particular, relevant methods include sketching algorithms, randomized Kaczmarz, stochastic gradient descent and their variants  \cite{Strohmer2009,Needell2010,Drineas2011,hogwild,Zouzias2013a, Needell2012, Ramdas2014, SAG, Takac2013,Johnson2013, Richtarik2015c, S2GD, proxSVRG, SAGA, mS2GD, ZhaoZhang2015,Needell2014, Dai2014, NeedellWard2015, Ma2015, Gower2015,Oswald2015,LiuWright-AccKacz-2016}
and randomized coordinate and subspace type methods and their variants \cite{Leventhal2010,Lin:2008:DCDM,ShalevTewari09,Nesterov2012,Wright:ABCRRO,
shotgun,Richtarik2014a,Nesterov2011,PCDM,tao2012stochastic,Necoara:Parallel,ICD,Necoara:rcdm-coupled,Richtarik2013a,Hydra2,SDCA,Fercoq-paralleladaboost,Fercoq2013a, Lee2013,Qu2015b,SPCDM,ALPHA,ESO,SPDC,Qu2015,NIPSdistributedSDCA,
Shalev-Shwartz2013b,APCG,Csiba2015,Gower2015}.

\subsection{The main results} 

We now describe two complexity theorems which form the core theoretical contribution of this chapter. The results hold for a wide family of distributions ${\cal D}$, which we describe next.

\paragraph{Weak assumption on ${\cal D}$.} In our analysis, we only impose a very weak assumption on $\cal D$. In particular, we only assume that the $m\times m$ matrix \begin{equation} \label{eq:H} H \eqdef \mathbf{E}_{S\sim {\cal D}} \left[ S\left(S^\top AB^{-1}A^\top S\right)^{\dagger}S^\top\right]\end{equation}
 is well defined and nonsingular.\footnote{Note that from Lemma~\ref{lem:pseudoposdef}, the pseudo pseudoinverse of a symmetric positive semidefinite matrix is again symmetric and positive semidefinite. As a result, if the expectation defining $H$ is finite, $H$ is also symmetric and positive semidefinite. Hence, we could equivalently assume that $H$ be positive definite.}
    Hence, we do not assume that $S$ is chosen from any particular random matrix ensemble. This makes it possible for practitioners to choose the best distribution specific to a particular application. 

We cast the first complexity result in terms of the primal iterates since solving \eqref{eq:P} is our main focus in this work. Let $\myRange{M}, \Rank{M}$ and $\lambda_{\min}^+(M)$ denote the range space, rank and the smallest nonzero eigenvalue of $M$, respectively.

\begin{theorem}[\bf Convergence of primal iterates and of the residual] \label{theo:Enormerror}
Assume that the matrix $H$, defined in \eqref{eq:H}, is nonsingular. Fix arbitrary $x^0\in \R^n$. The primal iterates $\{x^k\}$ produced by \eqref{eq:SDA-primal0} converge linearly in expectation to $x^* + t$, where  $x^*$ is the optimal solution of the primal problem \eqref{eq:P}, and $t$ is the projection of $x^0-c$ onto $\Null{A}$:
\begin{equation}\label{eq:def_of_t}t \eqdef \arg \min_{t'} \left\{ \|x^0-c - t'\|_B \;:\; t'\in \Null{A}\right\}.\end{equation}
In particular,  for all $k\geq 0$ we have
\begin{eqnarray}\label{eq:Enormerror} \quad \text{Primal iterates:} \quad &&\E{\norm{x^{k} - x^{*}- t}_{B}^2}\leq \rho^k \cdot \norm{x^{0} -  x^{*} - t}_{B}^2,\\
\label{eq:s98h09hsxxx} \text{Residual:} \quad && \E{\|A x^k-b\|_B} \leq \rho^{k/2} \|A\|_B \|x^0-x^*-t\|_B + \|At\|_B,\end{eqnarray}
where  $\|A\|_B \eqdef \max \{\|Ax\|_B \;:\; \|x\|_B\leq 1\}$ and 
\begin{equation}\label{ch:three:eq:rho} \rho \eqdef 1- \lambda_{\min}^+\left(B^{-1/2}A^\top H A B^{-1/2}\right).
\end{equation}
 Furthermore, the convergence rate is bounded by
\begin{equation} \label{eq:nubound}
1-\frac{\E{\Rank{S^\top A}}}{\Rank{A}}\leq \rho < 1.
\end{equation}

\end{theorem}

As show shown in Section~\ref{ch:two:sec:RK}, if we let $S$ be a unit coordinate vector chosen at random, $B$  be the identity matrix and set $c=0$, then \eqref{eq:SDA-primal0} reduces to the {\em randomized Kaczmarz (RK)} method proposed and analyzed in a seminal work of Strohmer and Vershynin \cite{Strohmer2009}. Theorem~\ref{theo:Enormerror}
implies that RK converges with an exponential rate so long as the system matrix has no zero rows (see Section~\ref{sec:discrete}).  To the best of our knowledge, such a result was not previously established: current convergence results for RK assume that the system matrix is full rank~\cite{Ma2015, Ramdas2014}. Not only do we show that the RK method converges to the least-norm solution for any consistent system, but we do so through a single all encompassing theorem covering a wide family of algorithms. Likewise, convergence of block variants of RK has only been established for full column rank~\cite{Needell2012,Needell2014}. Block versions of RK can be obtained from our generic method by choosing $B=I$ and $c=0$, as before, but letting $S$ to be a random column submatrix of the identity matrix. Again, our general complexity bound holds under no assumptions on $A$, as long as one can find $S$ such that $H$ becomes nonsingular.

The lower bound \eqref{eq:nubound} says that for a singular system matrix, the number of steps required by SDA to reach an expected accuracy is at best inversely proportional to the rank of $A$. If $A$ has row rank equal to one, for instance, then  RK converges in one step (this is no surprise, given that RK projects onto the solution space of a single row, which in this case, is the solution space of the whole system). Our lower bound in this case becomes $0$, and hence is 
tight. 

While Theorem~\ref{theo:Enormerror} is cast in terms of the primal  iterates, if we  assume that $x^0 = c+ B^{-1}A^\top y^0$ for some $y^0\in \R^m$, then an equivalent dual characterization follows by combining  \eqref{eq:primaliterates0} and \eqref{eq:iugs8gs}. In fact, in that case we can  also establish the convergence of the primal function values and of the duality gap. {\em No such results were previously known. }

\begin{theorem}[\bf Convergence of function values]\label{theo:2}
Assume that the matrix $H$, defined in \eqref{eq:H}, is nonsingular. Fix arbitrary $y^0\in \R^m$ and let $\{y^k\}$ be the SDA iterates produced by \eqref{eq:SDA-compact0}. Further, let $\{x^k\}$ be the associated primal iterates, defined by \eqref{eq:primaliterates0}, $OPT  \eqdef P(x^*)=D(y^*)$,
\[U_0 \eqdef \tfrac{1}{2}\|x^0-x^*\|_B^2 \overset{\eqref{eq:iugs8gs}}{=} OPT -D(y^0),\] 
and let $\rho$ be as in Theorem~\ref{theo:Enormerror}. Then for all $k\geq 0$ we have the following complexity bounds:
\begin{eqnarray} 
\quad \text{Dual suboptimality:} \quad &&\E{OPT - D(y^k)}\leq \rho^k U_0\label{eq:DUALSUBOPT} \\
\quad \text{Primal suboptimality:} \quad &&\E{P(x^k) - OPT}\leq \rho^k U_0 +  2 \rho^{k/2}  \sqrt{OPT \times U_0} \label{eq:PRIMALSUBOPT} \\
 \text{Duality gap:} \quad && \E{P(x^k) - D(y^k)}\leq 2\rho^k U_0 +  2 \rho^{k/2} \sqrt{OPT\times U_0}  \label{eq:GAPSUBOPT}
\end{eqnarray}

\end{theorem}

%Note that the dual objective function is {\em not} strongly concave  in general, and yet we prove linear convergence (see \eqref{eq:DUALSUBOPT}). It is known that for {\em some} structured optimization problems,  linear convergence results  can be obtained without the need to assume strong concavity (or strong convexity, for minimization problems). Typical approaches to such results  would be via the employment of error bounds \cite{LuoTseng93-AOR, Tseng95-JCAM, HongLuo2013,Ma-Tapp-Takac-FeasibleDescent-2015,NecoaraClipiciSIOPT2016}.
{\em In our analysis, no error bounds are necessary.}

\subsection{Chapter outline}

This chapter is structured as follows.  Section~\ref{sec:SDA}  describes the algorithm in detail, both in its dual and primal form, and establishes several useful identities. In Section~\ref{sec:discrete} we characterize discrete distributions for which our main assumption on $H$ is satisfied. We then specialize our method to several simple discrete distributions to better illustrate the results. We then show in Section~\ref{sec:gossip} how SDA can be applied to design new randomized gossip algorithms. We also show that our framework can recover some standard methods. Theorem~\ref{theo:Enormerror} is proved in Section~\ref{sec:proof} and Theorem~\ref{theo:2} is proved in Section~\ref{sec:proof2}. In Section~\ref{sec:experiments} we perform a simple experiment illustrating the convergence of the randomized Kaczmarz method on rank deficient linear systems. We then summarize in Section~\ref{sec:conclusion}. 

%To the appendix  we relegate two elementary but useful technical results which are needed multiple times in the text.

\section{Stochastic Dual Ascent} \label{sec:SDA}

By {\em stochastic  dual ascent} (SDA) we refer to a randomized optimization method for solving the dual problem \eqref{eq:Dualfunc} performing iterations of the form
\begin{equation}\label{eq:methoddual} y^{k+1} = y^k + S \lambda^k,\end{equation}
where $S$ is a random matrix with $m$ rows drawn in each iteration independently from a prespecified distribution. We shall not fix the number of columns of $S$; in fact, we even allow for the number of columns to be random. By performing steps of the form \eqref{eq:methoddual}, we are moving in the range space of the random matrix $S$, with $\lambda^k$ describing the precise linear combination of the columns used in computing the step. 
In particular, we shall choose $\lambda^k$ from the set 
\[Q^k \eqdef \arg \max_{\lambda} D(y^k + S\lambda) \overset{\eqref{eq:Dualfunc}}{=} \arg\max_{\lambda} \left\{ (b-Ac)^\top (y^k + S\lambda) - \tfrac{1}{2}\left\|A^\top (y^k + S \lambda)\right\|_{B^{-1}}^2\right\}.\]
Since $D$ is bounded above (a consequence of weak duality), this set is nonempty. Since $D$ is a concave quadratic, $Q^k$ consists of all those vectors $\lambda $ for which the gradient of the mapping $\phi_k(\lambda): \lambda \mapsto D(y^k + S\lambda)$ vanishes. 
%Since \[\nabla \phi_k(\lambda) = S^\top (b- Ac - AB^{-1}A^\top y^k) - S^\top A B^{-1}A^\top S\lambda,\]
This leads to the observation that $Q^k$ is the set of solutions of a random linear system:
\[Q^k =  \left\{\lambda \in \R^m \;:\; \left(S^\top A B^{-1}A^\top S \right) \lambda = S^\top \left(b - Ac - A B^{-1}A^\top y^k \right) \right\}.\]
If $S$ has a small number of columns, this is a small easy-to-solve system. 

A key feature of our method enabling us to prove exponential error decay despite the lack of strong concavity is the  way in which we choose $\lambda^k$ from $Q^k$. In  SDA, $\lambda^k$  is chosen to be the least-norm  element of $Q^k$,
\[\lambda^k \eqdef \arg\min_{\lambda \in Q^k} \|\lambda\|_2.\]
 Using Lemma~\ref{lem:pseudoleastnorm}, the least-norm solution to the above is given by
 % the pseudoinverse of the system matrix applied to the right hand side, that is
% of a linear system can be written down in a compact way using the (Moore-Penrose) pseudoinverse. In our case, we obtain the formula
\begin{equation}\label{eq:lambda_closed_form}
\lambda^k = \left(S^\top A B^{-1} A^\top S\right)^\dagger S^\top \left(b - Ac - A B^{-1}A^\top y^k  \right).
\end{equation}
Note that if $S$ has only a few columns, then~\eqref{eq:lambda_closed_form} requires projecting the origin onto a small linear system. % In particular, if $S$ is a vector, 
The SDA algorithm is obtained by combining \eqref{eq:methoddual} with \eqref{eq:lambda_closed_form}.

\begin{algorithm}[!h]
\begin{algorithmic}[1]
\State \textbf{parameter:} ${\cal D}$ = distribution over random matrices
\State Choose $y^0 \in \R^m$
\Comment Initialization
\For {$k = 0, 1, 2, \dots$}
	\State Sample an independent copy $S\sim {\cal D}$
	\State $\lambda^k = \left(S^\top A B^{-1} A^\top S\right)^\dagger S^\top \left(b - A c - AB^{-1}A^\top y^k  \right)$
	\State $y^{k+1} =  y^k + S  \lambda^k$
		\Comment Update the dual variable
\EndFor
\end{algorithmic}

\caption{Stochastic Dual Ascent (SDA)}
\label{alg:SDA}
% {}
\end{algorithm}

The method has one parameter: the distribution $\cal D$ from which the random matrices $S$ are drawn. Sometimes, one is interested in finding any solution of the system $Ax=b$, rather than the particular solution described by the primal problem 
\eqref{eq:P}. In such situations, $B$ and $c$ could also be seen as parameters.

\subsection{Optimality conditions} \label{subsec:OptCond}

For any $x$ for which $Ax=b$ and for any $y$ we have
\[P(x) - D(y) \overset{\eqref{eq:P}+\eqref{eq:Dualfunc}}{=} \tfrac{1}{2}\|x-c\|_B^2 + \tfrac{1}{2}\|A^\top y\|_{B^{-1}}^2 + (c-x)^\top A^\top y \geq 0,\]
where the inequality (weak duality) follows from the Fenchel-Young inequality\footnote{Let $U$ be a vector space equipped with an inner product $\langle \cdot, \cdot \rangle : U\times U \to \R$. Given a function $f:U \to \R$, its convex (or Fenchel) conjugate $f^*:U\to \R\cup \{+\infty\}$ is defined by $f^*(v) = \sup_{u \in U} \langle u, v \rangle - f(u)$. A direct consequence of this is the  Fenchel-Young inequality, which asserts that  $f(u) + f^*(v)\geq \langle u, v\rangle$ for all $u$ and $v$.  The inequality in the main text follows by choosing $f(u)=\tfrac{1}{2}\|u\|_B^2$ (and hence  $f^*(v)=\tfrac{1}{2}\|v\|^2_{B^{-1}}$), $u=x-c$ and $v=A^\top y$. If $f$ is differentiable, then equality holds if and only if $v=\nabla f(u)$. In our case, this condition is $x=c+B^{-1}A^\top y$. This, together with primal feasibility, gives  the optimality conditions \eqref{eq:opt_cond}. For more details on Fenchel duality, see \cite{bookBorweinLewis2006}.}. As a result, we obtain the following necessary and sufficient optimality conditions, characterizing primal and dual optimal points.

\begin{proposition} [Optimality conditions] \label{eq:prop_opt_cond}Vectors $x\in \R^n$ and $y\in\R^m$ are optimal for the primal \eqref{eq:P} and dual \eqref{eq:Dualfunc} problems respectively, if and only if they satisfy the following relation
\begin{equation} \label{eq:opt_cond}Ax = b, \qquad x = c + B^{-1} A^\top y.\end{equation}
\end{proposition}

In  view of this, it will be useful to define a linear mapping from $\R^m$ to $\R^n$ as follows:
\begin{equation}\label{eq:98s98hs}x(y) = c + B^{-1}A^\top y.\end{equation}

As an immediate corollary of Proposition~\ref{eq:prop_opt_cond} we observe that for any dual optimal $y^*$, the vector $x(y^*)$ must be primal optimal. Since the primal problem has a unique optimal solution, $x^*$,   we must necessarily have \begin{equation}\label{eq:opt_primal}x^* =x(y^*) = c + B^{-1} A^\top y^*.\end{equation}

Another immediate corollary of Proposition~\ref{eq:prop_opt_cond} is the following characterization of dual optimality:  $y$ is dual optimal if and only if
\begin{equation} \label{eq:98hs8h9sss}b  - Ac = AB^{-1}A^\top y.\end{equation}
Hence, the set of dual optimal solutions is
${\cal Y}^* = (AB^{-1}A^\top)^\dagger (b-Ac) + \Null{AB^{-1}A^\top}$. Since, $\Null{AB^{-1}A^\top} = \Null{A^\top}$ (see Lemma~\ref{lem:WGW}), we have
\[{\cal Y}^* = \left(AB^{-1}A^\top\right)^\dagger (b-Ac) + \Null{A^\top}.\]

Combining this with \eqref{eq:opt_primal}, we get

\[x^* = c + B^{-1}A^\top \left(AB^{-1}A^\top \right)^\dagger(b-Ac).\]

\begin{remark} [The dual is also a least-norm problem] Observe that:
\begin{enumerate}
\item 
The particular dual optimal point $y^* = (AB^{-1}A^\top)^\dagger (b-Ac)$ is the solution of the following optimization problem:
\begin{equation} \label{eq:iusiuh7ss}\min \left\{ \tfrac{1}{2}\|y\|^2_2 \;:\; A B^{-1} A^\top y = b-Ac\right\}.\end{equation}

Hence, this particular formulation of the dual problem has the same form as the primal problem: projection onto a linear system.

\item If $A^\top A$ is positive definite (which can only happen if $A$ is of full column rank, which means that $Ax=b$ has a unique solution and hence the primal objective function does not matter), and we choose $B=A^\top A$, then the dual constraint \eqref{eq:iusiuh7ss} becomes
\[A (A^\top A)^{-1}A^\top y = b - Ac.\]

This constraint has a geometric interpretation: we are seeking a vector $y$ whose orthogonal projection onto the column space of $A$ is equal to $b-Ac$. Hence the reformulated dual problem \eqref{eq:iusiuh7ss} is asking us to find the vector $y$ with this property having the least norm.

\end{enumerate}
\end{remark}

\subsection{Primal iterates associated with the dual iterates}

With the sequence  of dual  iterates $\{y^k\}$ produced by SDA  we can associate a sequence of primal iterates $\{x^k\}$  using the mapping \eqref{eq:98s98hs}:
\begin{equation} \label{eq:primaliterates} x^k \eqdef x(y^k) = c + B^{-1}A^\top y^k.\end{equation}
This leads to the following {\em primal version of the SDA method}.
\begin{algorithm}[!h]
\begin{algorithmic}[1]
\State \textbf{parameter:} ${\cal D}$ = distribution over random matrices
\State Choose $x^0 \in \R^n$
\Comment Initialization
\For {$k = 0, 1, 2, \dots$}
	\State Sample an independent copy $S\sim {\cal D}$
	\State $x^{k+1} =  x^k - B^{-1}A^\top S \left(S^\top A B^{-1} A^\top S\right)^\dagger S^\top (A x^k  -b )$
		\Comment Update the primal variable
\EndFor
\end{algorithmic}

\caption{Primal Version of Stochastic Dual Ascent (SDA-Primal)}
\label{alg:SDA-Primal}
\end{algorithm}

\begin{remark} \label{lem:5shsuss} 

A couple of observations:

\begin{enumerate}

\item {\em Self-duality.} If $A$ is positive definite, $c=0$, and if we choose $B=A$, then in view of \eqref{eq:primaliterates}  we have $x^k = y^k$ for all $k$, and hence Algorithms~\ref{alg:SDA} and \ref{alg:SDA-Primal} coincide. In this case, Algorithm~\ref{alg:SDA-Primal} can be described as {\em self-dual.}

\item {\em Space of iterates.} A direct consequence of the correspondence between the dual and primal iterates \eqref{eq:primaliterates}  is the following simple observation (a generalized version of this, which we prove later as Lemma~\ref{lem:error}, will be used  in the proof of Theorem~\ref{theo:Enormerror}): Choose $y^0\in \R^m$ and let $x^0 = c + B^{-1}A^\top y^0$. Then the iterates $\{x^k\}$ of Algorithm~\ref{alg:SDA-Primal} are of the form $x^k = c + B^{-1} A^\top y^k$ for some $y^k\in \R^m$.

\item {\em Starting point.} While we have defined the primal iterates of Algorithm~\ref{alg:SDA-Primal}  via a linear transformation of the dual iterates---see \eqref{eq:primaliterates}---we {\em can}, in principle,  choose $x^0$ arbitrarily, thus breaking the primal-dual connection which helped us to define the method. In particular, we can choose $x^0$ in such a way that there does not exist $y^0$ for which $x^0 = c + B^{-1}A^\top y^0$. As is clear from Theorem~\ref{theo:Enormerror}, in this case the iterates $\{x^k\}$ will not converge to $x^*$, but to $x^*+t$, where $t$ is the projection of $x^0-c$ onto the nullspace of $A$.
\end{enumerate}

\end{remark}

It is now clear that the iterates of Algorithm~\ref{alg:SDA-Primal} are the same as the iterates defined by the Random Update viewpoint~\eqref{eq:MP} of the sketch-and-project method~\eqref{ch:two:NF} in Chapter~\ref{ch:linear_systems}. Thus we have uncovered  a hidden dual nature of the sketch-and-project method. It is this dual relationship that allows us to formulate and prove convergence of the dual function values and duality gap proven in Theorem~\ref{theo:2}. In particular the duality gap gives a means to measure the distance from the optimality. This certificate of convergence is standard in optimization methods, but seldom appears in the numerical linear algebra literature.

\subsection{Relating the quality of the dual and primal iterates}

The following simple but insightful result (mentioned in the introduction) relates the ``quality'' of a dual vector $y$ with that of its primal counterpart, $x(y)$. It says that the dual suboptimality of $y$ in terms of function values is equal to the primal suboptimality of $x(y)$ in terms of distance.
 
\begin{proposition}\label{lem:correspondence} Let $y^*$ be any dual optimal point and $y\in \R^m$. Then
\[D(y^*) - D(y) = \tfrac{1}{2}\|x(y^*) - x(y)\|_B^2.\]
\end{proposition}
\begin{proof}
Straightforward calculation shows that
\begin{eqnarray*}
D(y^*)  -D(y) &\overset{\eqref{eq:Dualfunc}}{=}& (b-Ac)^\top (y^* - y) - \tfrac{1}{2}(y^*)^\top A B^{-1} A^\top y^*  + \tfrac{1}{2}y^\top A B^{-1} A^\top y\\
&\overset{\eqref{eq:98hs8h9sss}}{=}&(y^*)^\top A B^{-1} A^\top (y^* - y) - \tfrac{1}{2}(y^*)^\top A B^{-1} A^\top y^*  + \tfrac{1}{2}y^\top A B^{-1} A^\top y\\
&=& \tfrac{1}{2}(y-y^*)^\top AB^{-1} A^\top (y-y^*)\\
&\overset{\eqref{eq:98s98hs}}{=}& \tfrac{1}{2}\|x(y) - x(y^*)\|_B^2. 
\end{eqnarray*}
%\qed
\end{proof}

Applying this result to the sequence $\{(x^k,y^k)\}$ of dual iterates produced by SDA and their corresponding primal images, as defined in \eqref{eq:primaliterates}, we get the identity:
\[D(y^*)- D(y^k) = \tfrac{1}{2}\|x^k - x^*\|_B^2.\]
Therefore, {\em  dual convergence in function values $D(y^k)$  is equivalent to primal convergence in iterates $x^k$}. Furthermore, a direct computation leads to the following formula for the {\em duality gap}:
\begin{equation}\label{eq:dualitygap09709709}P(x^k ) - D(y^k) \overset{\eqref{eq:primaliterates}}{=} (AB^{-1}A^\top y^k + Ac - b)^\top y^k = -(\nabla D(y^k) )^\top y^k.\end{equation}

Note that computing the gap is significantly more expensive than the cost of a single iteration (in the interesting regime  when the number of columns of $S$ is small). Hence, evaluation of the duality gap should generally be avoided. If it is necessary to be certain about the quality of a solution however,  the above formula will be useful. The gap should then be computed from time to time only, so that this extra work does not significantly slow down the iterative process.

\section{Discrete Distributions} \label{sec:discrete}

Both the SDA algorithm and its primal counterpart are generic in the sense that the distribution  $\cal D$ is not specified beyond assuming that the matrix $H$ defined in \eqref{eq:H} is well defined and nonsingular. In this section we shall first characterize finite discrete distributions for which $H$ is nonsingular. We  then give a few examples of algorithms based on such distributions, and comment on our complexity results in more detail.

\subsection{Nonsingularity of $H$ for finite discrete distributions} \label{sec:Hnonsingular}

For simplicity, we shall focus on {\em finite discrete} distributions $\cal D$. That is, we set $S = S_i$ with probability $p_i>0$, where  $S_1,\dots,S_r$ are fixed matrices (each with $m$ rows).  The next theorem gives a necessary and sufficient condition for the matrix $H$ defined in \eqref{eq:H} to be nonsingular.
 
\begin{theorem}\label{thm:H} Let $\cal{D}$ be a finite discrete distribution, as described above. Then  $H$ is nonsingular if and only if
\[\myRange{[S_1S_1^\top A, \cdots, S_r S_r^\top A]} = \R^m .\]
\end{theorem}

\begin{proof} Let $K_i = S_i^\top AB^{-1/2}$. In view of the identity $\left(K_i K_i^\top \right)^{\dagger} = (K_i^\dagger )^\top K_i^\dagger$,  we can write 
\[H \overset{\eqref{eq:H}}{=} \sum_{i=1}^r H_i,\]
where $H_i = p_i S_i (K_i^\dagger)^\top K_i^\dagger S_i^\top$. Since $H_i$ are symmetric positive semidefinite,  so is $H$.  Now,  it is easy to check that $y^\top H_i y = 0$ if and only if $y  \in \Null{H_i}$ (this holds for any symmetric positive semidefinite $H_i$). Hence, $y^\top H y = 0$ if and only if $y \in \cap_i \Null{H_i}$ and thus
$H$ is positive definite if and only if 
\begin{equation} \label{eq:0h09sh0976}\bigcap_{i} \Null{H_i} = \{0\}.\end{equation}
In view of Lemma~\ref{lem:WGW}, $\Null{H_i} = \Null{\sqrt{p_i}K_i^\dagger S_i^\top} =  \Null{K_i^\dagger S_i^\top}$. Now, $y\in \Null{K_i^\dagger S_i^\top}$ if and only of $S_i^\top y \in \Null{K_i^\dagger} = \Null{K_i^\top} = \Null{A^\top S_i}$. Hence, $\Null{H_i} = \Null{A^\top S_i S_i^\top}$, which means that \eqref{eq:0h09sh0976} is equivalent to \\ $\Null{[S_1S_1^\top A, \cdots, S_r S_r^\top A]^\top} = \{0\}$.% \qed
\end{proof}

\bigskip
We have the following corollary.\footnote{We can also prove the corollary directly as follows: The first assumption implies that $S_i^\top A B^{-1} A^\top S_i$ is invertible for all $i$ and that $V \eqdef \mbox{Diag}\left(p_i^{1/2}(S_i^\top A{B^{-1}}A^\top S_i)^{-1/2}\right)$  is nonsingular. It remains to note that
\[
H \overset{\eqref{eq:H}}{=}  \E{ S\left(S^\top AB^{-1}A^\top S\right)^{-1} S^\top} \\
= \sum_i p_i  S_i\left(S_i^\top AB^{-1}A^\top S_i \right)^{-1} S_i^\top  = \mathbf{S}V^2 \mathbf{S}^\top.
\]
}

\begin{corollary}\label{cor:09hs09hs}
Assume that $S_i^\top A$ has full row rank for all $i$ and that  $\mathbf{S} \eqdef [S_1,\ldots, S_r]$ is of full row rank. Then $H$ is nonsingular. 
\end{corollary}

We now give a few illustrative examples:

\begin{enumerate}
\item \emph{Coordinate vectors.} Let $S_i = e_i$ ($i^{\text{th}}$ unit coordinate vector) for $i=1,2,\dots,r=m$. In this case,  $\mathbf{S} = [S_1,\dots,S_m]$ is the identity matrix in $\R^m$, and $S_i^\top A$ has full row rank for all $i$ as long as the rows of $A$ are all nonzero. By Corollary~\ref{cor:09hs09hs}, $H$ is positive definite.
\item \emph{Submatrices of the identity matrix.} We can let $S$ be a random column submatrix of the $m\times m$ identity matrix $I$. There are  $2^m-1$ such potential submatrices, and  we choose $1\leq r \leq 2^m-1$.  As long as we choose $S_1,\dots,S_r$ in such a way that each column of $I$ is represented in some matrix $S_i$, the matrix $\mathbf{S}$ will have full row rank.  Furthermore, if  $S_i^\top A$ has full row rank for all $i$, then by the above corollary, $H$ is nonsingular. Note that if the row rank of $A$ is $r$, then the matrices $S_i$ selected by the  above process will  necessarily have at most $r$ columns.
\item \emph{Count sketch and Count-min sketch.} Many other ``sketching'' matrices $S$ can be employed within SDA, including the count sketch \cite{CountSketch2002} and the count-min sketch \cite{CountMinSketch2005}. In our context (recall that we sketch with the transpose of $S$), $S$ is a count-sketch matrix (resp. count-min sketch) if it is assembled from random columns of $[I,-I]$ (resp $I$), chosen uniformly with replacement, where $I$ is the $m\times m$ identity matrix.
\end{enumerate}

\subsection{Randomized Kaczmarz is the primal process associated with  randomized coordinate ascent } \label{subsec:RKvsRCA}
Let $B=I$ (the identity matrix). The primal problem then becomes
\begin{eqnarray}  \text{minimize} \quad \ && P(x)\eqdef \tfrac{1}{2}\|x-c\|_2^2 \notag\\
\text{subject to} \quad \ && Ax=b \notag\\
&& x\in \R^n,\notag
\end{eqnarray}
and the dual problem is
\begin{eqnarray}\notag \text{maximize}\quad \  && D(y)\eqdef (b-Ac)^\top y - \tfrac{1}{2}y^\top A A^\top y\\
\text{subject to} \quad \ && y \in \R^m. \notag
\end{eqnarray}

\paragraph{Dual iterates.} Let us choose $S=e^i$ (unit coordinate vector in $\R^m$) with probability $p_i>0$ (to be specified later). The SDA method (Algorithm~\ref{alg:SDA}) then takes the form
\begin{equation}\label{eq:SDA-compact08986986098}
\boxed{\quad y^{k+1} =  y^k + \frac{b_i - A_{i}c - A_{i:} A^\top y^k }{\|A_{i:}\|_2^2}e_i   \quad }
\end{equation}
This is the randomized coordinate ascent method applied to the dual problem. In the form popularized by Nesterov~\cite{Nesterov2012}, it takes the form
\[y^{k+1} = y^k + \frac{e_i^\top \nabla D(y^k)}{L_i} e_i,\]
where $e_i^\top \nabla D(y^k)$ is the $i$th partial derivative of $D$ at $y^k$ and $L_i>0$ is the Lipschitz constant of the $i$th partial derivative, i.e., constant for which the following inequality holds for all $\lambda\in \R$:
\begin{equation}\label{eq:s97g98gs} | e_i^\top \nabla D(y + \lambda e_i) - e_i^\top \nabla D(y)  | \leq L_i |\lambda|.\end{equation}
It can be easily verified that \eqref{eq:s97g98gs}  holds with $L_i=\|A_{i:}\|_2^2$ and that $e_i^\top \nabla D(y^k)=b_i - A_{i:}c - A_{i:} A^\top y^k $.

\paragraph{Primal iterates.} The associated primal iterative process (Algorithm~\ref{alg:SDA-Primal})  takes the form
\begin{equation}\label{eq:SDA-primal009s09us0098}
\boxed{\quad x^{k+1} =  x^k - \frac{A_{i:} x^k   -b_i }{\|A_{i:}\|_2^2} A_{i:}^\top  \quad }
\end{equation}
This is the randomized Kaczmarz method of Strohmer and Vershynin \cite{Strohmer2009}.

\paragraph{The rate.}

\bigskip
Let us now compute the rate $\rho$ as defined in~\eqref{ch:three:eq:rho}.
It will be convenient, but {\em not} optimal, to choose the probabilities according to Theorem~\ref{theo:convsingleS},  that is
 \begin{equation}\label{eq:089h08hs98xx}p_i = \frac{\norm{ A_{i:} }_2^2}{\norm{A}_F^2},\end{equation}  where $\|\cdot\|_F$ denotes the Frobenius norm (we assume that $A$ does not contain any zero rows).  Since
\[H \overset{\eqref{eq:H}}{=} \E{S \left(S^\top A A^\top S \right)^{\dagger}S^\top } =  \sum_{i=1}^m p_i \frac{e_i e_i^\top }{\|A_{i:}\|_2^2} \overset{\eqref{eq:089h08hs98xx}}{=} \frac{1}{\norm{A}_F^2}I,\]
we have
\begin{equation}\label{eq:98hs8h8ss}\rho = 
 1-\lambda_{\min}^+\left(A^\top H A \right) = 1-\frac{\lambda_{\min}^+\left(A^\top A\right)}{\norm{A}_F^2}.
\end{equation}
%In general, the rate $\rho$ is a function of the probabilities $p_i$. The inverse problem: ``How to set the probabilities so that the rate is optimized?'' is difficult. As shown in Section~\ref{ch:two:sec:optprob}, even when $A$ is of full column rank, computing the optimal probabilities leads to a semidefinite program.
Furthermore, if $r= \Rank{A}$, then in view of \eqref{eq:nubound}, the rate  is bounded  as
\[ 1- \frac{1}{r}\leq  \rho <1. \]
Assume that $A$ is of rank $r=1$ and let  $A= uv^\top$. Then $A^\top A = (u^\top  u) v  v^\top$, and hence this matrix is also of rank 1. Therefore, $A^\top A$ has a single nonzero eigenvalue, which is equal to its trace. Therefore, $\lambda_{\min}^+(A^\top A) = \Tr{A^\top A} = \|A\|^2_F$ and hence $\rho =0$. Note that the  rate $\rho$ reaches its lower bound and the method converges in one step.

\paragraph{Remarks.} For randomized coordinate ascent applied to (non-strongly) concave quadratics, rate \eqref{eq:98hs8h8ss} has been established by Leventhal and Lewis \cite{Leventhal2010}. However, to the best of our knowledge, this is the first time this rate has also been established for the randomized Kaczmarz method. We do not only prove this, but show that this is because the iterates of the two methods are linked via a linear relationship. In the $c=0, B=I$ case, and for row-normalized matrix $A$, this linear relationship between the two methods was  recently independently observed by Wright~\cite{Wright:CoorDescMethods-survey}.  While all linear complexity results for  RK we are aware of require full rank assumptions, there exist nonstandard variants of RK which do not require such assumptions, one example being the asynchronous parallel version of RK studied by Liu, Wright and Sridhar \cite{Wright:AsyncPRK}.  Finally, no results of the type \eqref{eq:PRIMALSUBOPT} (primal suboptimality)  and \eqref{eq:GAPSUBOPT} (duality gap) previously existed for these methods in the literature. 

\subsection{Randomized block Kaczmarz is the primal process associated with randomized Newton}

Let $B=I$, so that we have the same pair of primal dual problems as in Section~\ref{subsec:RKvsRCA}.

\paragraph{Dual iterates.} Let us now choose $S$ to be a random column submatrix of the $m\times m$ identity matrix $I$. That is, we choose a  random subset $C\subset \{1,2,\dots,m\}$ and then let $S$ be the concatenation of columns $j\in C$ of $I$. We shall write $S=I_C$. Let $p_C$ be the probability that $S = I_C$. Assume that for each $j \in \{1,\dots,m\}$ there exists $C$ with $j \in C$ such that $p_C>0$. Such a random set is called {\em proper}~\cite{Qu2015}.  

The SDA method (Algorithm~\ref{alg:SDA}) then takes the form
\begin{equation}\label{eq:SDA-compact08986986098BLOCK}
\boxed{\quad y^{k+1} =  y^k + I_C \lambda^k  \quad }
\end{equation}
where $\lambda^k$ is chosen so that the dual objective is maximized (see \eqref{eq:lambda_closed_form}). This is a variant of the {\em randomized Newton method} studied in \cite{Qu2015}. By examining \eqref{eq:lambda_closed_form}, we see that this method works by ``inverting'' randomized submatrices of the ``Hessian'' $AA^\top$. Indeed, $\lambda^k$ is in each iteration computed by solving a system with the matrix $I_C^\top A A^\top I_C$. This is the random submatrix of $A A^\top$ corresponding to rows and columns in $C$.

\paragraph{Primal iterates.} In view of the equivalence between Algorithm~\ref{alg:SDA-Primal} and the sketch-and-project method~\eqref{ch:two:NF}, the  primal iterative process associated with the randomized Newton method has the form
\begin{equation}\label{eq:SDA-primal009s09us0098BLOCK}
\boxed{\quad x^{k+1} = \arg \min_{x} \left\{ \|x-x^k\| \;:\; I_C^\top Ax = I_C^\top b \right\} \quad }
\end{equation}
This method is a variant of the {\em randomized block Kaczmarz} method of Needell \cite{Needell2012}. The method proceeds by projecting the last iterate $x^k$ onto a subsystem of $Ax=b$ formed by equations indexed by the set $C$.

\paragraph{The rate.} Provided that $H$ is nonsingular, the shared  rate of the randomized Newton and randomized block Kaczmarz methods is
\[\rho   \overset{\eqref{ch:three:eq:rho} }{=} 1- \lambda_{\min}^+\left(A^\top \E{I_C\left(I_C^\top AA^\top I_C\right)^\dagger I_C^\top} A\right).\]

Qu et al \cite{Qu2015} study the randomized Newton method for the problem of minimizing a smooth strongly convex function and prove linear convergence. In particular, they study the above rate in the case when $AA^\top$ is positive definite. Here we show that linear converges also holds for {\em weakly} convex quadratics (as long as $H$ is nonsingular). 

%An interesting feature of the randomized Newton method, established in \cite{Qu2015}, is  that when  viewed as a family of methods indexed by the size $\tau=|C|$, it enjoys superlinear speedup in $\tau$. That is, as $\tau$ increases by some factor, the iteration complexity drops by a factor that is at least as large.  It is possible to conduct a similar study in our setting  with a possibly singular matrix $AA^\top$, but such a study is not trivial and we therefore leave it for future research.

\subsection{Self-duality  for positive definite $A$}

If $A$ is positive definite, then we can choose $B=A$. As mentioned before, in this setting SDA is self-dual: $x^k=y^k$ for all $k$. The primal problem then becomes
\begin{eqnarray}  \text{minimize} \quad \ && P(x)\eqdef \tfrac{1}{2}x^\top A x \notag\\
\text{subject to} \quad \ && Ax=b \notag\\
&& x\in \R^n.\notag
\end{eqnarray}
and the dual problem becomes
\begin{eqnarray}\notag \text{maximize}\quad \  && D(y)\eqdef b^\top y - \tfrac{1}{2}y^\top A  y\\
\text{subject to} \quad \ && y \in \R^m. \notag
\end{eqnarray}

Note that the primal objective function  does not play any role in determining the solution; indeed, the feasible set contains a single point only: $A^{-1}b$. However, it does affect the iterative process.
 
\paragraph{Primal and dual iterates.} As before, let us choose $S=e^i$ (unit coordinate vector in $\R^m$) with probability $p_i>0$, where the probabilities $p_i$ are arbitrary. Then both the  primal and the dual iterates take the form 
\[
\boxed{\quad y^{k+1} =  y^k - \frac{A_{i:}  y^k - b_i}{A_{ii}}e_i   \quad }
\]
This is the randomized coordinate ascent method applied to the dual problem. 

\paragraph{The rate.}

If we choose $p_i = A_{ii}/\Tr{A}$, then
\[H = \E{S\left(S^\top A S\right)^\dagger S^\top} = \frac{I}{\Tr{A}},\]
whence
\[\rho \overset{\eqref{ch:three:eq:rho} }{=} 1 - \lambda_{\min}^+ \left(   A^{1/2}H A^{1/2}\right)  = 1 - \frac{ \lambda_{\min}(A)}{\Tr{A}}.\]

It is known that for this problem,  randomized coordinate descent applied to the dual problem, with this choice of probabilities, converges with this rate\cite{Leventhal2010}.

\section{Application: Randomized Gossip Algorithms} \label{sec:gossip}

In this section we apply our method and results to the distributed consensus (averaging) problem.

Let $(V,E)$ be a connected network with $|V|=n$ nodes and $|E|=m$ edges, where each edge is an unordered pair $\{i,j\} \in E$ of distinct nodes. Node $i \in V$ stores a private value $c_i\in \R$. The goal of the \emph{distributed consensus problem} is for the network to compute the average of these private values in a distributed fashion \cite{Boyd2006,OlshevskyTsitsiklis2009}. This means that the exchange of information  can only occur along the edges of the network.

The nodes may represent people in a social network, with edges representing friendship and private value representing certain private information, such as salary. The goal would be to compute the average salary via an iterative process where only friends are allowed to exchange information. The nodes may represent sensors in a wireless sensor network, with an edge between two sensors if they are close to each other so that they can communicate. Private values  represent measurements of some quantity performed by the sensors, such as the temperature. The goal is for the network to compute the average temperature.

\subsection{Consensus as a projection problem}

We now show how one can model the consensus (averaging) problem in the form \eqref{eq:P}. Consider the  projection problem
\begin{eqnarray}  \text{minimize} \quad \ && \tfrac{1}{2}\|x - c\|_2^2 \notag \\
 \text{subject to} \quad \ & & x_1=x_2=\cdots = x_n, \label{eq:ohs09hud98yd}
\end{eqnarray}
and note that the optimal solution $x^*$ must necessarily satisfy \[x^*_i=\bar{c}\eqdef \frac{1}{n}\sum_{i=1}^n c_i,\] for all $i$. There are many ways in which the constraint forcing all coordinates of $x$ to be equal can be represented in the form of a linear system $Ax=b$.   Here are some examples:

\begin{enumerate}
\item {\em Each node is equal to all its neighbours.} Let the equations of the system $Ax=b$ correspond to constraints \[x_i=x_j,\] for $\{i,j\}\in E$. That is, we are enforcing all pairs of vertices joined by an edge to have the same value.  Each edge $e\in E$ can be written in two ways: $e = \{i,j\}$ and $e=\{j,i\}$, where $i,j$ are the incident vertices. In order to avoid duplicating constraints,  for each edge $e\in E$ we use $e=(i,j)$ to denote an arbitrary but fixed order of its incident vertices $i,j$. We then let $A\in \R^{m\times n}$ and $b=0\in \R^m$, where  \begin{equation}\label{eq:89hs87s8ys}(A_{e:})^\top = f_i - f_j,\end{equation} and where  $e=(i,j)\in E$, $f_i$ (resp.\ $f_j$) is the $i^{\text{th}}$ (resp.\ $j^{\text{th}}$) unit coordinate vector in $\R^n$.  Note that the constraint $x_i=x_j$ is represented only once in the linear system. Further, note that the matrix \begin{equation}\label{eq:Laplacian09709}L = A^\top A \end{equation} is the {\em Laplacian} matrix of the graph $(V,E)$:
 \[L_{ij} = \begin{cases}
 d_i & i=j\\
 -1 & i\neq j,\,\, (i,j)\in E\\
 0 & \text{otherwise,}
 \end{cases}\]
 where $d_i$ is the degree of node $i$.

\item {\em Each node is the average of its neighbours.}  Let the equations of the system $Ax=b$ correspond to constraints 
\[x_i = \frac{1}{d_i}\sum_{j \in N(i)} x_j,\]
for $i\in V$, where $N(i)\eqdef \left\{ j\in V \,\,: \,\, \{i,j\} \in E\right\}$ is the set of neighbours of node $i$ and $d_i \eqdef |N(i)|$ is the degree of node $i$.  That is, we require that the values stored at each node are equal to the average of the values of its neighbours. This corresponds to the choice $b=0$ and \begin{equation}\label{eq:s8h98s78gd}(A_{i:})^\top = f_i - \frac{1}{d_i}\sum_{j \in N(i)}f_j.\end{equation} Note that $A\in \R^{n\times n}$.

\item {\em Spanning subgraph.} Let $(V,E')$ be any connected subgraph of $(V,E)$. For instance, we can choose a spanning tree. We can now apply any of the 2 models above to this new graph and either require $x_i=x_j$ for all $\{i,j\}\in E'$, or require the value $x_i$ to be equal to the average of the values $x_j$ for all neighbours $j$ of $i$ in $(V,E')$.
 
\end{enumerate}

Clearly, the above list does not exhaust the ways in which the constraint $x_1=\dots=x_n$ can be modeled as a linear system. For instance, we could build the system from constraints such as $x_1 = x_2 + x_4 -  x_3$, $x_1 = 5 x_2 - 4 x_7$ and so on.

Different representations of the constraint $x_1=\cdots=x_n$, in combination with a choice of $\cal D$, will lead to a wide range of specific algorithms for the consensus problem \eqref{eq:ohs09hud98yd}. Some (but not all) of these algorithms will have the property that communication only happens along the edges of the network, and these are the ones we are interested in. The number of combinations is very vast. We will therefore only highlight two options, with the understanding that based on this, the interested reader can assemble other specific methods as needed.

\subsection{Model 1: Each node is equal to its neighbours}

 Let $b=0$ and $A$ be as in~\eqref{eq:89hs87s8ys}. Let the distribution $\cal D$ be defined by setting $S=e_i$  with probability $p_i>0$, where $e_i$ is the $i^{\text{th}}$ unit coordinate vector in $\R^m$. We have $B=I$, which means that Algorithm~\ref{alg:SDA-Primal} is the randomized Kaczmarz (RK) method \eqref{eq:SDA-primal009s09us0098} and Algorithm~\ref{alg:SDA} is the randomized coordinate ascent method  \eqref{eq:SDA-compact08986986098}. 

Let us take $y^0 = 0$ (which means that $x^0=c$), so that in Theorem~\ref{theo:Enormerror} we have $t=0$, and hence $x^k  \to x^*$. The particular choice of the starting point $x^0=c$ in the primal process has a very tangible meaning: for all $i$, node $i$ initially knows value $c_i$. The primal iterative process will dictate how the local values are modified in an iterative fashion so that eventually all nodes contain the optimal value $x^*_i = \bar{c}$.

\paragraph{Primal method.} In view of \eqref{eq:89hs87s8ys}, for each edge $e = (i,j)\in E$, we have  $\|A_{e:}\|_2^2=2$ and $A_{e:}x^k = x^k_i - x^k_j$. Hence, if the edge $e$ is selected by the RK method, \eqref{eq:SDA-primal009s09us0098} takes the specific form
\begin{equation}\label{eq:s98g98gsf66r6fs}\boxed{\quad x^{k+1} = x^k - \frac{x^k_i - x^k_j}{2} (f_i - f_j) \quad }\end{equation}
From \eqref{eq:s98g98gsf66r6fs} we see that only the $i^{\text{th}}$ and $j^{\text{th}}$ coordinates of $x^k$ are updated, via
\[x^{k+1}_i = x^k_i - \frac{x^k_i - x^k_j}{2} = \frac{x_i^k + x_j^k}{2}\]
and
\[x^{k+1}_j = x^k_j + \frac{x^k_i - x^k_j}{2} = \frac{x_i^k + x_j^k}{2}.\]
Note that in each iteration of RK, a random edge is selected, and the nodes on this edge replace their local values by their average. This is a basic variant of the {\em randomized gossip} algorithm \cite{Boyd2006, ZouziasFreris2009}. 

\paragraph{Invariance.}
Let $f$ be the vector of all ones in $\R^n$ and notice that from \eqref{eq:s98g98gsf66r6fs} we obtain
$f^\top x^{k+1} = f^\top x^k$ for all $k$. This means that  for all $k\geq 0$ we have the invariance property:
\begin{equation}\label{eq:invariance}\sum_{i=1}^n x_i^k = \sum_{i=1}^n c_i.\end{equation}

\paragraph{Insights from the dual perspective.} We can now bring new insight into the randomized gossip algorithm by considering the  dual iterative process.  The dual method \eqref{eq:SDA-compact08986986098} maintains  weights $y^k$ associated with the edges of $E$ via the process:
\[y^{k+1} = y^k - \frac{A_{e:} (c - A^\top y^k)}{2}e_e,\]
where $e$ is a randomly selected edge. Hence, only the weight of a single edge is updated in each iteration. At optimality, we have
$x^* =c + A^\top y^*$. That is, for each $i$
\[\delta_i\eqdef \bar{c} - c_i = x_i^* - c_i =  (A^\top y^*)_i = \sum_{e\in E} A_{ei} y^*_e,\]
where $\delta_i$ is the correction term which needs to be added to $c_i$ in order for node $i$ to contain the value $\bar{c}$. From the above we observe that these correction terms are maintained by the dual method as an inner product of the $i^{\text{th}}$ column of $A$ and $y^k$, with the optimal correction being $\delta_i = A_{:i}^\top y^*$.

\paragraph{Rate.} Both Theorem~\ref{theo:Enormerror} and Theorem~\ref{theo:2} hold, and hence we automatically get several types of convergence for the randomized gossip method. In particular, to the best of our knowledge, no primal-dual type of convergence exist in the literature.  Equation~\eqref{eq:dualitygap09709709}  gives a stopping criterion certifying convergence via the duality gap, which is also new. 

In view of \eqref{eq:98hs8h8ss} and \eqref{eq:Laplacian09709}, and since $\|A\|_F^2 = 2m$, the convergence rate appearing in  all these complexity results is given by
\[\rho = 1 - \frac{\lambda_{\min}^+(L)}{2m},\]
where $L$ is the Laplacian of $(V,E)$. While it is know that the Laplacian is singular, the rate depends on the smallest nonzero eigenvalue. This means that the number of iterations needed to output an $\epsilon$-solution in expectation scales as $O(\left(2m/\lambda_{\min}^+(L)\right)\log(1/\epsilon))$, i.e., linearly with the number of edges.

\subsection{Model 2: Each node is equal to the average of its neighbours}

Let $A$ be as in \eqref{eq:s8h98s78gd} and $b=0$. Let the distribution $\cal D$ be defined by setting $S=f_i$  with probability $p_i>0$, where $f_i$ is the $i^{\text{th}}$ unit coordinate vector in $\R^n$. Again, we have $B=I$, which means that Algorithm~\ref{alg:SDA-Primal} is the randomized Kaczmarz (RK) method \eqref{eq:SDA-primal009s09us0098} and Algorithm~\ref{alg:SDA} is the randomized coordinate ascent method  \eqref{eq:SDA-compact08986986098}. As before, we choose $y^0=0$, whence $x^0=c$.

\paragraph{Primal method.} Observe that $\|A_{i:}\|_2^2 = 1 + 1/d_i$. The RK method \eqref{eq:SDA-primal009s09us0098} applied to this formulation of the problem takes the form
\begin{equation} \label{eq:sihiuhd098d7d}
\boxed{\quad x^{k+1} =  x^k - \frac{x^k_i - \frac{1}{d_i}\sum_{j\in N(i)}x^k_j  }{1 + 1/d_i} \left(f_i - \frac{1}{d_i}\sum_{j\in N(i)}f_j \right) \quad }
\end{equation}
where $i$ is chosen at random. This means that only coordinates in $i\cup N(i)$ get updated in such an iteration, the others remain unchanged. For node $i$ (coordinate $i$), this update is

\begin{equation}\label{eq:9g8g98gssdd} x^{k+1}_i = \frac{1}{d_i+1} \left( x_i^k + \sum_{j\in N(i)}x^k_j \right).\end{equation}
That is, the updated value at node $i$ is the average of the values of its neighbours and the previous value at $i$. From \eqref{eq:sihiuhd098d7d} we see that the values at nodes $j\in N(i)$ get updated as follows:

\begin{equation}\label{eq:98889ff} x^{k+1}_j = x_j^{k} + \frac{1}{d_i+1}\left(x_i^k - \frac{1}{d_i}\sum_{j'\in N(i)}x^k_{j'} \right).\end{equation}

\paragraph{Invariance.}  Let $f$ be the vector of all ones in $\R^n$ and notice that from \eqref{eq:sihiuhd098d7d} we obtain
\[f^\top x^{k+1} = f^\top x^k -\frac{x^k_i - \frac{1}{d_i}\sum_{j\in N(i)}x^k_j  }{1 + 1/d_i} \left(1 - \frac{d_i}{d_i} \right)=f^\top x^k, \]
for all $k$.
It follows that the method satisfies the invariance property \eqref{eq:invariance}.

\paragraph{Rate.} The method converges with the rate $\rho$ given by \eqref{eq:98hs8h8ss}, where $A$ is given by \eqref{eq:s8h98s78gd}.   If $(V,E)$ is a complete graph (i.e., $m=\tfrac{n(n-1)}{2}$), then $L = \tfrac{(n-1)^2}{n}A^\top A$ is the Laplacian. In this case,  $\|A\|_F^2 = \Tr{A^\top A} = \tfrac{n}{(n-1)^2}\Tr{L}=\tfrac{n}{(n-1)^2}\sum_i d_i  = \tfrac{n^2}{n-1}$ and hence
\[\rho \overset{\eqref{eq:s8h98s78gd}}{=} 1 - \frac{\lambda_{\min}^+(A^\top A)}{\|A\|_F^2} = 1 - \frac{\tfrac{n}{(n-1)^2}\lambda_{\min}^+(L)}{\tfrac{n^2}{n-1}} = 1 - \frac{\lambda_{\min}^+(L)}{2m}.\]

\section{Proof of Theorem~\ref{theo:Enormerror}} \label{sec:proof}

In this section we prove Theorem~\ref{theo:Enormerror}. We proceed as follows:  in Section~\ref{subsec:error} we characterize the space in which the iterates move, in Section~\ref{subsec:inequality} we establish a certain key technical  inequality, in Section~\ref{subsec:convergence} we establish  convergence of iterates, in Section~\ref{subsec:residual} we derive a rate for the residual and finally, and in Section~\ref{subsec:lower_bound} we establish the lower bound on the convergence rate.

\subsection{An error lemma} \label{subsec:error}

The following result describes the space in which the iterates move.   It is an extension of the observation in item 2 of Remark~\ref{lem:5shsuss} to the case when $x^0$ is chosen arbitrarily.

\begin{lemma} \label{lem:error} Let the assumptions of Theorem~\ref{theo:Enormerror} hold. For all $k\geq 0$ there exists $w^k\in \R^m$ such that
$ x^k - x^* - t = B^{-1}A^\top w^k$.
\end{lemma}
\begin{proof}
We proceed by induction. Since by definition, $t$ is the projection of $x^0-c$ onto  $\Null{A}$ (see \eqref{eq:98hs8htt}), applying Proposition~\ref{prop:decomposition} we know that $x^0 -c = s+ t$, where $s = B^{-1}A^\top \hat{y}^0$ for some $\hat{y}^0\in \R^m$. Moreover, in view of \eqref{eq:opt_primal}, we know that $x^* =c + B^{-1}A^\top y^*$, where $y^*$ is any dual optimal solution. Hence, 
\[ x^0 - x^* - t = B^{-1}A^\top (\hat{y}^0-y^*).\]
Assuming the relationship holds for $k$, we have
\begin{eqnarray*} x^{k+1} - x^* - t
 &\overset{(\text{Alg}~\ref{alg:SDA-Primal})}{=} & \left[x^k - B^{-1}A^\top S (S^\top A B^{-1} A^\top S)^\dagger S^\top (A x^k  -b ) \right] - x^*-t\\
&= &  \left[x^* + t + B^{-1}A^\top w^k - B^{-1}A^\top S (S^\top A B^{-1} A^\top S)^\dagger S^\top (A x^k  -b )\right] - x^* - t\\
&=& B^{-1}A^\top w^{k+1},
\end{eqnarray*}
where $w^{k+1} = w^k - S (S^\top A B^{-1} A^\top S)^\dagger S^\top (A x^k  -b )$. %\qed
\end{proof}

\subsection{A key inequality} \label{subsec:inequality}
The following inequality is of key importance in the proof of the main theorem.

\begin{lemma}\label{lem:WGWtight}
Let $0\neq W \in  \R^{m\times n}$ and
$G \in \R^{m\times m}$ be symmetric positive definite. Then the matrix $W^\top G W$ has a positive eigenvalue, and the following inequality holds for all $y\in \R^m$:
\begin{equation}\label{eq:WGWtight}
 y^\top WW^\top G WW^\top y \geq \lambda_{\min}^+(W^\top G W)\|W^\top y\|_2^2.
 \end{equation}
 Furthermore, this bound is tight.
\end{lemma}
\begin{proof} Fix arbitrary $y\in \R^m$. By Lemma~\ref{lem:WGW}, $W^\top y\in \myRange{W^\top G W}$. Since $W$ is nonzero the positive semidefinite matrix $W^\top G W$ is also nonzero, and hence it has a positive eigenvalue. Hence, $\lambda_{\min}^+(W^\top G W)$ is well defined. 
Let $\lambda_{\min}^+(W^\top G W) = \lambda_{1}\leq \cdots\leq  \lambda_{\tau}$ be the positive eigenvalues of $W^\top GW$, with associated orthonormal eigenvectors $q_1,\dots,q_\tau$. We thus have
\[W^\top G W = \sum_{i=1}^\tau \lambda_i q_i q_i^\top.\] It is easy to see that these eigenvectors span $\myRange{W^\top GW}$. Hence,  we can write $W^\top y = \sum_{i=1}^{\tau} \alpha_i q_i$ and therefore
\[y^\top WW^\top G W W^\top y = \sum_{i=1}^{\tau} \lambda_{i}\alpha_i^2 \geq  
\lambda_1\sum_{i=1}^{\tau} \alpha_i^2 = \lambda_1 \|W^\top y\|_2^2.\]
Furthermore this bound is tight, as can be seen by selecting $y$ so that $W^\top y = q_1$. %\qed
\end{proof}

\subsection{Convergence of the iterates} \label{subsec:convergence}
Subtracting $x^*+t$ from both sides of the update step of Algorithm~\ref{alg:SDA-Primal}, and letting 
\begin{equation} \label{ch:three:Z}
Z\eqdef A^\top S (S^\top A B^{-1} A^\top S)^\dagger S^\top  A,\end{equation}
 we obtain the identity
\begin{equation} \label{eq:fixed0}
 x^{k+1} - (x^{*}+t) = (I-B^{-1}Z)(x^k-(x^{*}+t)),
 \end{equation}
 where we used that $t\in \Null{A}$. Left multiplying~\eqref{eq:fixed0} by $B^{1/2}$ we see that the residual 
 \[r^k \eqdef  B^{1/2}(x^{k+1} - (x^{*}+t)),\]
  satisfies the recurrence 
\begin{equation} \label{eq:fixedr}
r^{k+1} = (I-B^{-1/2}ZB^{-1/2})r^k.
 \end{equation}
 In view of $\E{Z} = A^\top H A$, and taking norms and expectations (in $S$) on both sides of~\eqref{eq:fixedr}  gives 
 \begin{eqnarray}
\E{\norm{r^{k+1}}_{2}^2\, | \, r^k} &=& 
\E{\norm{(I-{B^{-1/2}}ZB^{-1/2})r^k }^2_2} \nonumber \\
&\overset{\eqref{eq:B12ZB12proj}}{=}& \E{(r^k)^\top (I-{B^{-1/2}}ZB^{-1/2}) r^k} \nonumber\\
& =& \norm{r^k}_{2}^2-
(r^k)^\top  B^{-1/2} \E{Z} B^{-1/2}r^k \nonumber \\
&=& \norm{r^k}_{2}^2- (r^k)^\top B^{-1/2} A^\top H A  B^{-1/2} r^k, \label{eq:theostep1}
\end{eqnarray}
 In view of Lemma~\ref{lem:error}, let $w^k\in \R^{m}$ be such that $r^k = B^{1/2}(B^{-1}A^\top w^k) = B^{-1/2}A^\top w^k.$  Thus
\begin{eqnarray}
(r^k)^\top B^{-1/2}A^\top H A B^{-1/2}r^k \quad &=&(w^k)^\top AB^{-1}A^\top H A B^{-1}A^\top w^k\nonumber \\
&\overset{(\text{Lemma}~\ref{lem:WGWtight})}{\geq}& \lambda_{\min}^+(B^{-1/2}A^\top H A B^{-1/2})\cdot
\|B^{-1/2}A^\top w^k\|^2_2 \nonumber\\
&=& (1- \rho) \cdot\norm{r^k}_2^2,\label{eq:theostep2}
\end{eqnarray}
where we applied Lemma~\ref{lem:WGWtight} with $W = AB^{-1/2}$ and $G = H$,  so that $W^\top GW = B^{-1/2}A^\top H A B^{-1/2}.$
Substituting~\eqref{eq:theostep2} into~\eqref{eq:theostep1} gives
$\E{\norm{r^{k+1}}_{2}^2\, | \, r^k} \leq  \rho \cdot \norm{r^k}_2^2$.
Using the tower property of expectations, we obtain the recurrence
\[\E{\norm{r^{k+1}}_{2}^2} \leq  \rho \cdot \E{\norm{r^k}_2^2}. \]
To prove~\eqref{eq:Enormerror} it remains to  unroll the recurrence.

\subsection{Convergence of the residual} \label{subsec:residual}

We now prove \eqref{eq:s98h09hsxxx}. Letting $V_k = \|x^k-x^*-t\|_B^2$, we have
\begin{eqnarray*}
\E{\|Ax^k - b\|_B} &=& \E{\|A(x^k-x^*-t) + At\|_B}\\
&\leq &\E{\|A(x^k-x^*-t) \|_B} + \|At\|_B\\
&\leq &  \|A\|_B \E{\sqrt{V_k}} + \|At\|_B\\
&\leq &\|A\|_B \sqrt{\E{V_k}} + \|At\|_B\\
&\overset{\eqref{eq:Enormerror}}{\leq}&\|A\|_B \sqrt{\rho^k V_0 } + \|At\|_B,
 \end{eqnarray*} 
 where in the step preceding the last one we have used Jensen's inequality.

\subsection{Proof of the lower bound}\label{subsec:lower_bound}
Now we prove~\eqref{eq:nubound}.
Using Lemma~\ref{lem:WGW} with $G=H$ and $W=A B^{-1/2}$ gives
\begin{align*}
\myRange{B^{-1/2}A^\top HAB^{-1/2}}=\myRange{B^{-1/2}A^\top},
\end{align*}
from which we deduce that \begin{eqnarray*}\Rank{A} &=&
 \dim\left(\myRange{A^\top}\right) \\
&=& \dim\left(\myRange{B^{-1/2}A^\top}\right)\\
&=&  \dim\left(\myRange{B^{-1/2}A^\top HAB^{-1/2}}\right)\\
&=& \Rank{B^{-1/2} A^\top HA B^{-1/2}}.\end{eqnarray*}

Hence, $\Rank{A}$ is equal to the number of nonzero eigenvalues of $B^{-1/2}A^\top HAB^{-1/2}$, from which we immediately obtain the bound
\begin{eqnarray*}
\Tr{B^{-1/2}A^\top HA B^{-1/2}} &\geq & \Rank{A }\, \lambda_{\min}^+(B^{-1/2}A^\top  HAB^{-1/2}). 
\end{eqnarray*}
To conclude the proof, note that $\E{Z} = A^\top H A$ where $Z$ is defined in~\eqref{ch:three:Z}.
In order to obtain \eqref{eq:nubound},  it only remains to combine the above inequality with 
\[
\E{\Rank{S^\top A}}  \overset{ \eqref{eq:B12ZB12trace}}{=} =  \E{\Tr{B^{-1/2}Z B^{-1/2}}} \\
= \Tr{B^{-1/2}A^\top HAB^{-1/2}}.
\]
%where we applied~\eqref{eq:ugisug7sss} with $Z =Z_{S^\topA}.$
%
%\subsection{Proof of the lower bound}\label{subsec:lower_bound}
%Now we prove~\eqref{eq:nubound}.
%Since $\E{Z} = A^\top H A$, using Lemma~\ref{lem:WGW} with $G=H$ and $W=A B^{-1/2}$ gives
%\begin{align*}
%\myRange{B^{-1/2}\E{Z}B^{-1/2}}=\myRange{B^{-1/2}A^\top},
%\end{align*}
%from which we deduce that \begin{eqnarray*}\Rank{A} &=&
% \dim\left(\myRange{A^\top}\right) \\
%&=& \dim\left(\myRange{B^{-1/2}A^\top}\right)\\
%&=&  \dim\left(\myRange{B^{-1/2}\E{Z}B^{-1/2}}\right)\\
%&=& \Rank{B^{-1/2}\E{Z}B^{-1/2}}.\end{eqnarray*}
%
%Hence, $\Rank{A}$ is equal to the number of nonzero eigenvalues of $B^{-1/2}\E{Z}B^{-1/2}$, from which we immediately obtain the bound
%\begin{eqnarray*}
%\Tr{B^{-1/2}\E{Z}B^{-1/2}} &\geq& \Rank{A }\, \lambda_{\min}^+(B^{-1/2}\E{Z}B^{-1/2}). 
%\end{eqnarray*}
%In order to obtain \eqref{eq:nubound},  it only remains to combine the above inequality with
%\[
%\E{\Rank{S^\top A}}  \overset{ \eqref{eq:ugisug7sss}}{=} \E{\Tr{B^{-1}Z}}  =  \E{\Tr{B^{-1/2}Z B^{-1/2}}} \\
%= \Tr{B^{-1/2}\E{Z}B^{-1/2}}. 
%\]

\section{Proof of Theorem~\ref{theo:2}} \label{sec:proof2}

In this section we prove Theorem~\ref{theo:2}. We dedicate a subsection to each of the three complexity bounds.

\subsection{Dual suboptimality}

Since $x^0\in c+\myRange{B^{-1}A^\top}$, we have $t=0$ in Theorem~\ref{theo:Enormerror}, and hence \eqref{eq:Enormerror} says that \begin{equation}\label{eq:s98h98shs}\E{U_k} \leq \rho^k U_0.\end{equation}
It remains to apply Proposition~\ref{lem:correspondence}, which says that $U_k = D(y^*)-D(y^k)$.

\subsection{Primal suboptimality}

Letting $U_k = \tfrac{1}{2}\|x^k- x^*\|_B^2$, we can write
\begin{eqnarray}
P(x^k) - OPT &=& \tfrac{1}{2}\|x^k - c\|_B^2 - \tfrac{1}{2}\|x^* - c\|_B^2\notag\\
&=& \tfrac{1}{2}\|x^k - x^* + x^* - c\|_B^2  - \tfrac{1}{2}\|x^* - c\|_B^2\notag\\
&=& \tfrac{1}{2}\|x^k- x^*\|_B^2 +  (x^k-x^*)^\top B (x^*-c)\notag\\
&\leq & U_k + \|B^{1/2}(x^k-x^*)\|_2 \|B^{1/2}(x^*-c)\|_{2} \notag\\
&=& U_k +\|x^k-x^*\|_B  \|x^*-c\|_B \notag \\
&= & U_k + 2 \sqrt{U_k} \sqrt{OPT}.\label{eq:iuhs89h98s6s}
\end{eqnarray}
By taking expectations on both sides of \eqref{eq:iuhs89h98s6s}, and using Jensen's inequality, we obtain
\[\E{P(x^k)-OPT} \leq \E{U_k} + 2\sqrt{OPT} \sqrt{\E{U_k}} \overset{\eqref{eq:s98h98shs}}{\leq} \rho^k U_0 + 2 \rho^{k/2} \sqrt{OPT \times U_0},\]
which establishes the bound on primal suboptimality \eqref{eq:PRIMALSUBOPT}. 

\subsection{Duality gap}

Having established rates for primal and dual suboptimality, the rate for the duality gap follows easily:
\begin{eqnarray*}
\E{P(x^k) - D(y^k)} &=& \E{P(x^k) - OPT + OPT - D(y^k)}\\
&=&  \E{P(x^k) - OPT} + \E{OPT - D(y^k)}\\
&\overset{\eqref{eq:DUALSUBOPT}  + \eqref{eq:PRIMALSUBOPT}}{=}& 2 \rho^k U_0 + 2 \rho^{k/2} \sqrt{OPT \times U_0}.
\end{eqnarray*}

\section{Numerical Experiments: Randomized Kaczmarz Method with Rank-Deficient System}
\label{sec:experiments}
To illustrate some of the novel aspects of our theory, we perform numerical experiments with the Randomized Kaczmarz method~\eqref{eq:SDA-primal009s09us0098} (or equivalently the randomized coordinate ascent method applied to the dual problem~\eqref{eq:Dualfunc}) and compare the empirical convergence to the convergence predicted by our theory. We test several randomly generated rank-deficient systems and compare the evolution of the empirical primal error $\norm{x^k-x^*}_2^2/\norm{x^0-x^*}_2^2$ to the convergence dictated by the rate $\rho = 1-\lambda_{\min}^+\left(A^\top A\right)/\norm{A}_F^2$ given in~\eqref{eq:98hs8h8ss} and the lower bound $1-1/\Rank{A}\leq\rho$. From Figure~\ref{ch:three:fig:rand} we can see that the RK method converges despite the fact that the linear systems are rank deficient. While previous results do not guarantee that RK converges for rank-deficient matrices, our theory does as long as the system matrix has no zero rows. Furthermore, we observe in Figure~\ref{ch:three:fig:rand} that the lower the rank of the system matrix, the faster the  convergence of the RK method, and moreover, the closer the empirical convergence is to the convergence dictated by the rate $\rho $ and lower bound on $\rho$. In particular, on the low rank system in Figure~\ref{ch:three:fig:randa}, the empirical convergence is very close to both the convergence dictated by $\rho$ and the lower bound.
While on the full rank system in Figure~\ref{ch:three:fig:randd}, the convergence dictated by $\rho$ and the lower bound on $\rho$ are no longer an accurate estimate of the empirical convergence.

\begin{figure}[!h]
    \centering
\begin{subfigure}[t]{0.45\textwidth}
        \centering
\includegraphics[width =  \textwidth, trim= 80 270 80 280, clip ]{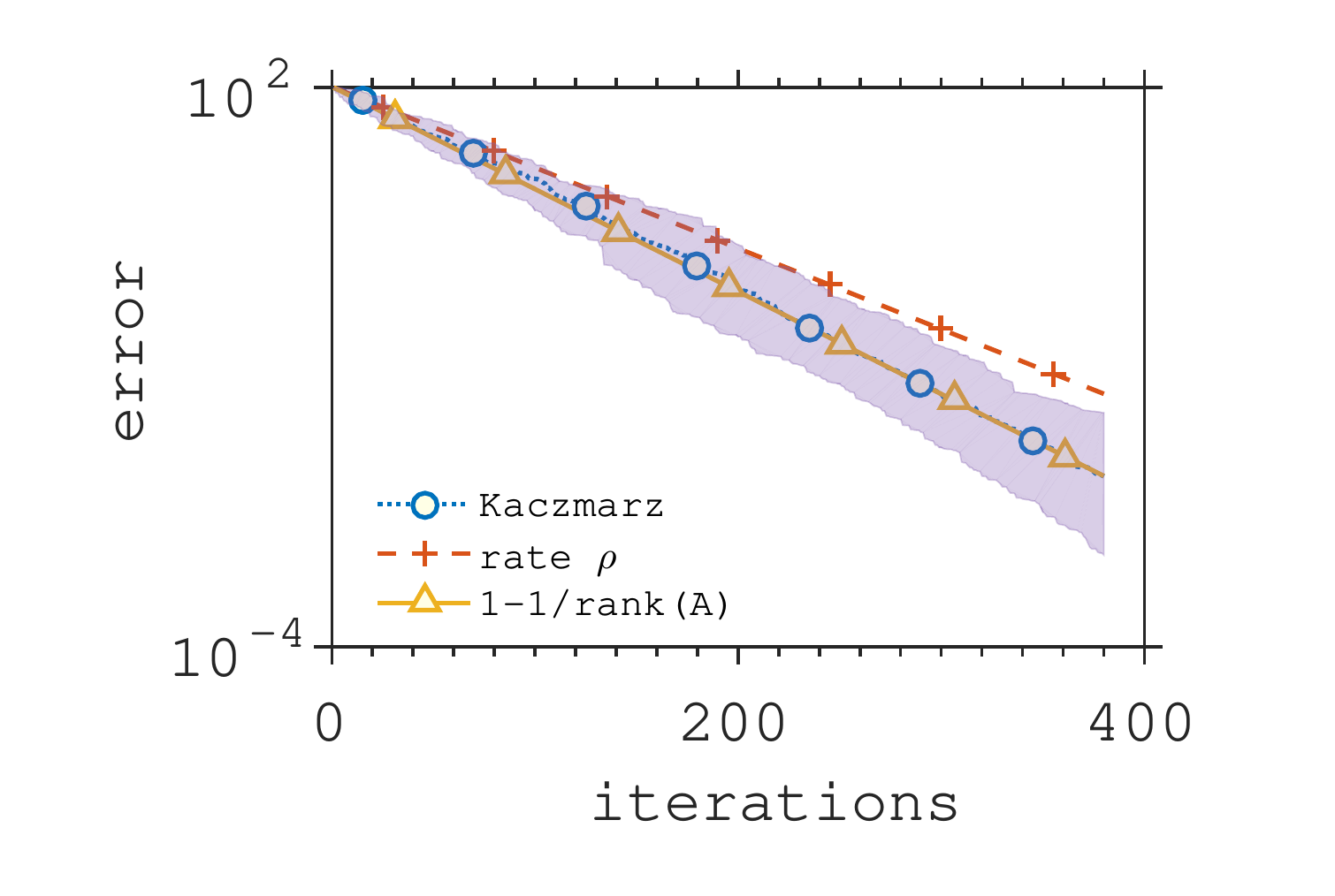}
        \caption{$\Rank{A}= 40$}\label{ch:three:fig:randa}
\end{subfigure}%
\begin{subfigure}[t]{0.45\textwidth}
        \centering
\includegraphics[width =  \textwidth, trim= 80 270 80 280, clip ]{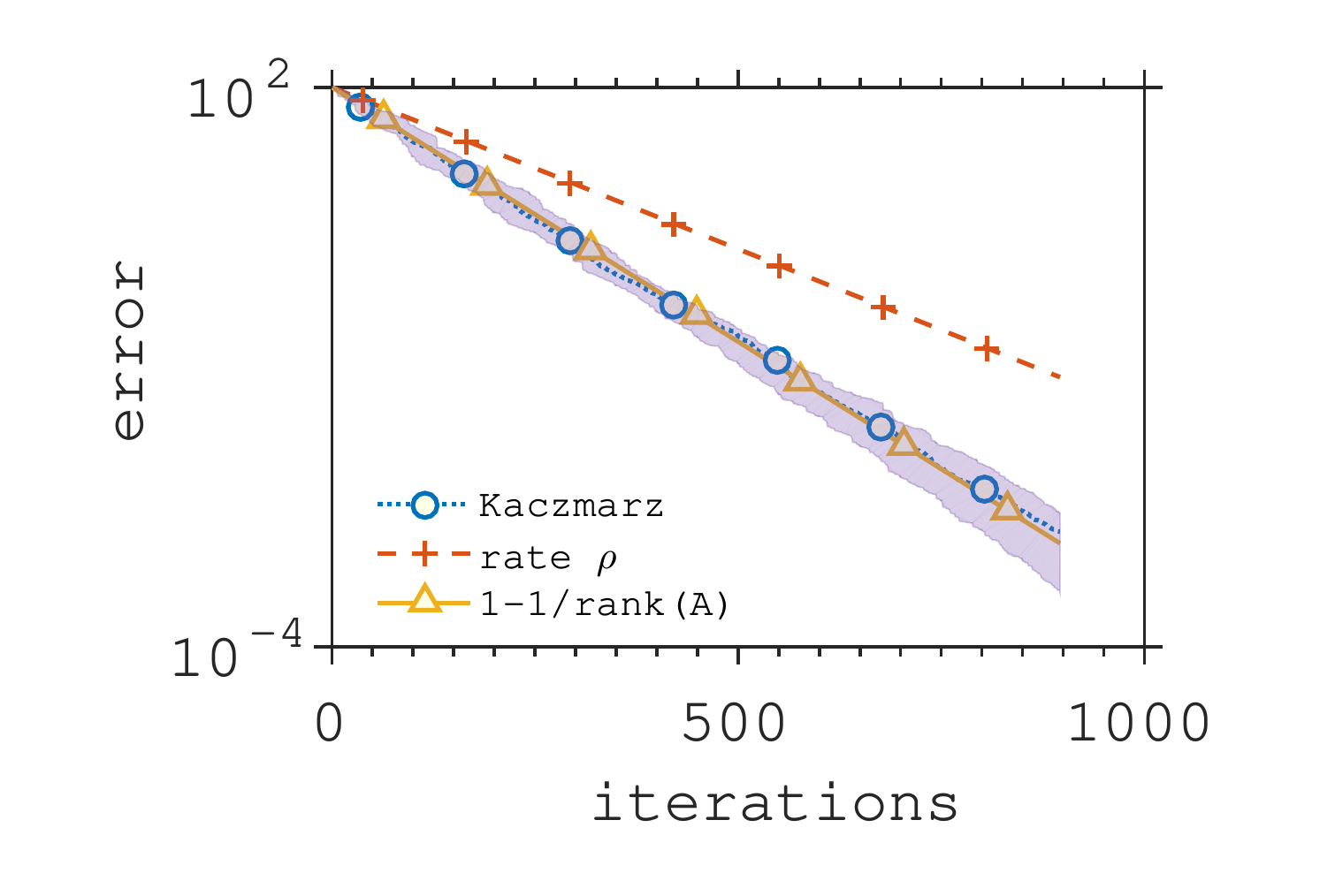}
        \caption{$\Rank{A} = 80$}\label{ch:three:fig:randb}
\end{subfigure}\\% 
\begin{subfigure}[t]{0.45\textwidth}
        \centering
\includegraphics[width =  \textwidth, trim= 80 270 80 280, clip ]{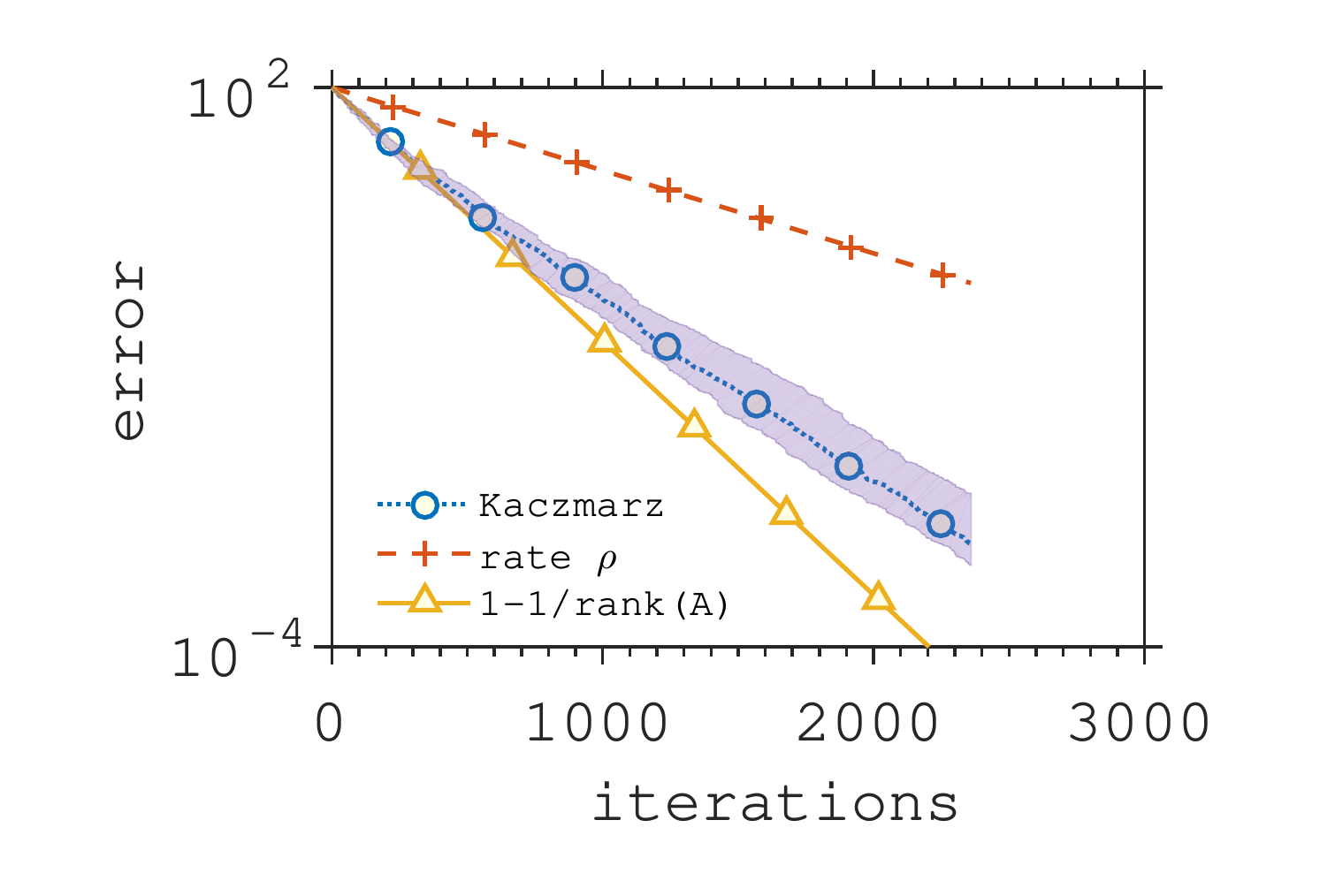}
        \caption{$\Rank{A} = 160$}\label{ch:three:fig:randc}
\end{subfigure}%
\begin{subfigure}[t]{0.45\textwidth}
        \centering
\includegraphics[width =  \textwidth, trim= 80 270 80 280, clip ]{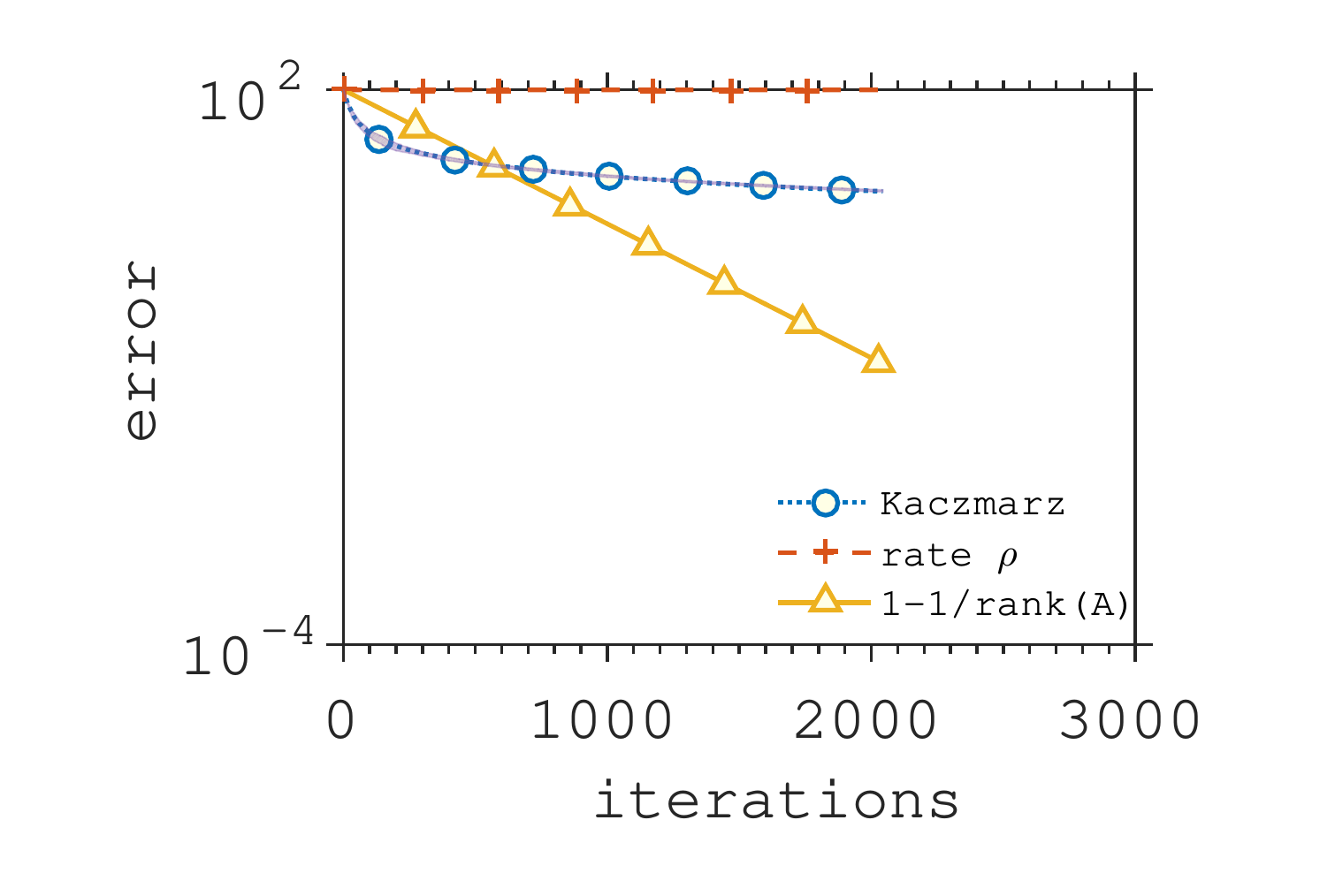}
        \caption{$\Rank{A} = 300$}\label{ch:three:fig:randd}
\end{subfigure}%
%%%%%%%%%%%%%%%%%%%%%%%%%%%%%%%%%%%%%%%
    \caption{Synthetic MATLAB generated problems.  Rank deficient matrix $A~=~\sum_{i=1}^{\Rank{A}} \sigma_i u_i v_i^\top $ where $\sum_{i=1}^{300} \sigma_i u_i v_i^\top =$\texttt{rand}$(300,300)$ is an svd decomposition of a $300\times 300$ uniform random matrix. We repeat each experiment ten times. The blue shaded region is the $90\%$ percentile of relative error achieved in each iteration.
    }\label{ch:three:fig:rand}
\end{figure}

\section{Summary} \label{sec:conclusion}

We have developed a versatile and powerful  algorithmic framework for solving linear systems: {\em stochastic dual ascent (SDA)}. The SDA method finds the projection of a given point, in a fixed but arbitrary Euclidean norm, onto the solution space of the system.  Our method is  dual in nature, but can also be described in terms of primal iterates via a simple affine transformation of the dual variables. Viewed as a dual method, SDA belongs to a  novel class of randomized optimization algorithms:  it updates the current iterate by adding the product of a random matrix, drawn independently from a fixed distribution, and a vector. The update is chosen as the best point lying in the random subspace spanned by the columns of this random matrix. 

While SDA is the first method of this type, particular choices for the distribution of the random matrix lead to several known algorithms: randomized coordinate descent \cite{Leventhal2010} and randomized Kaczmarz \cite{Strohmer2009} correspond to a discrete distribution over the columns of the identity matrix, randomized Newton method \cite{Qu2015} corresponds to a discrete distribution over column submatrices of the identity matrix, and Gaussian descent~\cite{Stich2014} corresponds to the case when the random matrix is a  Gaussian vector. 

We equip the method with several complexity results with the same rate of exponential decay in expectation (aka linear convergence) and establish a tight lower bound on the rate. In particular, we prove convergence of primal iterates, dual function values, primal function values, duality gap and of the residual. The method converges under very weak conditions beyond consistency of the linear system. In particular, no rank assumptions on the system matrix are needed. For instance,   randomized Kaczmarz method converges linearly as long as the system matrix contains no zero rows.

Further, we show that SDA can be applied to the distributed (average) consensus problem. We recover a standard randomized gossip algorithm as a special case, and show that its complexity is proportional to the number of edges in the graph and inversely proportional to the smallest nonzero eigenvalue of the graph Laplacian.  Moreover, we illustrate how our framework can be used to obtain new randomized algorithms for the distributed consensus problem. 

Our framework extends to several other problems in optimization and numerical linear algebra. For instance, one can apply it to develop new stochastic algorithms for computing the inverse of a matrix and obtain state-of-the art performance for inverting matrices of huge sizes, which is the subject of the next chapter.

%\section{Acknowledgements}
%
%The second author would like to thank Julien Hendrickx from Universit\'{e} catholique de Louvain for a discussion regarding randomized gossip algorithms.

  % Stochastic Dual Ascent
%\clearpage
%---------------------------------------------------------------------------------
%	CHAPTER four: Randomized Matrix Inversion
%---------------------------------------------------------------------------------
%\onehalfspace   %% official UoE spacing
\chapter{Randomized Matrix Inversion}
\chaptermark{Randomized Matrix Inversion}
\label{ch:inverse}

%Fuils an bairns soud never see things hauf duin.
%
%    It needs powers of perception and mature judgment to visualise the result of an enterprise. 
\section{Introduction}

Here we extend our randomized methods for solving linear systems to methods for inverting matrices. Though there exists applications where one needs the explicit inverse of a matrix\footnote{for instance when one needs to store a Schur complement or a projection matrix}, inverting a matrix is seldom required. In contrast, calculating the approximate inverse of a matrix finds many applications. Most notably, calculating an approximate inverse finds applications in preconditioning~\cite{Saad2003} and, if the approximate inverse is guaranteed to be positive definite, then an iterative scheme for inverting a matrix can be used to design variable metric optimization methods. The methods we propose here converge globally and linearly to the inverse matrix and thus are well suited for quickly calculating approximate inverse matrices. 

%Matrix inversion is a standard tool in numerics, needed for instance, in computing a projection matrix or a Schur complement, which are common place calculations in computational methods. 
When only an approximate inverse is required, then iterative methods 
are the methods of choice, for they can terminate the iterative process when the desired accuracy is reached. This can be far more efficient than using an all-or-nothing direct method. 
%Calculating an approximate inverse is a much needed tool in preconditioning~\cite{Saad2003} and, if the approximate inverse is guaranteed to be positive definite, then the iterative scheme can be used to design variable metric optimization methods. 
Furthermore, iterative methods can make use of an initial estimate of the inverse when available.

The driving motivation of this work is the need to develop algorithms capable of computing the approximate inverse of very large matrices, where standard techniques take an excessive amount of time or simply fail. In particular, using the sketch-and-project technique  we develop a family of randomized/stochastic methods for inverting a matrix with  specialized variants maintaining symmetry or positive definiteness of the iterates. All methods in the family  converge globally (i.e., from any starting point) and linearly (i.e., the error decays exponentially). 
% Furthermore, the convergence of the methods is determined by the same rate $\rho$~\eqref{eq:rho}.

  As special cases, we obtain stochastic block variants of several quasi-Newton updates, including   bad Broyden (BB), good Broyden (GB),  Powell-symmetric-Broyden (PSB), Davidon-Fletcher-Powell (DFP) and Broyden-Fletcher-Goldfarb-Shanno (BFGS). To the best of our knowledge, these are the first stochastic versions of quasi-Newton updates. Moreover, this is the first time that randomized quasi-Newton methods are shown to be iterative methods for inverting a matrix. We also offer a new interpretation of the quasi-Newton methods through a Lagrangian dual viewpoint. This new viewpoint uncovers a fundamental link between quasi-Newton updates and approximate inverse preconditioning.  
  
  We develop an adaptive variant  of randomized block BFGS, in which we modify the distribution underlying the stochasticity of the method throughout the iterative process to achieve faster convergence. Through extensive numerical experiments with matrices arising  from several applications, we demonstrate that AdaRBFGS is highly competitive when compared with the well established Newton-Schulz and minimal residual methods.  In particular, on large-scale problems our method outperforms the standard methods by orders of magnitude.
    
    The development of efficient methods for estimating the inverse of very large matrices is a  much needed tool for preconditioning and variable metric methods in the advent of the big data era.

 \subsection{Chapter outline} 
In Section~\ref{sec:contributions} we summarize the main contributions of this chapter. In Section~\ref{sec:QN} we describe the quasi-Newton methods, which is the main inspiration of our methods. Subsequently, Section~\ref{sec:SIMI-nonsym}  describes two  algorithms, each corresponding to a variant of the inverse equation,  for inverting general square matrices. We also provide insightful dual viewpoints for both methods. In Section~\ref{sec:SIMI-sym} we describe a method specialized to inverting symmetric matrices. Convergence in expectation is examined in Section~\ref{sec:conv}, were we consider two types of convergence: the convergence of i) the expected norm of the error, and the convergence of ii) the norm of the expected error. In Section~\ref{sec:discrete} we specialize our methods to discrete distributions, and comment on how one may construct a probability distribution leading to better complexity rates (i.e., importance sampling), and how to construct an adaptive probability distribution.
 In Section~\ref{sec:discretemethods} we detail several instantiations of our family of methods, and their resulting convergence rates. We show how via the choice of the parameters of the method,  we obtain {\em stochastic block variants} of several well known quasi Newton methods.
We also describe the simultaneous randomized Kaczmarz method here.
Section~\ref{sec:AdaRBFGS} is dedicated to the development of an adaptive variant of our randomized BFGS method, AdaRBFS, for inverting positive definite matrices. 
% This method adaptively changes the stochasticity of the method throughout the iterative process to obtain faster practical convergence behaviour.
  Finally, in Section~\ref{sec:numerical} we show through numerical tests that AdaRBFGS significantly outperforms state-of-the-art  iterative matrix inversion methods on large-scale matrices.

\subsection{Notation} \label{subsec:notation}

Let $I $ denote the $n\times n$ identity matrix. Let
\[\dotprod{X,Y}_{F(B)} \eqdef \Tr{X^\top BYB},\]
denote the weighted Frobenius inner product,
where $X,Y \in \R^{n\times n}$ and $B \in \R^{n\times n}$ is a symmetric positive definite ``weight'' matrix. As the trace is invariant under cyclic permutations, a fact we use repeatedly  throughout this chapter, we have 
\begin{equation}\label{eq:98y988ff}
\norm{X}_{F(B)}^2 = \Tr{X^\top BX B} = \Tr{B^{1/2}X^\top B X B^{1/2}} =
 \norm{B^{1/2}XB^{1/2}}_{F}^2,\end{equation}
where we have used the convention $F=F(I)$, since  $\|\cdot\|_{F(I)}$ is the standard Frobenius norm. Let $\norm{\cdot}_{2}$ denote the induced operator norm for square matrices defined via \[\|Y\|_{2} \eqdef \max_{\norm{v}_2=1} \norm{Yv}_2.\] Finally, we define the weighted induced norm via 
\[\norm{Y}_{B}^* \eqdef  \norm{B^{1/2}YB^{1/2}}_2. \]

\subsection{Previous work}\label{sec:previous}

A widely used iterative method for inverting matrices is the Newton-Schulz method~\cite{Schulz1933} introduced in 1933, and its variants which is still subject of ongoing research~\cite{Li2010}. The drawback of the Newton-Schulz methods is that they do not converge for any initial estimate. Instead,  an initial estimate $X_0$ such that 
$\norm{I-X_0A}_2 <1$ is required to guarantee convergence. Though note that such a $X_0$ always exists\footnote{Take for example $X_0 = \alpha A^T$ with $0 < \alpha < 2/\norm{A}_2$.}.
The Newton-Schulz method enjoys local quadratic convergence. As has been 
observed before~\cite{Pan1991}, and as we observe in our numerical experiments, the Newton-Schulz method can experience slow initial convergence before the asymptotic second order convergence rate sets in.
This is contrast with the methods we present here that enjoy global linear convergence and a fast initial convergence.

% On the other hand, both the Newton-Schulz and MR method enjoy local quadratic convergence, thus when the iterates are close to the solution, these methods enjoy a notable speed-up. 
%% possible initial slow convergence of any high-order convergent method before the asymptotic convergence rate sets in. See [4].

%is close to $A^{-1}$ (in some norm) is required. 
 Bingham~\cite{Bingham1941} describes a method that uses the characteristic polynomial to recursively calculate the inverse, though it requires the calculating the coefficients of the polynomial when initiated, which is costly, and the method has fallen into disuse.  Goldfarb~\cite{Goldfarb1972} uses Broyden's method~\cite{Broyden1965} for iteratively inverting matrices. Our methods include a stochastic variant of Broyden's method.
 
 The approximate inverse preconditioning (\emph{AIP}) methods~\cite{Chow1998,Saad2003,Gould1998,Benzi1999} calculate an approximate inverse by minimizing in $X \in \R^{n\times n}$ the residual $\norm{XA-I}_F$ (Frobenius norm). They accomplish this by
applying a number of iterations of the steepest descent or minimal residual method.  
   A considerable drawback of the AIP methods, is that the approximate inverses  are not guaranteed to be positive definite nor symmetric, even when $A$ is both. 
A solution to the lack of symmetry is to ``symmetrize'' the estimate between iterations, but then it is difficult to guarantee the quality of the new symmetric estimate.
Another solution is to calculate directly a factored form $LL^\top =X$ and minimize in $L$ the residual $\norm{L^\top AL-I}_F$. But now this residual is a non-convex function, and is thus difficult to minimize. A variant of our method naturally maintains symmetry of the iterates.

\section{Contributions and Overview} \label{sec:contributions}

In this section we  describe the main contributions of this chapter.
\subsection{New algorithms} \label{subsec:primal}
We develop a novel and surprisingly simple  family of stochastic algorithms for inverting matrices. The problem of finding the inverse of an $n\times n$ invertible matrix $A$ can be characterized as finding the solution to either one of the two {\em inverse equations}\footnote{One may use other equations uniquely defining the inverse, such as $AXA = A$, but we do not explore these in this thesis.} $AX=I$ or $XA=I.$
Our methods make use of randomized sketching~\cite{Pilanci2014,Gower2015,Pilanci2015,Pilanci2015a} to reduce the dimension of the inverse equations in an iterative fashion. To the best of our knowledge, these are the first stochastic algorithms for inverting a matrix with global complexity rates.

In particular, our nonsymmetric  method (Algorithm~\ref{alg:asym-row}) is based on the inverse equation $AX=I$,  and performs the  {\em sketch-and-project} iteration
\begin{equation}\label{eq:98hs8s}X_{k+1} = \arg \min_{X\in \R^{n\times n} } \tfrac{1}{2}\norm{X - X_{k}}_{F(B)}^2 \quad \mbox{subject to } \quad  S^\top AX =S^\top,\end{equation}
where $S\in \R^{n\times q}$ is a random matrix drawn in an i.i.d.\ fashion from a fixed distribution $\cal{D}$, and  $B \in \R^{n\times n}$ is symmetric positive definite.
 The distribution $\cal D$ and matrix $B$ are the parameters of the method. Note that if we choose $q\ll n$, the constraint in the projection problem  \eqref{eq:98hs8s} will be of a much smaller dimension than the original inverse equation, and hence the iteration \eqref{eq:98hs8s} will become cheap.

In an analogous way, we   design a method based on the inverse equation $XA=I$ (Algorithm~\ref{alg:asym-col}).
% \begin{equation}\label{eq:98hs8s2}X_{k+1} = \arg \min_{X\in \R^{n\times n} } \tfrac{1}{2}\norm{X - X_{k}}_{F(B)}^2 \quad \mbox{subject to } \quad XA S =S.\end{equation}
 By adding the symmetry constraint $X=X^\top $ to~\eqref{eq:98hs8s}, we obtain Algorithm~\ref{alg:sym}---a specialized method for inverting symmetric matrices capable of maintaining symmetric iterates.

\subsection{Dual formulation}

  Besides the {\em primal formulation} described in Section~\ref{subsec:primal}---\emph{sketch-and-project}---we also provide  {\em dual formulations} of all three  methods (Algorithms~\ref{alg:asym-row}, \ref{alg:asym-col} and \ref{alg:sym}).  For instance, the dual formulation of \eqref{eq:98hs8s} is \begin{equation}\label{eq:98hs8sssds}X_{k+1} = \arg_X \min_{X\in \R^{n\times n}, Y\in \R^{n\times q}} \tfrac{1}{2}\norm{X_{k}-A^{-1}}_{F(B)}^2 \quad \mbox{subject to } \quad  X = X_k + B^{-1}A^\top SY^\top .\end{equation}
 We call the dual formulation  \emph{constrain-and-approximate} as one seeks to perform the best approximation of the inverse (with respect to the weighted Frobenius distance) while constraining the search to  a random affine space of matrices passing through $X_k$. While the projection \eqref{eq:98hs8sssds} cannot be performed directly since $A^{-1}$ is not known, it can be performed indirectly via the equivalent primal formulation \eqref{eq:98hs8s}.

\subsection{Quasi-Newton updates and approximate inverse preconditioning}

As we will discuss in Section~\ref{sec:QN}, through the lens of the sketch-and-project formulation, Algorithm~\ref{alg:sym} can be seen as {\em randomized block extension of the quasi-Newton updates}~\cite{Broyden1965,Fletcher1960,Goldfarb1970,Shanno1971}. 
We distinguish here between quasi-Newton methods, which are algorithms  used in optimization, and quasi-Newton updates, which are the {\em matrix-update} rules used in the quasi-Newton methods.   Standard quasi-Newton updates work with $q=1$ (``block'' refers to the choice $q>1$) and $S$ chosen in a deterministic way, depending on the sequence of iterates of the underlying optimization problem.  To the best of our knowledge, this is the first time stochastic versions of quasi-Newton updates were designed and analyzed. On the other hand,  through the lens of the constrain-and-approximate formulation, our methods can be seen as {\em new variants of the approximate inverse preconditioning (\emph{AIP}) methods}~\cite{Chow1998,Saad2003,Gould1998,Benzi1999}. Moreover, the equivalence between these two formulations reveals deep connections  between what were before seen as distinct fields: the quasi-Newton and AIP literature.  Our work also provides several new insights for {\em deterministic} quasi-Newton updates. For instance, the {\em bad Broyden update}~\cite{Broyden1965,Griewank2012} is a particular best rank-1 update that minimizes the distance to the inverse of $A$ under the Frobenius norm. 
%However, the {\em good Broyden update}~\cite{Broyden1965,Griewank2012} does not directly minimize the distance to the inverse. 
The {\em BFGS update}~\cite{Broyden1965,Fletcher1960,Goldfarb1970,Shanno1971} can be seen as a projection of $A^{-1}$ onto a space of rank-2 symmetric matrices. To the best of our knowledge, this has not been observed before.

\subsection{Complexity} Our framework leads to global linear convergence (i.e., exponential decay) under very weak assumptions on $\cal D$.  In particular, we provide an explicit convergence rate $\rho$ for the exponential decay of the norm of the expected error of the iterates (line~2 of Table~\ref{tab:complexity}) and the expected norm of the error (line~3 of Table~\ref{tab:complexity}), where $\rho$ is the same rate provided in Chapter~\ref{ch:linear_systems}, namely 
\begin{equation}\label{ch:four:eq:rho}
\rho = 1- \lambda_{\min}(B^{-1/2}\E{Z}B^{-1/2}),
\end{equation}
where 
\[ Z \eqdef A^\top S(S^\top  A B^{-1}A^\top  S)^{-1}S A^\top .\] 
%% This defeats the purpuse of introducing notation: For the definition of the $\|\cdot\|_{B}^*$ norm, see Section~\ref{subsec:notation}.

% We show that the converges rate $\rho$ is always bounded between $0$ and $1$. Furthermore, we provide a lower bound on $\rho$ that shows that the rate can potentially improve as the number of columns in $S$ increases. 
 This sets our method apart from current methods for inverting matrices that lack global guarantees, such as Newton-Schulz, or the self-conditioning variants of the minimal residual method.
\begin{table}
\centering
\begin{tabular}{|c|c|}
\hline
& \\
 $\E {X_{k+1} -A^{-1}} = \left(I  - B^{-1}\E{Z}\right) \E{X_{k+1} - A^{-1}}  $ & Theorem 4.1\\
 & \\
$\norm{\E {X_{k+1} -A^{-1}}}_{B}^* \leq \rho \; \cdot \; \norm{\E{X_{k+1} - A^{-1}}}_{B}^*
 $ & Theorem~\ref{theo:normEconv}\\
 & \\
$ \E {\norm{X_{k+1} -A^{-1} }_{F(B)}^2 } \leq \rho \;\cdot\; \E{ \norm{X_{k+1} - A^{-1}}_{F(B)}^2}$  & Theorem~\ref{theo:Enormconv}\\
& \\
 \hline
 \end{tabular}
 \caption{Our main complexity results.}
 \label{tab:complexity}
 \end{table}
 By optimizing an upper bound on~\eqref{ch:four:eq:rho}, we also obtain a new practical importance sampling. This should be contrasted with the optimized rate in Section~\ref{ch:two:sec:optprob} which results in a SDP, which is rarely practical to solve.

\subsection{Adaptive randomized BFGS}
  
We develop an additional highly efficient method---adaptive randomized BFGS (AdaRBFGS)---for calculating an approximate inverse of {\em positive definite matrices}. 
In extensive numeric tests in Section~\ref{sec:numerical} we show that the AdaRBFGS method greatly outperforms the Newton-Schulz and approximate inverse preconditioning methods at obtaining an approximate inverse (with a relative precision of $99\%$). Furthermore, the AdaRBFGS method preserves positive definiteness, a quality not present in previous methods.
%Not only does the method perform outperform methods such as Newton-Schulz and approximate inverse preconditioning methods
%
% well in our numeric experiments in Section~\ref{sec:numerical} as compared to the 
%Newton-Schulz and approximate inverse preconditioning methods,
%%greatly outperform the state-of-the-art methods such as Newton-Schulz and approximate inverse preconditioning methods,
% but it also preserves positive definiteness, a quality not present in previous methods.
  Therefore, AdaRBFGS can be used to precondition positive definite systems and to design new variable-metric optimization methods.  Since the inspiration behind this method comes from the desire to design an {\em optimal adaptive} distribution for $S$ by examining the complexity rate $\rho$, this work  also highlights the importance of developing algorithms with explicit convergence rates.

%The methods we present here, in contrast, can inherit both symmetry and positive definiteness from $A$, making them a natural choice for preconditioning and designing variable metric methods.
%   
%Lend themselves to parallelisation~\eqref{Gould1998a}
% 

\subsection{Extensions}  This work opens up many possible avenues for extensions. For instance, new efficient methods could be achieved by experimenting and analyzing through our framework with different sophisticated sketching matrices $S$, such as the Walsh-Hadamard matrix~\cite{Lu2013,Pilanci2014}. Furthermore, our method produces low rank estimates of the inverse and can be adapted to calculate low rank estimates of any matrix. Our methods can  be applied to singular matrices, in which case they converge to a particular pseudo-inverse.  
  
   Our results can be used to push forward work into stochastic variable metric methods. Such as the work by Leventhal and Lewis~\cite{Leventhal2011}, where they present a randomized iterative method for estimating Hessian matrices that converge in expectation with known convergence rates for any initial estimate. Stich et al.\ \cite{Stich2015} use Leventhal and Lewis' method to design a stochastic variable metric method for black-box minimization, with explicit convergence rates, and promising numeric results. We leave these and other extensions to future work.

\section{Randomization of Quasi-Newton Updates}\label{sec:QN}

Our methods are inspired by, and in some cases can be considered to be, randomized block variants of the quasi-Newton updates. In this section we explain how our algorithms arise naturally from the quasi-Newton setting. Readers familiar with  quasi-Newton methods may  jump ahead to Section~\ref{subsec:iuh9898}.

\subsection{Quasi-Newton methods}

A problem of fundamental interest in  optimization is the unconstrained minimization problem
\begin{equation}\label{eq:opt}\min_{x\in \R^n} f(x),\end{equation}
where $f:\R^n\to \R$ is a sufficiently smooth  function. Quasi-Newton (QN) methods, first proposed by Davidon in 1959~\cite{Davidon1959}, are an extremely powerful and popular class of algorithms for solving this problem, especially in the regime of moderately large $n$. In each iteration of a QN method, one approximates the function locally around the current iterate $x_k$ by a quadratic of the form
\begin{equation}\label{eq:98h98hff}f(x_k+s)\approx f(x_k) + (\nabla f(x_k))^\top  s + \frac{1}{2}s^\top  B_k s,\end{equation}
where $B_k$ is a suitably chosen approximation of the Hessian: $B_k\approx \nabla^2 f(x_k)$. After this, a direction $s_k$ is computed by minimizing the quadratic approximation in $s$, obtaining
\begin{equation}\label{eq:QNdirection}s_k = - B_k^{-1}\nabla f(x_k),\end{equation}
if the matrix $B_k$ is invertible. The   next iterate is then set to
\[x_{k+1} = x_k+h_k, \quad h_k =\alpha_k s_k,\]
for a suitable choice of stepsize $\alpha_k$, often chosen by a line-search procedure (i.e., by approximately minimizing $f(x_k+ \alpha s_k)$ in $\alpha$). 

Gradient descent arises as a special case of this process by choosing $B_k$ to be constant throughout the iterations. A popular choice is $B_k=L I$, where $I$ is the identity matrix and $L\in \R_+$ is the Lipschitz constant of the gradient of $f$. In such a case, the quadratic approximation \eqref{eq:98h98hff} is a global upper bound on $f(x_k+s)$, which means that $f(x_k+s_k)$ is guaranteed to be at least as good (i.e., smaller or equal) as $f(x_k)$, leading to guaranteed descent. Newton's method also arises as a special case: by choosing $B_k = \nabla^2 f(x_k)$. These two algorithms are extreme cases on the opposite end of  a spectrum. Gradient descent benefits from a trivial  update rule for $B_k$ and from cheap iterations due to the fact that no linear systems need to be solved. However, curvature information is largely ignored, which slows down the practical convergence of the method. Newton's method utilizes the full curvature information contained in the Hessian, but requires the computation of the Hessian in each step, which is expensive for large $n$.  QN methods aim to find a sweet spot on the continuum between these two extremes. In particular, the QN methods choose  $B_{k+1}$ to be a matrix for which  the {\em secant equation} is satisfied:
\begin{equation}\label{eq:secant}B_{k+1}(x_{k+1}-x_k) = \nabla f(x_{k+1})-\nabla f(x_k).\end{equation}

The basic reasoning behind this requirement is the following: if $f$ is a convex quadratic  then the Hessian matrix satisfies the secant equation for all pairs of vectors $x_{k+1}$ and $x_k$. If $f$ is not a quadratic, the reasoning is as follows. Using the fundamental theorem of calculus, we have that 
\[ \left(\int_{0}^1 \nabla^2 f(x_k + t h_k) \; dt\right) (x_{k+1}-x_k)= \nabla f(x_{k+1}) -\nabla f(x_k). \]
By selecting $B_{k+1}$ that satisfies the secant equation, we are enforcing that $B_{k+1}$ mimics the action of the integrated Hessian along the line segment joining $x_k$ and $x_{k+1}$. Unless $n=1$, the secant equation \eqref{eq:secant} does not have a unique solution in $B_{k+1}$. All QN methods differ only in which particular solution is used. The formulas transforming $B_k$ to $B_{k+1}$ are called {\em QN updates}.

Since these matrices are used to compute the direction $s_k$ via \eqref{eq:QNdirection}, it is often more reasonable to instead maintain a sequence of inverses $X_k = B_k^{-1}$. By multiplying both sides of \eqref{eq:secant} by $X_{k+1}$, we arrive at the {\em secant equation for the inverse:} \begin{equation}\label{eq:secant2}X_{k+1}(\nabla f(x_{k+1})-\nabla f(x_k)) = x_{k+1}-x_k.\end{equation} 

%To determine a unique $X_{k+1}$ that satisfies~\eqref{eq:secant2}, the most popular   classes of QN updates choose $X_{k+1}$ as the closest matrix to $X_k$. That is, choose $X_{k+1}$
%that is solution to a projection problem under a suitable norm (usually a weighted Frobenius norm with various weight matrices)
%Standard versions of QN updates lead to rank-1 or rank-2 updates to $B_k$. Therefore, the inverses can be efficiently maintained via the Sherman-Morisson-Woodbury formula.
 The most popular classes of QN updates choose $X_{k+1}$ as the closest matrix to $X_k$, in a suitable norm (usually a weighted Frobenius norm with various weight matrices), subject to the secant equation, often with an  explicit symmetry constraint:
\begin{equation}\label{eq:QN:project_form}X_{k+1} = \arg \min_{X\in \R^{n\times n}} \left\{ \|X-X_k\| \;:\;  X y_k = h_k, \; X = X^\top \right\},\end{equation}
where  $y_k = \nabla f(x_{k+1}) - \nabla f(x_k)$,  

\subsection{Quasi-Newton updates}

Consider now problem \eqref{eq:opt} with the quadratic objective
\begin{equation} \label{eq:QN_quadratic}f(x) = \frac{1}{2}x^\top  A x - b^\top  x + c,\end{equation}
where $A$ is an $n\times n$ symmetric positive definite matrix, $b\in \R^n$ and $c\in \R$. Granted, this is not a typical problem for which QN methods would be used by a practitioner. Indeed, the Hessian of $f$ does not change, and hence one {\em does not have to} track it. The problem  can simply be solved by setting the gradient to zero, which leads to the system $Ax = b$,  the  solution being $x_*=A^{-1}b$. As solving a linear system is much simpler than computing the inverse $A^{-1}$, approximately tracking the (inverse) Hessian of $f$ along the path of the iterates $\{x_k\}$---the basic strategy of all  QN methods---seems like too much effort for what is ultimately a much simpler problem.

However, and this is one of the main insights of this work, instead of viewing QN methods as optimization algorithms, we can alternatively interpret them as iterative algorithms producing a sequence of matrices, $\{B_k\}$ or $\{X_k\}$, hopefully converging to some matrix of interest. In particular, one would hope that if  a QN method is applied to the quadratic problem \eqref{eq:QN_quadratic}, with any symmetric positive definite initial guess $X_0$, then the sequence $\{X_k\}$  converges to $A^{-1}$. 

For $f$ given by \eqref{eq:QN_quadratic}, the QN updates of the minimum distance variety given by \eqref{eq:QN:project_form} take the form
\begin{equation}\label{eq:QN:project_form2}X_{k+1} = \arg \min_{X\in \R^{n\times n}} \left\{ \|X-X_k\| \;:\;  X A h_k = h_k , \; X = X^\top \right\}.\end{equation}
%where $h_k = x_{k+1}-x_k$. 

\subsection{Randomized quasi-Newton updates}\label{subsec:iuh9898}

While the motivation for our work comes from optimization, having arrived at the update \eqref{eq:QN:project_form2}, we can dispense of some of the implicit assumptions and propose and analyze a wider class of methods. In particular, in this chapter we analyze a large class of {\em randomized algorithms} of the type \eqref{eq:QN:project_form2}, where the vector $h_k$ is replaced by a random matrix $S$ and $A$ is \emph{any} invertible, and not necessarily symmetric or positive definite matrix. This constitutes a randomized block extension of the QN updates.

\section{Inverting Nonsymmetric Matrices} \label{sec:SIMI-nonsym}

In this chapter we are concerned with the development and complexity analysis of a family of stochastic algorithms for computing the inverse of a nonsingular matrix $A\in \R^{n\times n}$. The starting point in the development of our methods is the simple observation that the inverse $A^{-1}$ is the (unique) solution of a linear matrix equation, which we shall refer to as {\em inverse equation}: 
\begin{equation} \label{eq:inverse_equations} AX = I.\end{equation}
Alternatively, one can use the inverse equation $XA = I$ instead. Since \eqref{eq:inverse_equations} is difficult to solve directly, our approach is to iteratively solve a small  randomly relaxed version of \eqref{eq:inverse_equations}. That is, we choose a random matrix $S\in \R^{n \times q}$, with $q\ll n$, and instead solve  the following {\em sketched inverse equation}: \begin{equation}\label{eq:inverse_eq_sketched}S^\top  AX = S^\top .\end{equation}
If we base the method on the second inverse equation, the  sketched inverse equation  $XA S = S$ should be used instead. Note that $A^{-1}$ satisfies \eqref{eq:inverse_eq_sketched}. If $q\ll n$, the sketched inverse equation is of a much smaller dimension than the original inverse equation, and hence easier to solve. However, the equation will no longer have a unique solution and  in order to design an algorithm, we need a way of picking a particular solution. Our algorithm defines $X_{k+1}$ to be the solution that is closest to the current iterate $X_k$ in a weighted Frobenius norm. This is repeated in an iterative fashion, each time drawing $S$ independently from a fixed distribution $\cal D$. The distribution $\cal D$ and the matrix $B$ can be seen as parameters of our method. The flexibility of being able to adjust $\cal D$ and $B$ is important: by varying these parameters we obtain various specific instantiations of the generic method, with varying properties and convergence rates. This gives the practitioner the flexibility to adjust the method to the structure of $A$, to the computing environment and so on.

% As we shall see in  Section~\ref{sec:discretemethods}, for various choices of these parameters we recover stochastic block variants of several well known quasi-Newton updates.

\subsection{Projection viewpoint: sketch-and-project}
 
The next iterate $X_{k+1}$ is the nearest point to $X_k$ that satisfies a \emph{sketched} version of  the inverse equation:
\begin{align}
\boxed{ X_{k+1} = \arg \min_{X\in \R^{n\times n}} \frac{1}{2} \norm{X - X_{k}}_{F(B)}^2 \quad \mbox{subject to } \quad  S^\top AX =S^\top } \label{eq:NF}
\end{align}
  In the special case when $S=I$, the only such matrix is the inverse itself, and 
\eqref{eq:NF} is not helpful. However, if  $S$ is ``simple'', \eqref{eq:NF} will be easy to compute and the hope is that through a sequence of such steps, where the matrices $S$ are sampled in an i.i.d. fashion from some distribution, $X_k$ will converge to $A^{-1}$. 

%We take this opportunity to point out that due to the sketched equation $S^\top AX=S^\top $ the transpose satisfies $X^\top A^\top S =S.$ Now this can be viewed as a sketched form of the equation $X^\top A^\top  = I$...
 
 Alternatively, we can sketch the equation $XA=I$ and project onto $XAS=S$:
\begin{equation}\label{eq:NFcols}
\boxed{X_{k+1} = \arg \min_{X\in \R^{n\times n}} \frac{1}{2} \norm{X - X_{k}}_{F(B)}^2 \quad \mbox{subject to } \quad  XAS =S}\end{equation}

While the method~\eqref{eq:NF} sketches the rows of $A$, the method~\eqref{eq:NF} sketches the columns of $A.$ Thus we refer to~\eqref{eq:NF} as the row variant and to~\eqref{eq:NFcols} as the column variant. The two variants~\eqref{eq:NF} and~\eqref{eq:NFcols} both converge to the inverse of $A$, as will be established in Section~\ref{sec:conv}.

 If $A$ is singular, then the iterates of~\eqref{eq:NFcols} converge to the left inverse, while the iterates of~\eqref{eq:NF} converge to the right inverse, an observation  we
 leave to future work.
  %leave this observation to future work, and do not explore it further in this paper.
%

\subsection{Optimization viewpoint: constrain-and-approximate} 

The row sketch-and-project method can be cast in an apparently different yet equivalent viewpoint:
%The next iterate $X_{k+1}$ is the best approximation to $A^{-1}$ restricted to a random affine space that passes through $X_{k}:$
\begin{align}
\boxed{X_{k+1} = \arg_X  \min_{X\in \R^{n\times n}, Y\in \R^{n\times q}} \frac{1}{2}\norm{X - A^{-1}}_{F(B)}^2 \quad \mbox{subject to } \quad  X=X_{k} +B^{-1}A^\top SY^\top }\label{eq:RF}
\end{align}
In this viewpoint, at each iteration~\eqref{eq:RF}, we select a random affine space
that passes through $X_k.$ After that, we select the point in this space that is as close as possible to the inverse. This random search space is special in that, independently of the input pair $(B,S)$ we can efficiently compute
the projection of $A^{-1}$ onto this space, without knowing $A^{-1}$ explicitly. 

The column variant~\eqref{eq:NFcols} also has an equivalent constrain-and-approximate formulation:
\begin{align}
\boxed{X_{k+1} = \arg_X  \min_{X\in \R^{n\times n}, Y\in \R^{n\times q}} \frac{1}{2}\norm{X - A^{-1}}_{F(B)}^2 \quad \mbox{subject to } \quad  X=X_{k} +YS^\top A^\top B^{-1}}\label{eq:RFcols}
\end{align}

 These two variants~\eqref{eq:RF} and~\eqref{eq:RFcols} can be viewed as new variants of the approximate inverse preconditioner (AIP) methods~\cite{Benzi1999,Gould1998,Kolotilina1993,Huckle2007}.
The AIP methods are a class of methods for computing approximate inverses of $A$ by  minimizing  $\norm{XA-I}_F$ via iterative optimization algorithms. In particular, the AIP methods use variants of the steepest descent or a minimal residual method to minimize $\norm{XA-I}_F$. The idea behind the AIP methods is to minimize the distance of $X$ from $A^{-1}$ in some sense. Our variants do just that, but under a weighted Frobenius norm. 
 Furthermore, our methods project onto a randomly generated affine space instead of employing steepest descent of a minimal residual method. 
% In contrast, our variants minimizes $\norm{A^{1/2}XA^{1/2}-I}_F$, despite not requiring access to the square root $A^{1/2}/$ Furthermore, our variant 
%project onto a randomly generated affine space instead of employing steepest descent of a minimal residual method.

\subsection{Equivalence}

We now prove that~\eqref{eq:NF} and~\eqref{eq:NFcols} are equivalent to~\eqref{eq:RF} and~\eqref{eq:RFcols}, respectively, and give their explicit solution.
 \begin{theorem}\label{theo:NFRF}
The viewpoints~\eqref{eq:NF} and~\eqref{eq:RF} are equivalent to~\eqref{eq:NFcols} and~\eqref{eq:RFcols}, respectively. 
Furthermore, if $S$ has full column rank, then the explicit solution to~\eqref{eq:NF}  is
   \begin{equation}
  \boxed{ X_{k+1} = X_{k} +B^{-1}A^\top S(S^\top  A B^{-1}A^\top  S)^{-1}S^\top (I-AX_{k})} \label{eq:Xupdate}
   \end{equation}
   and the explicit solution to~\eqref{eq:NFcols} is
   \begin{equation}
   \boxed{X_{k+1} = X_{k} +(I-X_{k}A^\top )S(S^\top  A^\top  B^{-1}A S)^{-1}S^\top  A^\top B^{-1}} \label{eq:Xupdatecols}
   \end{equation}
\end{theorem}
\begin{proof}
We will prove all the claims for the row variant, that is, we prove that~\eqref{eq:NF} are~\eqref{eq:RF} equivalent and that their solution is given by~\eqref{eq:Xupdate}. The remaining claims, that~\eqref{eq:NFcols} are~\eqref{eq:RFcols} are equivalent and that their solution is given by~\eqref{eq:Xupdatecols}, follow with analogous arguments. 

It suffices to consider the case when $B=I$, as we can perform a change of variables to recover the solution for any $B$. Indeed, in view of \eqref{eq:98y988ff}, with the change of variables
\begin{equation}\label{eq:varchangeW}\hat{X}\eqdef B^{1/2}X B^{1/2}, \quad \hat{X}_k\eqdef B^{1/2}X_kB^{1/2}, \quad \hat{A} \eqdef B^{-1/2}AB^{-1/2} \quad \mbox{and} \quad \hat{S} \eqdef B^{1/2}S,
\end{equation}
\eqref{eq:NF} becomes
\begin{align}
\min_{\hat{X} \in \R^{n\times n}} \frac{1}{2} \norm{\hat{X} - \hat{X}_{k}}_{F}^2 \quad \mbox{subject to } \quad  \hat{S}^\top \hat{A}\hat{X} =\hat{S}^\top . \label{eq:NFbar}
\end{align}
Moreover, if we let $\hat{Y} = B^{1/2}Y$, then ~\eqref{eq:RF} becomes
\begin{align}
\min_{\hat{X} \in \R^{n\times n}, \hat{Y}\in \R^{n\times q}} \frac{1}{2} \norm{\hat{X} - \hat{A}^{-1}}_{F}^2 \quad \mbox{subject to } \quad  \hat{X} =\hat{X}_k + \hat{A}^\top \hat{S} \hat{Y}^\top . \label{eq:RFbar}
\end{align}

By substituting the constraint in~\eqref{eq:RFbar} into the objective function, then differentiating to find the stationary point, we obtain that
\begin{equation} \label{eq:Xupdatebar}
\hat{X} = \hat{X}_k +\hat{A}^\top \hat{S}(\hat{S}^\top \hat{A}\hat{A}^\top \hat{S})^{-1}\hat{S
}^\top (I-\hat{A}\hat{X}_k),\end{equation}
is the solution to~\eqref{eq:RFbar}. After changing the variables back using~\eqref{eq:varchangeW}, the update~\eqref{eq:Xupdatebar} becomes~\eqref{eq:XZupdate}.

Now we prove the equivalence of~\eqref{eq:NFbar} and~\eqref{eq:RFbar} using Lagrangian duality. The sketch-and-project viewpoint~\eqref{eq:NFbar} has a convex quadratic objective function with linear constraints, thus strong duality holds. 
Introducing Lagrangian multiplier $\hat{Y} \in \R^{n\times q}$, the Langrangian dual of~\eqref{eq:NFbar} is given by
\begin{equation}\label{eq:lagdual}
L(\hat{X},\hat{Y})=   \frac{1}{2}\norm{\hat{X}-\hat{X}_k}_{F}^2 - \dotprod{{\hat{Y}}^\top ,\hat{S}^\top \hat{A}(\hat{X}-\hat{A}^{-1})}_{F}.
\end{equation}
Clearly
\[\eqref{eq:NFbar} =\min_{X\in\R^{n\times n}} \max_{\hat{Y} \in\R^{n\times q}} L(\hat{X}, \hat{Y}). \]
 %problem~\eqref{eq:NFbar} is equivalent to $\min_{X\in\R^{n\times n}} \max_{\hat{Y} \in\R^{n\times q}} L(\hat{X}, \hat{Y}).$
  We will now prove that
\[\eqref{eq:RFbar} =\max_{\hat{Y} \in\R^{n\times q}} \min_{X\in\R^{n\times n}}  L(\hat{X}, \hat{Y}), \]  
thus proving that~\eqref{eq:NFbar} and~\eqref{eq:RFbar} are equivalent by strong duality. 
Differentiating the Lagrangian in $\hat{X}$ and setting to zero gives
\begin{equation}\label{eq:primopt1}
\hat{X} = \hat{X}_k +\hat{A}^\top \hat{S}{\hat{Y}}^\top .
\end{equation}
Substituting back into~\eqref{eq:lagdual} gives
\begin{align*}
L(\hat{X},\hat{Y}) &=  \frac{1}{2}\norm{\hat{A}^\top \hat{S}\hat{Y}^\top }_{F}^2 - \dotprod{\hat{A}^\top \hat{S}{\hat{Y}}^\top , \hat{X}_k +\hat{A}^\top \hat{S}{\hat{Y}}^\top  -\hat{A}^{-1}}_{F} \\
&= -  \frac{1}{2}\norm{\hat{A}^\top \hat{S}{\hat{Y}}^\top }_{F}^2 - \dotprod{\hat{A}^\top \hat{S}{\hat{Y}}^\top , \hat{X}-\hat{A}^{-1} }_{F}.
\end{align*}
Adding $\pm \frac{1}{2}\norm{\hat{X}_k -\hat{A}^{-1}}_{F}^2$ to the above gives
\[L(\hat{X},\hat{Y}) =   -\frac{1}{2}\norm{\hat{A}^\top \hat{S}{\hat{Y}}^\top +\hat{X}_k-\hat{A}^{-1}}_{F}^2+\frac{1}{2}\norm{\hat{X}_k -\hat{A}^{-1}}_{F}^2.\]
Finally,  substituting~\eqref{eq:primopt1} into the above,  minimizing in $\hat{X}$ then maximizing in $\hat{Y}$, and dispensing of the term $\frac{1}{2}\norm{\hat{X}_k -\hat{A}^{-1}}_{F}^2$ as it does not depend on $\hat{Y}$ nor $\hat{X}$, we have that the dual problem is 
\[ \max_{\hat{Y}}\min_{\hat{X}} L(\hat{X},\hat{Y}) =\min_{\hat{X},\hat{Y}} \frac{1}{2}\norm{\hat{X}-\hat{A}^{-1}}_{F}^2 \quad \mbox{subject to} \quad \hat{X} =\hat{X}_k+ \hat{A}^\top \hat{S}{\hat{Y}}^\top .\]
It now remains to change variables using~\eqref{eq:varchangeW} and set $Y = B^{-1/2}\hat{Y}$ to obtain~\eqref{eq:RF}.
\end{proof}

Based on Theorem~\ref{theo:NFRF}, we can summarize the methods described in this section as Algorithm~\ref{alg:asym-row} and Algorithm~\ref{alg:asym-col}.

\begin{algorithm}[!t]
\begin{algorithmic}[1]
\State \textbf{input:} invertible matrix $A \in \R^{n\times n}$
\State \textbf{parameters:} ${\cal D}$ = distribution over random matrices; positive definite matrix $B\in \R^{n\times n}$
\State \textbf{initialize:} arbitrary square matrix $X_0\in \R^{n\times n}$
\For {$k = 0, 1, 2, \dots$}
	\State Sample an independent copy $S\sim {\cal D}$
	\State Compute $\Lambda = S(S^\top  A B^{-1}A^\top  S)^{-1}S^\top $
    \State $X_{k+1} = X_{k} + B^{-1}A^\top \Lambda (I-AX_{k})$
    \Comment This is equivalent to \eqref{eq:NF} and \eqref{eq:RF}
\EndFor
\State \textbf{output:} last iterate $X_k$
\end{algorithmic}
\caption{Stochastic Iterative Matrix Inversion (SIMI) -- nonsymmetric row variant}
\label{alg:asym-row}
\end{algorithm}

\begin{algorithm}[!h]
\begin{algorithmic}[1]
\State \textbf{input:} invertible matrix $A \in \R^{n\times n}$
\State \textbf{parameters:} ${\cal D}$ = distribution over random matrices; positive definite matrix $B\in \R^{n\times n}$
\State \textbf{initialize:} arbitrary square matrix $X_0\in \R^{n\times n}$
\For {$k = 0, 1, 2, \dots$}
	\State Sample an independent copy $S\sim {\cal D}$
	\State Compute $\Lambda = S(S^\top  A^\top  B^{-1}A S)^{-1}S^\top $
    \State $X_{k+1} = X_{k} + (I-X_{k} A^\top ) \Lambda  A^\top  B^{-1}$
    \Comment This is equivalent to \eqref{eq:NFcols} and \eqref{eq:RFcols}
\EndFor
\State \textbf{output:} last iterate $X_k$
\end{algorithmic}
\caption{Stochastic Iterative Matrix Inversion (SIMI) -- nonsymmetric column variant}
\label{alg:asym-col}
\end{algorithm}

The explicit formulas~\eqref{eq:Xupdate} and~\eqref{eq:Xupdatecols} for \eqref{eq:NF} and~\eqref{eq:NFcols} allow us to efficiently implement these methods, and  facilitate  convergence analysis. In particular, we can now see that the convergence analysis of~\eqref{eq:Xupdatecols} will follow trivially from analyzing~\eqref{eq:Xupdate}. 
 This is because~\eqref{eq:Xupdate} and~\eqref{eq:Xupdatecols} differ only in terms of a transposition.
That is, transposing~\eqref{eq:Xupdatecols} gives
\[ X_{k+1}^\top  = X_{k}^\top  +B^{-1}AS(S^\top  A^\top  B^{-1}A S)^{-1}S^\top (I-A^\top X_{k}^\top ),\]
which is the solution to the row variant of the sketch-and-project viewpoint but where the equation $A^\top X^\top  = I$ is sketched instead of $AX=I.$ Thus, since the weighted Frobenius norm is invariant under transposition, it suffices to study the convergence of~\eqref{eq:Xupdate}, then the convergence of~\eqref{eq:Xupdatecols} follows by simply swapping the role of $A$ for $A^\top .$ We collect this observation is the following remark.
\begin{remark}\label{rem:alg2conv}
The expression for the rate of convergence of Algorithm~\ref{alg:asym-col}  is the same as the expression for the rate of convergence of Algorithm~\ref{alg:asym-row}, but with every occurrence of $A$ swapped for $A^\top .$ 
\end{remark}

%
%\begin{equation}\label{eq:NFtrans}
%X_{k+1}^\top  = \arg \min_{X \in \R^{n\times n}} \frac{1}{2} \norm{X - X_{k}^\top }_{F(B)}^2 \quad \mbox{subject to } \quad  S^\top A^\top X =S^\top .\end{equation}
%
%swap the occurrences of $A^\top $ for $A$ and vice-verse, we arrive
%
%
%The convergence of both~\eqref{eq:NF} and~\eqref{eq:NFcols} can be understood by analysing only one of them, for instance, by analysing the row sketch-and-project method.
%The reason for this is that~\eqref{eq:NF} and~\eqref{eq:NFcols} differ only in terms of a transposition. That is, the transpose of $XAS =S$ is $S^\top A^\top X^\top =S^\top .$ 
%Furthermore,  as the weight Frobenius norm is invariant under transposition, 
%thus column sketching method is equivalent to 
%\begin{equation}\label{eq:NFtrans}
%X_{k+1}^\top  = \arg \min_{X \in \R^{n\times n}} \frac{1}{2} \norm{X - X_{k}^\top }_{F(B)}^2 \quad \mbox{subject to } \quad  S^\top A^\top X =S^\top .\end{equation}
%The convergence of $X_{k+1}^\top $ to $A^{-T}$ can be understood by merely swapping the role of $A$ for $A^\top $ in~\eqref{eq:NF}. Therefore the analysis of~\eqref{eq:NFtrans} follows trivially from the analysis of~\eqref{eq:NF}, 

\subsection{Relation to multiple linear systems}
\label{sec:relationlinear}
Any iterative method for solving linear systems can be applied to the $n$ linear systems that define the inverse through $AX=I$ to obtain an approximate inverse. Though   
not all methods for solving linear systems can be applied to solve these $n$ linear systems simultaneously, that is calculating each column of $X$ simultaneously, which is necessary for an efficient matrix inversion method. 

The sketch-and-project methods we described in Chapters~\ref{ch:linear_systems} and~\ref{ch:SDA} can be easily and efficiently generalized to inverting a matrix, and the resulting method is equivalent to our row variant method~\eqref{eq:NF} and~\eqref{eq:RF}.  To show this, we perform the change of variables  $\hat{X}_k= X_kB^{1/2},$ 
$\hat{A} =B^{-1/2} A$ and $\hat{S} = B^{1/2}S$ then~\eqref{eq:NF} becomes
\[\hat{X}_{k+1} \eqdef X_{k+1}B^{1/2} = \arg \min_{\hat{X} \in \R^{n\times n}} \frac{1}{2} \norm{B^{1/2}(\hat{X} - \hat{X}_{k})}_{F}^2 \quad \mbox{subject to } \quad   \hat{S}^\top \hat{A}\hat{X}=\hat{S}^\top .\]
The above is a separable problem and each column of $\hat{X}_{k+1}$ can be calculated separately. 
Let $\hat{x}_{k+1}^i$ be the $i$th column of $\hat{X}_{k+1}$ which can be calculated through
\[\hat{x}_{k+1}^i = \arg \min_{\hat{x} \in \R^n}\frac{1}{2} \norm{B^{1/2}(\hat{x} - \hat{x}_{k}^i)}_{2}^2 \quad \mbox{subject to } \quad  \hat{S}^\top  \hat{A}\hat{x} =\hat{S}^\top e_i.
\]
The above is exactly an iteration of the sketch-and-project method~\eqref{ch:two:NF}
 applied to the system $\hat{A}\hat{x} = e_i.$ Thus the convergence results established in~\cite{Gower2015} carry over to our row variant~\eqref{eq:NF} and~\eqref{eq:RF}. In particular, the theory in~\cite{Gower2015} proves that the expected norm difference of each column of $B^{1/2}X_k$ converges to $B^{1/2}A^{-1}$ with rate $\rho$ as defined in~\eqref{ch:four:eq:rho}.
This equivalence breaks down when we impose additional matrix properties through constraints, such as symmetry.

% Thus we need to prove convergence of~\eqref{eq:XZupdatesym}. Further differences arise between designing iterate methods for solving a positive definite linear system and inverting a positive definite matrix, as in the latter case we can use self-preconditioning strategies, as we illustrate later on in Section~\eqref{sec:AdaRBFGS}.

\section{Inverting Symmetric Matrices} \label{sec:SIMI-sym}

When $A$ is symmetric, it may be useful to maintain symmetry in the iterates, in which case the nonsymmetric methods---Algorithms~\ref{alg:asym-row} and \ref{alg:asym-col}---have an issue, as they do not guarantee that the iterates are symmetric. However, we can modify~\eqref{eq:NF} by adding a symmetry constraint. The resulting \emph{symmetric} method naturally maintains symmetry in the iterates. 
 
\subsection{Projection viewpoint: sketch-and-project}

The new iterate $X_{k+1}$ is the result of projecting $X_k$ onto the space of matrices that satisfy a sketched inverse equation and that are also symmetric, that is
\begin{align}
\boxed{ X_{k+1} = \arg \min_{X\in \R^{n\times n}}\frac{1}{2} \norm{X - X_{k}}_{F(B)}^2 \quad \mbox{subject to }\quad   S^\top AX =S^\top ,  \quad X = X^\top } \label{eq:NFsym}
\end{align}
 See Figure~\ref{fig:proj} for an illustration of the symmetric update~\eqref{eq:NFsym}. 

 This viewpoint can be seen as a randomized block version of the quasi-Newton methods~\cite{Goldfarb1970,Greenstadt1969}, as detailed in Section~\ref{sec:QN}.
%Though in the quasi-Newton literature, instead of the sketched inverse equation $S^\top AX = S^\top ,$ the secant equation $XA\delta = \delta$ where $\delta \in \R^n$ is used. 
%But this is not unlike our sketched equation, since transposing $S^\top AX = S^\top $ we get $XAS =S,$ where we used the symmetry of $A$ and $X.$ Furthermore, as the objective function in~\eqref{eq:NFsym} is invariant under transposition, replacing $S^\top AX = S^\top $ for $XAS =S$ results in the same method. Thus~\eqref{eq:NFsym} is a block version of the quasi-Newton methods, with $XAS=S$ in place of the secant equation.  
 The flexibility in using a weighted norm is important for choosing a norm that better reflects the geometry of the problem. For instance, when $A$ is symmetric positive definite, it turns out that $B =A$ results in a good method.  This added freedom of choosing an appropriate weighting matrix has proven very useful in the quasi-Newton literature, in particular, the highly successful BFGS method~\cite{Broyden1965,Fletcher1960,Goldfarb1970,Shanno1971} selects $B$ as an estimate of the Hessian matrix.

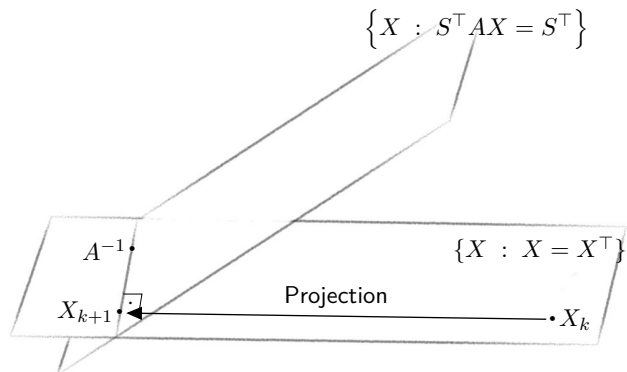
\begin{figure}
% {wrapfigure}{r}{0.30\textwidth}
\centering
\resizebox{0.6\textwidth}{!}{
\begin{tikzpicture}[>=triangle 45,font=\sffamily, ] 
   \node (planes)  at (0,0) {\includegraphics[width =10cm]{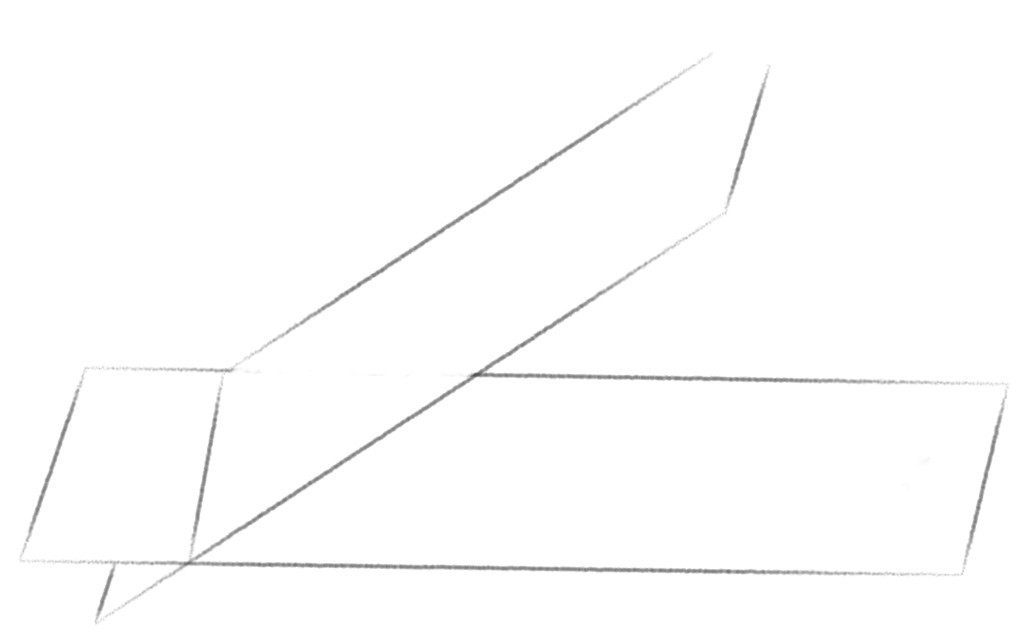}};
   \node (sym) at (3.5cm,-1cm) {$\{X \; : \; X= X^\top \}$};   
   \nodepoint{Prevp}{3.7cm}{-2.1cm};
   \node[inner sep =0mm, outer sep =0mm, right = 0pt of  Prevp] (Prev) {$X_{k}$};
   \node (action) at (2.5cm,2.5cm) {$\left\{ X \; : \; S^\top AX= S^\top \right\}$};
   \nodepoint{Nextp}{-3.1cm}{-2cm};
   \node[inner sep =0mm, outer sep =0mm, left = 0pt of  Nextp] (Next) {$X_{k+1}$}; 
   \nodepoint{Invp}{-2.9cm}{-1cm};
   \node[inner sep =0mm, outer sep =0mm, left = 0pt of  Invp] (Inv) {$A^{-1}$};  
    \draw [semithick,->] (Prevp) -- (Nextp)  node [midway,above=0.0cm] {Projection};
    \draw (Nextp)++(0.05,0.3cm) --  ++ (0.3cm,0.0cm) -- ++ (-0.05cm, -0.3cm);
    \node[inner sep =0mm, outer sep =0mm] (orthodot) at (-3.1cm+0.17cm,-2cm+0.15cm) {$\cdot$};
\end{tikzpicture}}
\caption{\footnotesize The new estimate $X_{k+1}$ is obtained by projecting $X_{k}$ onto the affine space formed by intersecting $\{X \; : \; X= X^\top \}$ and $\left\{ X \; :\; S^\top AX= S^\top \right\}$.}
\label{fig:proj}
\end{figure}

\subsection{Optimization viewpoint: constrain-and-approximate} 
The viewpoint~\eqref{eq:NFsym} also has an interesting dual viewpoint:
\begin{equation}
\boxed{X_{k+1} = \arg_{X}  \min_{X \R^{n\times n}, Y\in \R^{n\times q}} \frac{1}{2}\norm{X -A^{-1}}_{F(B)}^2 \quad \mbox{subject to} \quad
X = X_k + \frac{1}{2}(YS^\top AB^{-1} + B^{-1}A^\top SY^\top )} \label{eq:RFsym}
\end{equation}
The minimum is taken over matrices $X\in \R^{n\times n}$ and $Y\in \R^{n\times q}$. The next iterate $X_{k+1}$ is the best approximation to $A^{-1}$ restricted to a random affine space of symmetric matrices.
Furthermore, \eqref{eq:RFsym} is a symmetric equivalent of~\eqref{eq:RF}; that is, the constraint in~\eqref{eq:RFsym} is the result of projecting the constraint in~\eqref{eq:RF} onto the space of symmetric matrices. 

When $A$ is symmetric positive definite and we choose $B =A$ in~\eqref{eq:RF} and~\eqref{eq:RFcols}, then  
\[\norm{X - A^{-1}}_{F(A)}^2 = \Tr{ (X-A^{-1})A(X-A^{-1})A} = \norm{XA - I}_{F}^2.\] 
The above is exactly the objective function used in most  approximate inverse preconditioners (AIP)~\cite{Benzi1999,Gould1998,Kolotilina1993,Huckle2007}. 

\subsection{Equivalence}

We now prove that the two viewpoints~\eqref{eq:NFsym} and~\eqref{eq:RFsym} are equivalent, and show their explicit solution.
 \begin{theorem}
If $A$ and $X_k$ are symmetric, then the viewpoints~\eqref{eq:NFsym} and~\eqref{eq:RFsym} are equivalent. That is, they define the same $X_{k+1}$. Furthermore, if $S$ has full column rank, then the explicit solution to ~\eqref{eq:NFsym} and~\eqref{eq:RFsym} is
%   \begin{align}
%X_{k+1} &=X_{k} -(X_{k}AS -S)(S^\top  A B^{-1}A S)^{-1}S^\top  A B^{-1}  \nonumber \\
%& + B^{-1}AS(S^\top  A B^{-1}A S)^{-1} (S^\top AX_{k} -S^\top )\left(AS(S^\top  A B^{-1}A S)^{-1}S^\top  A B^{-1}- I \right) \label{eq:Xupdatesym}
%\end{align}
\begin{align}&\boxed{X_{k+1} =X_{k}-(X_{k}AS-S)\Lambda S^\top  A B^{-1} +B^{-1}AS\Lambda(S^\top AX_{k}-S^\top )\left(AS\Lambda S^\top AB^{-1}-I\right)} \label{eq:Xupdatesym}
 \end{align}
 where $\Lambda \eqdef (S^\top  AB^{-1}A S)^{-1}$.   
\end{theorem}
\begin{proof}
 We first prove the equivalence of~\eqref{eq:NFsym} and~\eqref{eq:RFsym} using Lagrangian duality.  It suffices to prove the claim for  $B=I$ as we did in the proof of Theorem~\ref{theo:NFRF}, since using the change of variables~\eqref{eq:varchangeW} applied to~\eqref{eq:NFsym} we have that~\eqref{eq:NFsym} is equivalent to
\begin{equation}
\min_{\hat{X} \in \R^{n\times n}} \frac{1}{2} \norm{\hat{X} - \hat{X}_{k}}_{F}^2 \quad \mbox{subject to } \quad  \hat{S}^\top \hat{A}\hat{X} =\hat{S}^\top , \quad \hat{X}= \hat{X}^\top . \label{eq:NFsymbar}
\end{equation}
Since~\eqref{eq:NFsym} has a convex quadratic objective with linear constraints, strong duality holds. Thus we will derive a dual formulation for~\eqref{eq:NFsymbar} then use the change of coordinates~\eqref{eq:varchangeW} to recover the solution to~\eqref{eq:NFsym}. 
Let $\hat{Y} \in \R^{n\times q}$ and $W \in \R^{n\times n}$ and consider the Lagrangian of~\eqref{eq:NFsymbar} which is
\begin{equation}\label{eq:lagsym}
L(\hat{X},\hat{Y}, W)= \frac{1}{2}\norm{\hat{X}-\hat{X}_k}_F^2 - \dotprod{\hat{Y}^\top ,\hat{S}^\top \hat{A}(\hat{X}-\hat{A}^{-1})}_{F} - \dotprod{W,\hat{X}-\hat{X}^\top }_F.
\end{equation}
 Differentiating in $\hat{X}$ and setting to zero gives
\begin{equation}\label{eq:optsym}
\hat{X} = \hat{X}_k+ \hat{A}^\top \hat{S}\hat{Y}^\top  +W-W^\top .
\end{equation}
Applying the symmetry constraint $X=X^\top $ gives
 \[W-W^\top  = \frac{1}{2}\left(\hat{Y}\hat{S}^\top \hat{A} - \hat{A}^\top \hat{S}\hat{Y}^\top \right).\]
 Substituting the above into~\eqref{eq:optsym} gives
\begin{equation}\label{eq:optsym2}
\hat{X} = \hat{X}_k +\frac{1}{2}\left(\hat{Y}\hat{S}^\top \hat{A} + \hat{A}^\top \hat{S}\hat{Y}^\top  \right).
\end{equation}
Now let $\Theta = \frac{1}{2}(\hat{Y}\hat{S}^\top \hat{A} + \hat{A}^\top \hat{S}\hat{Y}^\top )$ and note that, since the matrix $\Theta+\hat{X}_k -\hat{A}^{-1}$ is symmetric, we get
\begin{equation}\label{eq:symmob}
 \dotprod{\hat{A}^\top \hat{S}\hat{Y}^\top ,\Theta+\hat{X}_k -\hat{A}^{-1}}_F =  \dotprod{\Theta ,\Theta+\hat{X}_k -\hat{A}^{-1}}_F.  \end{equation}
Substituting~\eqref{eq:optsym2}  into~\eqref{eq:lagsym} gives
\begin{align}
L(\hat{X},\hat{Y},W) &= \frac{1}{2}\norm{\Theta}_F^2 - \dotprod{\hat{A}^\top \hat{S}\hat{Y}^\top ,\Theta+\hat{X}_k -\hat{A}^{-1}}_F \nonumber 
 \overset{\eqref{eq:symmob}}{=} \frac{1}{2}\norm{\Theta}_F^2 - \dotprod{\Theta,\Theta+\hat{X}_k -\hat{A}^{-1}}_F \nonumber \\
& = -\frac{1}{2}\norm{\Theta}_F^2 -\dotprod{\Theta,\hat{X}_k -\hat{A}^{-1}}_F . \label{eq:dualsym2}
 \end{align}
Adding $\pm\frac{1}{2}\norm{\hat{X}_k-\hat{A}^{-1}}_F^2$ to~\eqref{eq:dualsym2} gives
\[ L(\hat{X},\hat{Y},W) = -\frac{1}{2}\norm{\Theta +\hat{X}_k -\hat{A}^{-1}}_F^2+\frac{1}{2}\norm{\hat{X}_k-\hat{A}^{-1}}_F^2.  \]
Finally, using~\eqref{eq:optsym2} and maximizing over $\hat{Y}$ then minimizing over $X$ gives the dual problem
\[\min_{\hat{X},\hat{Y}} \frac{1}{2}\norm{\hat{X} -\hat{A}^{-1}}_F^2 \quad \mbox{subject to} \quad
\hat{X} = \hat{X}_k + \frac{1}{2}(\hat{Y}	\hat{S}^\top \hat{A} + \hat{A}^\top \hat{S}\hat{Y}^\top ). \]
It now remains to change variables according to~\eqref{eq:varchangeW} and set $Y = B^{-1/2}\hat{Y}.$

It was recently shown in~\cite[Section~2]{Gower2014c} and~\cite[Section~4]{Hennig2015}\footnote{To re-interpret methods for solving linear systems through Bayesian inference, Hennig constructs estimates of the inverse system matrix using the sampled action of a matrix taken during a linear solve~\cite{Hennig2015}.}  that~\eqref{eq:Xupdatesym} is the solution to~\eqref{eq:NFsym}.  But for completion, we now give a new simple proof.
% that~\eqref{eq:Xupdatesym} is the solution to~\eqref{eq:NFsym}.

From~\eqref{eq:optsym2} we see that the solution is solely determined by $\hat{Y}\hat{S}^\top \hat{A}$, and thus we focus on obtaining this matrix. To simplify notation, let $\Gamma =\hat{A}^\top  \hat{S}$ and let $\hat{Z} =\Gamma(\Gamma^\top  \Gamma)^{\dagger}\Gamma^\top .$ As $\hat{Z}$ is a projection matrix we have that $\hat{Z}^2 = \hat{Z}$ and $(I-\hat{Z})\hat{Z}=0,$ two properties we will use repeatedly.

 Using the sketch constraint in~\eqref{eq:NFsymbar} we have
\begin{equation} \label{eq:Gammasketch}
\Gamma^\top  \hat{X} =\hat{S}^\top \hat{A}\hat{X}=\hat{S}^\top  = \Gamma^\top  \hat{A}^{-1},\end{equation} 
therefore left multiplying~\eqref{eq:optsym2} by $\Gamma^\top $ gives
\begin{equation} \label{eq:sketchdualconst}
\Gamma^\top \hat{X} = \Gamma^\top \hat{X}_k +\frac{1}{2}\Gamma^\top \left(\hat{Y}\Gamma^\top  + \Gamma\hat{Y}^\top  \right) \overset{\eqref{eq:Gammasketch}}{=} \Gamma^\top  \hat{A}^{-1}.
\end{equation}
Let $R = \hat{A}^{-1}-\hat{X}_k$ which is a symmetric matrix. Rearranging~\eqref{eq:sketchdualconst} gives 
\[  (\Gamma^\top \Gamma)\hat{Y}^\top  =\Gamma^\top \left( 2R -\hat{Y}\Gamma^\top \right). \]
The least norm solution of the above in term of $\hat{Y}^\top $ is given by
\[\hat{Y}^\top   =(\Gamma^\top \Gamma)^{\dagger}\Gamma^\top \left( 2R -\hat{Y}\Gamma^\top \right).\]
Left multiplying the above by $\Gamma$ gives
\begin{equation}\label{eq:LDproj}
\Gamma\hat{Y}^\top   = \Gamma(\Gamma^T\Gamma)^{\dagger}\Gamma^T\left( 2R -\hat{Y}\Gamma^\top \right) =\hat{Z}\left( 2R -\hat{Y}\Gamma^\top \right).
\end{equation}
This shows that $\Gamma\hat{Y}^\top $ is equal to a projection matrix times an unknown matrix that is
\begin{equation}\label{eq:PsiZ}\Gamma\hat{Y}^\top  = \hat{Z}\Psi = \hat{Z}\Psi \hat{Z} +\hat{Z}\Psi (I-\hat{Z}), \end{equation}
where $\Psi \in \R^{n \times n}$ is  the unknown matrix. Note that in~\eqref{eq:PsiZ} we have decomposed the rows of $\hat{Z}\Psi$ into orthogonal components.
Substituting~\eqref{eq:PsiZ} into~\eqref{eq:LDproj} gives
\begin{equation} \label{eq:ZPsiZ}
\hat{Z}\Psi \hat{Z} +\hat{Z}\Psi (I-\hat{Z})  =\hat{Z}\left( 2R -\hat{Z}\Psi^\top  \hat{Z} \right),\end{equation}
where we used that $\hat{Z}(I-\hat{Z})=0.$
Right multiplying~\eqref{eq:ZPsiZ} by $\hat{Z}$ and re-arranging gives
\begin{equation}\label{eq:lefthatZ}
\hat{Z}(\Psi +\Psi^\top )\hat{Z} = 2\hat{Z}R\hat{Z}. 
\end{equation}
Right multiplying~\eqref{eq:ZPsiZ} by $I-\hat{Z}$  and re-arranging gives
\begin{equation}\label{eq:lefthatIZ} \hat{Z}\Psi (I-\hat{Z}) =2\hat{Z}R(I-\hat{Z}).
\end{equation}
Finally, inserting~\eqref{eq:PsiZ} into~\eqref{eq:optsym2} gives
\begin{eqnarray*}
\hat{X} &= & \hat{X}_k +\frac{1}{2}\left(\hat{Z}(\Psi+\Psi^\top ) \hat{Z} +(I-\hat{Z})\Psi^\top \hat{Z}  +\hat{Z}\Psi (I-\hat{Z}) \right)\\
&\overset{\eqref{eq:lefthatZ}+\eqref{eq:lefthatIZ}}{=}&\hat{X}_k 
+ \hat{Z}R\hat{Z} + \hat{Z}R(I-\hat{Z})+(I-\hat{Z})R\hat{Z}\\
&=&  \hat{X}_k+ R\hat{Z} +\hat{Z}R(I-\hat{Z})\\
&=& \hat{X}_k+ (\hat{S}-\hat{X}_k\hat{A}\hat{S})\Lambda \hat{S}^\top \hat{A} +\hat{A}\hat{S}\Lambda (\hat{S}^\top -\hat{S}^\top \hat{A}\hat{X}_k)(I-\hat{A}\hat{S}\Lambda\hat{S}^\top \hat{A}),
\end{eqnarray*} 
where we used that $\Lambda =(\hat{S}^\top \hat{A}\hat{A}^\top \hat{S})^{\dagger} = (S^\top AB^{-1}A^\top S)^{-1}.$ It now remains to use the change of variables~\eqref{eq:varchangeW} to obtain~\eqref{eq:Xupdatesym}.
\end{proof}

%\begin{algorithm}[H]\DontPrintSemicolon
%\KwIn{  ${\cal D}$ = distribution over random matrices, $0 \prec W \in \R^{n\times n},$  symmetric  $X_0 \in \R^{n\times n}$, $\epsilon >0$}
%{\bf Initialization} $ k=0$\;
%\While{$\norm{AX_{k}-I}_F^2 > \epsilon \norm{AX_0-I}_F^2$}{
%	 Sample an independent copy $S\sim {\cal D}$\;
%	 Compute $\Lambda = (S^\top  AB^{-1}A S)^{-1}$\;
%	 $
%	 X_{k+1} =X_{k}~-~(X_{k}AS-S)\Lambda S^\top  A B^{-1}  $ 
% $+~B^{-1}AS\Lambda(S^\top AX_{k}-~S^\top )~\left(AS\Lambda S^\top AB^{-1}-I\right)
%     $ \label{ln:Xupdatesym}
%$ k = k+1$
%}
%\KwOut{$X_k$}
%\caption{The Symmetric Method}
%\label{alg:sym}
%\end{algorithm}

\begin{algorithm}[!h]
\begin{algorithmic}[1]
\State \textbf{input:} symmetric invertible matrix $A \in \R^{n\times n}$
\State \textbf{parameters:} ${\cal D}$ = distribution over random matrices; symmetric positive definite  $B\in \R^{n\times n}$
\State \textbf{initialize:} symmetric matrix $X_0\in \R^{n\times n}$
\For {$k = 0, 1, 2, \dots$}
	\State Sample an independent copy $S\sim {\cal D}$
	\State Compute $\Lambda = S(S^\top  AB^{-1}A S)^{-1}S^\top $
	\State Compute $\Theta = \Lambda A B^{-1}$	
	\State Compute $M_k = X_k A - I$
    \State $X_{k+1} =X_{k}- M_k \Theta  - (M_k \Theta)^\top  + \Theta^\top  (AX_{k}A-A) \Theta$
    \Comment This is equivalent to \eqref{eq:NFsym} \& \eqref{eq:RFsym}    
\EndFor
\State \textbf{output:} last iterate $X_k$
\end{algorithmic}
\caption{Stochastic Iterative Matrix Inversion (SIMI) -- symmetric variant}
\label{alg:sym}
\end{algorithm}

\section{Convergence} \label{sec:conv}
We now analyze the convergence of the \emph{error}, $X_{k}-A^{-1}$, for iterates of Algorithms~\ref{alg:asym-row}, \ref{alg:asym-col}   and~\ref{alg:sym}. For the sake of economy of space, we only analyze Algorithms~\ref{alg:asym-row} and \ref{alg:sym}. Convergence of  Algorithm~\ref{alg:asym-col}  follows from convergence of Algorithm~\ref{alg:asym-row} by observing Remark~\ref{rem:alg2conv}.

 The first analysis we present in Section~\ref{sec:normEconv} is concerned with the convergence of $\norm{\E{X_{k}-A^{-1}}}^2,$ that is, the {\em norm of the expected error}.   We then analyze the convergence of $\E{\norm{X_{k}-A^{-1}}}^2,$ the {\em expected norm of the error}. The latter is a stronger type of convergence,  as explained in Lemma~\ref{lem:convrandvar}.

The convergence of  Algorithms~\ref{alg:asym-row} and \ref{alg:sym} can  be entirely characterized by studying the following random matrix
   \begin{equation}\label{eq:Z}
   Z \eqdef A^\top S(S^\top  AB^{-1}A^\top  S)^{-1}S^\top  A.
   \end{equation}
With this definition, the update step of Algorithm~\ref{alg:asym-row} can be re-written as a simple fixed point formula
   \begin{align}
   X_{k+1} -A^{-1}&=   \left(I-B^{-1}Z\right)(X_{k}-A^{-1}). \label{eq:XZupdate}
   \end{align}
We can also simplify the iterates of Algorithm~\ref{alg:sym} to 
   \begin{align}
X_{k+1}-A^{-1} &=  \left(I-B^{-1}Z\right) (X_{k}-A^{-1})\left(I -ZB^{-1} \right).
 \label{eq:XZupdatesym}
\end{align}

Much like our convergence proofs in Section~\ref{C2sec:convergence}, the
only stochastic component in our methods is contained in the matrix $Z$, and thus the convergence of the iterates will depend on the properties of $Z$ and its expected value $\E{Z}.$ In particular, recall from Lemma~\ref{ch:one:lem:Z} that $B^{-1}ZB^{-1}$ is an orthogonal projection.

%
%
%\begin{lemma}\label{ch:one:lem:Z} If $Z$ is defined as in~\eqref{eq:Z}, then
%\begin{enumerate}
%\item the eigenvalues of $B^{-1/2}ZB^{-1/2}$ are either $0$ or $1$,
%\item  the matrix $B^{-1/2}ZB^{-1/2}$ projects onto the $q$--dimensional subspace $\myRange{B^{-1/2}A^\top S}$. 
%\end{enumerate}
%\end{lemma}
%\begin{proof} Using~\eqref{eq:Z}, simply verify that $(B^{-1/2}ZB^{-1/2})^2 = B^{-1/2}ZB^{-1/2}$ proves that it is a projection matrix, and thus has eigenvalues $0$ or $1$. Furthermore, the matrix $B^{-1/2}ZB^{-1/2}$ projects onto $\myRange{B^{-1/2}A^\top S}$, which follows by verifying 
%\[B^{-1/2}ZB^{-1/2} (B^{-1/2}A^\top S) = B^{-1/2}A^\top S \quad \mbox{and} \quad B^{-1/2}ZB^{-1/2} y =0, \quad \forall y \in \Null{B^{-1/2}A^\top S}. \]
%Finally $\dim\left(\myRange{B^{-1/2}A^\top S}\right) =\Rank{B^{-1/2}A^\top S} = \Rank{S} =q.$  
%\end{proof}

\subsection{Norm of the expected error} \label{sec:normEconv}

We start by proving that the norm of the expected error of the iterates of Algorithm~\ref{alg:asym-row} and Algorithm~\ref{alg:sym} converges to zero. The following theorem is remarkable in that we do not need to make any assumptions on the distribution $S$, except that $S$ has full column rank. Rather, the theorem pinpoints that convergence  depends solely on the spectrum of $I-B^{1/2}\E{Z}B^{1/2}.$  
%\peter{ Thereo. 3.1 the iterates produced by (2) or (4) has the following property. Current language sounds like I'm proving. Equation (6) is a theorem, consequence. Need to reference the original.}

\begin{theorem} \label{theo:normEconv}
Let $S$ be a random matrix which has full column rank with probability~$1$ (so that $Z$ is well defined). Then the iterates $X_{k+1}$ of Algorithm~\ref{alg:asym-row} satisfy
\begin{equation} \label{eq:XXXX}
\E{ X_{k+1} -A^{-1} } =(I-B^{-1}\E{Z}) \E{X_{k} - A^{-1}}.
 \end{equation}
Let $X_0 \in \R^{n\times n}$. If $X_k$ is calculated in either one of these two ways
 \begin{enumerate}
 \item Applying $k$ iterations of  Algorithm~\ref{alg:asym-row}, 
 %or Algorithm~\ref{alg:asym-col}, 
 \item   Applying $k$ iterations of Algorithm~\ref{alg:sym} (assuming $A$ and $X_0$ are  symmetric),
\end{enumerate}
then $X_k$ converges to the inverse exponentially fast, according to
\begin{equation}\label{eq:normexpconv}
\norm{\E{X_{k}-A^{-1}}}_{B}^* \leq \rho^k \norm{X_{0}-A^{-1}}_{B}^*, 
%\norm{B^{1/2}\E{X_{k}-A^{-1}}B^{1/2}}_{2} \leq \rho^k \norm{B^{1/2}(X_{0}-A^{-1})B^{1/2}}_{2}, 
\end{equation}
where 
\begin{equation} \label{eq:rhoequiv}\rho \eqdef 1-\lambda_{\min}(B^{-1/2}\E{Z}B^{-1/2}).\end{equation}
Moreover, we have the following lower and upper bounds on the convergence rate:
\begin{equation}\label{eq:rholower}
0 \leq 1-\frac{\E{q}}{n} \leq \rho \leq 1.
\end{equation}
\end{theorem}
\begin{proof}
Let
\begin{equation}\label{eq:oihsoi8dhyY8J} R_k \eqdef B^{1/2}R_k B^{1/2}\quad and \quad  \hat{Z} \eqdef B^{-1/2}ZB^{-1/2},\end{equation}
for all $k$. Thus $\hat{Z}$ is a projection matrix (see Lemma~\ref{ch:one:lem:Z}) and
$\norm{R_k}_2 = \norm{X_{k}-A^{-1}}_B.$ Left and right multiplying~\eqref{eq:XZupdate} by $B^{1/2}$ gives
\begin{equation}\label{eq:barRknext}
R_{k+1} = (I -\hat{Z})R_k.
\end{equation}
Taking expectation with respect to $S$ in~\eqref{eq:barRknext} gives
\begin{equation}\label{eq:EXinXk} \E{R_{k+1} \;| \; R_k} = (I - \E{\hat{Z}}) R_k. 
\end{equation}
Taking full expectation in~\eqref{eq:barRknext} and using the tower rule gives
\begin{eqnarray}
 \E{ R _{k+1}} &=& \E{\E{ R _{k+1} \;|\; R_{k}}} \nonumber \\
 &\overset{\eqref{eq:EXinXk}}{=}& \E{(I -\E{\hat{Z}})R_{k}} \nonumber\\
 &=& (I -\E{\hat{Z}})\E{R_{k}}. \label{eq:Efullasym}
\end{eqnarray}
Applying the norm in~\eqref{eq:Efullasym} gives
\begin{align} \label{eq:normErecur}
 \norm{\E{ R_{k+1}}}_{2} &\leq \norm{I -\E{\hat{Z}}}_2 \norm{\E{ R_{k}}}_{2}.
\end{align}
Furthermore 
\begin{align}
\norm{I -\E{\hat{Z}}}_2 &= \lambda_{\max}\left(I -\E{\hat{Z}}\right) \nonumber \\
 &=1-\lambda_{\min}(\E{\hat{Z}}) \overset{\eqref{eq:rhoequiv}}{=} \rho, \label{eq:rhonorm}
\end{align}
where we used to symmetry of $(I-\E{\hat{Z}})$ when passing from the operator norm to the spectral radius. Note that the symmetry of $\E{\hat{Z}}$ derives from the symmetry of $\hat{Z}$.
It now remains to unroll the recurrence in~\eqref{eq:normErecur} to get~\eqref{eq:normexpconv}.

Now we analyze the iterates of  Algorithm~\ref{alg:sym}. Left and right multiplying~\eqref{eq:XZupdatesym} by $B^{1/2}$  we have 
\begin{equation}\label{eq:barRevol}
R_{k+1}= P(R_k)\eqdef \left(I-\hat{Z}\right) R_k\left(I -\hat{Z} \right). 
\end{equation}
Defining 
$\bar{P}: R \mapsto \E{P(R) \, | \, R_k}$, 
taking expectation in \eqref{eq:barRevol} conditioned on $R_{k}$, gives
\[ \E{R_{k+1}\; | \; R_k} =  \bar{P}(R_{k}).\]
As $\bar{P}$ is a linear operator, taking expectation again yields
\begin{equation} \label{eq:barPZERk}
 \E{R_{k+1}} = \E{\bar{P}(R_{k} )} =  \bar{P}(\E{R_{k} }). 
 \end{equation}

%\rob{**** Is this really clear? This is clearer when vectorizing ***************\\
% Let $Y = I-B^{-1/2}ZB^{-1/2},$ note that
%\[\E{\mbox{vec}(YRY)\, | \, R} =\E{(Y\otimes Y)}\mbox{vec}(R). \]
%Thus
%\[\E{\mbox{vec}(YRY) }=\E{\E{\mbox{vec}(YRY)\, | \, R}} =\E{(Y\otimes Y)}\mbox{vec}(\E{R}). \]
%*********************\\}

Let $||| \bar{P}|||_2 \eqdef \max_{\norm{R}_2=1} \norm{\bar{P}(R)}_2$
%\[||| \bar{P}_Z|||_2 \eqdef \max_{\norm{R}_2=1} \norm{\bar{P}_Z(R)}_2, \]
be the operator induced norm. 
Applying norm in~\eqref{eq:barPZERk} gives
\begin{eqnarray} 
\norm{\E{X_{k+1}-A^{-1}}}_{B}^* &=& \norm{\E{R_{k+1}}}_2 \\
&\leq & ||| \bar{P}|||_2 \norm{\E{R_{k}}}_2 \nonumber\\
& = &  ||| \bar{P}|||_2 \norm{\E{X_{k}-A^{-1}}}_{B}^*.\label{eq:normsplitbarP}
\end{eqnarray}

Clearly, $P$ is a \emph{positive linear map}, that is, it is linear and maps positive semi-definite matrices to positive semi-definite matrices. Thus, by Jensen's inequality, the map $\bar{P}$ is also a positive linear map.  As every positive linear map attains its norm at the identity matrix (see Corollary 2.3.8 in~\cite{bhatia07}), we have that 
\begin{eqnarray*}
||| \bar{P}|||_2 &= & \norm{\bar{P}(I)}_2 \\
&\overset{\eqref{eq:barRevol}}{=} &\norm{\E{ \left(I-\hat{Z}\right) I\left(I -\hat{Z} \right)}}_2\\
&\overset{(\text{Lemma}~\ref{ch:one:lem:Z})}{=}&\norm{\E{ I-\hat{Z}}}_2
 \overset{\eqref{eq:rhonorm}}{=} \rho.
\end{eqnarray*}
Inserting the above equivalence in~\eqref{eq:normsplitbarP}, unrolling the recurrence and using the substitution \eqref{eq:oihsoi8dhyY8J} gives~\eqref{eq:normexpconv}.

Finally~\eqref{eq:rholower} follows immediately from Lemma~\ref{lem:rho1} as $A$ is invertible and $S$ has full column rank.

\end{proof}
If $\rho =1$, this theorem does not guarantee convergence.
But when $\E{Z}$ is positive definite, as it will transpire in all practical variants of our method, some of which we describe in Section~\ref{sec:discretemethods},  the rate $\rho$ will be strictly less than one, and the norm  of the expected error will converge to zero.

\subsection{Expectation of the norm of the error}

Now we consider the convergence of the expected norm of the error.

\begin{theorem} \label{theo:Enormconv}
Let $S$ be a random matrix that has full column rank with probability~$1$ and such that
$\E{Z}$ is positive definite, where $Z$ is defined in~\eqref{eq:Z}.
Let $X_0 \in \R^{n\times n}$. If $X_k$ is calculated in either one of these two ways
 \begin{enumerate}
 \item Applying $k$ iterations of  Algorithm~\ref{alg:asym-row}, 
 %or Algorithm~\ref{alg:asym-col}, 
 \item   Applying $k$ iterations of Algorithm~\ref{alg:sym} (assuming both $A$ and $X_0$ are symmetric matrices),
\end{enumerate}
then $X_k$ converges to the inverse according to
\begin{equation} \label{eq:Enormconv}
\E{\norm{X_{k} -A^{-1} }_{F(B)}^2} \leq \rho^k \norm{X_{0} - A^{-1}}_{F(B)}^2.
 \end{equation}
\end{theorem}

\begin{proof}
First consider Algorithm~\ref{alg:asym-row}, where $X_{k+1}$ is calculated by iteratively applying~\eqref{eq:XZupdate}. Using again the substitution~\eqref{eq:oihsoi8dhyY8J}, then from~\eqref{eq:XZupdate} we have 
\begin{equation}\label{eq:js9hf7HyT}  R_{k+1} = \left(I-\hat{Z}\right) R_k. \end{equation}
From this we obtain
\begin{eqnarray}
%\norm{R_{k+1}}_{F(B)}^2 & \overset{\eqref{eq:98y988ff}}{=} & \norm{ R_{k+1}}_{F}^2 \\
 \norm{ R_{k+1}}_{F}^2 & \overset{\eqref{eq:js9hf7HyT}}{=} & \norm{\left(I-\hat{Z}\right) R_{k}}_{F}^2 \nonumber\\
 & = &\Tr{\left(I-\hat{Z}\right)\left(I-\hat{Z}\right) R_k R_{k}^\top } \nonumber\\
&\overset{(\text{Lemma}~\ref{ch:one:lem:Z})}{=}& \Tr{\left(I-\hat{Z}\right) R_k R_{k}^\top } \label{eq:asymmapplylem}\\
&= & \norm{ R_k}_{F}^2 - \Tr{\hat{Z} R_{k} R_k^\top }. \nonumber
\end{eqnarray}
Taking expectations, conditioned on $ R_k$, we get
\[\E{\norm{ R_{k+1}}_{F}^2  \, | \,  R_k} =\norm{ R_k}_{F}^2 - \Tr{\E{\hat{Z}} R_{k} R_k^\top }.\]
 Using that $\Tr{\E{\hat{Z}} R_{k} R_k^\top } \geq  \lambda_{\min}\left(\E{\hat{Z}}\right)\Tr{ R_k R_{k}^\top } $, which relies on the symmetry of $\E{\hat{Z}},$ we have that
\begin{align*}\E{\norm{ R_{k+1}}_{F}^2 \,| \,  R_k} 
%&=  \norm{ R_k}_{F}^2 - \Tr{ R_k^\top  R_k \E{\hat{Z}}}\\
&\leq  \left(1-\lambda_{\min}\left(\E{\hat{Z}}\right) \right)
\norm{ R_k}_{F}^2 =\rho\cdot  \norm{ R_k}_{F}^2.
\end{align*}
In order to arrive at \eqref{eq:Enormconv},
it now remains to take full expectation,  unroll the recurrence and use the substitution \eqref{eq:oihsoi8dhyY8J} together with $\norm{R_k}_{F}^2 \overset{\eqref{eq:98y988ff}}{=} \norm{X_k-A^{-1}}_{F(B)}^2.$
  
Now we assume that $A$ and $X_0$ are symmetric and $\{X_k\}$ are the iterates computed by Algorithm~\ref{alg:sym}. 
Left and right multiplying~\eqref{eq:XZupdatesym} by $B^{1/2}$  we have 
\begin{equation}\label{eq:barRevol2}
 R_{k+1}= \left(I-\hat{Z}\right)  R_k\left(I -\hat{Z} \right). 
\end{equation}
Taking norm we have
\begin{eqnarray}
\norm{ R_{k+1}}_{F}^2 & \overset{(\text{Lemma}~\ref{ch:one:lem:Z})}{=} &
\Tr{ R_k\left(I -\hat{Z} \right)  R_k\left(I -\hat{Z} \right)} \nonumber \\
&=& \Tr{ R_k  R_k\left(I -\hat{Z} \right)} -\Tr{ R_k \hat{Z} R_k\left(I -\hat{Z} \right)} \nonumber\\
& \leq& \Tr{ R_k  R_k\left(I -\hat{Z}\right) }, \label{eq:symEnorm1}
\end{eqnarray}
where in the last inequality we used that $I -\hat{Z}$ is an orthogonal projection and thus it is symmetric positive semi-definite, whence
\[\Tr{ R_k \hat{Z} R_k\left(I -\hat{Z} \right)} = 
\Tr{\hat{Z}^{1/2} R_k\left(I -\hat{Z} \right)R_k \hat{Z}^{1/2}} \geq 0. \]
The remainder of the proof follows similar steps as those we used in the first part of the proof from~\eqref{eq:asymmapplylem} onwards.
%Taking expectation with respect to $X_k$ in~\eqref{eq:symEnorm1} gives~\eqref{}

\end{proof}

Theorem~\ref{theo:Enormconv} establishes that for all three  methods, the expected norm of the error converges exponentially fast to zero. Moreover, the convergence rate $\rho$ is the same that appeared in Theorem~\ref{theo:normEconv}, where we established the convergence of the norm of the expected error. 

Using Lemma~\ref{lem:itercomplex}, both of the convergence results in Theorems~\ref{theo:normEconv} and~\ref{theo:Enormconv}  can be recast as iteration complexity bounds.
For instance, for a given $0<\epsilon <1$, Theorem~\ref{theo:normEconv} combined with Lemma~\ref{lem:itercomplex} (with $\alpha_k =(\norm{\E{X_k-A^{-1}}}_{B}^*)^2$) gives
\begin{equation} \label{eq:itercomplex}k \geq \left(\frac{1}{2}\right)\frac{1}{1-\rho} \log\left(\frac{1}{\epsilon}\right) \quad \Rightarrow \quad (\norm{\E{X_k-A^{-1}}}_{B}^*)^2 \leq \epsilon (\norm{X_0-A^{-1}}_{B}^*)^2.
\end{equation}
On the other hand, Theorem~\ref{theo:Enormconv} combined with Lemma~\ref{lem:itercomplex} (with $\alpha_k =\mathbf{E}[\norm{X_k-A^{-1}}_{F(B)}^2]$) gives
\begin{equation} \label{eq:itercomplex2}k \geq \frac{1}{1-\rho} \log\left(\frac{1}{\epsilon}\right) \quad \Rightarrow \quad \E{\norm{X_k-A^{-1}}_{F(B)}^2} \leq \epsilon \norm{X_0-A^{-1}}_{F(B)}^2.
\end{equation}
To push the expected norm of the error below the $\epsilon$ tolerance~\eqref{eq:itercomplex2}, we require double the amount of iterates, as compared with bringing the norm of expected error below the same tolerance~\eqref{eq:itercomplex}.  This is because in Theorem~\ref{theo:Enormconv} we determined that $\rho$ is the rate at which the expectation of the \emph{squared} norm error converges, while in Theorem~\ref{theo:normEconv} we determined that $\rho$ is the rate at which the norm, without the square, of the expected error converges. Though it takes double the number of iterations to decrease the expectation of the norm error, as proven in Lemma~\ref{lem:convrandvar}, the former is a stronger form of convergence. Thus, Theorem~\ref{theo:normEconv} does not give a stronger result than Theorem~\ref{theo:Enormconv}, but rather, these theorems give qualitatively different results and ultimately enrich our understanding of the iterative process.

\section{Discrete Random Matrices} \label{sec:discrete}

We now consider the case of a discrete random matrix $S$. We show that when $S$ is a \emph{complete discrete sampling}, then $\E{Z}$ is 
positive definite, and thus from Theorems~\ref{theo:normEconv} and~\ref{theo:Enormconv} together with Remark~\ref{rem:alg2conv}, Algorithms~\ref{alg:asym-row}, \ref{alg:asym-col} and~\ref{alg:sym} converge.

\begin{definition}[Complete Discrete Sampling]\label{def:complete}
The random matrix $S$ has a finite discrete distribution with $r$ outcomes. In particular,  $S= S_i \in \R^{n \times q_i}$ with probability  $p_i>0$ for $i=1,\ldots, r$, where $S_i$ is of full column rank.   We say that $S$ is a complete discrete sampling when
 $\mathbf{S} \eqdef [S_1, \ldots, S_r] \in \R^{m\times \sum_{i=1}^r q_i}$ has full row rank.
\end{definition}
Since we consider $A$ to be invertible in this chapter, 
the above definition of complete discrete sampling is in synchrony with the definition presented in Chapter~\ref{ch:linear_systems}.
%\rob{Note that this definition of complete discrete sampling is in accordance with definition of complete discrete sampling in Assumption~\ref{ass:complete}, except we have simplified this class of distributions given our assumption that $A$ is invertible. }

As an  example of a complete discrete sampling, let $S =e_i$ (the $i$th unit coordinate vector in $\R^n$) with probability $p_i =1/n$, for $i =1,\ldots, n.$ Then $\mathbf{S}$, as defined in Definition~\ref{def:complete},  is equal to the identity matrix: $\mathbf{S} =I$. Consequently, $S$ is a complete discrete sampling. In fact, from any basis of $\R^n$ we could construct a complete discrete sampling in an analogous way.

Next we establish that when $S$ is discrete random matrix, that $S$ having a complete discrete distribution is a necessary and sufficient condition for $\E{Z}$ to be positive definite. 

%We also determine simple formula for $\E{Z}.$ This will allow us to determine an optimized distribution for $S$ in Section~\ref{sec:discreteopt}.

\begin{proposition}\label{prop:Ediscrete} 
Let $S$ be a discrete random matrix with $r$ outcomes $S_r$ all of which have full column rank. The matrix $\E{Z}$ is  positive definite if and only if $S$ is a complete discrete sampling. Furthermore
\begin{equation}\E{Z} =  A^\top  \mathbf{S} D^2 \mathbf{S}^\top   A \label{eq:EZdiscrete}, \end{equation}
where \begin{equation} \label{eq:D}
D~\eqdef~\mbox{Diag}\left( \sqrt{p_1}(S_1^\top  AB^{-1}A^\top  S_1)^{-1/2}, \ldots, \sqrt{p_r}(S_r^\top  AB^{-1}A^\top  S_r)^{-1/2}\right).\end{equation}
\end{proposition}
\begin{proof} 
The equation~\eqref{eq:EZdiscrete} was established in Proposition~\ref{pro:Ediscrete}.
% Taking the expectation of $Z$ as defined in~\eqref{eq:Z} gives
%\begin{align}
%\E{Z} &= \sum_{i=1}^r A^\top  S_i (S_i^\top  AB^{-1}A^\top  S_i)^{-1}S_i^\top  A p_i \nonumber \\
%& =  A^\top \left(\sum_{i=1}^r   S_i \sqrt{p_i}(S_i^\top  AB^{-1}A^\top  S_i)^{-1/2} (S_i^\top  AB^{-1}A^\top  S_i)^{-1/2}  \sqrt{p_i}S_i^\top  \right) A \nonumber \\
%&= \left(  A^\top  \mathbf{S} D\right)  \left( D \mathbf{S}^\top   A\right), \nonumber
%\end{align} and $\E{Z}$ is clearly positive semi-definite.  
Since we assume that $S$ has full column rank with probability $1$, the matrix $D$ is well defined and nonsingular. Given that $\E{Z}$ is positive semi-definite, we need only show that $\Null{\E{Z}}$ contains only the zero vector if and only if $S$ is a complete discrete sampling.
 Let $v \in \Null{\E{Z}}$ and $v \neq 0,$ thus 
\[0=v^\top A^\top  \mathbf{S}  D^2 \mathbf{S}^\top   Av = \norm{D\mathbf{S}^\top   Av}_2^2, \]
which shows that $\mathbf{S}^\top   Av =0$ and thus $v \in \Null{\mathbf{S}^\top A}.$  As $A$ is nonsingular, it follows that $v=0$ if and only if  $\mathbf{S}^\top $ has full column rank. 
%Conversely, if $\mathbf{S}^\top $ does not have full column rank, then there exists $v \neq 0$ such that $v \in \Null{\mathbf{S}^\top A},$ consequentially $\E{Z}$ is not positive definite.
\end{proof}

With a closed form expression for $\E{Z}$ we can optimize $\rho$ over the possible distributions of $S$ to yield a better convergence rate.
%This shows that $S$ should be a complete sampling, so that 
%convergence is guaranteed by Theorems~\ref{theo:Enormconv} and~\ref{theo:normEconv}.

\subsection{Optimizing an Upper Bound on the Convergence Rate} \label{sec:discreteopt}

So far we have proven two different types of convergence for Algorithms~\ref{alg:asym-row}, \ref{alg:asym-col} and  \ref{alg:sym} in Theorems~\ref{theo:normEconv} and~\ref{theo:Enormconv}. Furthermore, both forms of convergence depend on the same convergence rate $\rho$ for which we have a closed form expression~\eqref{eq:rhoequiv}.

The availability of a closed form expression for the convergence rate opens up the possibility of designing particular distributions for $S$ optimizing the rate. In Section~\ref{ch:two:sec:optprob} we showed that for a complete discrete sampling, computing the optimal probability distribution, assuming that the matrices $\{S_i\}_{i=1}^r$ are fixed, leads to  a semi-definite program (SDP). Here we propose a more practical alternative: to optimize the following upper bound on the convergence rate:
\[ \rho = 1- \lambda_{\min}(B^{-1/2}\E{Z} B^{-1/2}) \leq  1 - \frac{1}{\Tr{B^{1/2} (\E{Z})^{-1} B^{1/2}}} \eqdef \gamma.\]

To emphasize the dependence of $\gamma$ and $Z$ on the probability distribution $p = (p_1,\ldots, p_r)\in \R^r$, let us denote
\begin{equation}\label{eq:98h9s8h8s}\gamma(p) \eqdef 1 - \frac{1}{\Tr{B^{1/2} (\E{Z_p})^{-1} B^{1/2}}}, \end{equation}
where we have added a subscript to $Z$ to indicate that it is a function of $p$. We now minimize $\gamma(p)$ over the probability simplex:
\[\Delta_r \eqdef \left\{p = (p_1,\dots,p_r) \in \R^r \;:\; \sum_{i=1}^r p_i =1, \; p\geq 0\right\}.\]
\begin{theorem}
Let $S$ be a complete discrete sampling and let $\overline{S}_i\in \R^{n\times q_i}$, for $i=1,2,\dots,r$, be such that 
$\mathbf{S}^{-T} = [\overline{S}_1, \ldots, \overline{S}_r]$.
 Then
\begin{equation}\label{eq:discoptrate}
\min_{p\in \Delta_r}\gamma(p) \quad=\quad 1 -\frac{1}{\left(\sum_{i=1}^r \norm{B^{-1/2}A^\top  S_i \overline{S}^\top _i  A^{-T} B^{1/2}}_F\right)^2} .
\end{equation}
\end{theorem} 
 \begin{proof} In view of \eqref{eq:98h9s8h8s}, minimizing $\gamma$ in $p$ is equivalent to minimizing  $\Tr{B^{1/2} (\E{Z_p})^{-1} B^{1/2}}$ in $p$.
Further, we have
\begin{eqnarray}
 \Tr{B^{1/2} (\E{Z_p})^{-1}B^{1/2}} &\overset{\eqref{eq:EZdiscrete}}{=} & \Tr{B^{1/2} ( A^\top  \mathbf{S} D^2 \mathbf{S}^\top   A)^{-1}B^{1/2}}\\
 &=& \Tr{B^{1/2} A^{-1} \mathbf{S}^{-T} D^{-2} \mathbf{S}^{-1}  A^{-T} B^{1/2}} \nonumber \\
&\overset{\eqref{eq:D}}{=}&\sum_{i=1}^r \frac{1}{p_i}\Tr{B^{1/2} A^{-1} \overline{S}_i (S_i^\top  AB^{-1}A^\top  S_i) \overline{S}^\top _i  A^{-T} B^{1/2}} \nonumber\\
%&=\sum_{i=1}^r \frac{1}{p_i}\Tr{   (S_i^\top  AB^{-1}A^\top  S_i) \overline{S}^\top _i  A^{-1} %BA^{-T}\overline{S}_i}\\
&=&\sum_{i=1}^r \frac{1}{p_i}\norm{B^{-1/2}A^{-1} \overline{S}_i S_i^\top   A B^{1/2}}_F^2.
\label{eq:pifracsum}
\end{eqnarray}
Applying Lemma~\ref{lem:fracsum} in the Appendix,  the optimal probabilities are given by
\begin{equation} \label{eq:discoptprob}
p_i = \frac{\norm{B^{-1/2}A^{-1}\overline{S}_i S_i^\top   A B^{1/2}}_F}{\sum_{j=1}^r \norm{B^{-1/2}A^{-1}\overline{S}_j S_j^\top   A B^{1/2}}_F}, \quad i=1,2,\dots,r
\end{equation}
Plugging this into~\eqref{eq:pifracsum} gives the result~\eqref{eq:discoptrate}.
 \end{proof}
 
Observe that in general, the optimal probabilities~\eqref{eq:discoptprob} cannot be calculated, since the formula involves the inverse of $A$, which is not known. However, if $A$ is symmetric positive definite, we can choose $B^{-1}=A^2$, which eliminates this issue. If $A$ is not symmetric positive definite, or if we do not wish to choose $B^{-1}=A^2$, we can approach the formula \eqref{eq:discoptprob}  as a recipe for a heuristic choice of the probabilities: we can use the iterates $\{X_{k}\}$ as a proxy for $A^{-1}$. With this setup, the resulting method is not guaranteed to converge by the theory developed in this thesis. However, in practice one would expect it to work well. We have not done extensive experiments to test this, and leave this to future research. To illustrate, let us consider a concrete simple example. Choose $B =I$ and $S_i =e_i$ (the unit coordinate vector in $\R^n$). We have $\mathbf{S} = [e_1,\dots,e_n] = I$, whence $\overline{S}_i =e_i$ for $i=1,\dots,r$. Plugging into \eqref{eq:discoptprob},   we obtain 
\[p_i = \frac{\norm{X_k e_i e_i^\top   A}_F}{\sum_{j=1}^r \norm{X_k e_j e_j ^\top  A}_F} = \frac{\norm{X_k e_i}_2\norm {e_i^\top   A}_2}{\sum_{j=1}^r \norm{X_k e_j}_2\norm{ e_j ^\top  A}_2}.\]
%This choice of probabilities requires a pass through the data $A$ and %$X_{k}$, which may be too costly.

%\subsection{Convenient sampling} \label{sec:discreteconv}
%
% 
%We now ask the following question:  given matrices $S_1, \dots, S_r$ defining a complete  discrete sampling, assign probabilities $p_i$ to  $S_i$ so that the convergence rate $\rho$ becomes {\em easy to interpret}. The following result was first stated in \cite{Gower2015} in the context of solving linear systems, and gives a convenient choice of probabilities resulting in the rate $\rho$ which depends on a (scaled) condition number of the original data matrix $A$. 
%
%\begin{proposition} \label{theo:convenconv}
%Let $S$ be a complete discrete sampling where $S =S_i$ with probability  
%\begin{equation} \label{eq:convprob}
%p_i = \left.\norm{B^{-1/2}A^\top  S_i}_F^2\right/\norm{B^{-1/2}A^\top \mathbf{S}}_F^2.
%\end{equation}
%Then the convergence rate takes the form
%\begin{equation}
%\rho = 1- \frac{1}{\kappa_{2,F}^2(B^{-1/2}A^\top \mathbf{S})}, \label{eq:rhoconv}  
%\end{equation}
%where
%\begin{align}\label{eq:kappalower}
%\kappa_{2,F}(B^{-1/2}A^\top \mathbf{S}) &\eqdef \norm{(B^{-1/2}A^\top \mathbf{S})^{-1}}_2 \norm{B^{-1/2}A^\top \mathbf{S}}_F  =
%\sqrt{ \frac{\Tr{\mathbf{S}^\top AB^{-1}A^\top \mathbf{S}}}{\lambda_{\min}\left(\mathbf{S}^\top  AB^{-1}A^\top \mathbf{S} \right)} } \geq \sqrt{n}.
%\end{align}
%\end{proposition}
%\begin{proof}
%For the proof of~\eqref{eq:rhoconv}, see Theorem~5.1 in~\cite{Gower2015}. The bound in~\eqref{eq:kappalower} follows trivially.
%\end{proof}
%
%

\subsection{Adaptive Samplings}

In Theorem~\ref{theo:convsingleS} we determined a discrete probability distribution $\mathbf{P}(S= S_i) =p_i$ that yields a convergence rate $\rho$ that is easy to interpret. We will now make use of this convenient probability distribution and pose the question:  Having decided on the probabilities $p_1,\dots,p_r$, how should we choose the matrices $S_1,\ldots, S_r$ if we want $\rho $ to be as small as possible?

To answer this question, first we observe that the convergence rate in  Theorem~\ref{theo:convsingleS} is proportional to the scaled condition number defined by
\begin{align}\label{eq:kappalower}
\kappa_{2,F}(B^{-1/2}A^\top \mathbf{S}) &\eqdef \norm{(B^{-1/2}A^\top \mathbf{S})^{-1}}_2 \norm{B^{-1/2}A^\top \mathbf{S}}_F  =
\sqrt{ \frac{\Tr{\mathbf{S}^\top AB^{-1}A^\top \mathbf{S}}}{\lambda_{\min}\left(\mathbf{S}^\top  AB^{-1}A^\top \mathbf{S} \right)} } \geq \sqrt{n}.
\end{align}
That is, 
 if we select the probabilities
\begin{equation} \label{eq:convprob}
p_i = \left.\norm{B^{-1/2}A^\top  S_i}_F^2\right/\norm{B^{-1/2}A^\top \mathbf{S}}_F^2,
\end{equation}
then Theorem~\ref{theo:convsingleS} combined with~\eqref{eq:kappalower}  gives that the resulting convergence rate is
\begin{equation}
\rho = 1- \frac{\Tr{\mathbf{S}^\top AB^{-1}A^\top \mathbf{S}}}{\lambda_{\min}\left(\mathbf{S}^\top  AB^{-1}A^\top \mathbf{S} \right)}= 1- \frac{1}{\kappa_{2,F}^2(B^{-1/2}A^\top \mathbf{S})}. \label{eq:rhoconv}  
\end{equation}
Furthermore, following from Remark~\ref{rem:alg2conv}, we can determine a convergence rate for Algorithm~\ref{alg:asym-col} based on~\eqref{eq:rhoconv}. That is, by merely transposing each occurrence of $A$ we have that, by selecting $S_i$ with probability 
\begin{equation} \label{eq:convprob2}
p_i = \left.\norm{B^{-1/2}A S_i}_F^2\right/\norm{B^{-1/2}A\mathbf{S}}_F^2,
\end{equation}
then Algorithm~\ref{alg:asym-col}  converges at the rate 
 \begin{equation}\label{eq:rhoconv2}
\rho_2 = 1- \frac{1}{ \kappa_{2,F}^2(B^{-1/2}A\mathbf{S})}.
\end{equation}

 Since these rates improve as the condition number $\kappa^2_{2,F}(B^{-1/2}A^\top \mathbf{S})$ (or $\kappa_{2,F}^2(B^{-1/2}A\mathbf{S})$ for Algorithm~\ref{alg:asym-col}) decreases, we should aim for matrices $S_1,\ldots, S_r$ that minimize the condition number. For instance, the lower bound in~\eqref{eq:kappalower} is reached for $\mathbf{S} = (B^{-1/2}A^\top )^{-1} = A^{-T}B^{1/2}$. While we do not know $A^{-1}$, we can use our best current approximation of it, $X_k$, in its place. This leads to a method which {\em adapts} the probability distribution governing $S$ throughout the iterative process. This observation inspires a very efficient modification of Algorithm~\ref{alg:sym}, which we call AdaRBFGS (Adaptive Randomized BFGS), and describe in Section~\ref{sec:AdaRBFGS}.

Notice that, luckily and surprisingly, our twin goals of computing the inverse and optimizing the convergence rate via the above adaptive trick are compatible.  Indeed, we wish to find $A^{-1}$, whose knowledge gives us the  optimal rate. This should be contrasted with the SDP approach mentioned earlier in Section~\ref{ch:two:sec:optprob}: i) the SDP could potentially be harder than the inversion problem, and ii) having found the optimal probabilities $\{p_i\}$, we are still not guaranteed the optimal rate. Indeed, optimality is relative to the choice of the matrices $S_1,\dots,S_r$, which can be suboptimal.

\begin{remark}[Adaptive sampling]
The convergence rate~\eqref{eq:rhoconv2} suggests how one can select a sampling distribution for $S$ that would result in faster practical convergence. We now detail several practical choices for $B$ and indicate how to sample $S$. These suggestions require that the distribution of $S$ depends on the iterate $X_k$, and thus no longer fit into our framework. Nonetheless, we collect these suggestions here in the hope that others will wish to extend these ideas further, and as a  demonstration of the utility of developing convergence rates.  
\begin{enumerate}
\item If $B =I$, then Algorithm~\ref{alg:asym-row} converges at the rate $\rho = 1- 1/\kappa_{2,F}^2(A^\top \mathbf{S})$, and hence $S$ should be chosen so that $\mathbf{S}$ is a preconditioner of $A^\top $. For example $\mathbf{S}=X_{k}^\top ,$ that is, $S$ should be a sampling of the rows of $X_{k}$. 
\item If $B =I$, then Algorithm~\ref{alg:asym-col} converges at the rate $\rho = 1- 1/\kappa_{2,F}^2(A\mathbf{S})$, and hence $S$ should be chosen so that $\mathbf{S}$ is a preconditioner of $A$. For example $\mathbf{S}=X_{k}$; that is, $S$ should be a sampling of the columns of $X_{k}$. 
\item If $A $ is symmetric positive definite, we can choose $B=A$, in which case Algorithm~\ref{alg:sym} converges at the rate $\rho =1- 1/\kappa_{2,F}^2(A^{1/2}\mathbf{S}).$  This rate suggests that $S$ should be chosen so that $\mathbf{S}$ is an approximation of $A^{-1/2}.$ In Section~\ref{sec:AdaRBFGS} we develop this idea further, and design the AdaRBFGS algorithm.
\item If $B=A^\top A$, then Algorithm~\ref{alg:asym-row} can be efficiently implemented with $S=AV$, where $V$ is a complete discrete sampling. Furthermore $\rho = 1- 1/\kappa_{2,F}^2(A\mathbf{V}),$ where  $\mathbf{V} \eqdef [V_1,\ldots, V_r]$.  This rate suggests that $V$ should be chosen so that  $\mathbf{V}$ is a preconditioner of  $A$. For example $\mathbf{V}=X_{k}$; that is, $V$ should be a sampling of the rows of $X_{k}$. 
\item If $B=AA^\top $, then Algorithm~\ref{alg:asym-col} can be efficiently implemented with $S=A^\top V$, where $V$ is a complete discrete sampling. From~\eqref{eq:rhoconv2}, the convergence rate of the resulting method is given by
$1- 1/ \kappa_{2,F}^2(A^\top \mathbf{V}).$
This rate suggests that $V$ should be chosen so that  $\mathbf{V}$ is a  preconditioner of  $A^\top $. For example, $\mathbf{V}=X_{k}^\top $; that is, $V$ should be a sampling of the columns of $X_{k}$. 
\item If $A$ is symmetric positive definite, we can choose $B = A^{-2}$, in which case Algorithm~\ref{alg:sym} can be efficiently implemented with  $S =AV.$ Furthermore $\rho = 1- 1/\kappa_{2,F}^2(A\mathbf{V}).$ This rate suggests that $V$ should be chosen so that  $\mathbf{V}$ is a  preconditioner of  $A$. For example $\mathbf{V}=X_{k},$ that is, $V$ should be a sampling of the rows or the columns of $X_{k}$. 
\end{enumerate}
\end{remark}

\section{Randomized Quasi-Newton Updates} \label{sec:discretemethods}

Algorithms~\ref{alg:asym-row}, \ref{alg:asym-col} and \ref{alg:sym} are in fact families of algorithms indexed by the two parameters: i) positive definite matrix $B$ and ii) distribution $\cal D$ (from which we pick random matrices  $S$). This allows us to design a myriad of specific methods by varying these parameters. Here we highlight some of these possibilities, focusing on complete discrete distributions for $S$ so that convergence of the iterates is guaranteed through Theorems~\ref{theo:normEconv} and~\ref{theo:Enormconv}.
We also compute the convergence rate $\rho$ for these special methods for the convenient probability distribution given by~\eqref{eq:convprob} and~\eqref{eq:convprob2} (Theorem~\ref{theo:convsingleS}) so that the convergence rates~\eqref{eq:rhoconv} and~\eqref{eq:rhoconv2} depend on a scaled condition number which is easy to interpret. We will also make some connections to existing quasi-Newton  and Approximate Inverse Preconditioning methods. Table~\ref{tbl:QN} provides a guide through this section.

\begin{table}[!h]
\begin{center}
{
\footnotesize
\begin{tabular}{|c|c|c|c|c|c|}
\hline
$A$                                          & $B^{-1}$      & $S$ &  Inverse Equation & Randomized Update & Section\\
\hline
any        &  any & invertible & any & One Step & \ref{subsec:1step}\\
any        &  $ I$ & $e_i$ & $AX = I$ & Simultaneous Kaczmarz (SK) & \ref{subsec:09u09u09Kacz}\\
any       &  $I$ & vector & $XA = I$ & Bad Broyden (BB) & \ref{subsec:09u09u09}\\
sym. & $I$ & vector & $AX = I, X=X^\top $ & Powell-Symmetric-Broyden (PSB) & \ref{subsec:PSB} \\
any       &  $ I$ & vector & $XA^{-1} = I$ & Good Broyden (GB) & \ref{subsec:09u09u09good}\\
sym.      &  $A^{-1}-X_k$ & vector & $AX =I$ or $XA=I$ & Symmetric Rank 1 (SR1) & \ref{subsec:SR1}\\
%symmetric       &  $X_k$ & vector & $AX =I, X=X^\top $ & Greenstadt & \ref{subsec:greenstadt}\\
s.p.d.       &  $A$ & vector & $XA^{-1} =I , X=X^\top $ & Davidon-Fletcher-Powell (DFP) & \ref{sec:RDFP} \\
s.p.d.       &  $A^{-1}$ & vector & $AX = I, X=X^\top $ & Broyden-Fletcher-Goldfarb-Shanno (BFGS) & \ref{sec:RBFGS} \\
any      & $(A^\top  A)^{-1}$ & vector & $AX=I$ & Column & \ref{sec:RCM}\\
\hline
\end{tabular}
}
\end{center}
\caption{Specific randomized updates for inverting matrices discussed in this section, obtained as special cases of our algorithms. First column: ``sym'' means ``symmetric'' and ``s.p.d.'' means ``symmetric positive definite''. Block versions of all these updates are obtained by choosing $S$ as a matrix with more than one column (i.e., not as a vector).}\label{tbl:QN}
\end{table}
 
 % The existing methods fall under the column variant  of Algorithm~\ref{alg:asym}, while the row variant of Algorithm

\subsection{One Step Update}\label{subsec:1step}

We have the freedom to select $S$ as almost any random matrix that has full column rank. This includes choosing $S$ to be a constant and invertible matrix, such as the identity matrix $I$,  in which case 
$X_{1}$ must be equal to the inverse. Indeed, the sketch-and-project formulations of all our algorithms reveal that. For Algorithm~\ref{alg:asym-row}, for example, the sketched system is $S^\top  AX = S^\top $, which is equivalent to $AX=I$, which has as its unique solution $X=A^{-1}$.  Hence, $X_1 = A^{-1}$, and we have convergence in one iteration/step. Through inspection of the complexity rate, we see that $B^{-1/2}\E{Z}B^{-1/2} = I$ and  $\rho = \lambda_{\min}(B^{-1/2}\E{Z}B^{-1/2}) =1$, thus this one step convergence  is predicted in theory by Theorems~\ref{theo:normEconv} and~\ref{theo:Enormconv}.

\subsection{Simultaneous randomized Kaczmarz update}\label{subsec:09u09u09Kacz}

Perhaps the most natural choice for the weighting matrix $B$ is the identity $B=I.$ With this choice, Algorithm~\ref{alg:asym-row} is equivalent to applying the randomized  Kaczmarz update simultaneously to the $n$ linear systems encoded in $AX=I$. To see this, 
note that the sketch-and-project viewpoint~\eqref{eq:NF} of Algorithm~\ref{alg:asym-row} is
 \begin{align} \label{eq:NFbadbroyadj}
X_{k+1} = &\arg \min_{X\in \R^{n\times n}} \frac{1}{2}\norm{X - X_{k}}_{F}^2 \quad \mbox{subject to }\quad  S^\top AX =S^\top ,
\end{align} 
which, by~\eqref{eq:Xupdate}, results in the explicit update 
\begin{equation} \label{eq:badbroyadj}
X_{k+1} = X_{k}+ A^\top S(S^\top AA^\top S)^{-1}S^\top  (I-AX_{k}).
\end{equation}
If $S$ is a random coordinate vector, then~\eqref{eq:NFbadbroyadj} is equivalent to projecting the $j$th column of $X_k$ onto the solution space of $A_{i:}x = \delta_{ij},$ which is exactly an iteration of the randomized Kaczmarz update applied to solving $Ax=e_j.$
%is the randomized Kaczmarz update applied simultaneously to the $n$ equations $AX=I.$ 
In particular, if $S = e_i$ with probability $p_i = \norm{A_{i:}}_2^2/\norm{A}_F^2$ then according to~\eqref{eq:rhoconv}, the rate of convergence of update~\eqref{eq:badbroyadj} is given by
\[\E{\norm{X_{k} -A^{-1} }_F^2} \leq \left(1- \frac{1}{\kappa_{2,F}^2(A)}\right)^k \norm{X_{0} -A^{-1} }_F^2 \]
where we used that $\kappa_{2,F}(A) = \kappa_{2,F}(A^\top ).$ This is exactly the rate of convergence given by Strohmer and Vershynin in~\cite{Strohmer2009} for the randomized Kaczmarz method.

\subsection{Randomized bad Broyden update}\label{subsec:09u09u09}

The update~\eqref{eq:badbroyadj} can also be viewed as an adjoint form of the bad Broyden update~\cite{Broyden1965,Griewank2012}. To see this,
if we use  Algorithm~\ref{alg:asym-col} with $B=I$, then the iterative process is
\begin{equation} \label{eq:badbroy}
X_{k+1} = X_{k}+ (I-X_{k}A)S(S^\top A^\top AS)^{-1}S^\top A^\top . 
\end{equation}
 This update~\eqref{eq:badbroy} is a randomized block form of the {\em bad Broyden update}~\cite{Broyden1965,Griewank2012}. In the quasi-Newton setting, $S$ is not random, but rather the previous step direction $S = \delta \in \R^n$. Furthermore, if we rename $\gamma \eqdef AS\in \R^n$, then~\eqref{eq:badbroy} becomes  
\begin{equation}\label{eq:badbroyden}
X_{k+1} = X_{k}+ \frac{\delta-X_k\gamma}{\norm{\gamma}_2^2}\gamma^\top ,
\end{equation}
which is the standard way of writing the bad Broyden update~\cite{Griewank2012}.
The update~\eqref{eq:badbroyadj} is an adjoint form of the bad Broyden in the sense that, if we transpose~\eqref{eq:badbroyadj}, then set $S = \delta$ and denote $\gamma = A^\top S$, we obtain the bad Broyden update, but applied to $X_k^\top $  instead.

From the constrain-and-approximate viewpoint~\eqref{eq:RFcols} we give a new interpretation to the bad Broyden update, namely, the update~\eqref{eq:badbroyden} can be written as
\[X_{k+1} = \arg_X \min_{X \in \R^{n\times n}, \; y\in \R^n} \frac{1}{2}\norm{X - A^{-1}}_{F}^2 \quad \mbox{subject to } \quad  X=X_{k} +y \gamma^\top . \]
Thus, {\em the bad Broyden update is the best rank-one update  approximating the inverse.}

We can determine the rate at which our randomized variant of the BB update~\eqref{eq:badbroy} converges by using~\eqref{eq:rhoconv2}.
In particular, if $S =S_i$ with probability   $p_i = \left.\norm{A S_i}_F^2\right/\norm{A\mathbf{S}}_F^2$, then~\eqref{eq:BFGSas} converges  with the rate
\[\E{\norm{X_{k} -A^{-1} }_F^2} \leq \left(1- \frac{1}{\kappa_{2,F}^2(A\mathbf{S})}\right)^k \norm{X_{0} -A^{-1} }_F^2. \]
%$1- 1/ \kappa_{2,F}^2(A\mathbf{S}).$ 

\subsection{Randomized Powell-Symmetric-Broyden update}\label{subsec:PSB}

If $A$ is symmetric and we use  Algorithm~\ref{alg:sym} with $B=I$, the iterates are given by 
\begin{align}
 X_{k+1} &= X_{k} +AS^\top  (S^\top A^2S)^{-1}SA
(X_{k}AS -S)\left((S^\top A^2S)^{-1} S^\top  A-I \right) \nonumber \\
&-(X_{k}AS -S)(S^\top A^2S)^{-1} S^\top  A, \label{eq:PSB}
 \end{align}
which is a randomized block form of the Powell-Symmetric-Broyden update~\cite{Gower2014c}.
If $S =S_i$ with probability $p_i = \norm{AS_i}_F^2/\norm{A\mathbf{S}}_F^2$, then according to~\eqref{eq:rhoconv}, the iterates~\eqref{eq:PSB} and~\eqref{eq:badbroyadj} converge according to
\[\E{\norm{X_{k} -A^{-1} }_F^2} \leq \left(1  - \frac{1}{\kappa^2_{2,F}(A^\top \mathbf{S})}\right)^k \norm{X_{0} - A^{-1}}_F^2.\]
% To the best of our knowledge, this is  the {\em first randomized Powell-Symmetric-Broyden update}. 

\subsection{Randomized good Broyden update}
\label{subsec:09u09u09good}

Next we present a method that shares certain properties with Gaussian elimination and can be viewed as a randomized block variant of the good Broyden update~\cite{Broyden1965,Griewank2012}. This method requires the following adaptation of Algorithm~\ref{alg:asym-col}: instead of sketching the inverse equation, consider the update~\eqref{eq:guasssketch} that performs a column sketching of the equation $XA^{-1} = I$ by right multiplying with $Ae_i$, where $e_i$ is the $i$th coordinate vector. Projecting an  iterate $X_k$ onto this sketched equation gives 
\begin{align}\label{eq:guasssketch}
X_{k+1} = &\arg \min_{X \in \R^{n\times n}} \frac{1}{2} \norm{X - X_{k}}_{F}^2 \quad \mbox{subject to }\quad  Xe_i =Ae_i.
\end{align}
 The iterates defined by the above are given by
\begin{equation}\label{eq:gaussinv}X_{k+1}=X_{k} +(A-X_k)e_ie_i^\top .
\end{equation}
Given that we are sketching and projecting onto the solution space of $XA^{-1}=I$, the iterates of this method converge to $A$. Therefore the inverse iterates $X^{-1}_k$ converge to $A^{-1}.$ We can efficiently compute the inverse iterates 
by using the Woodbury formula~\cite{Woodbury1950} which gives
\begin{equation}\label{eq:gauss}
X_{k+1}^{-1} = X_k^{-1}-\frac{(X_k^{-1}A-I)e_ie_i^\top X_k^{-1}}{e_i^\top X_k^{-1}Ae_i}.
\end{equation}
This update~\eqref{eq:gauss} behaves like Gaussian elimination in the sense that, if $i$ is selected in a cyclic fashion, that is $i=k$ on the $k$th iteration, then from~\eqref{eq:gaussinv} it is clear that 
\[ X_{k+1}e_i = Ae_i, \quad \mbox{thus} \quad X_{k+1}^{-1}Ae_i =e_i, \quad \mbox{for }i=1\ldots k.\]
That is, on the $k$th iteration, the first $k$ columns of the matrix $X_{k+1}^{-1}A$ are equal to the first $k$ columns of the identity matrix. Consequently,  $X_n = A$ and $X^{-1}_n = A^{-1}.$ If instead, we select $i$ uniformly at random, then we can adapt~\eqref{eq:rhoconv} by swapping each occurrence of $A^\top $ for $A^{-1}$ and observing that $S_i = Ae_i$ thus $\mathbf{S}=A$. Consequently the iterates~\eqref{eq:gaussinv} converge to $A$  at a rate of
\[\rho = 1-\kappa_{2,F}^2 \left(A^{-1}A\right)= 1-\frac{1}{n}, \]
and thus the lower bound~\eqref{eq:rholower} is achieved and $X_k$ converges to $A$ according to
\[\E{\norm{X_{k} -A }_F^2} \leq \left(1-\frac{1}{n}\right)^k \norm{X_{0} -A}_F^2. \]
 Despite this favourable convergence rate, this does not say anything about how fast 
$X^{-1}_k$ converges to $A^{-1}$. Therefore~\eqref{eq:gauss} is not an efficient method for calculating an approximate inverse.
If we replace $e_i$  by a \emph{step} direction $\delta_k \in \R^d$, then the update~\eqref{eq:gauss} is known as the {\em good Broyden} update~\cite{Broyden1965,Griewank2012}.

\subsection{Approximate inverse preconditioning}\label{sec:AIP} 

When $A$ is symmetric positive definite, we can choose $B =A$, and Algorithm~\ref{alg:asym-row} is given by
\begin{equation} \label{eq:BFGSas} X_{k+1} = X_{k} +S(S^\top AS)^{-1}S^\top (I-AX_{k}). \end{equation}
The constrain-and-approximate viewpoint~\eqref{eq:RF} of this update is
 \[X_{k+1} = \arg_X \min_{X\in \R^{n\times n}, Y\in \R^{n\times q}}  \frac{1}{2}\norm{A^{1/2}XA^{1/2} - I}_{F}^2 \quad \mbox{subject to } \quad  X=X_{k} + S Y^\top .\]
This viewpoint reveals that the update~\eqref{eq:BFGSas} is akin to the Approximate Inverse Preconditioning (\emph{AIP}) methods~\cite{Benzi1999,Gould1998,Kolotilina1993,Huckle2007}. 

% In the AIP methods one minimizes the residual $\norm{XA - I}_{F}^2$
%using the  steepest descent or a minimal residual type method~\cite{Chow1998,Saad2003}. Here, instead, we minimize the residual on a randomly generated subspace that passes through $X_k.$

 We can determine the rate a which~\eqref{eq:BFGSas}  converges using~\eqref{eq:rhoconv}. In particular, if $S =S_i$ with probability   $p_i = \Tr{S_i^\top  A S_i}/\Tr{\mathbf{S}^\top A\mathbf{S}}$, then~\eqref{eq:BFGSas} converges with rate 
\begin{equation}
\rho \overset{\eqref{eq:rhoconv}}{=} 1- \frac{1}{\kappa_{2,F}^2(A^{1/2}\mathbf{S})} = 1- \frac{\lambda_{\min}(\mathbf{S}^\top A\mathbf{S})}{\Tr{\mathbf{S}^\top A\mathbf{S}}}, \label{eq:BBonv}  
\end{equation}
and according to
\[\E{\norm{A^{1/2}X_k A^{1/2} - I }_F^2} \leq \left(1- \frac{\lambda_{\min}(\mathbf{S}^\top A\mathbf{S})}{\Tr{\mathbf{S}^\top A\mathbf{S}}}\right)^k \norm{A^{1/2}X_0A^{1/2} - I}_F^2.\]
This, as we will see in Section~\ref{sec:RBFGS}, 
is the same rate of convergence as a randomized variant of the BFGS method.
\subsection{Randomized SR1} \label{subsec:SR1}

The Symmetric Rank-1 (SR1) update~\cite{Davidon1968,Murtagh1969} does not explicitly fit into our framework, and nor does it fit into the traditional quasi-Newton framework, since it requires a $B$ that is not positive definite.  
Despite this, we present the update since it is  still commonly used.

Before choosing $B$, note that, though our theory requires that $B$ be positive definite, Algorithms~\ref{alg:asym-row} and~\ref{alg:asym-col} can be defined with a matrix $B^{-1}$ even if $B$ does not exist! For instance, 
when $A$ is symmetric and $B^{-1} = A^{-1}-X_k$ then from~\eqref{eq:Xupdate} or~\eqref{eq:Xupdatecols} we get
\begin{equation} \label{eq:SR1}
X_{k+1} = X_k +(I-AX_k)^\top S(S^\top (A-AX_kA)S)^{-1}S^\top (I-AX_k).
\end{equation}
This choice for $B$ presents problems, namely, the update~\eqref{eq:SR1} is not always well defined because it requires inverting $S^\top (A-AX_kA)S$ which is not necessarily invertible.
To fix this, we should select the sketching matrix $S$ so that $ S^\top (A-AX_kA)S$ is invertible. But this in turn means that $S$ will depend on $X_k$ and most likely cannot be sampled in an i.i.d fashion. Alternatively,   we can use the pseudoinverse of $S^\top (A-AX_kA)S$ in place of the inverse.

Since $B$ is not positive definite, our convergence theory says nothing about this update.

\subsection{Randomized DFP update}\label{sec:RDFP}

If $A$ is symmetric positive definite then we can choose $B =A^{-1}.$ Furthermore, if we adapt the sketch-and-project formulation~\eqref{eq:NF} to sketch the equation $XA^{-1} = I$ by right multiplying by $AS,$ and additionally impose symmetry on the iterates, we arrive at the following update.
\begin{align}\label{eq:DFPsketch}
X_{k+1} = &\arg \min_{X \in \R^{n\times n}} \frac{1}{2} \norm{X - X_{k}}_{F(A^{-1})}^2 \quad \mbox{subject to }\quad  XS =AS, \quad X =X^\top .
\end{align}
The solution to the above is given by\footnote{To arrive at this solution, one needs to swap the occurrences of $AS$ for $S$ and plug in  $B=A^{-1}$ in~\eqref{eq:Xupdatesym}. This is because, by swapping $AS$ for $S$ in~\eqref{eq:DFPsketch} gives~\eqref{eq:NFsym}.}
\begin{equation}\label{eq:DFPinv}
X_{k+1}= AS (S^\top AS)^{-1} S^\top A +\left( I- AS (S^\top AS)^{-1} S^\top 
\right)X_k\left(I -S (S^\top AS)^{-1} S^\top  A\right).  
\end{equation}
Using the Woodbury formula~\cite{Woodbury1950}, we find that 
\begin{equation}\label{eq:DFP} X_{k+1}^{-1} = X_{k}^{-1} + AS (S^\top AS)^{-1} S^\top A  -
X_k^{-1}  S \left( S^\top  X_k^{-1}  S  \right)^{-1} S^\top  X_k^{-1}.\end{equation}
The update~\eqref{eq:DFP} is a randomized variant of the Davidon-Fletcher-Powell (DFP) update~\cite{Davidon1959,Fletcher1960}. We can adapt~\eqref{eq:rhoconv} to determine the rate at which $X_k$ converges to $A$ by swapping each occurrence of $A^\top $ for $A^{-1}$. Indeed, for  example, let $S_i= Ae_i$  with probability $ p_i = \left. \lambda_{\min}(A) \right/ \Tr{A},$
%\[p_i = \left.\norm{A^{1/2}A^{-1} Ae_i}_F^2\right/\norm{A^{1/2}}_F^2 = \left.\lambda_{\min}(A) \right/ \Tr{A}.\]
then the iterates~\eqref{eq:gaussinv} converge to $A$  at a rate of
\begin{equation}\label{eq:convDFP}
\E{\norm{X_{k} -A }_{F(A^{-1})}^2} \leq \left(1-\frac{\lambda_{\min}(A)}{\Tr{A}}\right)^k \norm{X_{0} - A}_{F(A^{-1})}^2.
\end{equation}
Thus $X_k$ converges to $A$ at a favourable rate.
But this does not indicate at what rate does $X_{k}^{-1}$ converge to $A^{-1}$. This is in contrast to the randomized BFGS, which produces iterates that converge to $A^{-1}$ at this same favourable rate, as we show in the next section.  
This sheds new light on why BFGS update performs better than the DFP update.

\subsection{Randomized BFGS update}\label{sec:RBFGS}

% The update~\eqref{eq:BFGSas} is different to previous AIP updates in that the constraint is a $q$--dimensional subspace, thus the update can be interpreted as a block version of an AIP method.
 
If $A$ is symmetric and positive definite, we can  choose $B=A$ and apply  Algorithm~\ref{alg:sym} to maintain symmetry of the iterates. The iterates are given by
\begin{equation}\label{eq:qunac}
X_{k+1}  =S(S^\top AS)^{-1}S^\top + \left(I-S(S^\top AS)^{-1}S^\top A\right) X_{k} \left(I -AS(S^\top AS)^{-1}S^\top  \right).
\end{equation}
This is a block variant, see~\cite{Gower2014c}, of the BFGS update~\cite{Broyden1965,Fletcher1960,Goldfarb1970,Shanno1971}.
The constrain-and-approximate viewpoint gives a new interpretation to the Block BFGS update. That is, from~\eqref{eq:NFsym}, the iterates~\eqref{eq:qunac} can be equivalently defined by
\[X_{k+1} = \arg_X  \min_{X\in \R^{n\times n}, Y\in \R^{n\times q}}  \frac{1}{2}\norm{XA -I}_{F}^2 \quad \mbox{subject to} \quad
X = X_k + SY^\top  +YS^\top .\]

Thus the block BFGS update, and the standard BFGS update, can be seen as a method for calculating an approximate inverse subject to a particular symmetric affine space passing through $X_k.$ This is a completely new way of interpreting the BFGS update.

If $p_i = \Tr{S_i^\top AS_i}/\Tr{\mathbf{S}A\mathbf{S}^\top }$, then according to ~\eqref{eq:rhoconv}, the update~\eqref{eq:qunac}  converges according to
\begin{equation}\label{eq:convqunac}
\E{\norm{X_{k}A -I }_{F}^2} \leq \left(1 - \frac{1}{\kappa^2_{2,F}(A^{1/2}\mathbf{S})}\right)^k \norm{X_{0}A - I}_{F}^2.
\end{equation}

A remarkable property of the update~\eqref{eq:qunac} is that it preserves positive definiteness of $A$.  Indeed, assume that $X_k$ is positive definite and let $v \in \R^n$ and $P \eqdef S(S^\top AS)^{-1}S^\top .$ Left and right multiplying~\eqref{eq:qunac} by $v^\top $ and $v$, respectively, gives
\[v^\top X_{k+1}v = v^\top Pv+ v^\top \left(I-PA\right) X_{k} \left(I -AP\right)v \geq 0.\]
Thus $v^\top X_{k+1}v =0$ implies that $Pv=0$ and $ \left(I -AP \right)v=0,$  which when combined gives $v =0.$ This proves that $X_{k+1}$ is positive definite.
Thus the update~\eqref{eq:qunac} is particularly well suited for calculating the inverse of a positive definite matrices.

In Section~\eqref{sec:AdaRBFGS}, we detail an update designed to improve the convergence rate in~\eqref{eq:convqunac}. The result is a method that is able to invert large scale positive definite matrices orders of magnitude faster than the state-of-the-art.

\subsection{Randomized Column update}\label{sec:RCM}
We now describe an update that has no connection to any previous updates, yet the convergence rate we determine~\eqref{eq:RRMconv} is favourable, and comparable to all the other updates we develop.
 
For this update, we need to perform a linear transformation of the sampling matrices. For this, let $V$ be a complete discrete sampling where $V= V_i\in \R^{n\times q_i}$ with probability $p_i>0,$ for $i =1,\ldots,r.$ Let $\mathbf{V} =[V_1,\ldots, V_r].$
 Let the sampling matrices be defined as $S_i=AV_i \in \R^{n \times q_i}$ for $i =1,\ldots,r$. As $A$ is nonsingular, and $\mathbf{S} = A \mathbf{V}$, then $S$ is a complete discrete sampling. With these choices and $B=A^\top A$, the sketch-and-project viewpoint~\eqref{eq:NF} is given by 
\begin{align}
X_{k+1} =\arg \min_{X\in \R^{n\times n}} \frac{1}{2}\norm{X - X_{k}}_{F(A^\top A)}^2 \nonumber\quad  \mbox{subject to } \quad V_i^\top A^\top AX = V_i^\top A^\top .
\end{align} 
The solution to the above are the iterates of  Algorithm~\ref{alg:asym-row}, which is given by
\begin{equation}\label{eq:RRM}
X_{k+1} = X_{k} +V_i(V_i^\top  A^\top  AV_i)^{-1}V_i^\top (A^\top -A^\top AX_{k}).
\end{equation}
 From the constrain-and-approximate viewpoint~\eqref{eq:RF}, this can be written as
 \[X_{k+1} = \arg  \min_{X\in \R^{n\times n}, Y\in \R^{n\times q}}  \frac{1}{2}\norm{A(XA^\top -I)}_{F}^2 \quad \mbox{subject to } \quad  X=X_{k} + V_i Y^\top.\]
 
With these same parameter choices for $S$ and $B$, the iterates of Algorithm~\ref{alg:sym} are given by
    \begin{align}
X_{k+1}  &=X_{k} + V_i(V_i^\top A^2 V_i)^{-1}V_i^\top (AX_{k} -I)\left(A^2V_i(V_i^\top  A^2V_i)^{-1}V_i^\top  - I \right)\nonumber \\
& -(X_{k}A -I)AV_i(V_i^\top  A^2 V_i)^{-1}V_i^\top . \label{eq:symROM}
\end{align}
If we choose
$p_i = \norm{(AA^\top )^{-1/2}AA^\top V_i}_F^2 /\norm{(AA^\top )^{-1/2}AA^\top \mathbf{V}}_F^2 = \norm{A^\top V_i}_F^2/\norm{A^\top \mathbf{V}}_F^2,$
 then according to ~\eqref{eq:rhoconv}, the iterates~\eqref{eq:RRM} and~\eqref{eq:symROM} converge exponentially in expectation to the inverse according to
\begin{equation}\label{eq:RRMconv}
\E{\norm{A(X_kA^\top -I) }_{F}^2} \leq \left(1 - \frac{1}{\kappa^2_{2,F}(A\mathbf{V})}\right)^k \norm{A(X_0A^\top -I)}_{F}^2.
\end{equation}
 There also exists an analogous ``row'' variant of~\eqref{eq:RRM}, which arises by using  Algorithm~\ref{alg:asym-col}, but we do not explore it here.

\section{AdaRBFGS: Adaptive Randomized BFGS}\label{sec:AdaRBFGS}

All the updates we have developed thus far use a sketching matrix $S$ that is sampled in an i.i.d.\ fashion from a fixed distribution $\cal D$ at each iteration. In this section we assume that $A$ is symmetric positive definite, and propose AdaRBFGS: a variant of the RBFGS update, discussed in Section~\ref{sec:RBFGS}, which  {\em adaptively} changes the distribution $\cal D$ throughout the iterative process. Due to this change,  Theorems~\ref{theo:normEconv} and~\ref{theo:Enormconv} are no longer applicable. Superior numerical efficiency of this update is verified through extensive numerical experiments in Section~\ref{sec:numerical}.  

\subsection{Motivation}
We now motivate the design of this new update by examining the convergence rate~\eqref{eq:convqunac} of the RBFGS iterates~\eqref{eq:qunac}. Recall that in RBFGS we choose $B=A$ and $S=S_i$ with probability 
\begin{equation}\label{eq:ihsih0(0hI}p_i = \Tr{S_i^\top AS_i}/\Tr{\mathbf{S}A\mathbf{S}^\top }, \quad i=1,2,\dots,r,\end{equation} 
where $S$ is a complete discrete sampling and ${\bf S} = [S_1,\dots,S_r]$. The convergence rate is
\[\rho = 1 - \frac{1}{\kappa_{2,F}^2(A^{1/2}\mathbf{S})}  \overset{\eqref{eq:kappalower}}{=} 1 - \frac{\lambda_{\min}({\bf S}^\top  A {\bf S})}{\Tr{{\bf S}^\top  A {\bf S}}}.\]

Consider now the question of choosing the  matrix $\bf S$ in such a way that $\rho$ is as small as possible. Note that the  optimal choice is any $\bf S$ such that \[{\bf S}^\top  A {\bf S}= I.\] Indeed, then $\rho = 1-1/n$, and the lower bound~\eqref{eq:kappalower} is attained. For instance, the choice ${\bf S} = A^{-1/2}$ would be optimal. This means that in each iteration we would choose $S$ to be a  random column (or random column submatrix) of $A^{-1/2}$. Clearly, this is not a feasible choice, as we do not know the inverse of $A$. In fact, it is $A^{-1}$  which we are trying to find! However, this  leads to the following interesting observation: {\em the goals of finding the inverse of $A$ and of designing an optimal distribution $\cal D$ are in synchrony.} 

\subsection{The algorithm}

While we do not know $A^{-1/2}$, we can use the information of the iterates $\{X_k\}$ themselves to construct a good {\em adaptive} sampling. Indeed,  the iterates contain information about the inverse and hence we can use them to design a better sampling $S$. In order to do so, it will be useful to maintain a factored form of the iterates,  \begin{equation}\label{eq:X_k_98987H}X_k  = L_kL_k^\top ,\end{equation} where $L_k \in \R^{n\times n}$ is invertible. With this in place, let us choose $S$  to be a random column submatrix of $L_k$. In particular, let $C_1,C_2,\dots,C_r$  be nonempty subsets that form a partition of $\{1,2,\dots,n\}$, and at iteration $k$ choose \begin{equation}\label{eq:siug98sUJi}S = L_k I_{:C_i} \eqdef S_i, \end{equation} with probability $p_i$ given by \eqref{eq:ihsih0(0hI} for $i=1,2,\dots,r$. 
For simplicity, assume that $C_1=\{1,\dots,c_1\}$, $C_2 = \{c_1+1,\dots,c_2\}$ and so on, so that, by the definition of $\bf S$, we have
\begin{equation}\label{eq:S_9898y8}{\bf S} = [S_1, \dots, S_r] = L_k.\end{equation}
Note that now both $\bf S$ and $p_i$ depend on $k$.  The method described above satisfies the following recurrence.

\begin{theorem}\label{lem:9hs9hIKKK} Consider one step of the AdaRBFGS method described above. Then
\begin{equation}\label{eq:98h98hJJJ}\E{\norm{X_{k+1} -A^{-1}}_{F(A)}^2\, | \, X_k} \leq \left(1-\frac{\lambda_{\min }(AX_k )}{\Tr{AX_k}}\right) \norm{X_k -A^{-1}}_{F(A)}^2. \end{equation}
\end{theorem}
\begin{proof}
Using the same arguments as those in the proof of Theorem~\ref{theo:Enormconv}, we obtain
\begin{equation}\label{eq:Ec}
\E{\norm{X_{k+1} -A^{-1}}_{F(A)}^2\, | \, X_k} \leq \left(1-\lambda_{\min}\left(A^{-1/2}\E{Z \;|\; X_k}A^{-1/2}\right)\right) \norm{X_k -A^{-1}}_{F(A)}^2,
\end{equation}
where
   \begin{equation}\label{eq:Zxxx}
   Z \overset{\eqref{eq:Zxxx}}{=}AS_i(S_i^\top  A S_i)^{-1}S_i^\top  A.
\end{equation}
So, we only need to show that
\[\lambda_{\min}\left(A^{-1/2}\E{Z \;|\; X_k}A^{-1/2}\right) \geq \frac{\lambda_{\min }(AX_k )}{\Tr{AX_k}}.\]
Since $S$ is a complete discrete sampling, Proposition~\ref{prop:Ediscrete} applied to our setting says that
\begin{equation}\label{eq:hYhGfR6}\E{Z \;|\; X_k} =  A \mathbf{S} D^2 \mathbf{S}^\top   A, \end{equation}
where \begin{equation} \label{eq:Dxxx}
D~\eqdef~\mbox{Diag}\left( \sqrt{p_1}(S_1^\top  A  S_1)^{-1/2}, \ldots, \sqrt{p_r}(S_r^\top  A  S_r)^{-1/2}\right).\end{equation}
We now have
\begin{eqnarray*}
\lambda_{\min}\left(A^{-1/2}\E{Z \;|\; X_k}A^{-1/2}\right) &\overset{\eqref{eq:hYhGfR6}+\eqref{eq:S_9898y8}}{\geq}& \lambda_{\min}\left(A^{1/2}L_k L_k^\top  A^{1/2}\right) \lambda_{\min}(D^2) \\
& \overset{\eqref{eq:X_k_98987H}}{=} &   \frac{\lambda_{\min }(AX_k )}{\lambda_{\max}(D^{-2})}\\
&\overset{\eqref{eq:Dxxx}}{=}& \frac{\lambda_{\min }(AX_k )}{\max_{i} \lambda_{\max}(S_i^\top  A S_i) / p_i}\\
&\geq & \frac{\lambda_{\min }(AX_k )}{\max_{i} \Tr{S_i^\top  A S_i} / p_i}\\
&\overset{\eqref{eq:ihsih0(0hI} + \eqref{eq:S_9898y8}}{=} &
\frac{\lambda_{\min }(AX_k )}{\Tr{A X_k}},
\end{eqnarray*}
where in the second equality we have used the fact that the largest eigenvalue of  a block diagonal matrix is equal to the maximum of the largest eigenvalues of the blocks.
\end{proof}

If $X_k$ converges to $A^{-1}$, then necessarily the one-step rate of AdaRBFGS proved in Theorem~\ref{lem:9hs9hIKKK} asymptotically reaches the lower bound
\[\rho_k \eqdef  1 - \frac{\lambda_{\min}(AX_k)}{\Tr{A X_k}} \to 1-\frac{1}{n}.\]

In other words, as long as this method works, the convergence rate gradually improves, and  becomes asymptotically optimal and independent of the condition number. We leave a deeper analysis of this and other adaptive variants of the methods developed in this chapter to future work.

\subsection{Implementation}
To implement the AdaRBFGS update, we need to maintain the iterates $X_k$ in the factored form \eqref{eq:X_k_98987H}. Fortunately,  a factored form of the update~\eqref{eq:qunac} was introduced in~\cite{Gratton2011}, which we shall now describe and adapt to our objective. Assuming that $X_k$ is symmetric positive definite such that $X_k = L_k L_k^\top $, we shall describe how to obtain a corresponding factorization of $X_{k+1}$. Letting $G_k =  (S^\top L_k^{-T}L_k^{-1}S)^{1/2}$ and $R_k= (S^\top AS)^{-1/2}$, it can be verified through direct inspection~\cite{Gratton2011} that $X_{k+1} = L_{k+1} L_{k+1}^\top $, where
\begin{equation}\label{eq:Chol}L_{k+1} = L_k +S R_k \left(G_k^{-1}S^\top  L_k^{-T}-R_k ^\top  S^\top A L_k\right).
\end{equation}
If we instead of \eqref{eq:siug98sUJi} consider the more general update $S = L_k \tilde{S}$, where $\tilde{S}$ is chosen in an i.i.d.\ fashion from some fixed distribution $\cal \tilde{D}$, then 
\begin{equation}
 L_{k+1}  = L_k +L_k \tilde{S} R_k \left(({\tilde{S}}^\top  \tilde{S})^{-1/2}\tilde{S}^\top  -R_k^\top  \tilde{S}^\top  L_k^\top  A L_k\right). \label{eq:Cholimprov}
\end{equation}
The above can now be implemented efficiently, see Algorithm~\ref{alg:AdaRBFGS}. 

\begin{algorithm}[!h]
\begin{algorithmic}[1]
\State \textbf{input:} symmetric positive definite matrix $A$
\State \textbf{parameter:} ${\cal \tilde{D}}$ = distribution over random matrices with $n$ rows 
\State \textbf{initialize:} pick invertible $L_0 \in \R^{n\times n}$

\For {$k = 0, 1, 2, \dots$}
	\State Sample an independent copy $\tilde{S} \sim {\cal \tilde{D}}$
	\State Compute $S = L_k \tilde{S}$ 
	\Comment $S$ is sampled adaptively, as it depends on $k$
	\State Compute $R_k = (\tilde{S}^\top  A \tilde{S})^{-1/2}$ 
	\State $L_{k+1} = L_k +S R_k \left((\tilde{S}^\top  \tilde{S})^{-1/2}\tilde{S}^\top  -R_k^\top  S^\top  A L_k\right)$
	\Comment Update the factor
\EndFor
\State \textbf{output:} $X_k = L_k L_k^\top $
\end{algorithmic}

\caption{ Adaptive Randomized BFGS  (AdaRBFGS)}
\label{alg:AdaRBFGS}
% {}
\end{algorithm}

In Section~\ref{sec:numerical} we test two variants based on~\eqref{eq:Cholimprov}. The first is  the \emph{AdaRBFGS\_gauss} update, in which the entries of $\tilde{S}$ are standard Gaussian. The second is \emph{AdaRBFGS\_cols}, where $\tilde{S} =I_{:C_i}$, as described above, and  $|C_i| =q$ for all $i$ for some $q$.

\section{Numerical Experiments} \label{sec:numerical}

Given the demand for approximate inverses of positive definite matrices in preconditioning and in variable metric methods in optimization, and the author's own interest in the aforementioned applications, we restrict our test to inverting positive definite matrices. 

We test four iterative methods for inverting matrices. This rules out the all-or-nothing direct methods such as Gaussian elimination of LU based methods.

For our tests we use two variants of  Algorithm~\ref{alg:AdaRBFGS}: AdaRBFGS\_gauss, where $\tilde{S}\in\R^{n\times q}$ is a normal Gaussian matrix, and  
AdaRBFGS\_cols, where $\tilde{S}$ consists of a collection of $q$ distinct coordinate vectors in $\R^n$, selected uniformly at random. 
%For all our algorithms we use $q=\sqrt{n}$ and $X_0=I$ in all tests.
At each iteration the AdaRBFGS methods compute the inverse of a small matrix $S^\top AS$ of dimension $q\times q$.
To invert this matrix we use MATLAB's inbuilt \texttt{inv} function, which uses $LU$ decomposition or Gaussian elimination, depending on the input. Either way, \texttt{inv} costs $O(q^3).$ For simplicity, we selected $q=\sqrt{n}$ in all our tests. 
%so that the cost of inverting $O(q^3)$ and forming $O(n^2q)$ the matrix $S^\top AS$ are of the same order.

We compare our method to two well established and competitive methods, the \emph{Newton-Schulz method}~\cite{Schulz1933}
and the global self-conditioned Minimal Residual (\emph{MR}) method~\cite{Chow1998}.  The Newton-Schulz method arises 
from applying the Newton-Raphson method to solve the equation $X^{-1}=A$, which gives
\begin{equation}\label{eq:Newton-Schulz}
X_{k+1} = 2X_{k} -X_{k} A X_{k}.
\end{equation}
The MR method was designed to calculate  approximate inverses, and it does so by minimizing the norm of the residual along the preconditioned residual direction, that is
\begin{equation} \label{eq:MRdef}
\norm{I-AX_{k+1}}_F^2 = \min_{\alpha \in \R} \left\{\norm{I-AX}_F^2 \quad \mbox{subject to} \quad X = X_k +\alpha X_k(I-AX_k)\right\},
\end{equation}
see~\cite[chapter 10.5]{Saad2003} for a didactic introduction to MR methods.
The resulting iterates of the MR method are given by
\begin{equation} \label{eq:MRiter}
 X_{k+1} =X_k +\frac{\Tr{R_k^\top AX_kR_k}}{\Tr{(AX_kR_k)^\top AX_kR_k}}X_kR_k,
\end{equation}
where $R_k = I-AX_k$.
%The MR method is part of a group of methods referred to as approximate inverse preconditioners~\cite{Saad2003}, as they are commonly used in preconditioning linear systems. 

We perform two sets of tests. On the first set, we choose a different starting matrix for each method which is optimized, in some sense, for that method. We then compare the empirical convergence of each method, including the time taken to calculate $X_0$. In particular, 
  the  Newton-Schulz is only guaranteed to converge for an initial matrix $X_0$ such that $\rho(I-X_0A)< 1$. Indeed, the Newton-Schulz method did not converge in most of our experiments when $X_0$ was not carefully chosen according to this criteria.
  To remedy this, we choose $X_0 = 0.99 \cdot A^\top /\rho^2(A)$ for the Newton-Schulz method\footnote{During a revision of this thesis, it was kindly pointed out by the committee that setting $X_0 = \alpha A^\top$ with $\alpha = 1/\norm{A}_F^2$ would suffice to guarantee convergence and avoid the cost of calculating $\rho(A).$ Though we should note that the cost in calculating $\rho(A)$ was partly offset by using a C++ implementation, while all other methods were coded in native MATLAB.
%   any constant $\alpha $ such that $\alpha $
    }, so that $\rho(I-X_0A)<1$ is satisfied. To compute $\rho(A)$ we used the inbuilt MATLAB function \texttt{normest} which is coded in C++. While for MR we followed the suggestion in~\cite{Saad2003} and used the projected identity for the initial matrix $X_0 = (\Tr{A}/\Tr{AA^\top }) \cdot I.$ For our AdaRBFGS methods we simply used $X_0 =I$, as this worked well in practice.
  
  In the second set of tests, which we relegate to the Appendix of this chapter,  we compare the empirical convergence of the methods starting from the same matrix, namely the identity matrix $X_0 = I$. 
%
%  In the second set of tests, 
% The literature for the MR and Newton-Schulz methods suggest particular starting matrices $X_0$ that results in the best performance for each method. 
% 
%In contrast with the AdaRBFGS algorithms, which are globally convergent, the  Newton-Schulz and MR are only guaranteed to converge for an initial matrix $X_0$ such that $\rho(I-X_0A)< 1$. 
%Consequentially, the Newton-Schulz method with $X_0 =I$ in particular, did not converge in all our experiments. 

We run each method until the relative residual $\norm{I-AX_{k}}_{F}/\norm{I-AX_0}_{F}$ is below $10^{-2}.$  All experiments were performed  and run in MATLAB R2014b. To appraise the performance of each method we plot the relative residual against time taken and against the number of floating point operations (\emph{flops}).

% MATLAB's in-built function \texttt{norm()} uses Arnoldi Iterations to obtain %an estimate of the largest singular value.

\subsection{Experiment 1: synthetic matrices}
First we compare the four methods on synthetic matrices generated 
using the \texttt{rand} function as follows:  $A~=~\bar{A}^\top \bar{A}$ where $\bar{A}=$\texttt{rand}$(n)$. The resulting matrix $A$ is positive definite with high probability. To appraise the difference in performance of the methods as the dimension of the problem grows, we tested for $n=1000$, $2000$ and $5000.$ As the dimension grows, only the two variants of the AdaRBFGS method are able to reach the $10^{-2}$ desired tolerance in a reasonable amount time and number of flops (see Figure~\ref{fig:pdsynth}).

\begin{figure}
    \centering
    \begin{subfigure}[t]{0.80\textwidth}
        \centering %left, bottom,right, top
\includegraphics[width =  \textwidth, trim= 40 300 50 300, clip ]{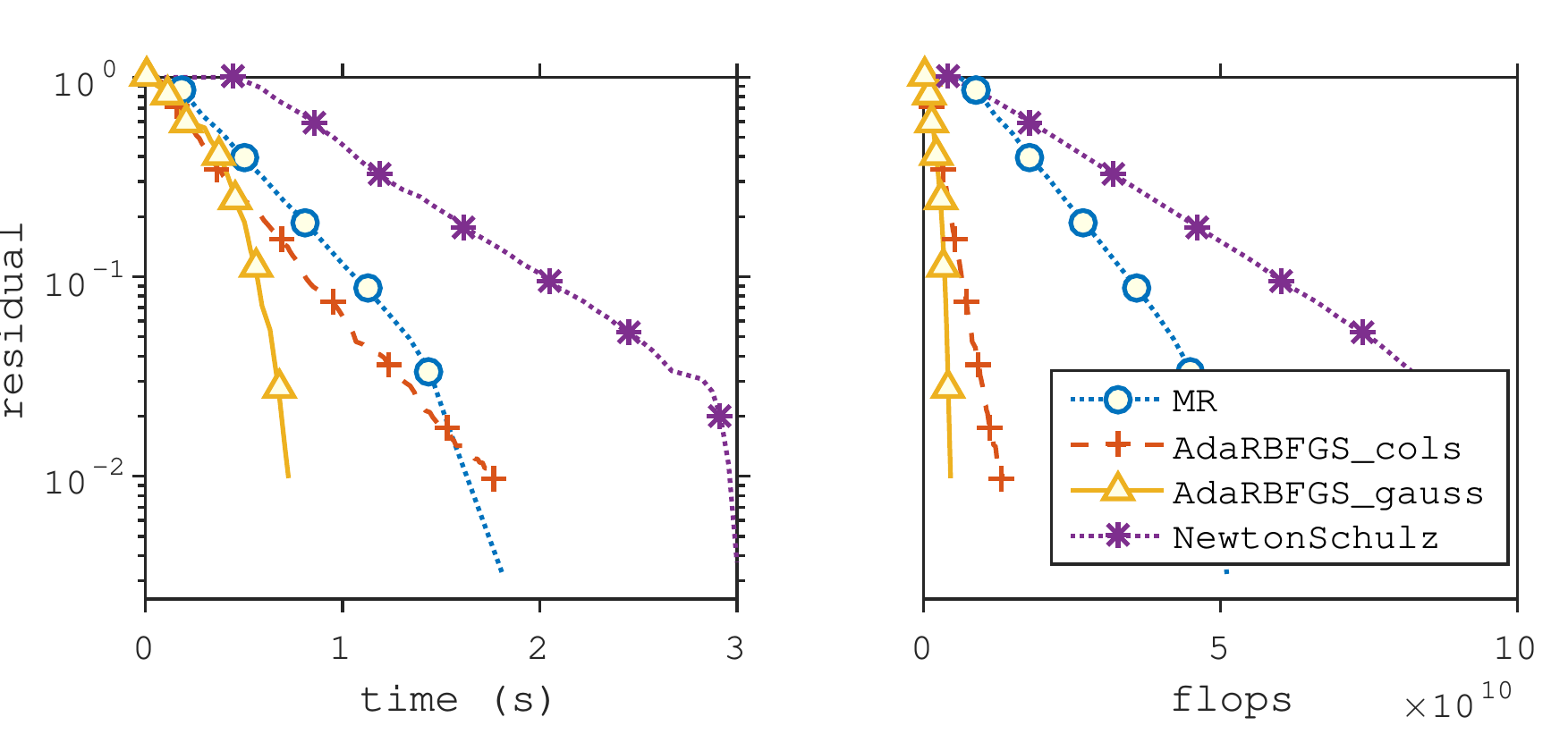}
        \caption{\texttt{rand} with $n=10^4$ }\label{fig:rand}
    \end{subfigure}
    % \hspace{0.05\textwidth}
         \begin{subfigure}[t]{0.80\textwidth}
        \centering
\includegraphics[width =  \textwidth, trim= 40 300 50 300, clip ]{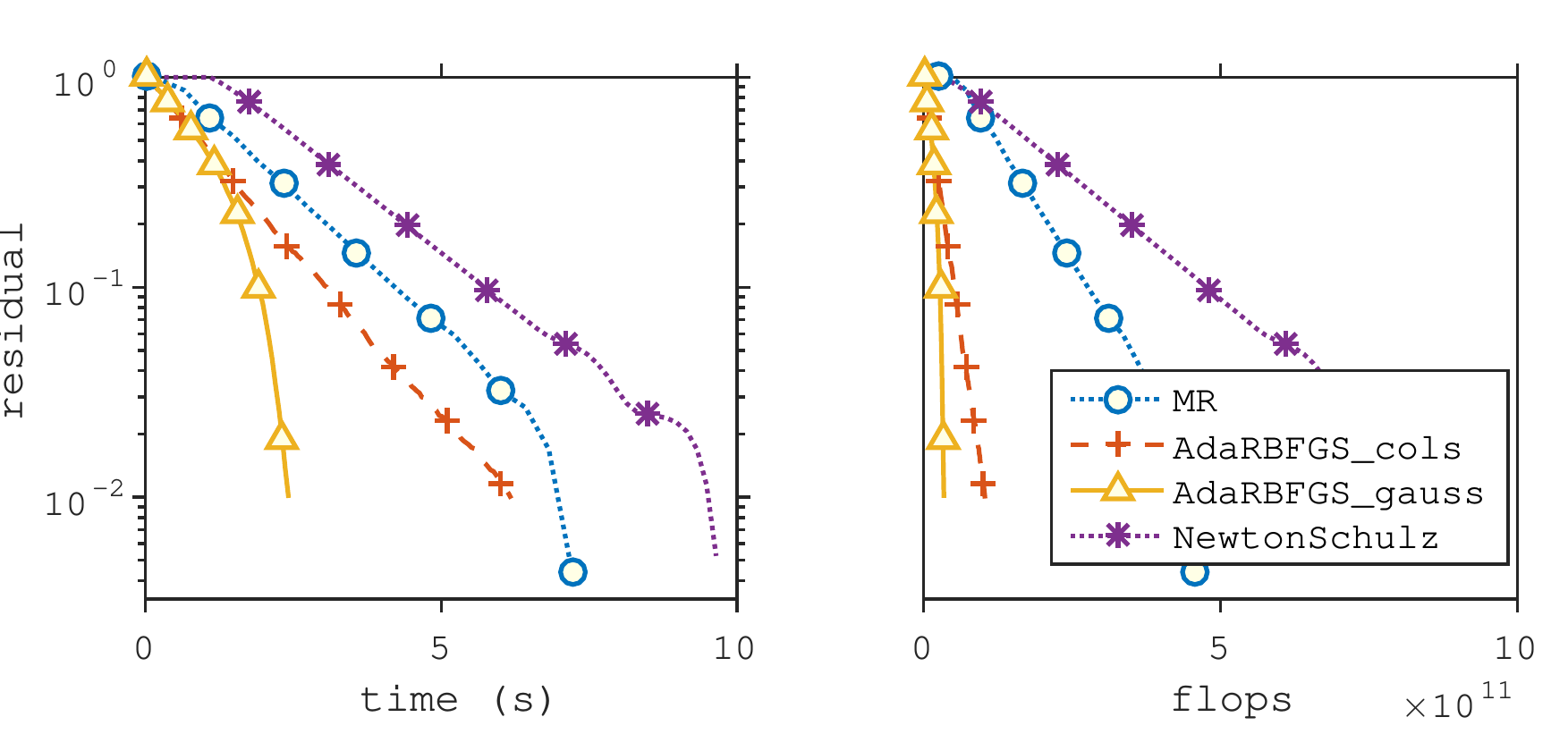}
        \caption{\texttt{rand} with $n=2 \cdot 10^4$ }\label{fig:rand2}
    \end{subfigure}\\
  %   \hspace{0.01\textwidth}
         \begin{subfigure}[t]{0.80\textwidth}
        \centering
\includegraphics[width =  \textwidth, trim= 40 300 50 300, clip ]{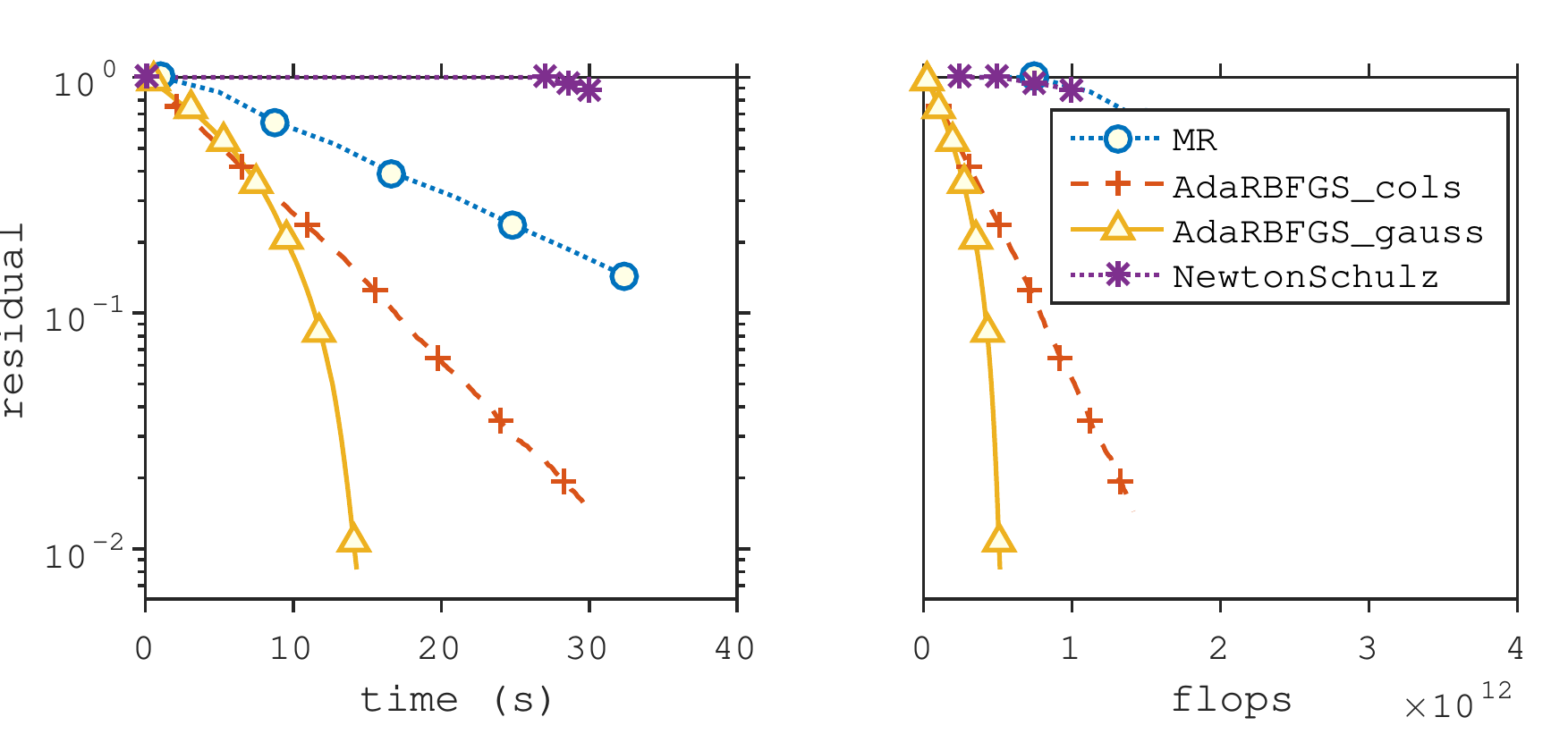}
        \caption{\texttt{rand} with $n=5 \cdot 10^4$ }\label{fig:rand3}
    \end{subfigure}%
%%%%%%%%%%%%%%%%%%%%%%%%%%%%%%%%%%%%%%%
%    \begin{subfigure}[t]{0.30\textwidth}
%     %   \centering
%\includegraphics[width =  \textwidth, trim= 70 230 50 230, clip ]{randn-20000-20000}
%        \caption{\texttt{rand} with $n=2 \cdot 10^4$}\label{fig:randH}
%    \end{subfigure}%
    \caption{Synthetic MATLAB generated problems.  Uniform random matrix $A~=~\bar{A}^\top \bar{A}$ where $\bar{A}=$\texttt{rand}$(n).$ 
}\label{fig:pdsynth}
\end{figure}
%Even when excluding the time taken to start-up the Newton-Schulz method, the RBFGS method is considerably faster. In Figure~\ref{fig:randH} we increase the dimension to $2\cdot 10^4$, on which RBFGS method converged in $800$ seconds. The Newton-Schulz method exceeded 30min, thus we omit the Newton-Schulz plot.
% This last test illustrates how iterative methods are indispensable on large dimensional problems. If a direct method were used, such as Gaussian-elimination, a total in order of $n^3 = 8\cdot 10^{12}$ flops would be  required. 
 %Using MATLAB's in-built {\tt svd} function took $566.74$ seconds to invert %this matrix.

%In Figure~\ref{fig:sprandn} we compare the methods on a sparse linear system generated  using the MATLAB sparse random matrix function \texttt{sprandn}($m,n$,{\tt density,rc}), where {\tt density} is the percentage of nonzero entries and {\tt rc} is the reciprocal of the condition number. 

\subsection{Experiment 2: LIBSVM matrices} 

Next we invert the Hessian matrix $\nabla^2 f(x)$ of four ridge-regression problems of the form
 \begin{equation}\label{eq:ridgeMatrix}
\min_{x\in \R^n}f(x)\eqdef \frac{1}{2}\norm{Ax-b}_2^2 + \frac{\lambda}{2} \norm{x}_2^2,\quad  \quad\nabla^2 f(x) = A^\top A+\lambda I,
\end{equation}
using data from LIBSVM~\cite{Chang2011}, see Figure~\ref{fig:LIBSVM}. We use $\lambda =1$ as the regularization parameter. On the two problems of smaller dimension, \texttt{aloi} and \texttt{protein}, the four methods have a similar performance, and encounter the inverse in a few seconds. On the two larger problems, \texttt{gisette-scale} and \texttt{real-sim}, the two variants of AdaRBFGS significantly outperform the MR and the Newton-Schulz method.
\begin{figure}
    \centering
    \begin{subfigure}[t]{0.65\textwidth}
        \centering
\includegraphics[width =  \textwidth, trim= 40 300 50 300, clip ]{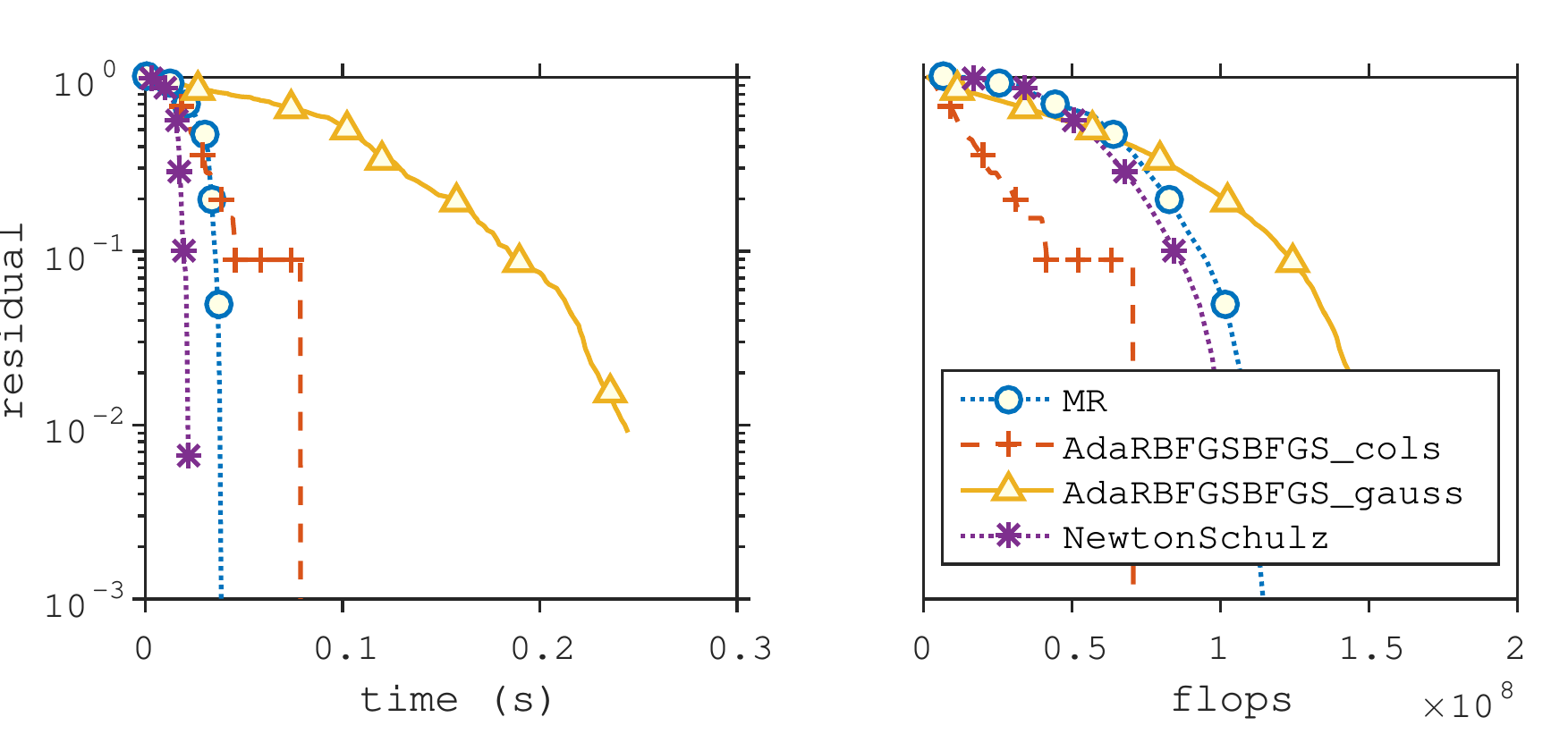}
        \caption{\texttt{aloi}}
    \end{subfigure}\\
%%%%%%%%%%%%%%%%%%%%%%%%%%%%%%%%%   
    \begin{subfigure}[t]{0.65\textwidth}
        \centering
\includegraphics[width =  \textwidth, trim=40 300 50 300, clip ]{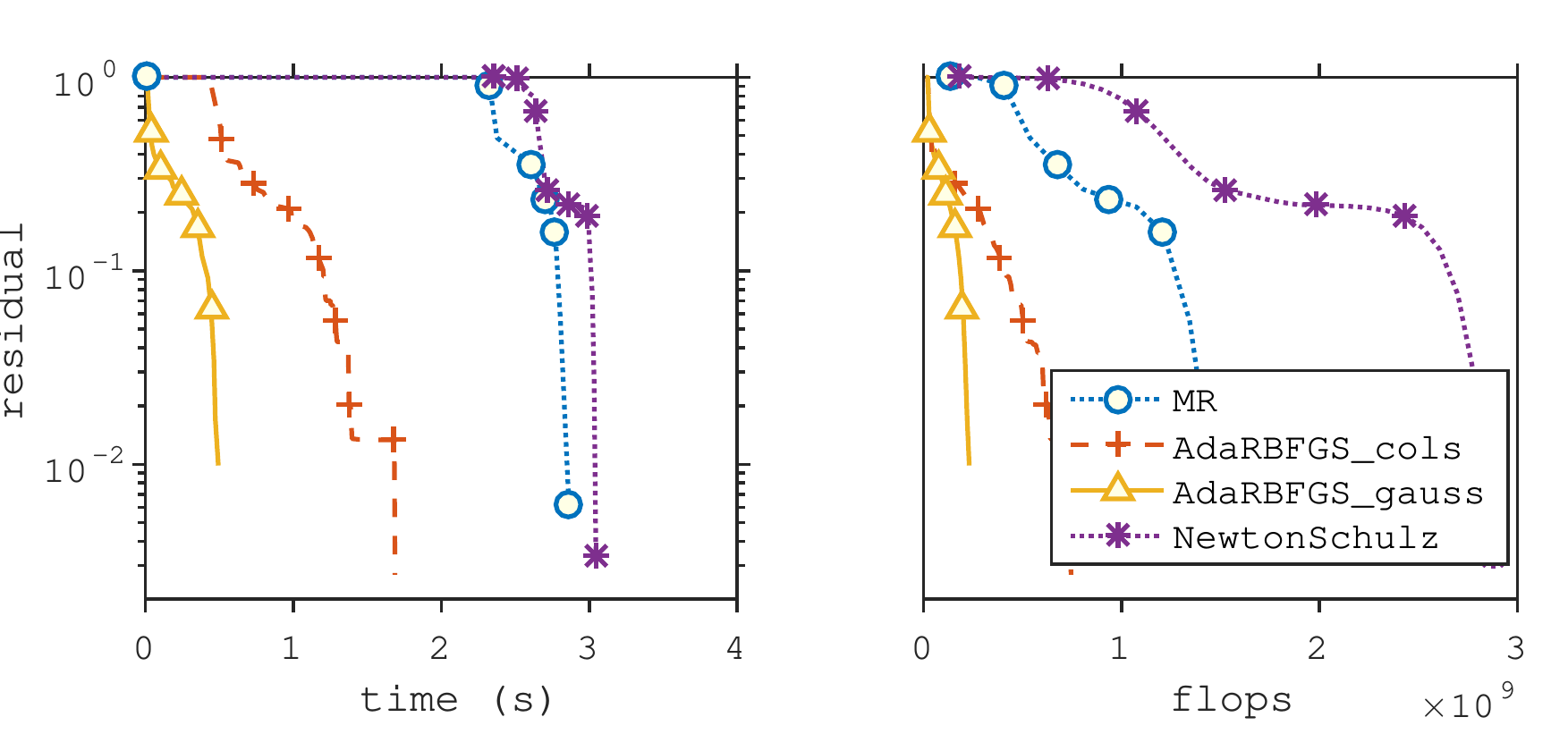}
        \caption{\texttt{protein}}
    \end{subfigure}\\
%%%%%%%%%%%%%%%%%%%%%%%%%%%%%%%%%   
        \begin{subfigure}[t]{0.65\textwidth}
        \centering
\includegraphics[width =  \textwidth, trim=40 300 50 300, clip ]{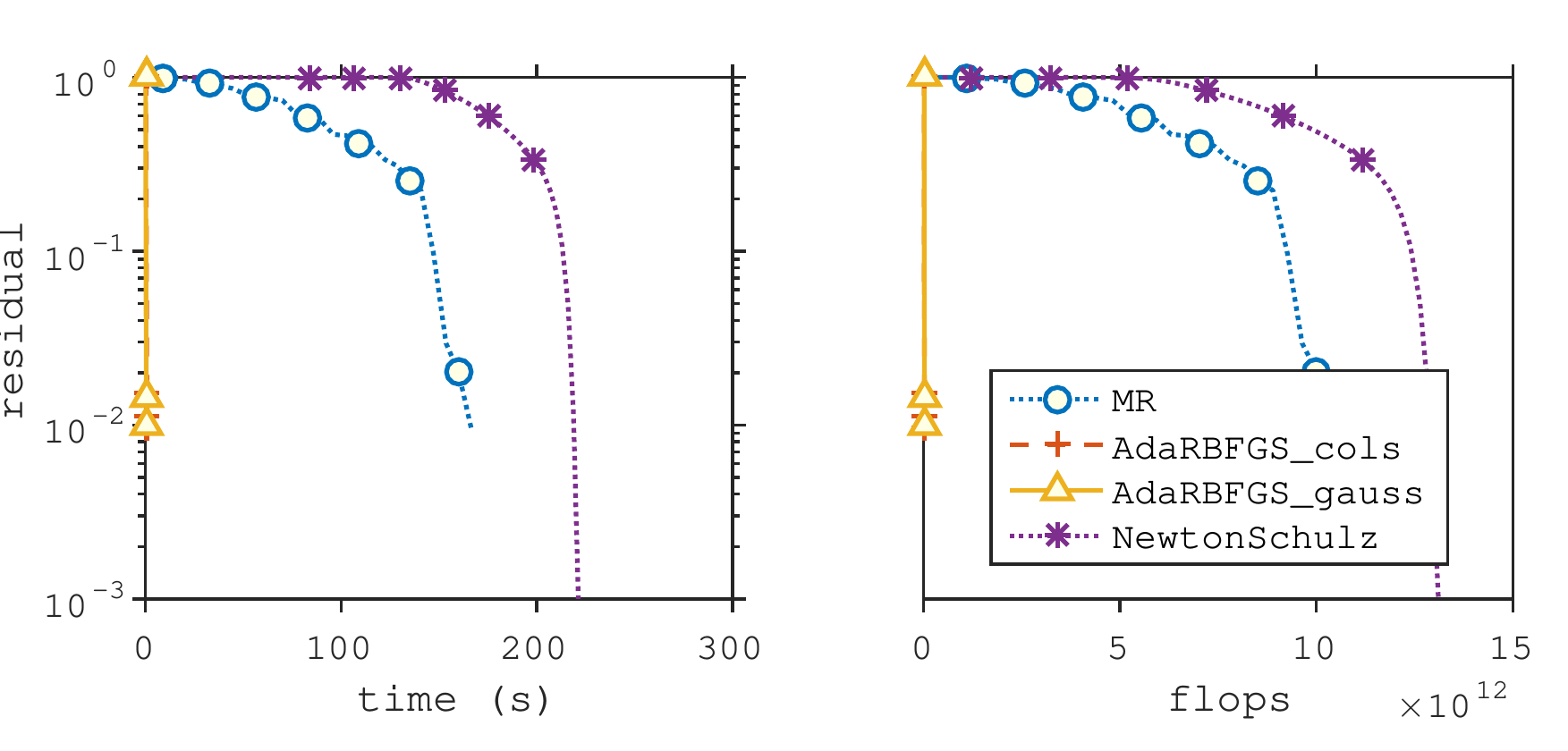}
        \caption{\texttt{gisette\_scale}}
    \end{subfigure}\\
%%%%%%%%%%%%%%%%%%%%%%%%%%%%%%%%%   
    \begin{subfigure}[t]{0.65\textwidth}
        \centering
\includegraphics[width =  \textwidth, trim= 40 300 50 300, clip ]{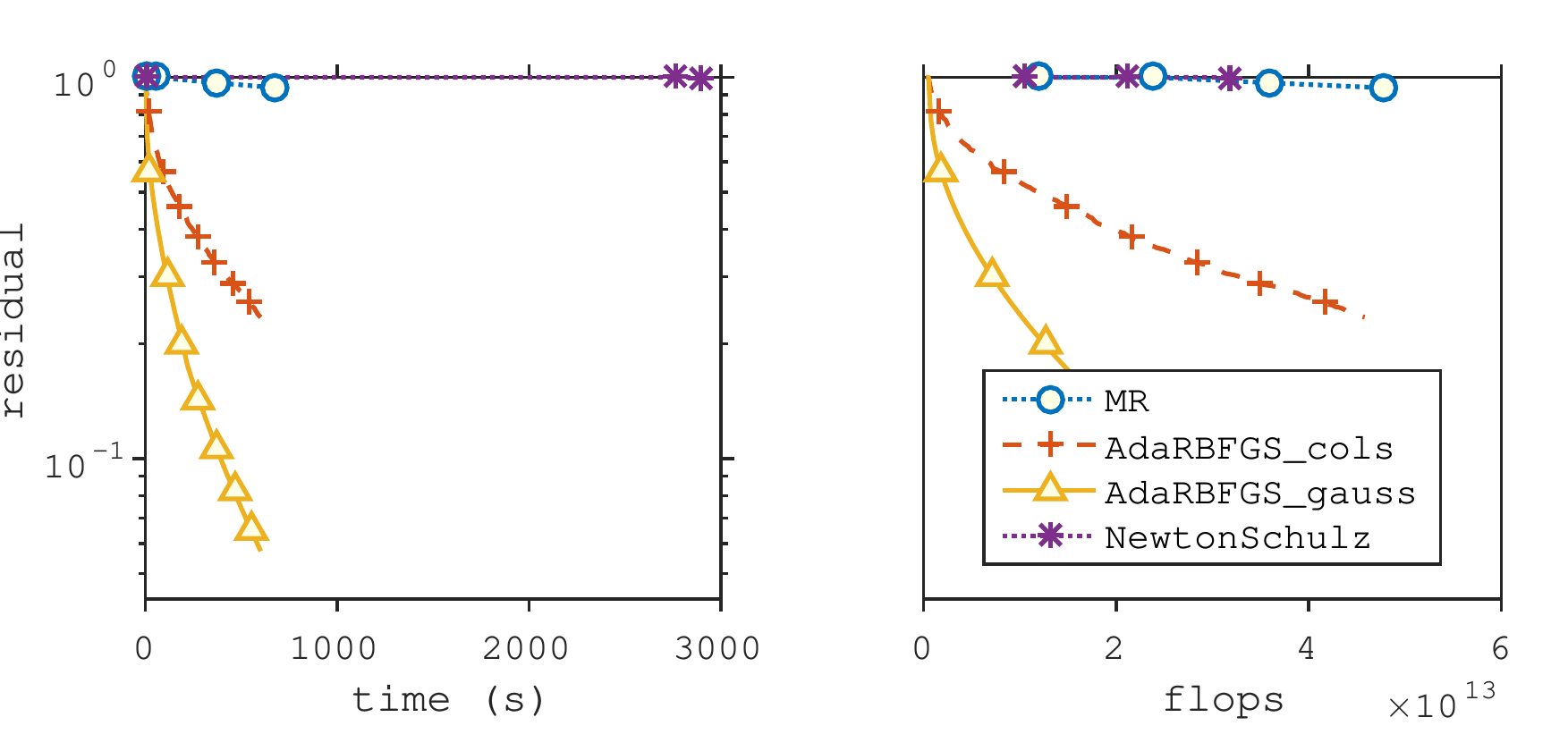}
        \caption{\texttt{real\_sim}}
    \end{subfigure}
    \caption{The performance of Newton-Schulz, MR, AdaRBFGS\_gauss and AdaRBFGS\_cols methods on the Hessian matrix of four LIBSVM test problems: (a) \texttt{aloi}: $(m;n)=(108,000;128)$ (b) \texttt{protein}: $(m; n)=(17,766; 357)$ (c)  \texttt{gisette\_scale}: $(m;n)=(6000; 5000)$ (d) \texttt{real-sim}: $(m;n)=(72,309; 20,958)$.} \label{fig:LIBSVM}
\end{figure}

\subsection{Experiment 3: UF sparse matrices}

For our final batch of tests, we calculate an approximate inverse of several sparse matrices from the Florida sparse matrix collection~\cite{Davis:2011}\footnote{One should never calculate an exact inverse of a sparse matrix in practice since the resulting matrix is dense. But given that the AdaBFGS method performs low rank updates, this allows us to \emph{implicitly} form an approximate inverse by storing the updates, and not explicitly forming the matrix. We explore this in the paper~\cite{GowerGold2016}}. We have selected six problems from six different applications, so that the set of matrices display a varied sparsity pattern and  structure, see Figures~\ref{fig:UF1} and~\ref{fig:UF2}.

On the matrix \texttt{Bates/Chem97ZtZ} of moderate size, the four methods perform well, with the Newton-Schulz method converging first in time and AdaRBFGS\_cols first in flops. While on the matrices of larger dimension, the two variants of AdaRBFGS converge much faster, often orders of magnitude before the MR and Newton-Schulz method reach the required precision.

\begin{figure}
    \centering
\begin{subfigure}[t]{0.65\textwidth}
        \centering
\includegraphics[width =  \textwidth, trim=40 300 50 300, clip ]{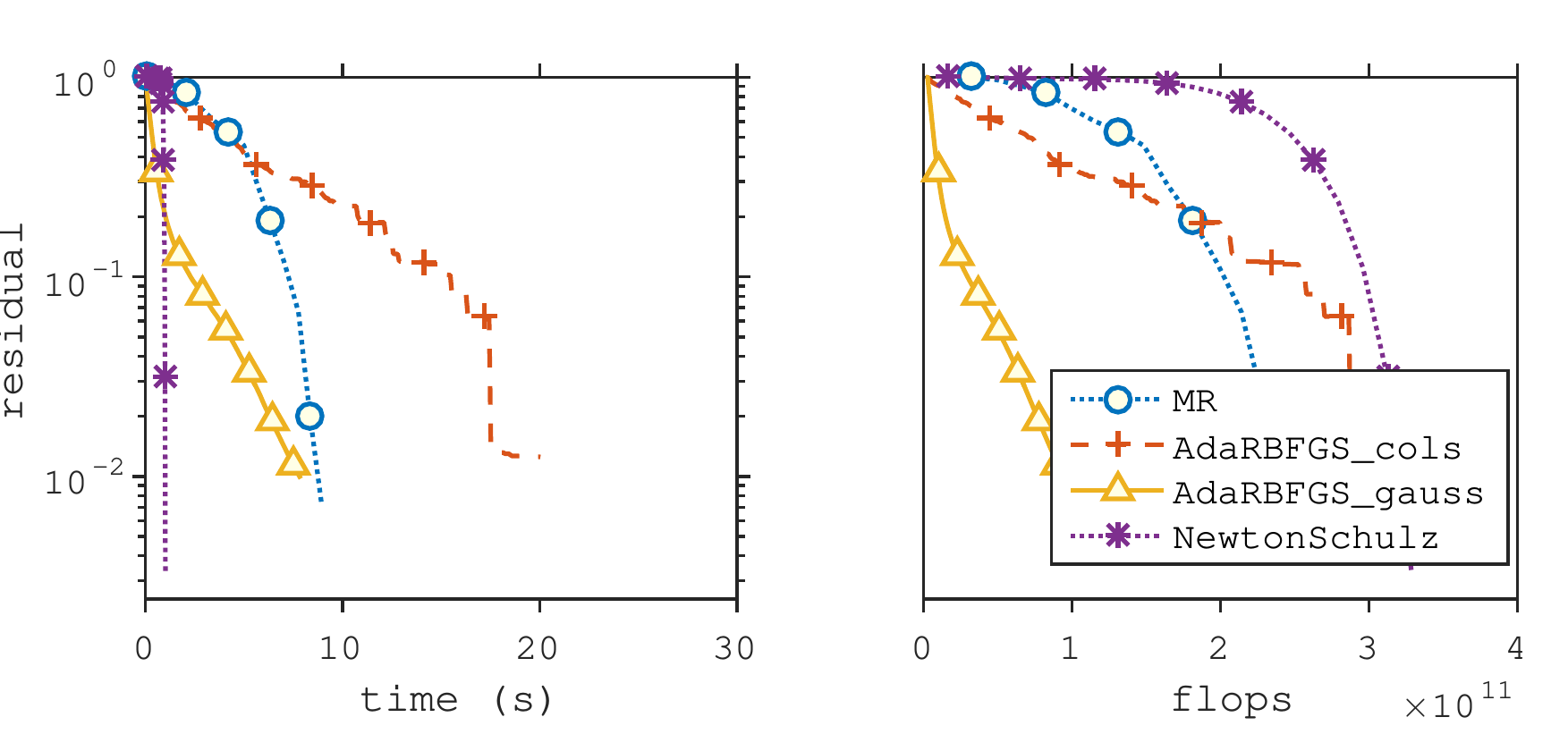}
        \caption{\texttt{Bates/Chem97ZtZ}}
\end{subfigure}\\ %\hspace{0.2\textwidth}    
\begin{subfigure}[t]{0.65\textwidth}
        \centering
\includegraphics[width =  \textwidth, trim= 40 300 50 300, clip ]{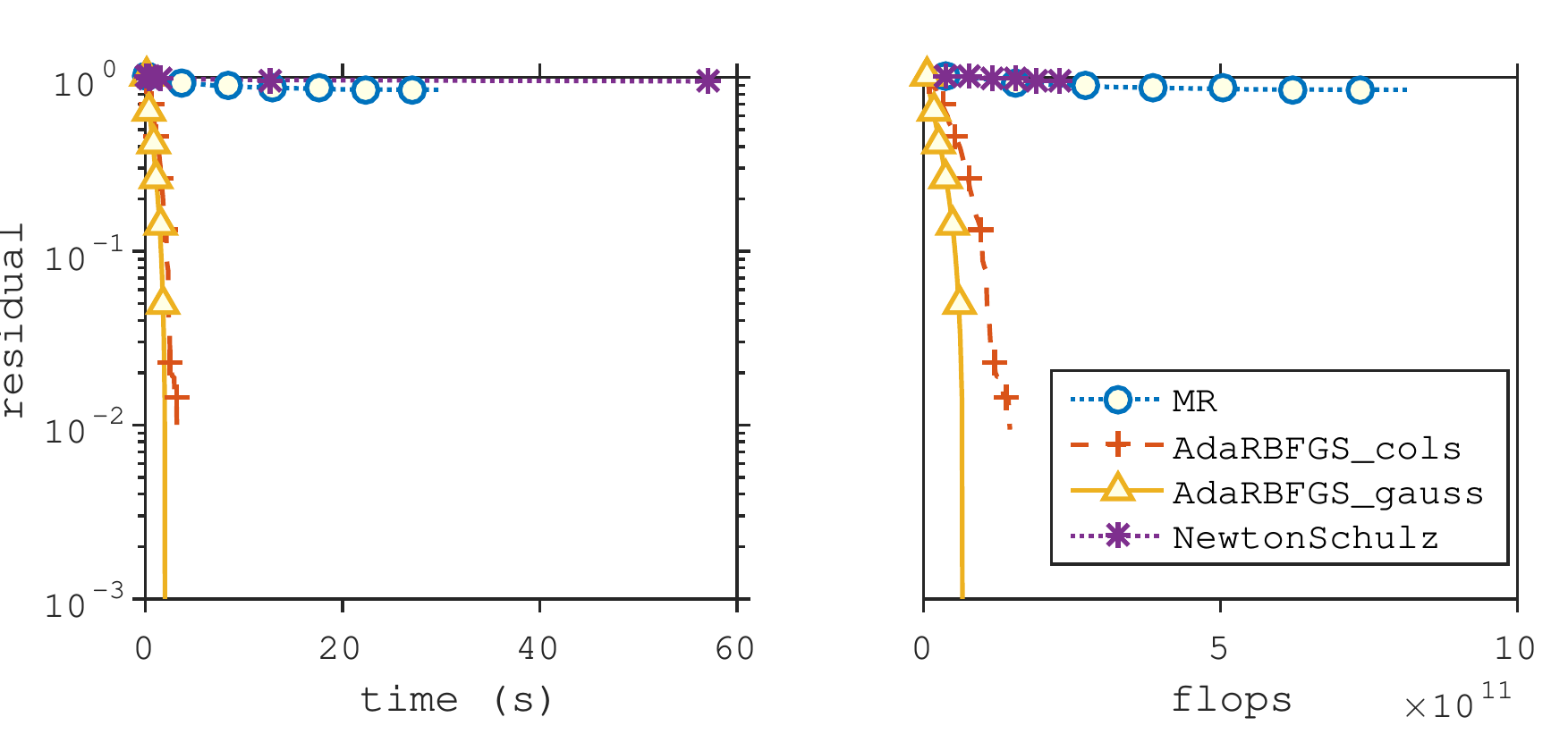}
        \caption{\texttt{FIDAP/ex9}}
\end{subfigure} \\%
%%%%%%%%%%%%%%%%%%%%%%%%%%%%%%%%%   
\begin{subfigure}[t]{0.65\textwidth}
        \centering
\includegraphics[width =  \textwidth, trim=40 300 50 300, clip ]{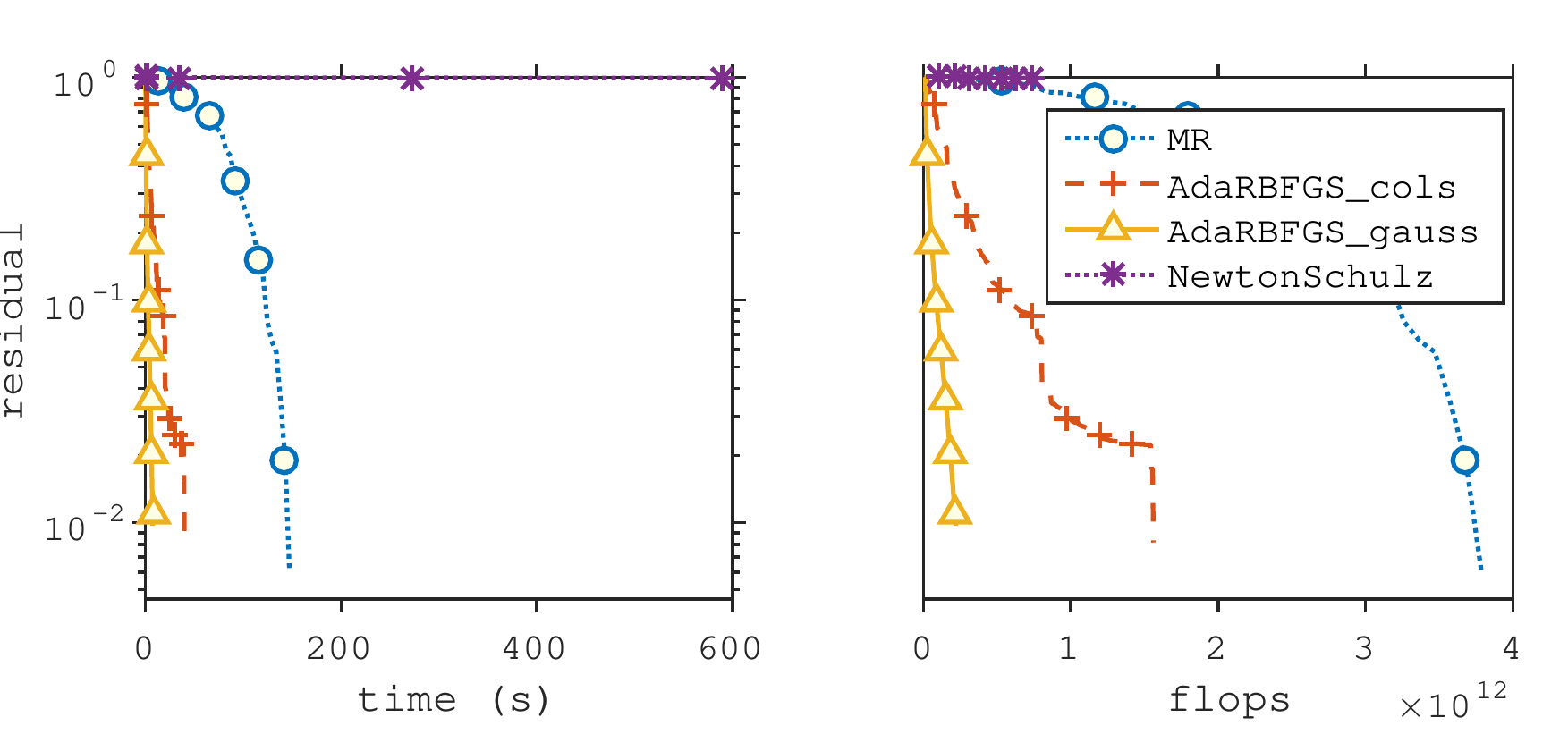}
        \caption{\texttt{Nasa/nasa4704}}
\end{subfigure} \\ %\hspace{0.2\textwidth}
\begin{subfigure}[t]{0.65\textwidth}
        \centering
\includegraphics[width =  \textwidth, trim=40 300 50 300, clip ]{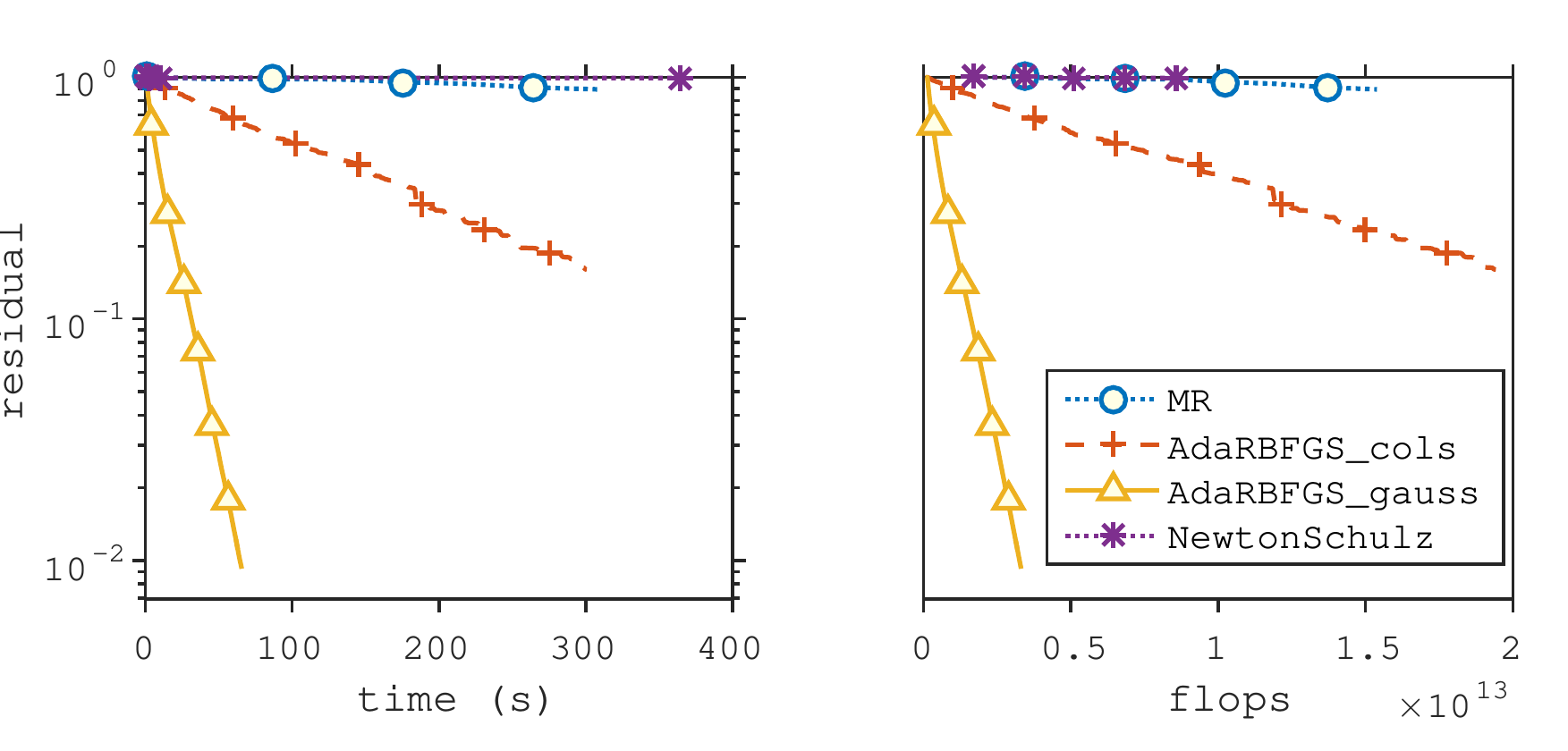}
        \caption{\texttt{HB/bcsstk18}}
\end{subfigure}%
%%%%%%%%%%%%%%%%%%%%%%%%%%%%%%%%%   
  \caption{The performance of Newton-Schulz, MR, AdaRBFGS\_gauss and AdaRBFGS\_cols on 
(a) \texttt{Bates-Chem97ZtZ}: $n= 2\,541$,   
  (b) \texttt{FIDAP/ex9}: $n = 3,\,363 $, (c) \texttt{Nasa/nasa4704}: $n= 4\,,704$, 
  (d) \texttt{HB/bcsstk18}: $n=11,\,948$.  } \label{fig:UF1}
  \end{figure}

\begin{figure}
    \centering
%%%%%%%%%%%%%%%%%%%%%%%%%%%%%%%%%   
\begin{subfigure}[t]{0.65\textwidth}
        \centering
\includegraphics[width =  \textwidth, trim=40 300 50 300, clip ]{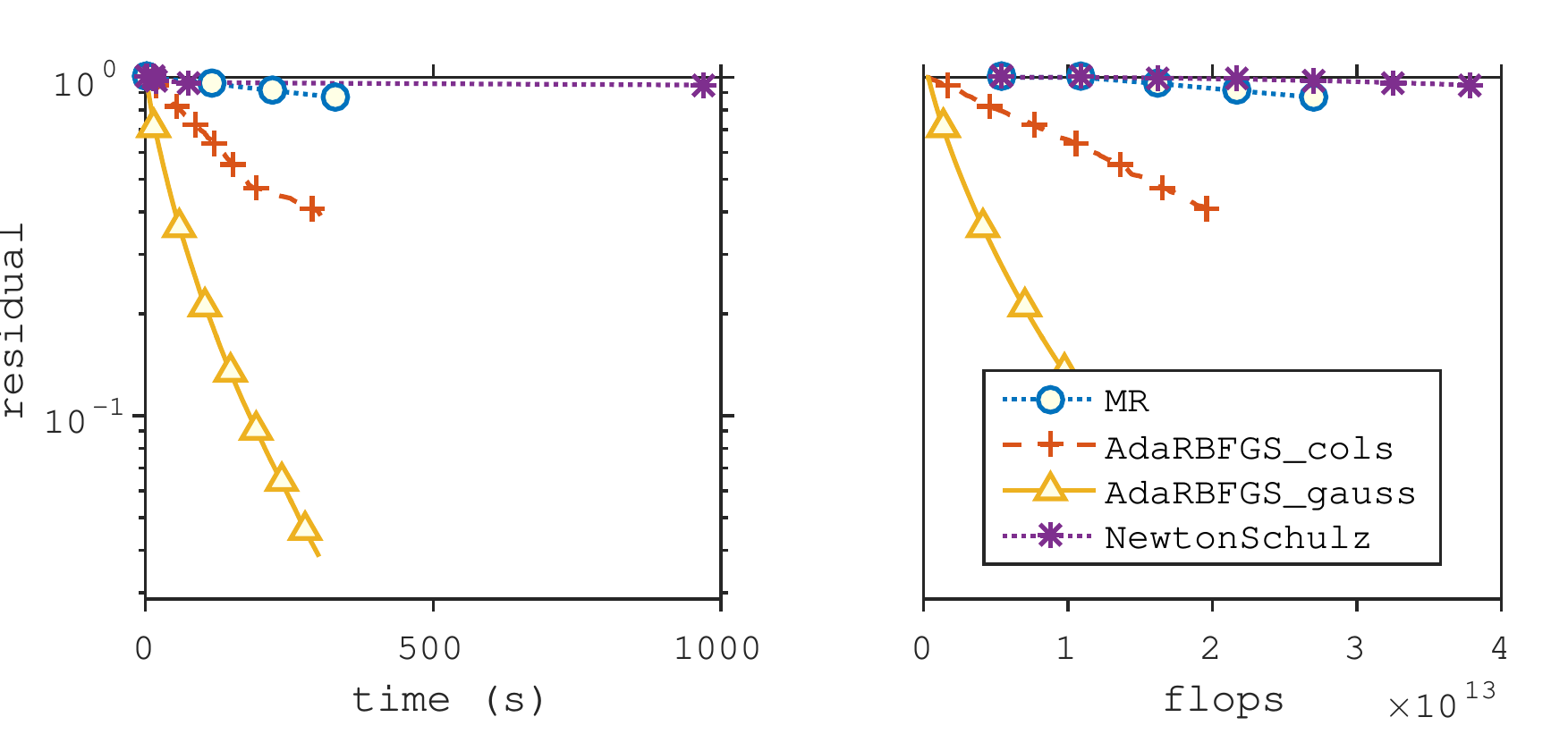}
        \caption{\texttt{Pothen/bodyy4}}
\end{subfigure} \\%\hspace{0.2\textwidth}
\begin{subfigure}[t]{0.65\textwidth}
        \centering
\includegraphics[width =  \textwidth, trim=40 300 50 300, clip ]{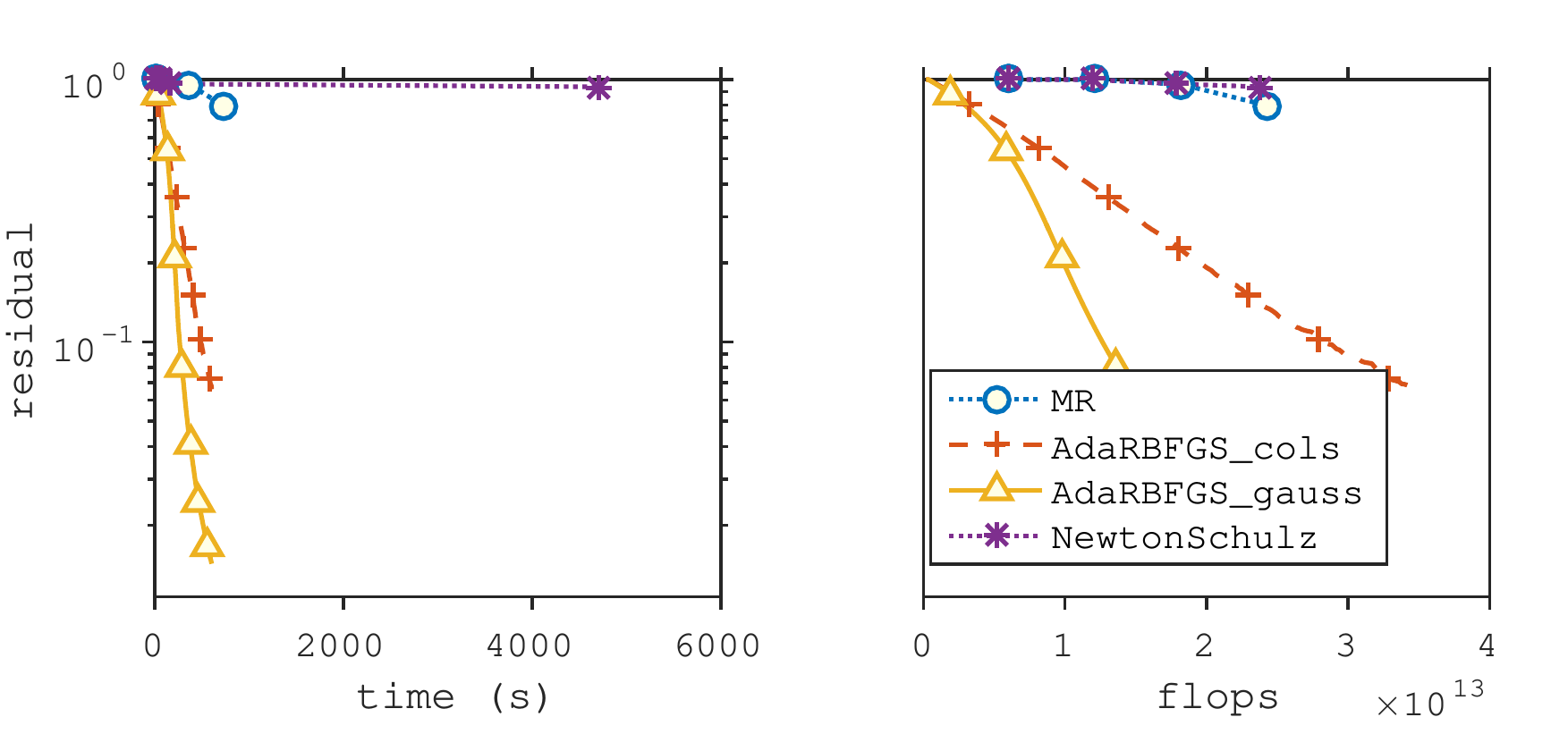}
        \caption{\texttt{ND/nd6k}}
\end{subfigure}\\%
\begin{subfigure}[t]{0.65\textwidth}
        \centering
\includegraphics[width =  \textwidth, trim=40 300 50 300, clip ]{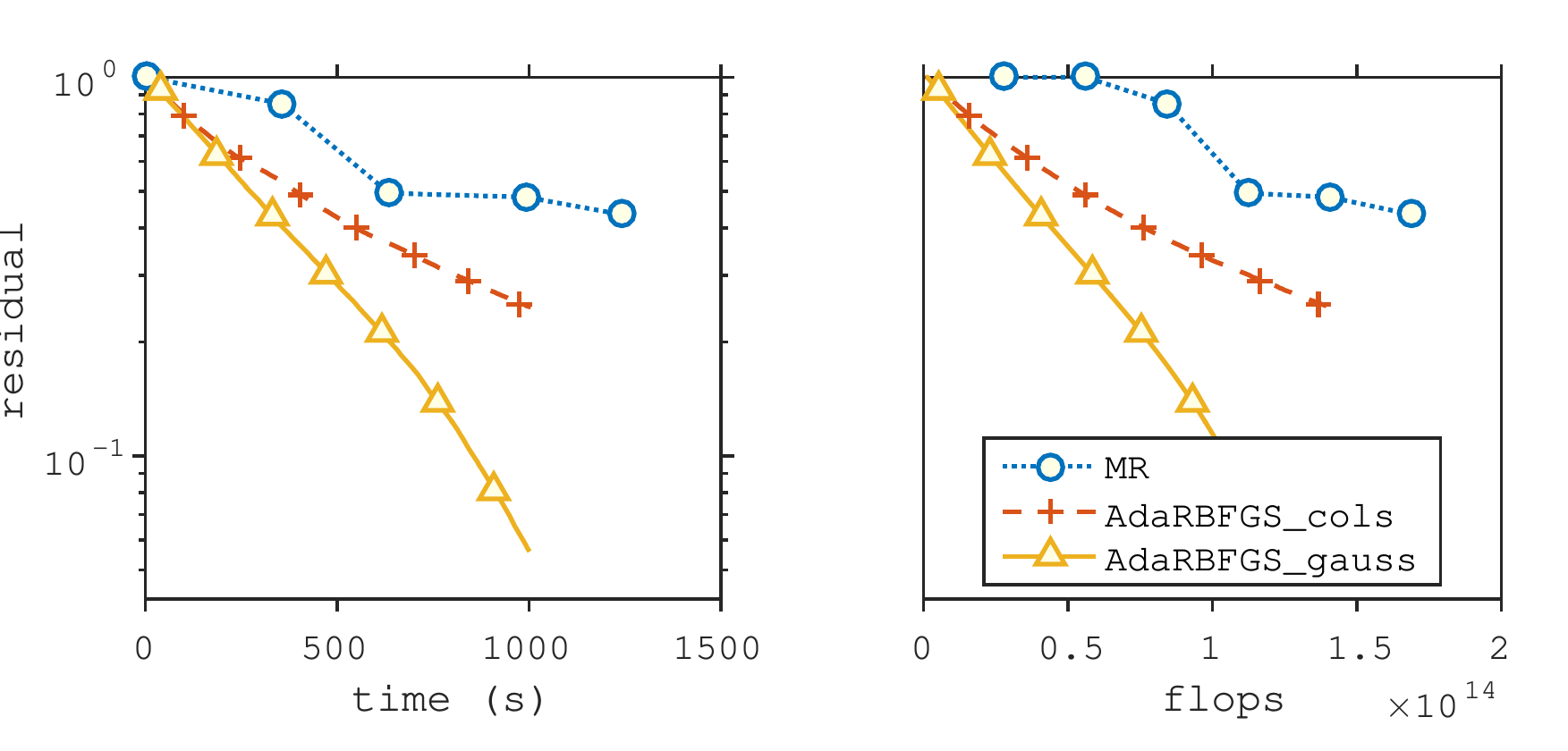}
        \caption{\texttt{GHS\_psdef/wathen100}}
\end{subfigure}%
  \caption{The performance of Newton-Schulz, MR, AdaRBFGS\_gauss and AdaRBFGS\_cols on 
  (a) \texttt{Pothen/bodyy4}: $n =17,\,546$ (b) \texttt{ND/nd6k}: $n=18,\,000$ (c) \texttt{GHS\_psdef/wathen100}: $n = 30, \,401$.  } \label{fig:UF2}
  \end{figure}

 The significant difference between the performance of the methods on large scale problems can be, in part, explained by their iteration cost. The iterates of the Newton-Schulz and MR method compute $n\times n$ matrix-matrix products. While the cost of an iteration of the AdaRBFGS methods is dominated by the cost of a $n\times n$ matrix by $n\times q$ matrix product. As a result, and because we set $q = \sqrt{n},$ this is difference of $n^3$ to $n^{2+1/2}$ in iteration cost, which clearly shows on the larger dimensional instances.

 \subsection{Conclusion of numeric experiments}
 Through our extensive numeric experiments, it is clear that the AdaRBFGS methods  are highly efficient at calculating a low precision approximate inverse of a positive definite matrix. Furthermore, in many of these experiments the 
 Newton-Schulz and MR method suffer from an initially slow convergence, particularly so on large-scale problems, see Figures~\ref{fig:pdsynth}c,~\ref{fig:LIBSVM}d,~\ref{fig:UF1}b,
 ~\ref{fig:UF1}c and~\ref{fig:UF1}d  for example. But after a sufficient number of iteration, the asymptotic second order convergence rate of the Newton-Schulz and MR method sets in, see Figures~\ref{fig:pdsynth}a,~\ref{fig:pdsynth}b,~\ref{fig:LIBSVM}a ,~\ref{fig:LIBSVM}b and~\ref{fig:LIBSVM}c for example. 
% Thus when an approximation of the inverse with relative residual below \%1 is required, the  Newton-Schulz and MR methods can outperform the AdaBFGS methods.

 These experiments also indicate that the AdaRBFGS method enjoys super linear convergence as can be seen, for example, in Figure~\ref{fig:pdsynth} the residual decreases superlinearly in both time and flops. Yet we have still to prove that the AdaRBFGS methods are even linearly convergent. Thus it is now an open question whether or not the AdaRBFGS convergent linearly or superlinearly.

\section{Summary}

%{\bf Extend to calculating low rank approximations to matrices\cite{Deshpande2006}
%
%Explore connections between approximate inverse preconditioners and quasi-Newton.
%
%Add a symmetry constraint to the Constrain-Approximate formulation for new symmetric methods (and perhaps new quasi-Newton methods)?
%
%solving matrix equations..
%
%
%} 

We developed a family of stochastic methods for iteratively inverting matrices, with a specialized variant for asymmetric, symmetric and positive definite matrices. The methods have two dual viewpoints, a sketch-and-project viewpoint which is an extension of the least-change formulation of the quasi-Newton methods,  
 and a constrain-and-approximate viewpoint which is related to the approximate inverse preconditioning (API) methods. The equivalence between these two viewpoints reveals a new connection between the quasi-Newton and the API methods, which were previously considered to be unrelated. 
   
     Under mild conditions, we prove convergence rates through two different perspectives, the convergence of the expected norm of the error, and the norm of the expected error. Our convergence theorems are general enough to accommodate discrete samplings and continuous  samplings, though we only explore discrete sampling here in more detail.  
       
         For discrete samplings, we determine a probability distribution for which the convergence rates are equal to a scaled condition number, and thus are easily interpretable. Furthermore, for discrete sampling, we determining a practical optimized sampling distribution, that is obtained by minimizing an upper bound on the convergence rate. We develop new randomized block variants of the quasi-Newton updates, including the BFGS update, complete with convergence rates, and provide new insights into these methods using our dual viewpoint. 
           
            For positive definite matrices, we develop an Adaptive Randomized BFGS methods (AdaRBFGS), which in large-scale numerical experiments, prove to be orders of magnitude faster (in time and flops) then the self-conditioned minimal residual method and the Newton-Schulz method. In particular, only the AdaRBFGS methods are able to approximately invert  
the $20,958 \times 20,958$ ridge regression matrix based on the \texttt{real-sim} data set in reasonable time and flops. 

   This work opens up many possible venues for future work, including,  developing methods that use continuous random sampling, implementing a limited memory approach akin to the LBFGS~\cite{Nocedal1980} method, which could maintain an operator that serves as an approximation to the inverse.  As recently shown in~\cite{Gower2015c}, an analogous method applied to linear systems converges with virtually no assumptions on the system matrix. This can be extended to calculating the pseudoinverse matrix, something we leave for future work. 

\section{Appendix: Optimizing an Upper Bound on the Convergence Rate}
  
\begin{lemma}\label{lem:fracsum} Let $a_1,\dots,a_r$ be positive real numbers. Then
\[\left[ \frac{\sqrt{a_1}}{\sum_{i=1}^r \sqrt{a_i}},\ldots, \frac{\sqrt{a_n}}{\sum_{i=1}^r \sqrt{a_i}} \right] =\arg\min_{p\in \Delta_r} \sum_{i=1}^r \frac{a_i}{p_i}.\]
\end{lemma}
\begin{proof}
Incorporating the constraint $\sum_{i=1}^r p_i =1$ into the Lagrangian we have 
\[\min_{p\geq 0} \sum_{i=1}^r \frac{a_i}{p_i} + \mu\sum_{i=1}^r (p_i-1),\]
where $\mu \in \R.$ Differentiating in $p_i$ and setting to zero, then isolating $p_i$ gives
\begin{equation} \label{eq:muaipi}
 p_i = \sqrt{\frac{a_i}{\mu}}, \quad \mbox{for }i=1,\ldots r.\end{equation}
Summing over $i$ gives
\[ 1 = \sum_{i=1}^r \sqrt{\frac{a_i}{\mu}} \quad \Rightarrow \quad \mu = \left(\sum_{i=1}^r \sqrt{a_i}\right)^2.\]
Inserting this back into~\eqref{eq:muaipi} gives $p_i = \sqrt{a_i}/\sum_{i=1}^r \sqrt{a_i}.$
\end{proof}

\section{Appendix: Numerical Experiments with the Same Starting Matrix}\label{app:numerics}
We now investigate the empirical convergence of the methods MR, AdaRBFGS\_cols and AdaRGFBS\_gauss when initiated with the same starting matrix $X_0 = I,$ see Figures~\ref{fig:LIBSVM2},~\ref{fig:UFapp1} and~\ref{fig:UFapp2}. We did not include the Newton-Schultz method in these figures because it diverged on all experiments when initiated from $X_0 =I.$
Again we observe that, as the dimension grows, only the two variants of the AdaRBFGS are capable of inverting the matrix to the desired $10^{-2}$ precision in a reasonable amount of time. Furthermore, the AdaRBFGS\_gauss variant had the overall best best performance.

\begin{figure}
    \centering
    \begin{subfigure}[t]{0.65\textwidth}
        \centering
\includegraphics[width =  \textwidth, trim= 40 300 50 300, clip ]{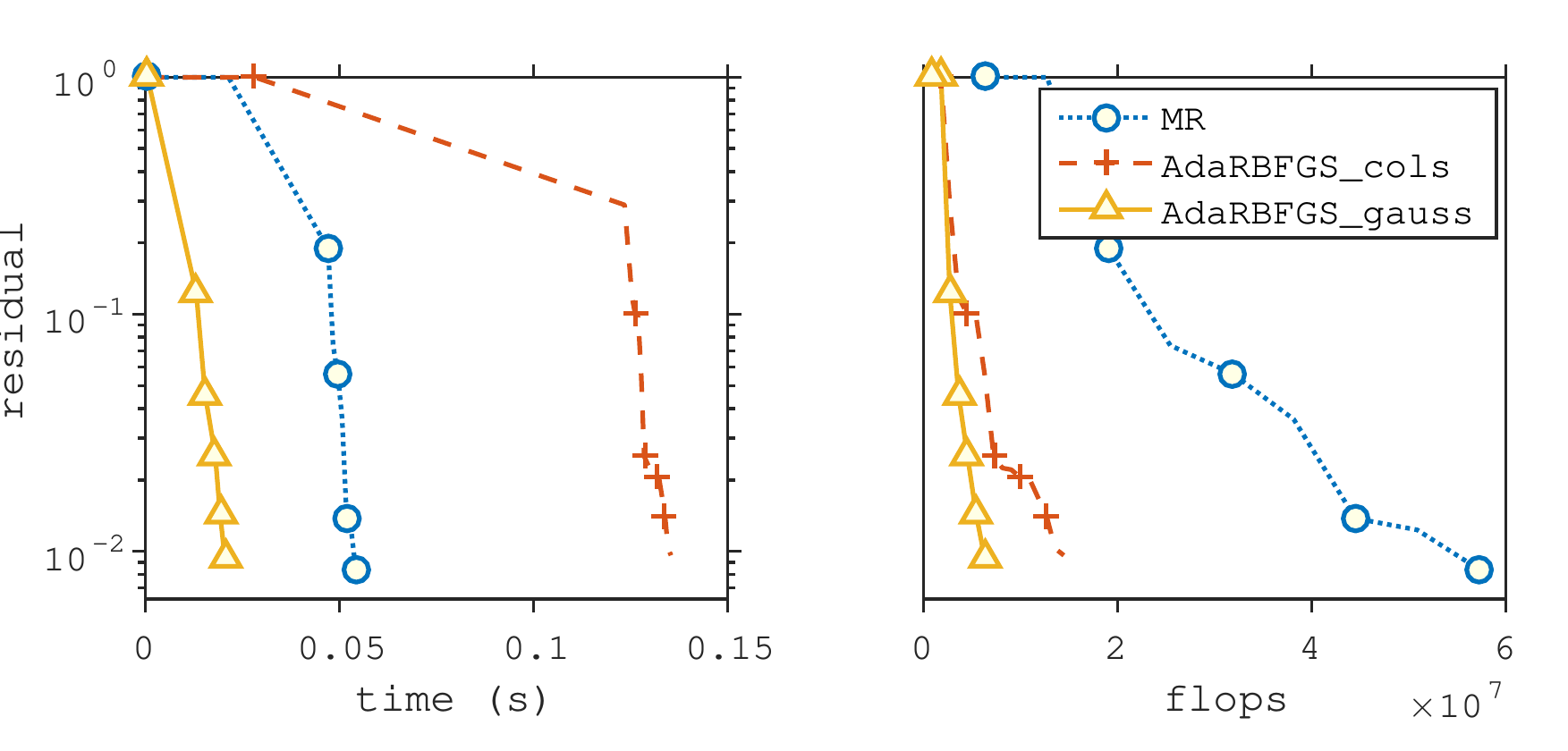}
        \caption{\texttt{aloi}}
    \end{subfigure} \\%
%%%%%%%%%%%%%%%%%%%%%%%%%%%%%%%%%   
    \begin{subfigure}[t]{0.65\textwidth}
        \centering
\includegraphics[width =  \textwidth, trim=40 300 50 300, clip ]{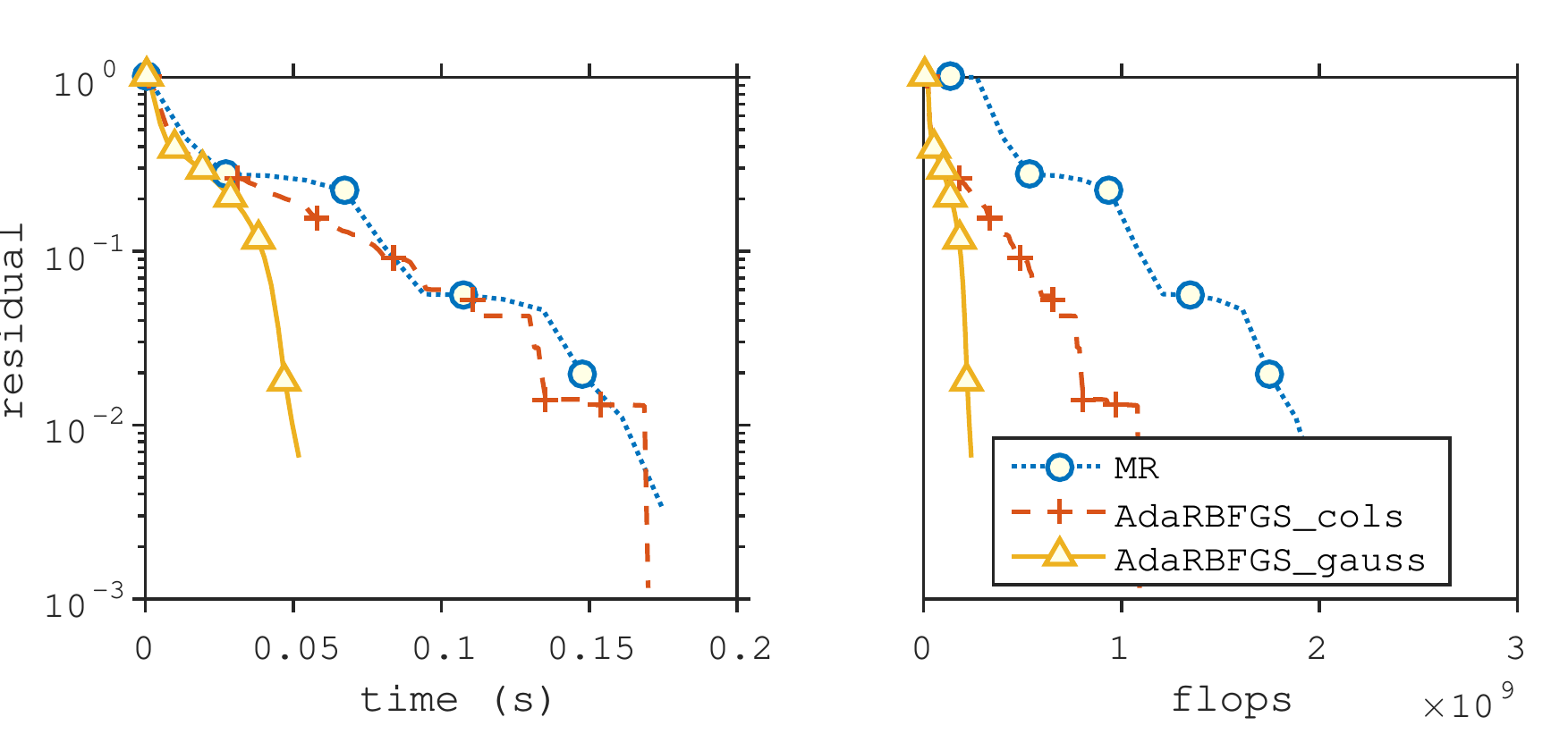}
        \caption{\texttt{protein}}
    \end{subfigure}
%%%%%%%%%%%%%%%%%%%%%%%%%%%%%%%%%   
        \begin{subfigure}[t]{0.65\textwidth}
        \centering
\includegraphics[width =  \textwidth, trim=40 300 50 300, clip ]{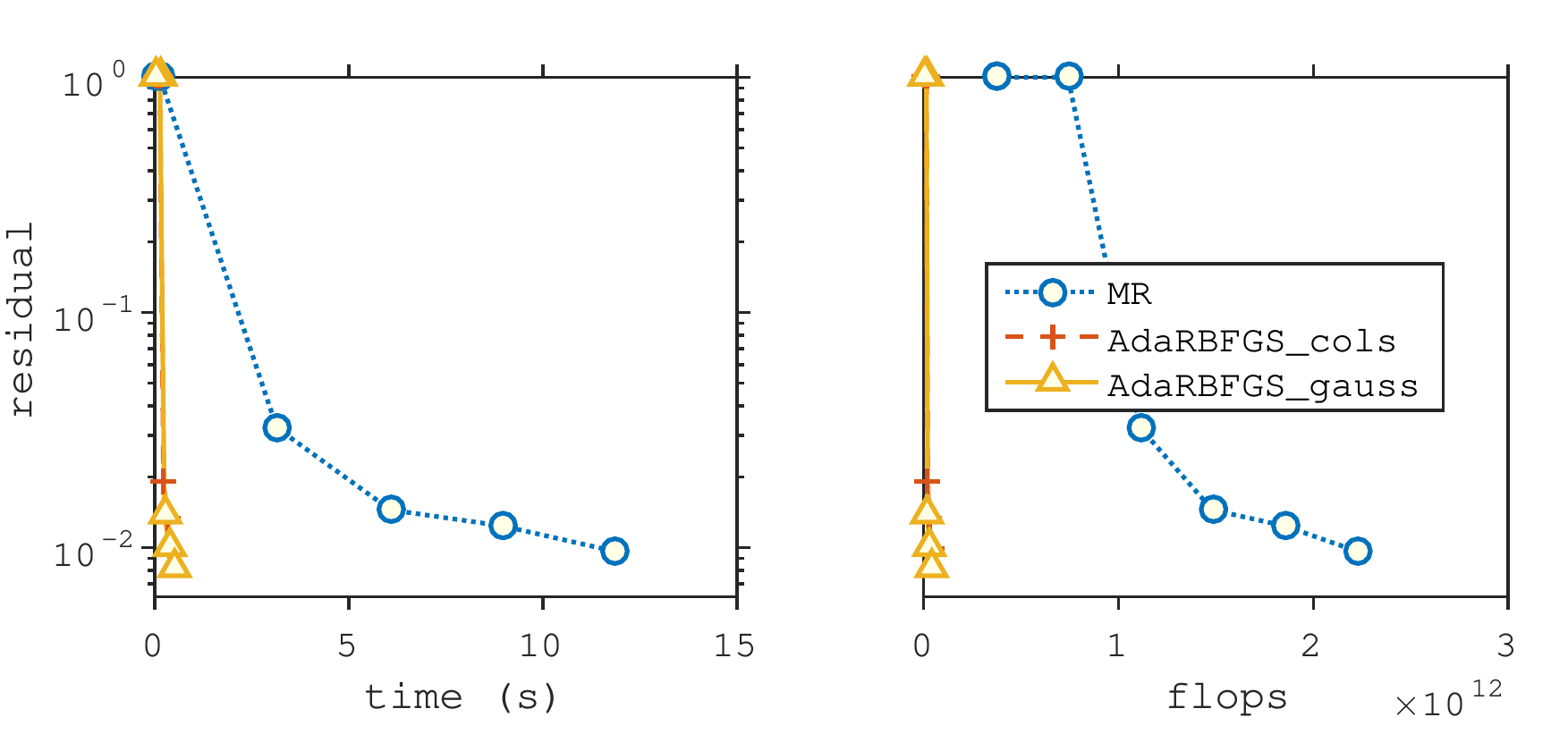}
        \caption{\texttt{gisette\_scale}}
    \end{subfigure} \\%
%%%%%%%%%%%%%%%%%%%%%%%%%%%%%%%%%   
    \begin{subfigure}[t]{0.65\textwidth}
        \centering
\includegraphics[width =  \textwidth, trim= 40 300 50 300, clip ]{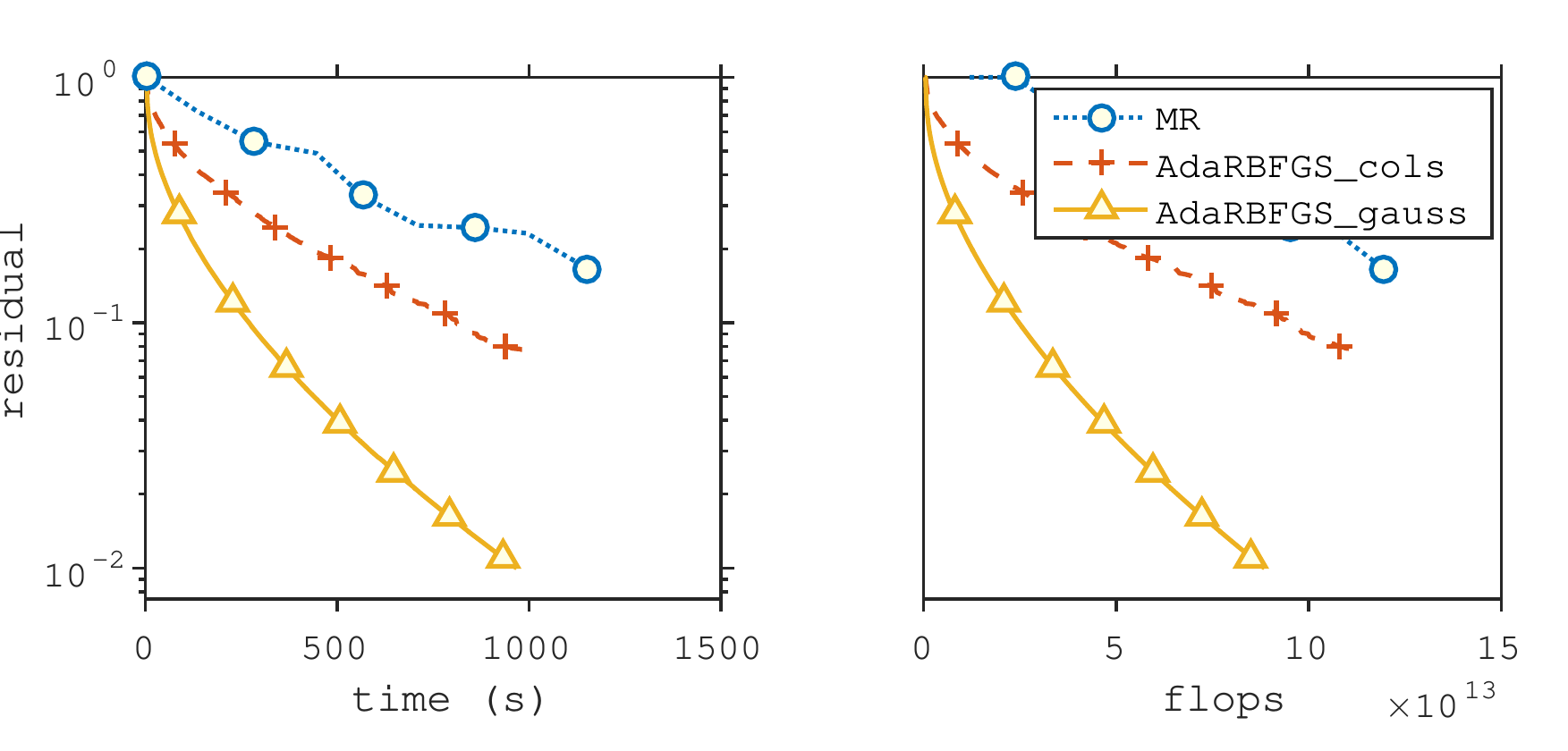}
        \caption{\texttt{real\_sim}}
    \end{subfigure}
    \caption{The performance of Newton-Schulz, MR, AdaRBFGS\_gauss and AdaRBFGS\_cols methods on the Hessian matrix of four LIBSVM test problems: (a) \texttt{aloi}: $(m;n)=(108,000;128)$ (b) \texttt{protein}: $(m; n)=(17,766; 357)$ (c)  \texttt{gisette\_scale}: $(m;n)=(6000; 5000)$ (d) \texttt{real-sim}: $(m;n)=(72,309; 20,958)$. The starting matrix $X_0=I$ was used for all methods.} \label{fig:LIBSVM2}
\end{figure}

\begin{figure}
    \centering
\begin{subfigure}[t]{0.65\textwidth}
        \centering
\includegraphics[width =  \textwidth, trim=40 300 50 300, clip ]{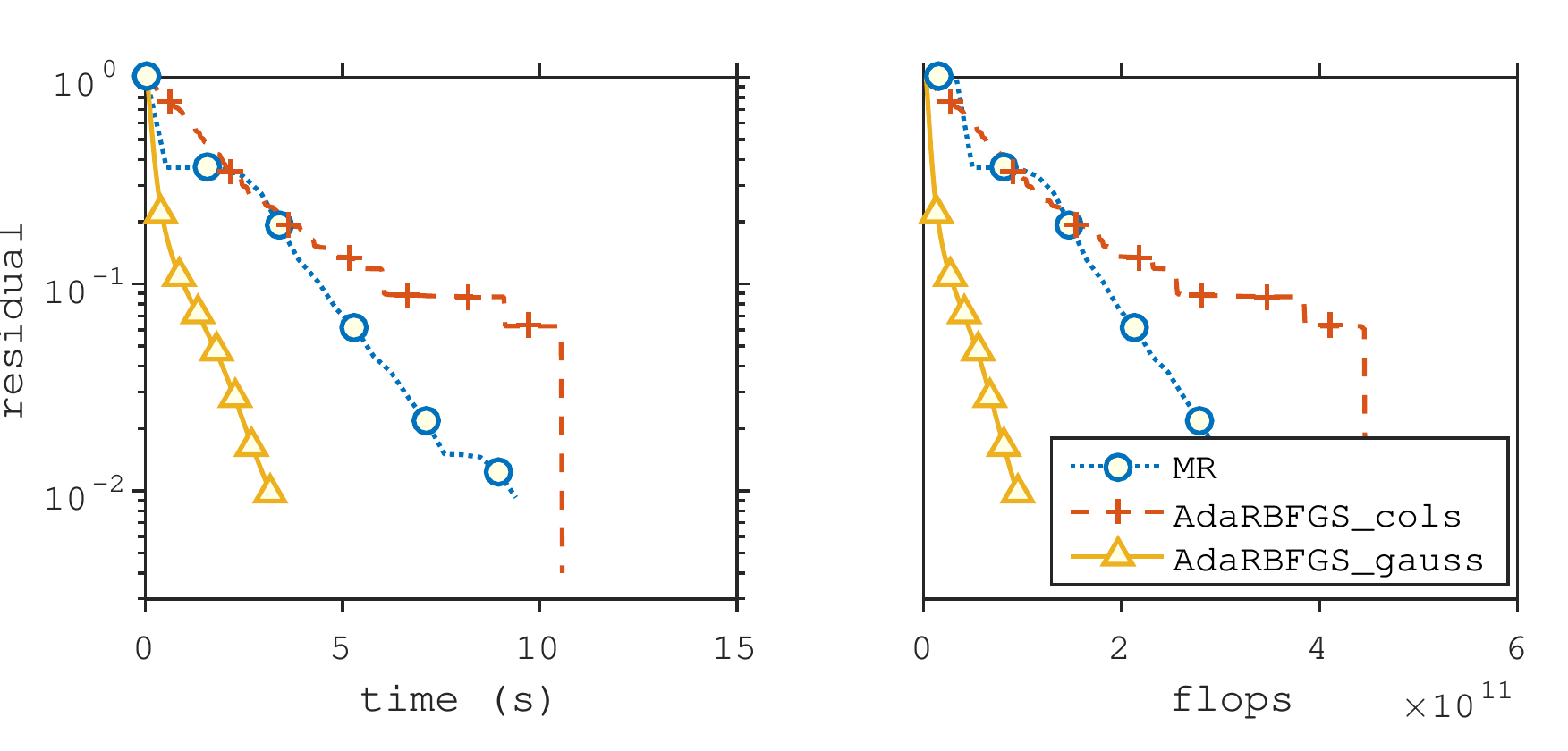}
        \caption{\texttt{Bates/Chem97ZtZ}}
\end{subfigure}  \\ %\hspace{0.2\textwidth}    
\begin{subfigure}[t]{0.65\textwidth}
        \centering
\includegraphics[width =  \textwidth, trim= 40 300 50 300, clip ]{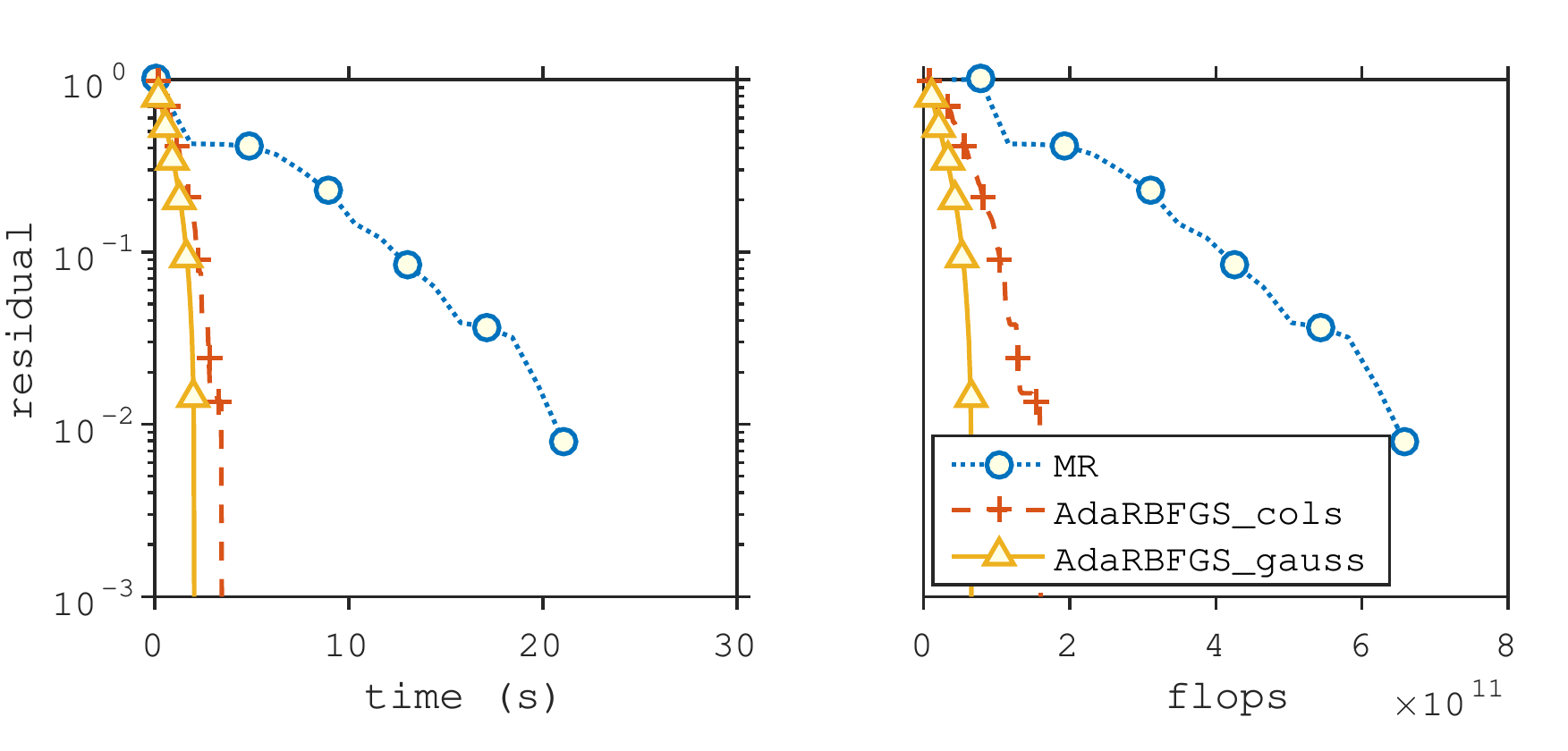}
        \caption{\texttt{FIDAP/ex9}}
\end{subfigure} \\%
%%%%%%%%%%%%%%%%%%%%%%%%%%%%%%%%%   
\begin{subfigure}[t]{0.65\textwidth}
        \centering
\includegraphics[width =  \textwidth, trim=40 300 50 300, clip ]{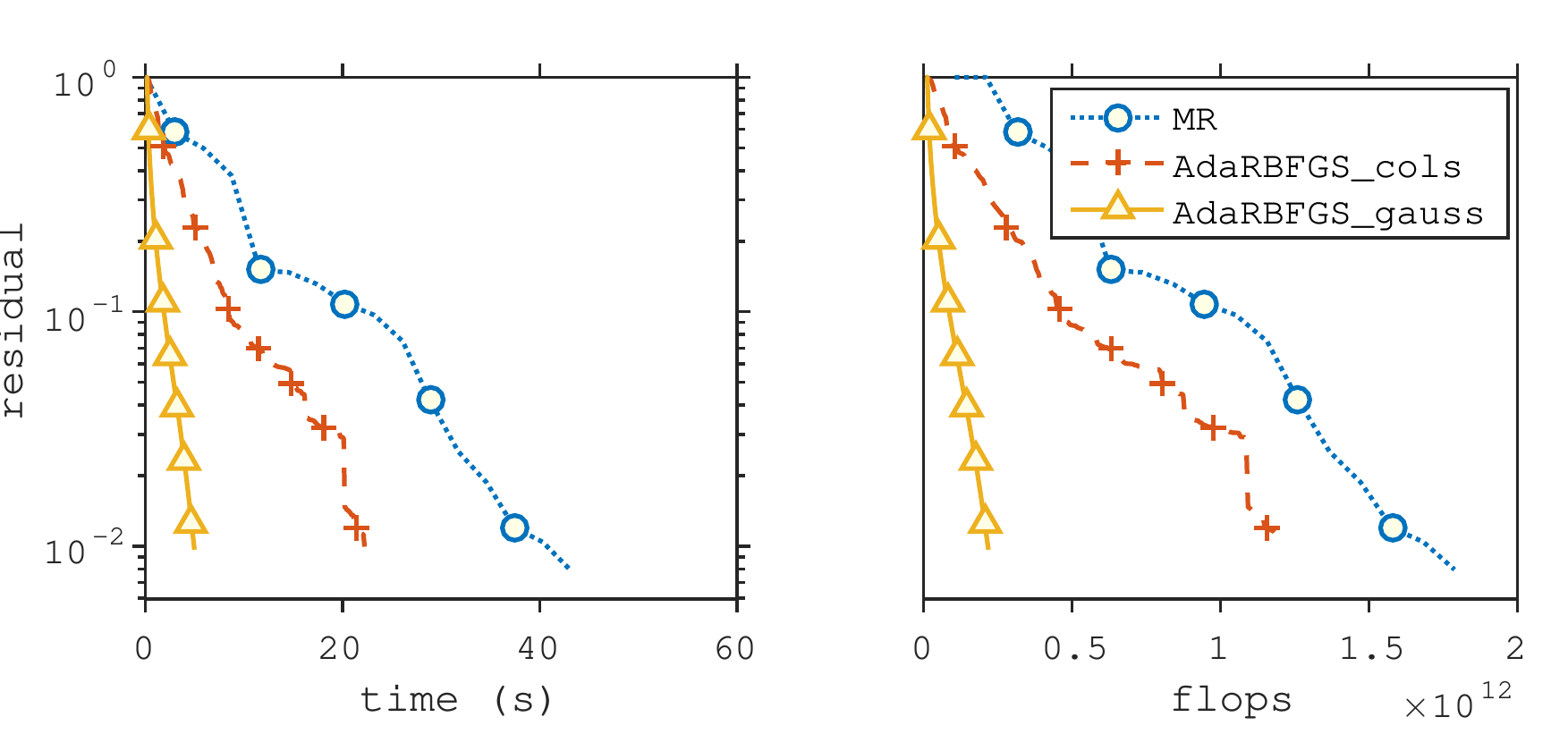}
        \caption{\texttt{Nasa/nasa4704}}
\end{subfigure}  \\%\hspace{0.2\textwidth}
\begin{subfigure}[t]{0.65\textwidth}
        \centering
\includegraphics[width =  \textwidth, trim=40 300 50 300, clip ]{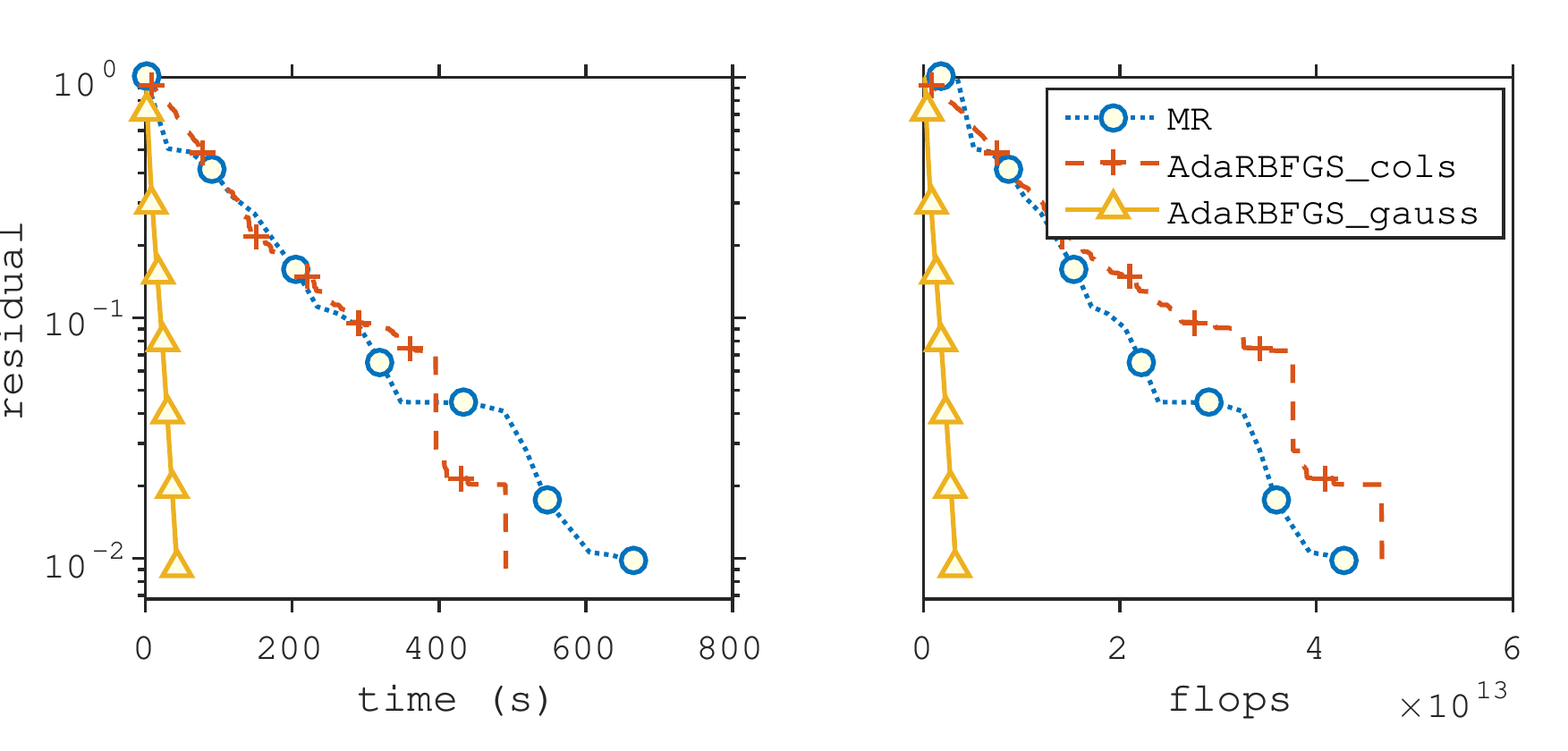}
        \caption{\texttt{HB/bcsstk18}}
\end{subfigure}%
%%%%%%%%%%%%%%%%%%%%%%%%%%%%%%%%%   
  \caption{The performance of Newton-Schulz, MR, AdaRBFGS\_gauss and AdaRBFGS\_cols on 
(a) \texttt{Bates-Chem97ZtZ}: $n= 2\,541$,   
  (b) \texttt{FIDAP/ex9}: $n = 3,\,363 $, (c) \texttt{Nasa/nasa4704}: $n= 4\,,704$, 
  (d) \texttt{HB/bcsstk18}: $n=11,\,948$. The starting matrix $X_0=I$ was used for all methods. } \label{fig:UFapp1}
  \end{figure}
  
\begin{figure}
    \centering
%%%%%%%%%%%%%%%%%%%%%%%%%%%%%%%%%   
\begin{subfigure}[t]{0.65\textwidth}
        \centering
\includegraphics[width =  \textwidth, trim=40 300 50 300, clip ]{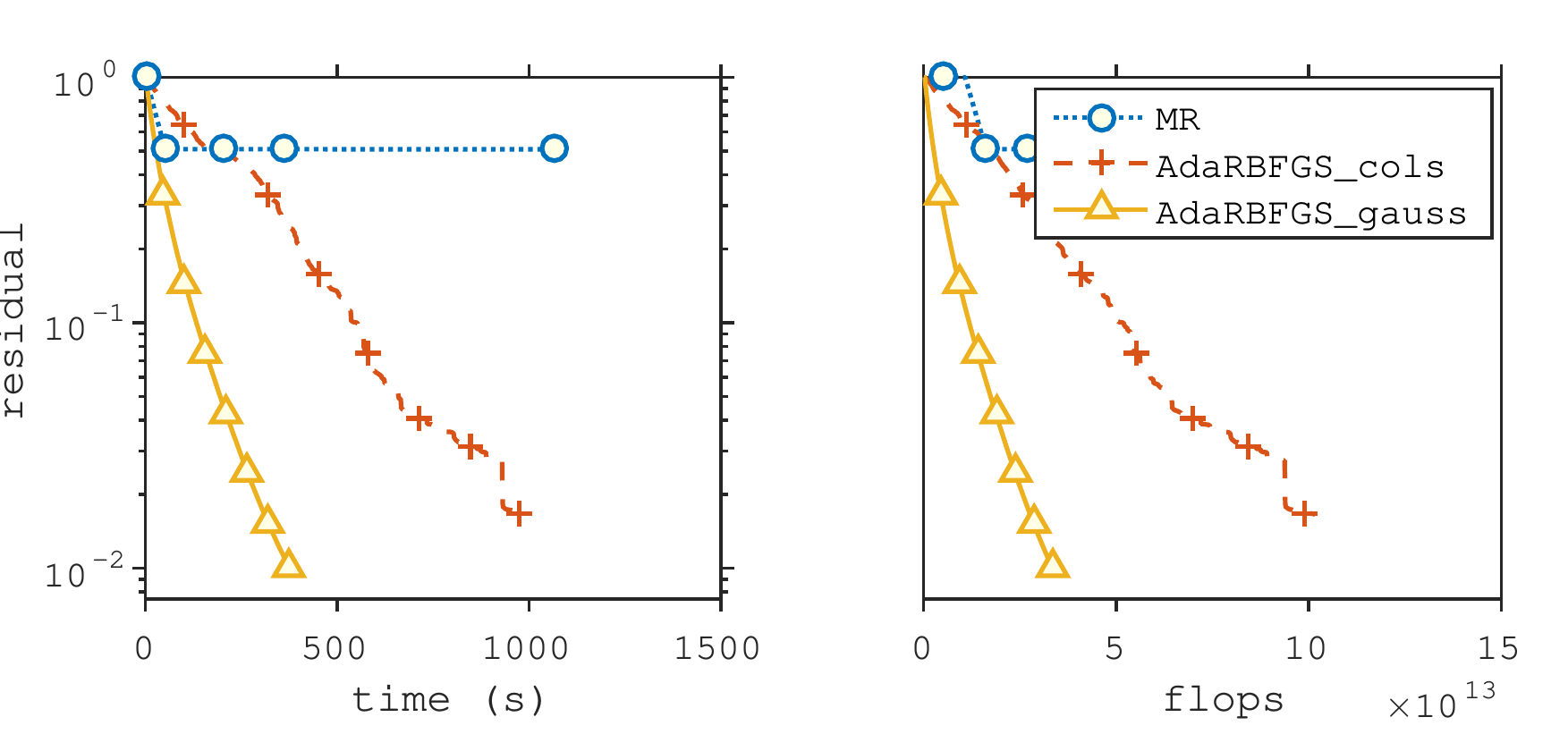}
        \caption{\texttt{Pothen/bodyy4}}
\end{subfigure} \\ %\hspace{0.2\textwidth}
\begin{subfigure}[t]{0.65\textwidth}
        \centering
\includegraphics[width =  \textwidth, trim=40 300 50 300, clip ]{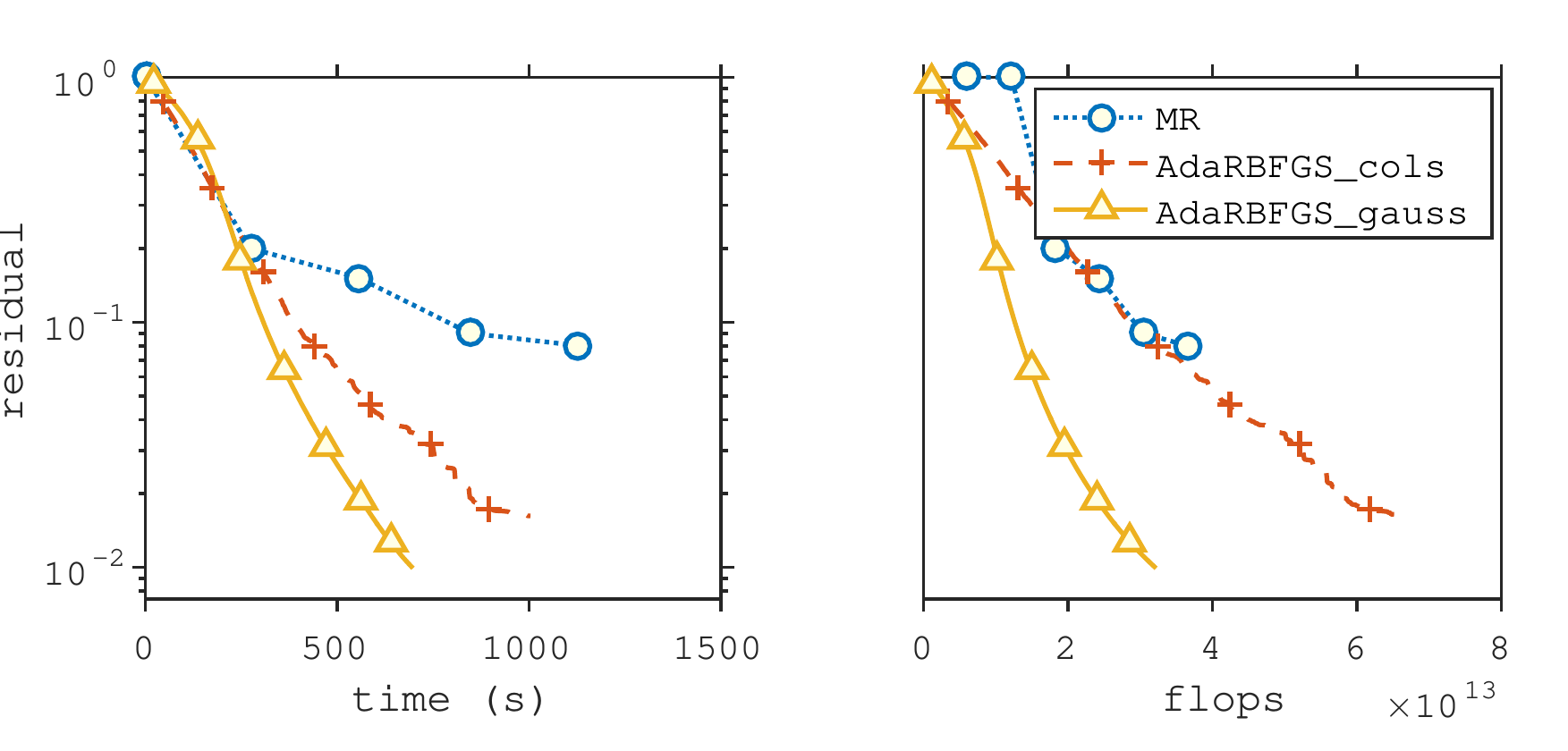}
        \caption{\texttt{ND/nd6k}}
\end{subfigure} \\%
\begin{subfigure}[t]{0.65\textwidth}
        \centering
\includegraphics[width =  \textwidth, trim=40 300 50 300, clip ]{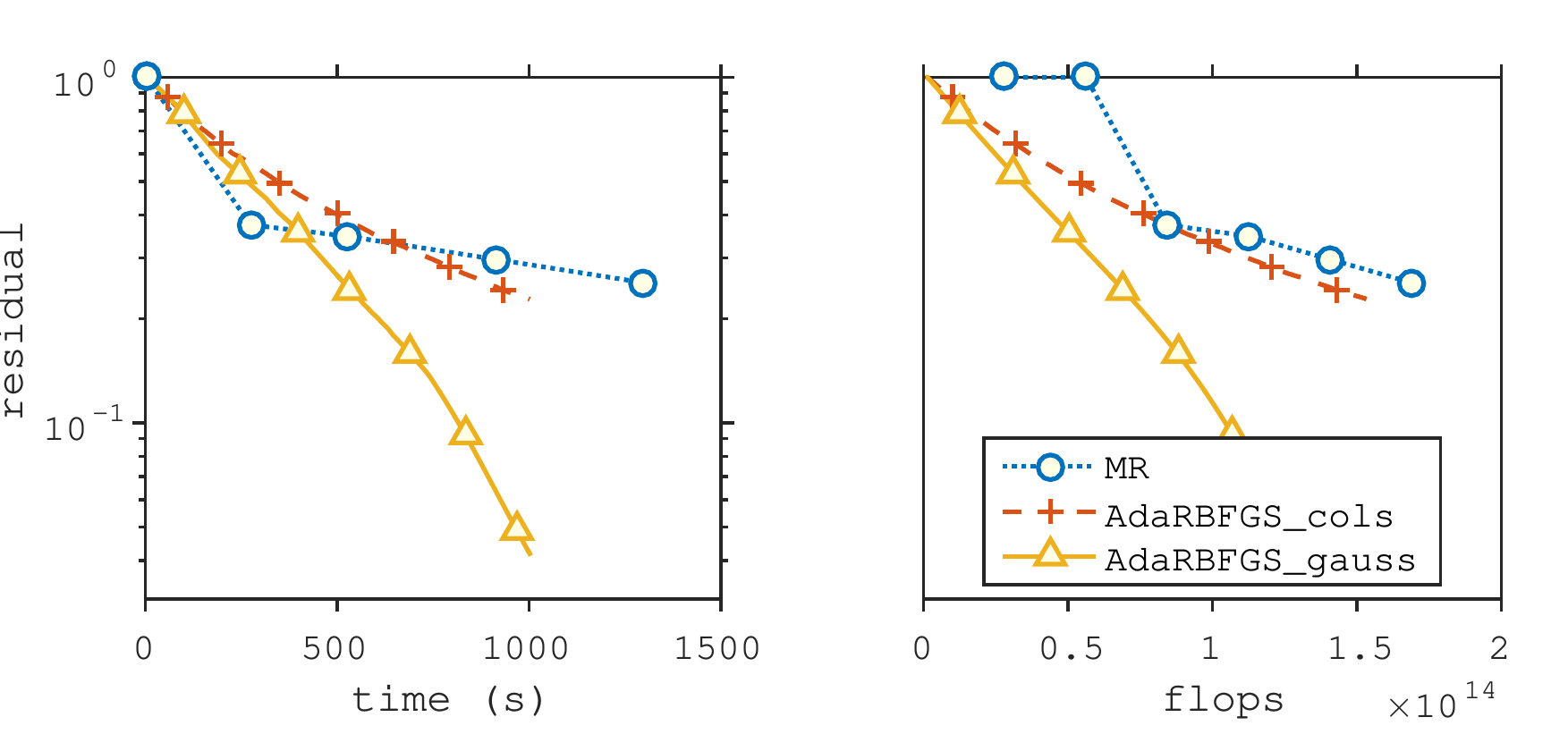}
        \caption{\texttt{GHS\_psdef/wathen100}}
\end{subfigure}%
  \caption{The performance of Newton-Schulz, MR, AdaRBFGS\_gauss and AdaRBFGS\_cols on   (a) \texttt{Pothen/bodyy4}: $n =17,\,546$ (b) \texttt{ND/nd6k}: $n=18,\,000$ (c) \texttt{GHS\_psdef/wathen100}: $n = 30, \,401$. The starting matrix $X_0=I$ was used for all methods. } \label{fig:UFapp2}
  \end{figure}

%\section{Acknowledgements}
%
%The second author would like to thank Julien Hendrickx from Universit\'{e} catholique de Louvain for a discussion regarding randomized gossip algorithms.

  % Randomized Matrix Inversion
%
%---------------------------------------------------------------------------------
%	CHAPTER four: Randomized Matrix Inversion
%---------------------------------------------------------------------------------
%\onehalfspace   %% official UoE spacing
\chapter{Conclusion and Future Work}
\chaptermark{Conclusion and Future Work}
\label{ch:Conclusion}
{
\epigraph{\emph{Fuils an bairns soud never see things hauf duin.} \\
 %Translation:
 It needs powers of perception and mature judgement to visualise the result of an enterprise.
}{Scottish proverb.} 

%It's not whether you win or lose that counts. In fact, nothing counts, and death is coming for us all.}{Jimmy Kimmel, supposedly an old Norwegian saying}
\let\clearpage\relax
}

This thesis laid the foundational work for a class of randomized methods for solving linear systems, equipt with convergence analysis, and general enough to accommodate several existing methods (Randomized Kaczmarz and Coordinate descent), but also allows for the design of 
%creates the potential for designing
 completely  new methods with continuous samplings, optimized samplings and block variants. Thus there is still much to explore in designing new  randomized methods for solving linear systems. One particularly promising direction is to use new sophisticated sketching matrices $S$, such as the Walsh-Hadamard matrix~\cite{Lu2013,Pilanci2014}, to design new practical and efficient methods.

Using duality theory, we redeveloped the sketch-and-project method through a dual perspective which we refer to as the Stochastic Dual Ascent (SDA) method. This perspective allowed us to extend the application of the sketch-and-project methods to that of finding the projection of a given vector onto the solution space of a linear system. 
This extensions enables us to solve the distributed consensus problem and it reveals that a standard randomized gossip algorithm is a special case of the sketch-and-project methods.

 Applying the same sketch and project principle to the inverse equations ($AX=I$ or $XA=I$), we developed new randomized methods for inverting nonsymmetric and symmetric matrices. These randomized methods can be viewed as randomized quasi-Newton updates. Through a dual perspective, we establish a new connection  connection between the quasi-Newton methods and the approximate inverse preconditioning methods. Furthermore, we design a highly efficient method, AdaRBFGS, for calculating approximate inverses of positive definite matrices. This opens up many avenues for developing stochastic preconditioning and  variable metric methods. Indeed, the AdaRBFGS method has already been used as the basis of a new stochastic variable metric method~\cite{GowerGold2016}.

Perhaps the most exciting direction for future work is to extend the 
sketch-and-project framework to solve linear equations in other settings. For instance, our framework could be extended to solve linear equations in a general Euclidean place. This would open up several new application areas including linear matrix equations such Sylvester equation, Lyapunov equation and more~\cite{Simoncini2014}. Finally, perhaps the 
sketch-and-project framework can be extended to solving linear equations defined by bounded linear operators between two Hilbert spaces. Such a development would lead to new randomized methods for solving  differential and integral equations.
% and results in new methods for large-scale linear differential equations and linear matrix equations for which applications abound.

%Conclusion and Future Work
%
%(a) Areas for future research. (b) Overall assessment of the methods, particularly wrt applications. (c) Discussion of choice of norm. (d) Accuracy demands: these methods designed for low accuracy; are they appropriate for, and can they deliver, medium or high accuracy (see the "struggle to bring" comment on page 60)? 

  % Randomized Matrix Inversion
%----------------------------------------------------------------------- % %
%
% 	END OF CHAPTERS
%
%----------------------------------------------------------------------- % %
\titleformat{\chapter}[hang]{\normalfont\Huge}{\thechapter}{0pt}{}  % Formats the Chapter headings of appendices and References to look like in the frontmatter
\titlespacing*{\chapter}{0cm}{10pt}{20pt}[0pt]
%----------------------------------------------------------------------- % %
%
% 	REFERENCES
%
%----------------------------------------------------------------------- % %
%\renewcommand*{\bibfont}{\footnotesize}									% Small font in bibilography. Saves paper and noone reads it anyways.
\addcontentsline{toc}{chapter}{References}
\baselineskip=10pt  																					% Distance between lines
\printbibliography
%\printbibliography[title={References}, sorting=nty]									% Sorted by name, title, year
%----------------------------------------------------------------------- % %
%
%% 	APPENDIX
%%
%----------------------------------------------------------------------- % %
%\begin{appendices}
%\renewcommand{\chaptermark}[1]{\markboth{\chaptername\ \thechapter.\ #1}{}}	  % To get the Header as I want
%\fancyhead[LE,RO]{ \leftmark}	 % Same
%\renewcommand\chaptername{Appendix} % Renaming Chapter to Appendix
%\input{header/appendix_template}
%\end{appendices}
\end{document}